\documentclass[10pt,oneside]{report}
\usepackage{amsmath}
\usepackage{amsfonts}
\usepackage{amssymb}
\usepackage{amsthm}
\usepackage{mathrsfs}
\usepackage{enumerate}
\usepackage{epstopdf}
\usepackage{nicefrac}
\usepackage{graphicx}
\usepackage{MnSymbol,faktor, systeme, bbm}
\usepackage{amscd}
\usepackage{cancel}
\usepackage{array}
\usepackage[all]{xy}
\usepackage{epsfig}
\usepackage{graphicx}
\usepackage{color}
\usepackage{lscape}
\usepackage{subfig}
\usepackage{multirow}
\usepackage{hhline}
\usepackage{longtable}
\usepackage{pbox}
\usepackage[top=1in, bottom=1in, left=1.5in, right=1in]{geometry}
\usepackage{fancyhdr}
\usepackage{floatpag}
\usepackage{euscript, calligra}
\usepackage[T1]{fontenc}
\usepackage{titlesec}
\usepackage[subfigure]{tocloft}
\usepackage{hyperref}
\usepackage{caption}
\usepackage[toc,page]{appendix}

\allowdisplaybreaks

\titleformat{\chapter}[block]{\bfseries\centering\Large}{\chaptertitlename\space\thechapter}{15pt}{}
\titleformat{\section}[hang]{\bfseries\large}{\thesection}{10pt}{}

\floatpagestyle{fancy}% Page style for float-page only

\newcounter{thecounter}
\numberwithin{thecounter}{chapter}
\newtheorem{lemma}[thecounter]{Lemma}
\newtheorem{proposition}[thecounter]{Proposition}
\newtheorem{theorem}[thecounter]{Theorem}
\newtheorem{corollary}[thecounter]{Corollary}
\newtheorem{definition}[thecounter]{Definition}
\newtheorem{conjecture}[thecounter]{Conjecture}
\newtheorem{rem}[thecounter]{Remark}

\newcommand{\ds}{\displaystyle}

\newcommand{\cB}{{\mathcal B}}

\newcommand{\cH}{{\mathcal H}}
\newcommand{\cI}{{\mathcal I}}
\newcommand{\cJ}{{\mathcal J}}

\newcommand{\LL}{{\mathcal L}}
\newcommand{\cM}{{\mathcal M}}

\newcommand{\cO}{{\mathcal O}}

\newcommand{\cS}{{\mathcal S}}

\newcommand{\cT}{{\mathcal T}}
\newcommand{\cU}{{\mathcal U}}

\newcommand{\ovP}{{\overline{P}_1}}
\newcommand{\vac}{{\mathbbm{1}}}
\newcommand{\ox}{{\otimes}}
\newcommand{\vect}[1]{\boldsymbol{#1}}
\newcommand{\ovect}[1]{\overline{\boldsymbol{#1}}}
\newcommand{\bgd}{\begin{displaymath}}
\newcommand{\edd}{\end{displaymath}}

\newcommand{\dnote}[1]{{\bf [Diego: {#1}]}}

\newcommand{\N}{{\mathbb{N}}}
\newcommand{\Z}{{\mathbb{Z}}}
\newcommand{\R}{{\mathbb{R}}}
\newcommand{\C}{{\mathbb{C}}}
\newcommand{\F}{{\mathbb{F}}}

\newcommand{\FF}{{\mathcal{F}}}
\newcommand{\Aff}{{\mathcal{A}\mathit{ff}}}
\newcommand{\Fib}{{\mathcal{F}ib}}
\newcommand{\HH}{{\mathfrak{h}}}
\newcommand{\la}{{\langle}}
\newcommand{\ra}{{\rangle}}

\def\bm{\bmatrix}
\def\ebm{\endbmatrix}

\begin{document}

% makes the page numbers roman numerals, doesn't count
% these pages in the table of contents
%\frontmatter

%\vbox to 1truein{}
\chapter*{}
\vspace{-40pt}
{\openup 1em
\centerline{DECOMPOSITION OF THE RANK 3 KAC-MOODY LIE ALGEBRA $\FF$}
\centerline{WITH RESPECT TO}
\centerline{THE RANK 2 HYPERBOLIC SUBALGEBRA $\Fib$}}
\vskip 170pt

\thispagestyle{empty}

\centerline{BY}
\vskip 10pt

\centerline{DIEGO A. PENTA}
\vskip 10pt

\centerline{BS, Massachusetts Institute of Technology, 1996}
\centerline{MS, Rutgers University, 2010}

\vskip 170pt

\centerline{DISSERTATION}
\vskip 7pt

\centerline{Submitted in partial fulfillment of the requirements for}
\centerline{the degree of Doctor of Philosophy in Mathematical Sciences}
\centerline{in the Graduate School of}
\centerline{Binghamton University}
\centerline{State University of New York}
\centerline{2016}

\newpage

\thispagestyle{empty}
\centerline{}
\vskip 594pt

\centerline{\copyright\ Copyright by Diego A. Penta 2016}

\vspace{10pt}
\centerline{All Rights Reserved}

\newpage

{\baselineskip = 10pt

\vbox to 2.0truein{}
\vskip 280pt

\thispagestyle{plain}

\centerline{Accepted in partial fulfillment of the requirements for}
\centerline{the degree of Doctor of Philosophy in Mathematical Sciences}
\centerline{in the Graduate School of}
\centerline{Binghamton University}
\centerline{State University of New York}
\centerline{2016}
\vskip 12pt

\centerline{Defense Date:   May 10, 2016}
\vskip 12pt

\centerline{Alex J. Feingold, Chair and Faculty Advisor}
\centerline{Department of Mathematical Sciences, Binghamton University}
\vskip 12pt

\centerline{Fernando Guzman, Member}
\centerline{Department of Mathematical Sciences, Binghamton University}
\vskip 12pt

\centerline{Leslie C. Lander, Outside Examiner}
\centerline{Department of Computer Science, Binghamton University}
\vskip 12pt

\centerline{Marcin Mazur, Member}
\centerline{Department of Mathematical Sciences, Binghamton University}
\vskip 12pt

\newpage
\fontsize{11}{20pt} \selectfont
%%\setlinespacing{2.0}

\chapter*{Abstract}
In 1983 Feingold-Frenkel studied the structure of a rank 3 hyperbolic Kac-Moody algebra $\FF$ containing the affine KM algebra $A^{(1)}_1$. In 2004 Feingold-Nicolai showed that $\FF$ contains all rank 2 hyperbolic KM algebras with symmetric Cartan matrices, $A=\bm 2 & -a \\ -a & 2 \ebm, a\geq 3$. The case when $a=3$ is called $\Fib$ because of its connection with the Fibonacci numbers (Feingold 1980). Some important structural results about $\FF$ come from the decomposition with respect to its affine subalgebra $A^{(1)}_1$. Here we study the decomposition of $\FF$ with respect to its subalgebra $\Fib$. We find that $\FF$ has a grading by $\Fib$-level, and prove that each graded piece, $\Fib(m)$ for $m\in\Z$, is an integrable $\Fib$-module. We show that for $|m|>2$, $\Fib(m)$ completely reduces as a direct sum of highest- and lowest-weight modules, and for $|m|\leq 2$, $\Fib(m)$ contains one irreducible non-standard quotient module $V^{\Lambda_m}=V(m)\Big/U(m)$. We then prove that the quotient $\Fib(m)\Big/V(m)$ completely reduces as a direct sum of one trivial module (on level 0), and standard modules. We give an algorithm for determining the inner multiplicities of any irreducible $\Fib$-module, in particular the non-standard modules on levels $|m|\leq 2$. We show that multiplicities of non-standard modules on levels $|m|=1,2$ do not follow the Kac-Peterson recursion (as does the non-standard adjoint representation on level 0), but instead appear to follow a recursion similar to Racah-Speiser, the recursion associated to standard modules. We also give an algorithm for finding outer multiplicities in the decompositions of all levels. We then use results of Borcherds and Frenkel-Lepowsky-Meurman and construct vertex algebras from the root lattices of $\FF$ and $\Fib$, and study the decomposition within this setting. We find a representation $\pi_\FF$ of $\FF$ in $\ovP$, a quotient of physical-1 space by a suitable subalgebra prescribed by Borcherds. We then define an action of $\Fib$ on $\ovP$ and we find extremal vectors for $\Fib$ in $\ovP$ which are not in $\pi_\FF(\FF)$. We conjecture the existence of a ``recognition algorithm'' on the Schur polynomials of vectors in $\ovP$ which will allow one to determine which vectors in $\ovP$ are extremal with respect to $\Fib$.

\fontsize{11}{20pt} \selectfont
%%\setlinespacing{2.0}
\chapter*{Acknowledgments}
First of all I am infinitely grateful to my PhD advisor Professor Alex Feingold for his constant support and guidance throughout the duration of this project. I thank him for being so generous with his time, and for teaching me about Kac Moody algebras and vertex operator theory. His dedication to mathematics, his love of teaching, and his devotion to his sculptures, his ice skating, and his family always inspired me to achieve. The knowledge I gained from his expertise in math and in life is immeasurable. 

I am also grateful to Dr. Axel Kleinschmidt of the Max Planck Institute for Gravitational Physics for creating in Mathematica the beautiful three-dimensional images of the root system for $\FF$ and its intersections by $\Fib$-planes. His programs allowed me to explore the $\Fib$-levels in great detail and gave me a deeper understanding and appreciation of the problem. I am grateful to him for his permission to use these images in my thesis.

I also wish to thank my committee members, Professors Fernando Guzman and Marcin Mazur of the Math Department, and Leslie Lander of the Computer Science Department.

I also owe my gratitude to my Masters advisor Lisa Carbone of Rutgers University for first introducing me to many of the topics discussed in the present work, and for having faith in me and in my ability to be a mathematician.

I thank my PhD cohorts, Quincy Loney and Christopher Mauriello, who always leant their kind support and advice. I also thank the math community and the graduate student community at large at Binghamton University, all of whose warmth, friendship, and support helped to make Binghamton feel like a second home. In particular, my good friends Alex Schaeffer, Odie Santiago, Nick Devin, Rachel Skipper, Simon Joyce, Joe Brennan, Melissa Fuentes, Mauricio Bustamante, Lorena Campuzano, and my girlfriend Gangotryi Sorcar, who was always there for me.

Lastly, I am indebted to my mother who has dedicated her life to my success, and without whose love and support I would not be here today. 

\newpage
\fontsize{11}{18pt} \selectfont
%%\setlinespacing{1.0}
\tableofcontents

\newpage
% add a new chapter without a chapter # for the references
{\listoftables}  % print list of tables
\addcontentsline{toc}{chapter}{List of Tables}

\newpage

% add a new chapter without a chapter # for the references
{\listoffigures}	% print list of figures
\addcontentsline{toc}{chapter}{List of Figures}

\newpage
% Changes page numbers to regular numbers, resets the counter
%\mainmatter

% This gives 11pt font with 20pt spacing, text from here should be double spaced
%\fontsize{11}{20pt} \selectfont
%%\setlinespacing{2.0}
% \include puts in the .tex file with the given name
% make sure that these files don't have any preamble material

% add a new chapter without a chapter # for the introduction
\pagenumbering{arabic}
\addcontentsline{toc}{chapter}{Introduction}

%--------------------------------------------------------%
%--------------------------------------------------------%
%--------------------- Chapter 0 ------------------------%
%--------------------------------------------------------%
%--------------------------------------------------------%

\chapter*{Introduction}\label{cht:intro}
\pagenumbering{arabic}
\setcounter{page}{1}
The theory of Kac-Moody Lie algebras has a rich history. Since their discovery in 1968 by Victor Kac and Robert Moody (independently), they have been shown to exhibit deep connections to a wide range of fields, from classical mathematics to theoretical physics. In particular, the affine Kac-Moody algebras were among the first examples to be studied, and the results in this field showed striking relationships to known combinatorial identities. Kac \cite{K1} and Moody \cite{M} showed that the Weyl denominator formulas for affine root systems are given by Macdonald identities. For example, the Weyl denominator formula for the affine algebra $A_1^{(1)}$, whose Cartan matrix is 
$$\ds\left(\begin{array}{cc}
	2 & -2 \\
	-2 & 2 \\
		\end{array}\right),$$
is the Jacobi triple product.  Throughout the late 1970s and early 1980s, further investigation into the affine algebras resulted in their explicit realization and showed more connections to power series identities. In particular, Feingold and Lepowsky \cite{FL} showed that the values of the classical partition function are exactly the weight multiplicities in the fundamental modules for two affine algebras, one of which is $A_1^{(1)}$. Results such as these helped to make the class of affine algebras better understood than those of more general Kac-Moody algebras.

A natural question that arose is whether or not similar identities exist for Kac-Moody algebras of indefinite type, the easiest examples of which (in the sense of complexity) are the hyperbolic Kac-Moody algebras. Since the late 1970s a substantial amount of research has been devoted to this question, and as of yet no satisfactory closed formula for root multiplicities for the hyperbolic algebras has been found. 

In 1983, the rank 3 hyperbolic extension of $A_1^{(1)}$, referred to by $HA_1^{(1)}$ and $\FF$ in the literature and whose Cartan matrix is
$$\ds\left(\begin{array}{ccc}
	2 & -2 & 0 \\
	-2 & 2 & -1\\
	0 & -1 & 2\\
		\end{array}\right),$$
was studied by Feingold and Frenkel \cite{FF}. The authors constructed a subalgebra $\Aff\subset\FF$ isomorphic to $A_1^{(1)}$, and showed that $\FF$ can be expressed as a direct sum of $\Aff$-modules, graded by level:
$$\FF=\bigoplus_{m\in\Z} \Aff(m).$$
Decomposing $\FF$ in this way reduced the problem to finding weight multiplicity formulas for each level with respect to $\Aff$, which for $|m|\geq 1$ decomposed as a direct sum of irreducible integrable highest or lowest-weight (i.e., standard) $\Aff$-modules. For example, Kac had already shown that the multiplicities of roots on level 0 of $\FF$ with respect to $\Aff$ are all 1, and the result of \cite{FL} gave the multiplicities for the basic representation in level 1.  In \cite{FF} the authors gave formulas for level 2 which involve a `modified' partition function. Subsequent works by Kang (e.g., \cite{Ka1}, \cite{Ka2}) gave root multiplicities up to level 5.

In 2004, Feingold and Nicolai \cite{FN} proved that every symmetric rank 2 hyperbolic KM Lie algebra $\mathcal{H}(n)$, whose Cartan matrices are
$$\ds\left(\begin{array}{cc} 2&-n\\-n&2\end{array}\right), \ \  n\geq 3,$$
 is contained in $\FF$. The simplest such algebra is $\mathcal{H}(3)$, referred to as the `Fibonacci' algebra because of its relationship to the Fibonacci numbers discovered by Feingold in 1980 \cite{F1}. The purpose of this paper is to take an alternative but similar approach to \cite{FF}, by constructing a subalgebra of $\Fib\subset\FF$ isomorphic to $\mathcal{H}(3)$, and attempting to find the decomposition of $\FF$ with respect to $\Fib$,
 $$\FF=\bigoplus_{m\in\Z} \Fib(m).$$ 

 %We wish to determine structural results for $\Fib(m)$ for all $m\in\Z$, in particular we wish to determine outer multiplicities of irreducible modules on each level. In particular, we ask whether a decomposition of $\FF$ into a direct sum of $\Fib$-modules for will shed new light on the root multiplicities of $\FF$. 

\thispagestyle{plain}
In Chapter \ref{ch:bg} we give background on $\FF$ and its rank 2 subalgebras, and review the Feingold-Frenkel decomposition with respect to $\Aff$. We also find an interesting aspect of the decomposition with respect to $\Fib$ which will be investigated in more detail in Chapter \ref{ch:nonstd}, namely the emergence of non-standard $\Fib$-modules, that is, modules which are neither highest- nor lowest-weight. 

In Chapter \ref{ch:fib} we construct a $\Fib$ subalgebra in $\FF$ and set up notation. In Chapter \ref{ch:modules} we define and study the $\Fib$-modules $\Fib(m)$, $m\in\Z$, in the $\Z$-grading of $\FF$, and discuss their properties, symmetries, and weight diagrams. We analyze the decomposition of each $\Fib(m)$ into irreducible $\Fib$-modules. In particular, we will show that for $|m|>2$, $\Fib(m)$ completely reduces, and for $|m|\leq 2$, $\Fib(m)$ contains only one irreducible non-standard quotient module, one trivial module (on level 0), and standard modules.

In Chapter \ref{ch:decomp0}, we begin investigating the decomposition of $\FF$ with respect to $\Fib$ by finding $\Fib$-modules on level 0. As in \cite{FF} the adjoint representation is contained in level 0, but unlike the affine case, additional modules are found, including a trivial representation for $\Fib$ and multiple copies of highest- and lowest-weight modules generated by imaginary root vectors whose roots are in the fundamental chamber for the Weyl group of $\Fib$. Using data from Chapter 11 of \cite{K2} and the Racah-Speiser algorithm for determining inner multiplicities of highest/lowest-weight irreducible modules, we determine all summands of the level 0 decomposition involving roots of height up to $12$ (with respect to the root system of $\Fib$). %We then look for a method to locate all extremal vectors for $\Fib$ inside $\FF$. 

In Chapter \ref{ch:nonstd} we investigate $\Fib$-levels $0<|m|\leq2$. As a formula for determining inner multiplicities of non-standard modules is not yet known, we present an algorithm for computing these multiplicities involving finding bases of multibrackets for their weight spaces. Data for these computations are presented in tables in Appendix \ref{appendix:C}. It is known that the root multiplicities of the adjoint representation of $\Fib$ (which occurs as a non-standard module on level 0) obey the Kac-Peterson recursion (cf. exercises 11.11-11.12, \cite{K2}), so it would seem reasonable to expect that the other four non-standard modules also obey a Kac-Peterson recursion. However, we will show that their multiplicities obey a recursion that appears to be of Racah-Speiser type, which indicates they have more in common with highest- and lowest-weight modules than with the adjoint.

In Chapter \ref{ch:vertex}, we employ methods in \cite{B} and \cite{FLM} to construct a vertex algebra $V_\FF$ from the root lattice of $\FF$, and a Lie algebra representation $\pi_\FF: \FF \rightarrow \ovP$ where $\ovP$ is a quotient of physical-1 space $P_1$ of $V_\FF$ by a suitable subalgebra, as prescribed by Borcherds. The restriction $\pi_\FF|_\Fib$ then gives a representation of $\Fib$ which is compatible with the similar construction of $V_\Fib$ from the root lattice of $\Fib$. We then investigate the decomposition of $\Fib(0)$ with respect to $\Fib$ within this setting by finding extremal vectors for $\Fib$ in one of the weight spaces of $\ovP$. Although this approach will show to be more computationally intensive than that of Chapters \ref{ch:decomp0} and \ref{ch:nonstd}, an interesting result will lead us to conjecture the existence of a ``recognition algorithm'' on the Schur polynomials in a given weight space of $\ovP$ that may give extremal vectors for $\Fib$. 
 \thispagestyle{plain}
%Also, we conjecture that the type III modules decompose as an alternating sum of nonstandard Verma modules defined by a choice of nonstandard Borel subalgebra. 

%--------------------------------------------------------%
%--------------------------------------------------------%
%--------------------- Chapter 1 ------------------------%
%--------------------------------------------------------%
%--------------------------------------------------------%
\chapter{Background}\label{ch:bg}
%--------------------------------------------------------%
%---------------------Section 1.1------------------------%
%--------------------------------------------------------%

\section{Lie algebras}\label{sec:finite} \
\quad We start with some basic definitions and results from the theory of finite-dimensional Lie algebras (see Humphreys \cite{H}) that will be used throughout the present work. 
\begin{definition}
A vector space $\LL$ over a field $\F$ with a bilinear operation $\LL \times \LL \rightarrow \LL$ denoted by $(x,y)\mapsto [x,y]$ (called the \textbf{bracket} of $x$ and $y$) is called a \textbf{Lie algebra} over $\F$ if the following axioms are satisfied for all $x, y, z\in \LL$:
$$[x,x]=0\qquad \text{and}\qquad[x,[y,z]] + [y,[z,x]]+[z,[x,y]]=0.$$
\end{definition}
As an example, consider the set of linear transformations $End(V)$ on a vector space $V$. We give $End(V)$ the structure of a Lie algebra by defining the bracket to be the commutator, $[x,y]=x\circ y-y\circ x$. The axioms are easily verified. This Lie algebra is called the \textit{general linear algebra} and is denoted $\mathfrak{gl}(V)$. In fact, any associative algebra may be made a Lie algebra by defining the bracket in this way.

Let $X, Y$ be subsets of a Lie algebra $\LL$. The bracket $[X,Y]$ is defined to be
$$[X,Y]=Span(\{[x,y]\mid x\in X, y \in Y \}).$$
\begin{definition}
Let $\cM$ be a vector subspace of $\LL$. Then
\begin{itemize}
\item $\cM$ is called an \textbf{ideal} of $\LL$ if $[\cM, \LL] \subseteq \cM$, and
\item $\cM$ is called a \textbf{Lie subalgebra} of $\LL$ if $[\cM, \cM] \subseteq \cM$.
\end{itemize}
\end{definition}

If $\cJ$ is an ideal of $\LL$, then the quotient vector space $\LL/\cJ$ is also a Lie algebra with bracket defined by $[\bar{x}, \bar{y}] = \overline{[x,y]}$
for all $x, y \in \LL$, where $\bar{x}=x+\cJ \in \LL/\cJ $.

\begin{definition}
The \textbf{universal enveloping algebra} $\cU(\LL)$ of a Lie algebra $\LL$ is the associative algebra $\cU(\LL) = \cT(\LL) / \cI$
where 
$\cT(\LL) = \F\oplus \LL \oplus \Big(\LL\otimes\LL\Big) \oplus \Big(\LL\otimes\LL\otimes\LL\Big) \oplus \cdots$
is the tensor algebra of the underlying vector space of $\LL$, and $\cI$ is the ideal of $\cT(\LL$) generated by elements of the form $x\otimes y - y\otimes x-[x,y]$ for $x,y\in \LL$. 
\end{definition}

\begin{definition}
Given Lie algebras $\LL$ and $\cM$ over $\F$, a linear transformation $\phi: \LL\rightarrow \cM$ is called a \textbf{Lie algebra homomorphism} if $\phi([x,y])=[\phi(x),\phi(y)]$ for all $x, y \in \LL$.
\end{definition}

\begin{definition}
A \textbf{representation} of a Lie algebra $\LL$ is a homorphism $\phi: \LL\rightarrow \mathfrak{gl}(V)$ for some vector space $V$ over $\F$. 
\end{definition}
Sometimes $V$ itself is referred to as the representation of $\LL$ in the literature.
\begin{definition}
The \textbf{adjoint representation} of a Lie algebra $\LL$ is the representation $ad: \LL\rightarrow \mathfrak{gl}(\LL)$ defined by $ad_x(y) = [x,y]$. 
\end{definition}

\begin{definition}
Given a Lie algebra $\LL$, a vector space $V$ is called an $\mathbf{\LL}$\textbf{-module} if there is a bilinar action $\LL\times V \rightarrow V$, denoted $(x,v)\mapsto x\cdot v$, such that $[x,y]\cdot v = x \cdot (y \cdot v) - y \cdot (x \cdot v)$
for $x,y\in \LL$ and $v\in V$. Also, a subspace $W$ of $V$ is called an $\mathbf{\LL}$\textbf{-submodule of} $\mathbf{V}$ if $W$ is itself an $\LL$-module, that is, $\LL\cdot W \subseteq W$.
\end{definition}
Having a representation $\phi:\LL\rightarrow \mathfrak{gl}(V)$ is equivalent to $V$ being an $\LL$-module, by
$x\cdot v = \phi(x) v.$

\begin{definition}
An $\LL$-module $V$ is called \textbf{irreducible} if its only submodules are the trivial module, $\{0\}$, and $V$ itself. 
\end{definition}

\begin{definition}
Given a collection of $\LL$-modules $\{V_i\}_{i\in I}$ where $I$ is an index set, then the \textbf{vector space direct sum} $\ds\bigoplus_{i \in I} V_i$ is also an $\LL$-module, where the action is component-wise.
\end{definition}

\begin{definition}
A module is called \textbf{completely reducible} if it is a direct sum of irreducible modules.
\end{definition}
%--------------------------------------------------------%
%---------------------Section 1.2------------------------%
%--------------------------------------------------------%
\section{Kac-Moody Lie algebras}\label{sec:kacmoody} 
Kac and Moody generalized the theory of finite-dimensional Lie algebras to include certain other classes of infinite-dimensional Lie algebras, which will be described shortly. 

\begin{definition} An $\ell \times \ell$ integer matrix $A=(a_{ij})$ is called a \textbf{generalized Cartan matrix (GCM)} if it has the the following properties:
$$a_{ii}=2 \text{ for }1\leq i \leq \ell, \quad a_{ij}\leq 0 \text{ if } i\not= j, \quad \text{ and } \ a_{ij}=0 \text{ implies }a_{ji}=0.$$

A GCM $A$ is called \textbf{symmetrizable} if there exists a diagonal matrix $D=diag(d_1,\ldots, d_\ell)$, $d_i>0$, and $A=DS$ for $S=(s_{ij})$ symmetric.

$A$ is called \textbf{indecomposable} if it cannot be put into block-diagonal form 
$\ds\left(\begin{array}{cc} A_1 & 0  \\ 0 & A_2 \\ \end{array}\right),$
where each $A_i$ is non-trivial, through a relabeling of the rows and columns.\end{definition}

A symmetrizable GCM $A=DS$ is of:
\begin{itemize}
\item \textbf{finite type} if $S$ is positive definite.
\item \textbf{affine type} if $S$ is positive semi-definite and has corank 1.
\item \textbf{indefinite type} if $S$ is indefinite. In addition, if every submatrix of $A$ is of either finite or affine type, then $A$ is of \textbf{hyperbolic type}.
\end{itemize}
From now on we assume all GCMs to be symmetrizable.

\begin{definition} Given a GCM $A$, a \textbf{realization of $A$ over $\F$} is a triple $(\HH, \Pi, \Pi^\vee)$ where $\HH$ is a complex vector space, $\Pi=\{\alpha_1, \ldots, \alpha_\ell\}\subset\HH^*$ (the set of \textbf{simple roots}), and $\Pi^\vee=\{\alpha_1^\vee, \ldots \alpha_\ell^\vee\}\subset\HH$ (the set of \textbf{simple coroots}), satisfying the conditions
\begin{itemize}
\item $\Pi$ and $\Pi^\vee$ are linearly independent,
\item $\alpha_i(\alpha_j^\vee)=a_{ji}$ for $1\leq i,j \leq \ell$,
\item $\ell \ - $ \emph{rank}$(A) = \dim \HH - \ell.$
\end{itemize}
\end{definition}

\begin{definition}\label{def:kml} Let $A$ be an $\ell \times \ell$ GCM with realization $(\HH, \Pi, \Pi^\vee)$. The \textbf{Kac-Moody Lie algebra}, or \textbf{KM algebra}, $\LL=\LL(A)$ over $\F$ associated to $A$ is the Lie algebra generated by $\{e_i,f_i \ | \  i=1\ldots \ell\}$ (called the \textbf{Chevalley generators}) and $\HH$, subject to the \textbf{Serre relations}, 
\begin{itemize}
  \item[-] $[h, h']=0$ for all $h,h'\in\HH$,
  \item[-] $[h,e_i]=\alpha_i(h) e_i$ \ and \ $[h,f_i]=-\alpha_i(h)f_i$, \ for all $h\in\HH^*$,
  \item[-] $[e_i,f_j]=\delta_{ij}\alpha_i^\vee, $
  \item[-] $(ad_{e_i})^{-a_{ij}+1}(e_j)=0,\ i\neq j\qquad\textrm{and}\qquad (ad_{f_i})^{-a_{ij}+1}(f_j)=0,\ i\neq j$.
\end{itemize}
The abelian subalgebra $\HH$ of $\LL$ is called the \textbf{Cartan subalgebra}. 
\end{definition}

Under the adjoint action, $ad_h: \LL\to\LL$ given by $ad_h(x)=[h,x]$, $\HH$ acts simultaneously diagonalizably on $\LL$.  The simultaneous eigenspaces 
$$\LL_{\alpha}\ =\ \{x\in \LL\mid[h,x]=\alpha(h)x,\ h\in \HH\}$$  
are labelled by certain linear functionals $\alpha\in\HH^*$. In particular, for $1\leq i \leq \ell$,
$$\LL_0=\HH, \quad\LL_{\alpha_i}=\F e_i\quad\hbox{ and }\quad\LL_{-\alpha_i}=\F f_i.$$
\begin{definition} We call $\alpha \in \HH^*$ a \textbf{root} when $\alpha\not= 0$ and $\LL_{\alpha}\not=0$, in which case $\LL_\alpha$ is called the $\mathbf{\alpha}$ \textbf{root space}, and the \textbf{multiplicity} of root $\alpha$ is
$$Mult_\LL(\alpha) =\dim (\LL_\alpha).$$
The \textbf{root lattice} is
$$Q = \sum_{i=1}^\ell \Z\alpha_i \subset \HH^*.$$
The set of all roots is denoted by $\Delta$ (so $\Delta\subset Q)$. 
\end{definition}
For roots $\alpha, \beta \in \Delta \cup \{0\}$, we have that $[\LL_\alpha, \LL_\beta] \subseteq \LL_{\alpha+\beta}$.

\begin{definition}
The \textbf{Cartan involution} of $\LL$ is an order-two automorphism $\nu:\LL\rightarrow\LL$ determined by $\nu(e_i)=-f_i, \ \nu(f_i)=-e_i$ for $1\leq i \leq \ell$, and $\nu(h)=-h$ for $h\in \HH$.
\end{definition}
For every $\alpha\in\Delta$, we have that $\nu(\LL_\alpha)= \LL_{-\alpha}$, thus $Mult(\alpha)=Mult(-\alpha)$. We also let $\nu(\alpha)=-\alpha$.

The set of roots $\Delta$ can be partitioned into  ``positive'' and ``negative'' roots,
$$\Delta = \Delta_+ \cup \Delta_-,$$
where
$$\Delta_+ = \Big\{  \sum^\ell_{i=1}n_i \alpha_i \in \Delta \mid 0\leq n_i\in\Z \Big\}$$
and $\Delta_-=-\Delta_+$.
There is a partial order on $\Delta$ defined by
\begin{equation}\label{eq:partial}\mu\leq\lambda \text{ if and only if } \lambda-\mu=\sum_{i=1}^\ell k_i\alpha_i \text{ where } k_i\in\Z_{\geq0}.\end{equation}
\begin{definition}
The \textbf{height} of a root $\alpha = \ds\sum_{i=1}^\ell n_i\alpha_i \in \Delta$ is $ht(\alpha)=\ds \sum_{i=1}^\ell n_i.$
\end{definition} 
Let $\mathfrak n_+$ (resp. $\mathfrak n_-$) be the subalgebra generated by  $\{e_i \mid 1 \leq i \leq \ell\}$ (resp. $\{f_i \ | \ 1 \leq i \leq \ell \}$). Then 
$$\mathfrak n_\pm = \bigoplus_{\alpha\in \Delta_+}\LL_{\pm\alpha}$$
and by Theorem 1.2 of \cite{K2}, $\LL$ has the triangular decomposition and root space decomposition,
$$\LL = \mathfrak n_- \oplus \HH \oplus \mathfrak n_+ \quad \text{and} \quad  \LL = \HH \oplus \bigoplus_{\alpha\in \Delta} \LL_{\alpha},$$
respectively.

We also have the following fundamental results of Kac. 
\begin{theorem}[\cite{K2}]
Let $\LL=\LL(A)$ be the KM algebra associated to a symmetrizable GCM $A$. Then 
\begin{itemize}
\item there exists a non-degenerate symmetric bilinear form $(\cdot,\cdot)$ on $\LL$ that is invariant (i.e., $([x,y],z) = (x,[y,z]) \ \text{for all } x, y\in\LL)$, and%Theorem 2.2 of Kac
\item the restriction $( \cdot,\cdot) |_\HH$ is non-degenerate, giving an induced non-degenerate bilinear form $(\cdot , \cdot )$ on $\HH^*$ such that $( \alpha_i, \alpha_j ) = s_{ij}$  and $d_i=\frac{2}{(\alpha_i,\alpha_i)}$ for $1\leq i,j \leq \ell$.
\end{itemize}
\end{theorem}
From now on, $(\cdot ,\cdot)$ shall refer to the induced bilinear form on $\HH^*$.
\begin{definition} For $\alpha, \beta\in\Delta$, the \textbf{angle bracket} of $\alpha, \beta$ is given by
$$\la \alpha, \beta \ra := \frac{2( \alpha, \beta )}{( \alpha, \alpha )}.$$
\end{definition}
Note that $\la \ , \ \ra$ is linear only in the second, and that $\la \beta, \alpha \ra = \ds\frac{(\alpha, \alpha)}{(\beta, \beta)}\la \alpha, \beta \ra.$ Also, $\la \alpha_i, \alpha_j \ra = d_i s_{ij} = a_{ij}=\alpha_j(\alpha_i^\vee),$ thus
$$A=\Big( \la \alpha_i, \alpha_j  \ra \Big)_{1\leq i,j \leq \ell}.$$
The integral entries $\la \alpha_i, \alpha_j \ra$ ($1\leq i,j\leq\ell$) of the GCM are referred to as \textit{Cartan integers}. Also since the angle bracket is linear in the second, we have that $\la \alpha_i, \alpha \ra = \alpha(\alpha_i^\vee)$ for $\alpha\in\HH^*$.
\begin{definition}The \textbf{squared-length} of a root $\alpha$ is $||\alpha||^2=( \alpha,\alpha )$.
\end{definition}
For indefinite KM algebras, it is possible for roots to have negative squared-length. 

For each $1\leq i \leq \ell$, let 
$$h_i=\frac{\alpha_i^\vee}{d_i} = \frac{(\alpha_i,\alpha_i)}{2}\alpha^\vee_i\in\HH.$$

\begin{rem}\label{aij=aji}
Note that if $A$ is symmetric, then $(\alpha_i,\alpha_j)=a_{ij}$, $||\alpha_i||^2=2$, and $h_i=\alpha_i^\vee$, giving us $\alpha_i(h_j) = \alpha_i(\alpha_j^\vee) = a_{ji}=a_{ij}.$
\end{rem}

\begin{definition}\label{fund}
Given a KM algebra $\LL$ and its associated simple roots $\Pi$, the \textbf{fundamental weights} of $\LL$, $\omega_j\in\HH^*$ for $1\leq j \leq \ell$, are defined by
$$\la \alpha_i,\omega_j \ra= \omega_j(\alpha_i^\vee) = \delta_{ij}  \quad \text{or, equivalently, } \quad \la \omega_i,\alpha_j \ra= \frac{(\alpha_i,\alpha_i)}{(\omega_i,\omega_i)}\delta_{ij}.$$
The \textbf{weight lattice} $P$ of $\LL$ is the set of \textbf{integral weights},
$$P = \sum_{i=1}^\ell \Z\omega_i,$$
and the \textbf{dominant integral weights} are defined as
$$P^+ = \sum_{i=1}^\ell \Z^+\omega_i = \Big\{ \sum_{i=1}^\ell n_i\omega_i \ | \ 0 \leq n_i \in \Z\Big\} = \{\omega\in P \ | \ \la \alpha_j, \omega\ra \geq 0, 1\leq j \leq \ell\}.$$
\end{definition}
Writing a simple root $\alpha_j$ as a linear combination of the fundamental weights, 
$$\alpha_j = \sum^\ell_{k=1}c_{kj}\omega_k,\quad \Rightarrow\quad a_{ij} = \la \alpha_i, \alpha_j\ra = \sum^\ell_{k=1}c_{kj}\la \alpha_i, \omega_k\ra = \sum^\ell_{k=1}c_{kj}\delta_{ik} = c_{ij}.$$
Hence,
$$\alpha_j=\sum^\ell_{i=1}a_{ij}\omega_i.$$
Thus, the $j$-th column of the GCM is the coordinate vector of $\alpha_j$ with respect to the basis of the fundamental weights.  If $A$ is invertible, the columns of the inverse of the GCM give the coefficients of the fundamental weights in the basis of the simple roots:  
$$\omega_j=\sum^\ell_{i=1}(A^{-1})_{ij}\alpha_i.$$
Moreover, since $\la \omega_i, \alpha_j \ra = \ds\frac{(\alpha_i,\alpha_i)}{(\omega_i,\omega_i)}\delta_{ij}$ we have the following: 
$$\la \omega_i, \omega_j \ra = \la \omega_i, \sum^\ell_{i=1}(A^{-1})_{ij}\alpha_i \ra =  \frac{(\alpha_i,\alpha_i)}{(\omega_i,\omega_i)} (A^{-1})_{ij}.$$
Hence the inverse of the Cartan matrix is: 
$$\Big(A^{-1}\Big)_{1\leq i,j \leq \ell } = diag\Bigg(\frac{(\omega_i,\omega_i)}{(\alpha_i,\alpha_i)}\Bigg)_{1\leq i \leq \ell} \Big(\la \omega_i, \omega_j \ra\Big)_{1\leq i,j\leq \ell} $$

\begin{definition}
We define the function $wt: P \rightarrow \Z$ by $wt(\omega)=\ds \sum_{i=1}^\ell n_i$ if $\omega = \ds\sum_{i=1}^\ell n_i\omega_i \in P$. 
We call $\omega$ a \textbf{positive weight} if $wt(\omega)>0$ and a \textbf{negative weight} if $wt(\omega)<0$.
\end{definition}

\begin{definition}
The \textbf{Weyl group} $W=W(A)$ associated to $\LL(A)$ is the group generated by \textbf{simple reflections} $\{w_i \mid i=1\leq i \leq \ell\}$, where for all $\lambda \in\HH^*$,
$$
w_i \lambda = \lambda - \la \alpha_i, \lambda \ra  \alpha_i.$$%p 29
\end{definition}
%Modified with Alex on 2.15
Note that $w_i$ fixes the hyperplane $\alpha^\perp_i = \{ \lambda \in \HH_\R^*\ | \  (\lambda, \alpha_i)=0 \}$ pointwise and $w_i\alpha_i = - \alpha_i$. Thus, each $w_i$ is a reflection, so $W$ is a group of orthogonal transformations with respect to $(\cdot , \cdot)$. There is a group homomorphism $sgn: W\rightarrow \{\pm 1\}$ defined by $sgn(w_i)=-1$.

Furthermore, Kac gives us:
\begin{lemma}[\cite{K3}] Let $\LL$ be a KM algebra with Weyl group $W$, set of roots $\Delta$, simple roots $\Pi$, and root lattice $Q$.
\begin{enumerate}[a)]
\item The set of roots $\Delta$ is $W$-invariant, and $Mult(\alpha)=Mult(w(\alpha))$ for every $w\in W$ and $\alpha\in\Delta$. The set $\Delta_+ \backslash \{\alpha_i\}$ is invariant with respect to $w_i$.
\item The set $\Delta_+$ is uniquely defined by the following properties:
\begin{enumerate}[i)]
\item $\Pi\subset\Delta_+\subset Q; \ 2\alpha\notin\Delta_+$ if $\alpha\in\Pi,$
\item If $\alpha\in\Delta_+ \backslash\Pi$, then $\alpha+k\alpha_i\in\Delta_+$ if and only if $-p\leq k\leq q, k\in \Z,$ where $p, q$ are some non-negative integers satisfying $p-q=\la \alpha_i, \alpha\ra .$
\end{enumerate}
\item If $A$ is indecomposable of affine or hyperbolic type, then if $\alpha\in\Delta_+$, there exists $\beta\in\Pi$ such that $\alpha+\beta\in \Delta_+$.
\end{enumerate}
\end{lemma}

If $A$ is a finite-type Cartan matrix, then $W$ and $\Delta$ are both finite,  $( \cdot, \cdot )$ is positive definite, and $W\Pi=\Delta$. Also, $\Delta$ embeds into Euclidean space $\R^\ell$, and satisfies the properties of a finite root system \cite{H}. 

If $A$ is a GCM we call $\Delta$ a \textbf{generalized root system}, and it partitions into \textit{real roots} and \textit{imaginary roots},
$$\Delta=\Delta^{re}\cup\Delta^{im},\quad \text{where}\quad 
\Delta^{re}=W\Pi \quad \text{and}\quad\Delta^{im}=\Delta \backslash \Delta^{re}.$$

\begin{definition}
Let $V$ be an $\LL$-module corresponding to representation $\phi: \LL\rightarrow End(V)$. We say that $x\in \LL$ is \textbf{locally nilpotent} on $V$ if for all $v\in V$, there exists a positive integer $n$ such that $\phi(x)^n  v = 0$. We sometimes write more briefly $x^n \cdot v$ for $\phi(x)^n v$ with an abuse of notation.
\end{definition}

\begin{definition}
For $\lambda\in\HH^*$, the $\lambda$ \textbf{weight space} $V_\lambda$ of an $\LL$-module $V$ is 
$$V_\lambda=\{v\in V \mid h\cdot v = \lambda(h)v \text{ for } h\in\HH\},$$
and $\lambda$ is a \textbf{weight} of $V$ if $V_\lambda\neq0$.
\end{definition}

\begin{definition}
An $\LL$-module $V$ is called $\HH$\textbf{-diagonalizable} if $V=\ds\bigoplus_{\lambda\in\HH^*}V_\lambda$.
\end{definition}

Given an $\HH$-diagonalizable $\LL$-module $V$, let 
$$P(V)=\{\lambda\in\HH^*\mid V_\lambda\neq 0\}\quad \text{and} \quad D(\lambda)=\{\mu\in\HH^*\mid\mu\leq \lambda\}.$$

%Kac chapter 9
\begin{definition}\label{cato}
\textbf{The category} $\cO$ is the category whose objects are $\HH$-diagonalizable $\LL$-modules $V$ with finite-dimensional weight spaces and such that for each $V$, there exists a finite number of elements $\lambda_1, \ldots, \lambda_s\in\HH^*$ such that 
$$P(V)\subset \bigcup_{i=1}^s D(\lambda_i).$$
The morphisms of the category $\cO$ are $\LL$-module homomorphisms.
\end{definition}

\begin{definition}\label{hwmod}
An $\LL$-module $V$ is a \textbf{highest-weight module} with \textbf{highest weight} $\lambda\in\HH^*$ if there exists a non-zero vector $v_\lambda\in V$, called a \textbf{highest-weight vector (HWV)}, with the following properties:
$$\mathfrak n_+\cdot v_\lambda = 0, \qquad h\cdot v_\lambda = \lambda(h)v_\lambda \ \text{for all }h\in\HH^*,\qquad \text{and}\qquad \cU(\mathfrak n_-)\cdot v_\lambda = V.$$
If $V$ is irreducible, we write $V=V^\lambda$. The $\mu$ weight space of $V^\lambda$ is denoted $V^\lambda_\mu.$ In particular, note that the second property implies $v_\lambda \in V^\lambda_\lambda$. 
\end{definition}

A highest-weight $\LL$-module is in the category $\cO$.

%\begin{definition}\label{def:type}
%Given a KM algebra $\LL$, we define the following two types of irreducible $\LL$-modules:
%\begin{itemize}
%\item \textbf{Type I}: an irreducible highest-weight module generated by a vector of highest weight $\lambda$, %such that  $\lambda(h_i)\geq 0$ $(1\leq i \leq \ell)$,
%\item  \textbf{Type II}: an irreducible lowest-weight module generated by a vector of lowest-weight $\lambda$, such that  $\lambda(h_i)\leq  0$ $(1\leq i \leq \ell)$.
%\end{itemize}
%\end{definition}

\begin{definition}\label{high}
Let $V$ be a completely reducible $\LL$-module in the category $\cO$, and let $\mu\in P(V)$. Define
$$High_V(\mu) =\{ v \in V_\mu \ | \ e_i \cdot v =0, \ 1\leq i \leq \ell \}$$
to be the subspace of highest-weight vectors in $V_\mu$. Define
$$M_V(\mu)=\dim High_V(\mu)$$
to be the \textbf{outer multiplicity of} $\mu$ \textbf{in} $V$, and define
$$P'(V) = \{ \mu \in P(V) \mid M_V(\mu) >0 \}$$
to be the set of weights whose weight spaces of $V$ contain highest-weight vectors.
\end{definition}
Let $\cB(\mu)=\{ v_{\mu,i}\in V_\mu \ | \ 1 \leq i \leq M_V(\mu) \}$ be a basis for $High_V(\mu)$. Then each $v_{\mu,i}$ generates an irreducible highest-weight $\LL$-module, $\cU(\LL)\cdot v_{\mu,i} \simeq V^\mu$. 
\begin{definition}
An $\HH$-diagonalizable module $V$ over $\LL$ is \textbf{integrable} if the generators $\{e_i, f_i \mid 1\leq i \leq \ell\}$ of $\LL$ are all locally nilpotent on $V$.
\end{definition}

\begin{lemma}[\cite{K2} 10.1]\label{dominant}
Let $\LL$ be a KM algebra and $V^\lambda$ be an irreducible highest-weight $\LL$-module. Then $V^\lambda$ is integrable if and only if $\lambda\in P^+$.
\end{lemma}

By the lemma, we have for any integrable highest-weight $\LL$-module $V$,
\begin{equation}\label{eq:p'}P'(V)\subset P^+.\end{equation}

\begin{definition}\label{string}
For fixed $\alpha\in\Delta$ and $\lambda\in P$, the set $S_\alpha(\lambda)=\{\lambda-i\alpha \mid 0\leq i \leq \la \alpha,\lambda  \ra\}$ is called the \textbf{$\alpha$-weight string through $\lambda$}. A subset $R\subset P$ of integral weights is called \textbf{saturated} if for all $\lambda\in R$, $\alpha\in \Delta$, $S_\alpha(\lambda)\subset R$. 
\end{definition}
Observe that weight strings and saturated sets of weights are $W$-invariant. If $V$ is an $\LL$-module, then $P(V)$ is a saturated set since
\begin{equation}\label{eq:saturated}\ds\bigcup_{i=1}^\ell S_{\alpha_i}(\mu)\subset P(V) \ \ \ \text{for all} \ \ \mu\in P(V)\end{equation}

%later, It's High_{\Fib(0)}. The module is $\Fib(0)$, not $\Fib$.
%Note that the involution $\nu$ interchanges these two subspaces, thus they have the same dimension, which we denote by $M(\lambda).$
\begin{theorem}[\cite{K2} 10.7]\label{reducible}
Let $\LL$ be a KM algebra and let $V$ be an $\LL$-module in the category $\cO$. Then $V$ is integrable if and only if $V$ is completely reducible, that is, $V$ has a decomposition 
$$V=\bigoplus_{\lambda\in P'(V)} M_V(\lambda)V^\lambda.$$
%In particular, any integral type I or type II $\LL$-module is completely reducible.
\end{theorem}
%We refer to $M_V(\lambda)$ as the \textit{outer multiplicity} of $V^\lambda$ in the decomposition of $V$.

\begin{definition}
Let $V$ be an $\LL$-module from the category $\cO$. The \textbf{formal character of } $V$ is defined as
$$Ch(V) = \sum_{\lambda\in P(V)} \dim V_\lambda e^\lambda,$$
where $e^\lambda$ is a formal exponential satisfying $e^\lambda e^\mu = e ^{\lambda+\mu}.$
\end{definition}
If $V$ is an integrable highest-weight module and $\mu\in P(V)$, then define
\begin{equation}\label{eq:multv}Mult_V(\mu) = \dim_V(V_\mu).\end{equation}
\begin{proposition}[\cite{K2} 10.1] If $V=V^\lambda$ is an irreducible highest-weight $\LL$-module with highest weight $\lambda\in P^+$, then for all $\mu\in P(V)$ and $w\in W$,
$$Mult_V(\mu)=Mult_V(w\mu).$$
In particular, $P(V)$ is $W$-invariant.
\end{proposition}

\begin{theorem}[\cite{K2} 10.4]
Let $V^\lambda$ be an irreducible $\LL$-module with highest weight $\lambda\in P^+$, and let $\rho=\ds\sum_{i=1}^\ell \omega_i.$ Then
$$Ch(V^\lambda) = \frac{\ds\sum_{w\in W}sgn(w)e^{w(\lambda+\rho)-\rho}}{\ds\prod_{\alpha\in \Delta_+}(1-e^{-\alpha})^{Mult (\alpha)}}.$$
This is known as the \textbf{Weyl-Kac character formula}.
\end{theorem}

Racah and Speiser gave a recursive algorithm based on the Weyl-Kac character formula (the so-called ``Racah-Speiser recursion'') for determining weight multiplicities of irreducible highest-weight modules for KM algebras. Section \ref{sec:inner} gives a detailed description of how this recursion is used in the setting of $\Fib$-modules in $\FF$. Kac and Peterson gave a separate algorithm for determining multiplicities of weights in the adjoint representation (roots) of a KM algebra \cite{K2}. This ``Kac-Peterson recursion'' is outlined briefly in Appendix \ref{appendix:A}. 

\begin{definition}\label{multinot}
Given a GCM $A$ and its associated Kac-Moody Lie algebra $\LL(A)$, we define the shorthand notation $e_{i_n \ldots i_2 i_1}$ to stand for the multibracket  $ad_{e_{i_n}} \cdots ad_{e_{i_2}} e_{i_1}$, which lies in the root space $\LL_\alpha$ for $\alpha = \sum_{j=1}^n \alpha_{i_j} \in \Delta_+$. Similarly, we write $f_{i_n \ldots i_2 i_1}$ as shorthand for the multibracket $ad_{f_{i_n}} \cdots ad_{f_{i_2}} f_{i_1}\in \LL_{-\alpha}$.
\end{definition}

For a fixed $\alpha \in\Delta_+$, the set of multibrackets $e_{i_n \ldots i_2 i_1} \in \LL_\alpha$ (where by definition, $ \sum_{j=1}^n \alpha_{i_j}=\alpha\})$ spans $\LL_\alpha$, and the set of multibrackets $f_{i_n \ldots i_2 i_1} \in \LL_{-\alpha}$ spans $\LL_{-\alpha}$ . 
%Thus, each $\alpha \in \Delta_+$, $\alpha$ can be written in the form 
%$$\alpha =  \sum_{j=1}^n \alpha_{i_j},$$ 
%where the $\alpha_{i_j}$ are simple roots, not necessarily distinct.  
Furthermore, if $e_{i_n \ldots i_2 i_1} \neq 0$, then for all $1\leq k\leq n$, $e_{i_k \ldots i_2 i_1} \neq 0$. In other words, each partial sum $\sum_{j=1}^k \alpha_{i_j}$ is also a root in $\Delta_+$.  Similarly, if $f_{i_n \ldots i_2 i_1} \neq 0$, then for all $1\leq k\leq n$, $f_{i_k \ldots i_2 i_1} \neq 0$ and $-\sum_{j=1}^k \alpha_{i_j}$ is a root in $\Delta_-$. 

\label{basis}We now briefly outline how one may determine weight multiplicities of an irreducible highest-weight $\LL$-module $V^\lambda$ by recursively computing bases for its weight spaces. (A detailed description of this process for irreducible $\Fib$-modules is given in Chapter \ref{ch:nonstd}.) First, $V^\lambda_\lambda$ has basis consisting of a single highest-weight vector $\{v_\lambda\}$. Let $\lambda\neq\mu\in P(\lambda)$ and assume for each $1\leq i \leq \ell$ we have the previously determined basis $\cB_i$ for $V^\lambda_{\mu+\alpha_i}$. Then acting with $f_i$ on each vector in $\cB_i$ for $1\leq i \leq \ell$ gives a spanning set $\cS_\mu = \{ v_1, \ldots , v_n\}$ for $V^\lambda_\mu$, where $n=\ds\sum_{i=1}^\ell |\cB_i|$. Linear dependence relations on the vectors in $\cS_\mu$ are then found by solving the homogeneous system of linear equations determined by setting
$$e_i \cdot \sum_{j=1}^n c_jv_j = \sum_{j=1}^n c_j(e_i\cdot v_j) = 0$$
for $1\leq i \leq \ell$. Since $V^\lambda$ is an irreducible module and therefore cannot contain any more highest-weight vectors, any nontrivial solutions will yield dependence relations on the vectors in $\cS_\mu$. Then choose vectors to delete from the spanning set to obtain a basis $\cB(\mu)$ for $V^\lambda_\mu$. Finally, we have $Mult_\lambda(\mu)=|\cB(\mu)|$.

%----- Look in K2 p. 106, 9.4.1, 9.4.2, maybe add the contravariant bilinear form ?
%There is an analogous theory of \textit{lowest-weight modules} in the opposite \textit{category} $\cO^{op}$. Similar statements about complete reducibility and character formulae exist for these modules.

\begin{definition}\label{lwmod} An $\LL$-module $V$ is a \textbf{lowest-weight module} with \textbf{lowest weight} $\lambda\in\HH^*$ if there exists a non-zero vector $v_\lambda\in V$, called a \textbf{lowest-weight vector (LWV)}, with the following properties:
$$\mathfrak{n}_- \cdot v_\lambda = 0,\qquad h\cdot v_\lambda = \lambda(h)v_\lambda,\qquad \text{and} \qquad\cU(\mathfrak n_+)\cdot v_\lambda = V.$$
If $V$ is irreducible we write $V=V^\lambda$ and denote the $\mu$-weight space of $V$ by $V^\lambda_\mu$.
\end{definition}

Given an irreducible integrable $\HH$-diagonalizable highest-weight $\LL$-module $V^\lambda$ in $\cO$ with HWV $v_\lambda$, we may view the vector space $V^\lambda$ as a different $\LL$-module in the opposite category $\cO^{op}$ under the action of $\LL$ twisted by the Cartan involution $\nu$.  This new module is denoted $V^{-\lambda}$ and has set of weights $P(V^{-\lambda})=-P(V^\lambda)$. Define the new action $\circ$ of $\LL$ on $V^\lambda$,
$$x \circ v := \nu(x) \cdot v \quad \text{for} \ x\in \LL, v\in V,$$
so that
$$[x,y]\circ v = \nu[x,y] \cdot v = [\nu(x),\nu(y)]\cdot v = \nu(x) \cdot (\nu(y)\cdot v)-\nu(y)\cdot(\nu(x)\cdot v)$$
$$= x \circ (y \circ  v) - y \circ (x \circ v). $$
Then if $v_\mu \in V^\lambda_\mu$ where $\mu\in P(V^\lambda)$, we have
$$h \circ v_\mu = -h \cdot v_\mu = -\mu(h) v_\mu \quad \text{and} \quad f_i \circ v_\lambda = -e_i \cdot v_\lambda = 0 \quad \text{for} \quad  1\leq i \leq \ell.$$ 
Then $V^{-\lambda}_{-\mu}$ is the weight space $V^\lambda_\mu$ viewed under the $\circ$ action as the $-\mu$-weight space of $V^{-\lambda}$. Thus the module $V^{-\lambda}=\ds\bigoplus_{\lambda\in-P(V^\lambda)} V^{-\lambda}_{-\mu}$ constructed in this way is called the \textit{contragredient module} to $V^\lambda$, and this construction of a lowest-weight module from a given highest-weight module is equivalent to a functor relating $\cO$ and $\cO^{op}$.

\begin{rem}
For every statement on category $\cO$ modules previously mentioned, an analogous statement also holds for modules in category $\cO^{op}$. In particular, we have
\begin{enumerate}
\item (Definition \ref{high}) For $V$ a not necessarily irreducible $\LL$-module in category $\cO^{op}$ and $\mu\in P(V)$, define $Low_V(\mu) = \{ v\in V_\mu \mid f_i\cdot v=0, \ 1\leq i\leq \ell \}$ to be the subspace of lowest-weight vectors in $V_\mu$. Then the outer multiplicity $M_V(\mu) = \dim Low_V(\mu)$. If $V=V^\lambda$ is irreducible, then $M_{V^\lambda}(\mu) = M_{V^{-\lambda}}(-\mu)$.
\item (Lemma \ref{dominant}) If $V^\lambda$ is an irreducible lowest-weight module, then $V^\lambda$ is integrable if and only if $\lambda\in P^-$.
\item (Theorem \ref{reducible}) If $V$ is a module in category $\cO^{op}$, then $V$ is integrable if and only if $V$ is completely reducible, so that $V=\ds\bigoplus_{\lambda\in P'(V)} M_V(\lambda) V^\lambda$, where $P'(V)\subset P^-$ is the set of weights whose weight spaces of $V$ contain LWV's.
\end{enumerate}
\end{rem}
\begin{definition}\label{Vpm}
Let $V$ be an $\HH$-diagonaliable $\LL$-module with weights $P(V)$. Define
$$P(V)_+ = \{\mu\in P(V) \mid wt(\mu)\geq 0\}, \qquad \text{and} \qquad P(V)_- = \{\mu\in P(V) \mid wt(\mu)< 0\},$$
so $P(V)=P(V)_-\cup P(V)_+.$ Then we define
$$V_\pm = \bigoplus_{\mu\in P(V)_\pm} V_\mu,$$
so $V= V_-\oplus V_+$.
\end{definition}

Lastly, we introduce \textit{non-standard modules}, which will be explored in Chapters \ref{ch:modules} $-$ \ref{ch:nonstd}. Our choice of properties in part $(ii)$ below was motivated by the fact that highest and lowest weights satisfy properties $(2)-(4)$, but not property $(1)$.
%----"Irreducible" to immediately exclude Verma modules

\begin{definition}\label{nstd}
Let $V$ be an irreducible, integrable, $\HH$-diagonalizable $\LL$-module. So for any $0\neq v \in V$, we have $V=\cU(\LL)\cdot v$.
\begin{enumerate}[i)]
\item We say $V$ is \textbf{standard} if it is a highest- or lowest-weight module. In other words, there exists a unique $\lambda\in P(V)$ that is either a highest weight (so $\lambda\in P^+$) or a lowest weight (so $\lambda\in P^-$). In either case, we write $V=V^\lambda$.
\item We say $V$ is \textbf{non-standard} if it is not a standard module, so there exists a weight $\lambda \in P(V)$ such that 1) $\lambda\notin P^\pm$, 2) $||\lambda||$ is maximal in $\{||\mu|| \mid \mu \in P(V)\}$, 3) $|wt(\lambda)|$ is minimal among those, and 4)  $V=\cU(\LL)\cdot v_\lambda$ for some $v_\lambda\in V_\lambda$. We may write $V=V^\lambda$ for any $\lambda$ that satisfies these properties. 
\end{enumerate}
\end{definition}

Although the label $\lambda$ for non-standard $V^\lambda$ is not uniquely determined, there are only finitely many weights that satisfy the properties of the definition. Note that $\LL$ itself is a non-standard module $V^{\alpha_i}$ for any $\alpha_i\in\Pi$.

Let $V^\lambda$ be either a standard or non-standard $\LL$-module. If we identify the set of fundamental weights with a basis for $\R^\ell$ we obtain a visualization of $P(V^\lambda)$ called the \textit{weight diagram} of $V^\lambda$, where each dot corresponds to a weight $\mu\in P(V^\lambda)$, and is labeled with its inner multiplicity. A weight diagram without multiplicities is called an unlabeled weight diagram. 

Since $P(V^\lambda)$ is a saturated set of weights, we have $P(V^\lambda) = R_1 \cup R_2$ where
$$R_1 =\bigcup_{w\in W} \{S_{\alpha_i}(w\lambda) \mid 1 \leq i \leq \ell \} \quad \text{and} \quad R_2 =  \bigcup_{\mu\in  R_1} S_{\alpha_i}(\mu)$$ 
(cf. \eqref{eq:saturated} above). This gives a recursive method for determining the weight diagram of $V^\lambda$. Examples of weight diagrams are shown in Figures \ref{fig:aff(1)}, \ref{subfig-1:fibmult}, \ref{subfig-2:rhomult}, \ref{subfig-2:2rho},  and \ref{fig:nonstdwts}.

%------------------------------------------------------%
%------------------------------------------------------%
%----------------Section 1.3. $\FF$ -------------%
%------------------------------------------------------%
%------------------------------------------------------%
\section{The hyperbolic Kac-Moody Lie algebra $\FF$}\label{sec:F} \
Consider the $3\times 3$ the generalized Cartan matrix
$$A = (a_{ij}) = \ds\left(\begin{array}{ccc}
	2 & -2 & 0 \\
	-2 & 2 & -1\\
	0 & -1 & 2\\
		\end{array}\right).$$
As outlined in section \ref{sec:kacmoody}, $A$ has realization $(\HH, \Pi=\{\alpha_i\}, \Pi^\vee=\{h_i\})_{i=1,2,3}$ (where $h_i=\alpha_i^\vee$ since $A$ is symmetric), and to $A$ we associate a KM algebra $\FF=\FF(A)$ over $\C$, generated by the elements $\{h_i,e_i,f_i \ | \  i=1,2,3\}$ subject to the relations presented in Definition \ref{def:kml}. We also have:

\begin{itemize}
\item the abelian Cartan subalgebra $\HH$ with basis $\Pi^{\vee}$
\item root lattice $Q=Q_\FF\subset \HH^*$
\item root system $\Phi\subset Q\subset \HH^*$ with basis of simple roots $\Pi$.
\item $\alpha_i(h_j)=a_{ij}$
\item root spaces $\FF_{\alpha}\ =\ \{x\in \FF\mid[h,x]=\alpha(h)x,\ h\in \HH\}$.
\item $\FF_{\alpha_i}=\C e_i$ and $\FF_{-\alpha_i}=\C f_i$
\item Cartan involution  $\nu=\nu_\FF:\FF\rightarrow\FF$ given by $\nu(e_i)=-f_i$ and $\nu(h_i)=-h_i$ for $i=1,2,3$.
\item root space and triangular decompositions, $\FF = \HH \oplus \bigoplus_{\alpha\in\Phi} \FF_{\alpha} = \HH\  \oplus\  \FF^+  \oplus\ \FF^-$
\item Weyl group $W_\FF=W(A)=\la w_i \mid 1\leq i \leq 3 \ra$
\item set of integral weights $P=\{ \omega\in\HH^* \mid \la \alpha_j, \omega \ra \in \Z , \ 1\leq j \leq 3 \}$,
\item fundamental weights $\omega_1, \omega_2, \omega_3\in P$,
\item dominant integral weights $P^+ = \{ \omega \in P \mid \la \alpha_j, \omega \ra \geq 0, 1\leq j \leq 3 \}$
\end{itemize}

Define the real vector space $\HH_\R=\R h_1 \oplus \R h_2 \oplus \R h_3$
and its dual
$\HH^*_\R=\{\beta=x\alpha_1+y\alpha_2+z\alpha_3 \in\HH^* \ | \  x, y, z \in \R\}$
with indefinite non-degenerate quadratic form of signature (2,1) determined by the GCM $A$, 
$$||\beta||^2 = [\beta]^t A [\beta] = 2(x^2-2xy+y^2-yz +z^2)$$
where $[\beta]$ is the coordinate vector of $\beta$ with respect to $\Pi$, so $\HH^*_\R \simeq \R^{(2,1)}.$  For $c\in\R$, 
\begin{equation}\label{eq:surface}S_c = \{\beta \in\HH^*_\R \ | \ ||\beta||^2=c \}\end{equation}
is a quadric surface, either a hyperboloid or a cone (see Figure \ref{fig:hyperbs} %and others 
for examples).

%%%%%%%%%%%%%%%
If we set $a=z-y, \ b=x-y, \ c=-z,$
we obtain
\begin{equation}\label{eq:sqlength}
||\beta||^2 = 2(b^2-ac)=-2 \det \ds\left[\begin{array}{cc}
	a & b \\
	b & c \\
		\end{array}\right].
		\end{equation}
We are therefore motivated to represent any $\beta \in\HH^*_\R$ as a $2\times 2$ real symmetric matrix:		
$$\beta = \ds\left[\begin{array}{cc}
	a & b \\
	b & c \\
		\end{array}\right] = \left[\begin{array}{cc}
	z-y & x-y \\
	x-y & -z \\
		\end{array}\right] = x\left[\begin{array}{cc}
	0 & 1 \\
	1 & 0 \\
		\end{array}\right] + y\left[\begin{array}{cc}
	-1 & -1 \\
	-1 & 0 \\
		\end{array}\right] + z\left[\begin{array}{cc}
	1 & 0 \\
	0 & -1 \\
		\end{array}\right],$$
so the simple roots are:
\begin{equation}\label{eq:simple}\alpha_1 = \ds\left[\begin{array}{cc}
	0 & 1 \\
	1 & 0 \\  
		\end{array}\right], \qquad		
\alpha_2 = \ds\left[\begin{array}{cc}
	-1 & -1 \\
	-1 & 0 \\
		\end{array}\right], \qquad \textrm{and} \qquad
\alpha_3 = \ds\left[\begin{array}{cc}
	1 & 0 \\
	0 & -1 \\
		\end{array}\right].\end{equation}

%%%
By polarization of the above quadratic form, we obtain the following non-degenerate linear form (in the second entry) on $\HH^*_\R$ represented by $A$ with respect to the basis $\Pi$:
$$\ds\la \alpha, \alpha' \ra =  [\alpha]^t A [\alpha'] = \frac{1}{2}\left(||\alpha + \alpha' ||^2 - || \alpha||^2  - || \alpha'||^2\right).$$

For $\alpha=\left[\begin{array}{cc}
	a & b \\
	b & c \\
		\end{array}\right], \ \alpha'=\left[\begin{array}{cc}
	a' & b' \\
	b' & c' \\
		\end{array}\right] \in \HH^*_\R$, we have
\begin{align}
\la \alpha , \alpha' \ra &= \ds\frac{1}{2}\Big(2\big[(b+b')^2-(a+c')(c+c')\big]-2[b^2-ac]-2[b'^2-a'c']\Big) \nonumber \\ 
&= 2bb'-ac'-a'c. \label{angleform}
\end{align}
Note that for $i,j=1,2,3$, this formula gives $\la \alpha_i, \alpha_j \ra = a_{ij}=\alpha_j(h_i).$

The Weyl group $W$ has a presentation as a Coxeter group 
$$W = \la w_1, w_2, w_3 \ | \ w_i^2 = 1, \ (w_1w_3)^2=(w_2 w_3)^3 = 1 \ra,$$
which is actually a hyperbolic triangle group and matrix group,
$$W \simeq T(2, 3, \infty) \simeq PGL_2(\Z) = \{ M \in \Z^2_2 \ | \ \det M = \pm 1 \} / \{ \pm I \}.$$ 
The latter isomorphism will prove especially useful since it will allow us to realize elements of $W$ by $2\times 2$ matrices.

We have the partition of $\Phi=\Phi^{re}\cup\Phi^{im},$
where
$$\Phi^{re}=\{\alpha\in Q \ | \ ||\alpha||^2=2 \} \qquad \hbox{and}\qquad \Phi^{im}=\{\alpha\in Q \ | \ ||\alpha||^2\leq 0 \}.$$
\begin{figure}[!ht]
    \subfloat[The hyperboloid of one sheet $S_2$ and real roots $\Phi^{re}$.\label{subfig-1:hyperboloid}]{%
     \includegraphics[trim = 0cm 0cm 0cm 0cm, clip=true, width=2.9 in]{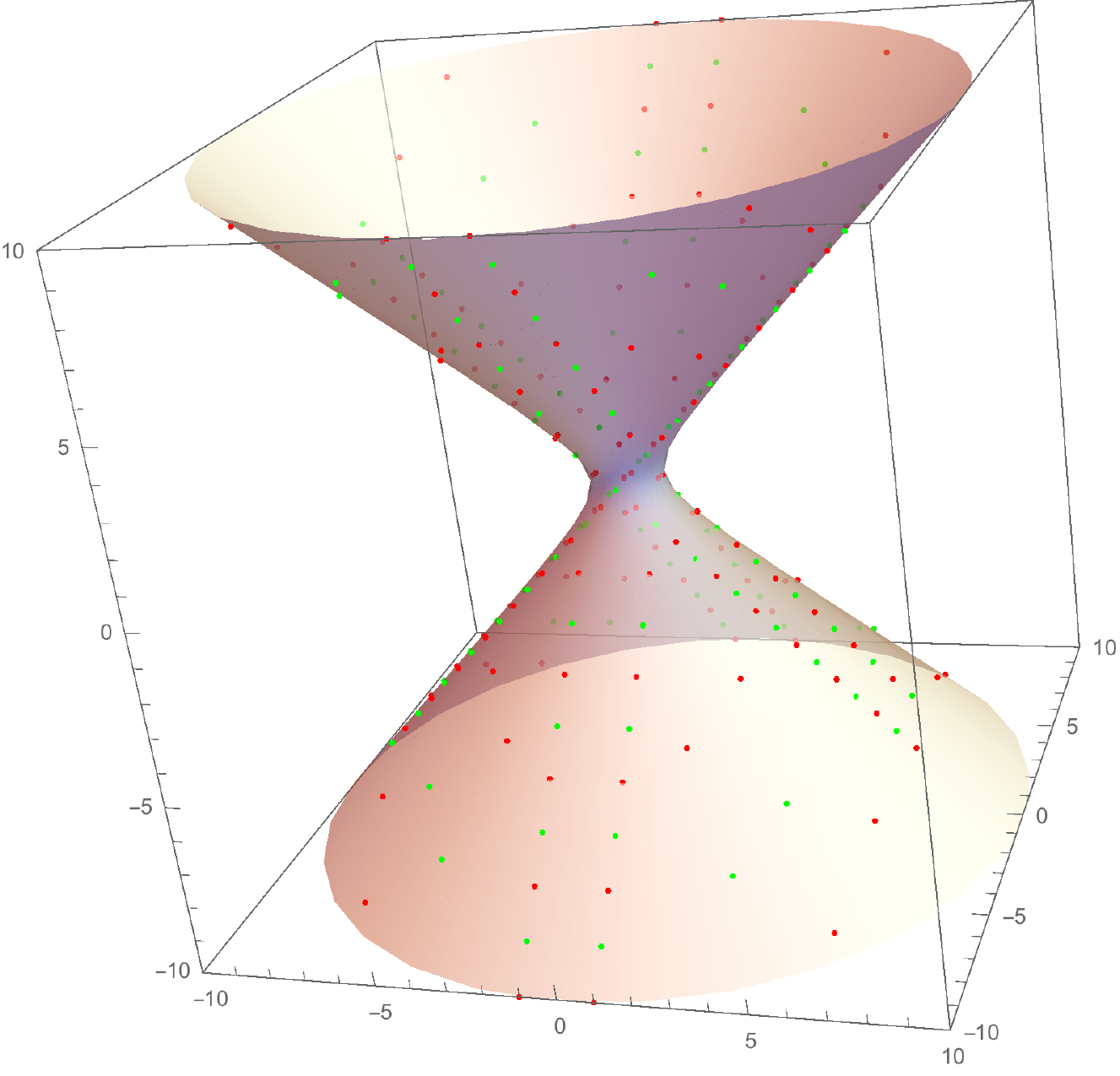}
         }
    \hfill
    \subfloat[A hyperboloid of two sheets $S_{-2}$ and its imaginary roots.\label{subfig-2:imaginaryhyp}]{%
      \includegraphics[trim = 3.3cm 11cm 2cm 2.5cm, clip=true, width=2.7 in]{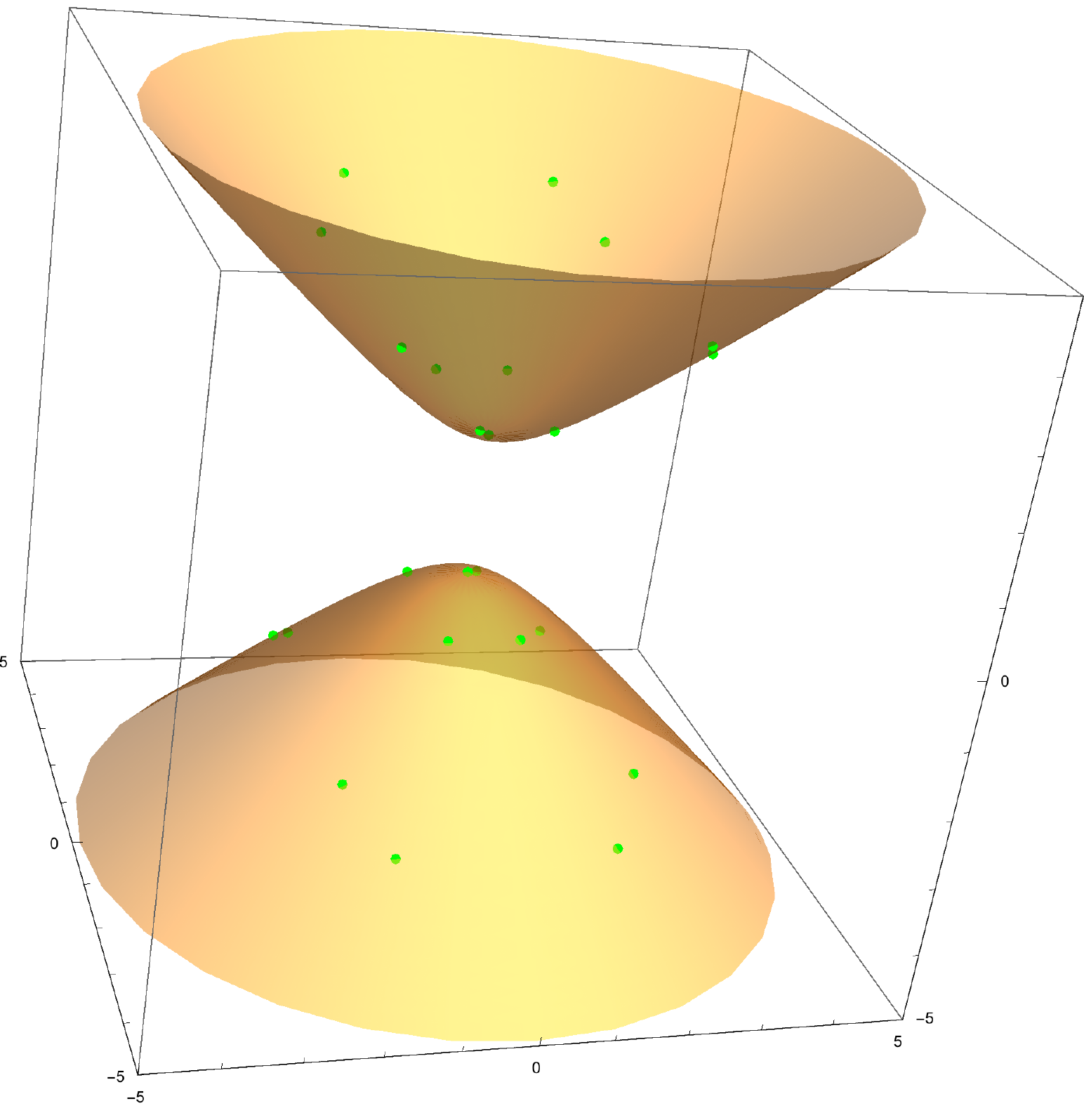}
    } 
    \caption{Quadric surfaces of constant squared-length in $\HH^*_\R$. }
    \label{fig:hyperbs}
  \end{figure}  
$\Phi^{re}$ is contained in the hyperboloid of one sheet, $S_2$, shown in Figure \ref{subfig-1:hyperboloid}.
 (Note that the roots on this hyperboloid are partitioned into two colors, corresponding to the partition of $\Phi^{re}$ into two disjoint $W$-orbits of $\Pi$ \cite{CCFP}.) Also, 
$$\Phi^{im} \subset \bigcup_{n=0}^\infty S_{-2n},$$
where $S_0$ is the \textit{null-cone} and $S_{-2n}$ for $n\geq 1$ are hyperboloids of two sheets strictly inside the \textit{light-cone}, defined to be
$$LC = \{\alpha\in\HH^*_{\R} \mid (\alpha, \alpha) \leq 0 \}.$$
If $||\alpha||^2=0$ (so $\alpha$ lies on the null-cone) then $Mult(\alpha)=1$.  A portion of $S_{-2}$ is shown in Figure \ref{subfig-2:imaginaryhyp}.  The light-cone is partitioned into the \textit{backward} and \textit{forward} light-cones,
$$LC = LC_- \cup LC_+,$$
such that $\Phi^{im}_\pm \subset LC_\pm$. For $\alpha=\left[\begin{array}{cc}
	a & b \\
	b & c \\
		\end{array}\right] \in\Phi^{im}$ with $||\alpha||^2=2(b^2-ac)<0,$ then if $c>0$, then $\alpha\in\Phi_-^{im}\subset LC_-$, and if $c<0$, then $\alpha\in\Phi_+^{im} \subset LC_+$ .

As Figure \ref{subfig-2:imaginaryhyp} indicates, for each $n> 0$, there is a ``positive'' sheet of $S_{-n}$ that lies in $LC_+$ and a ``negative'' sheet that lies in $LC_-$. Since each sheet is $W$-invariant, we have that $W(LC_\pm)=LC_\pm$, and
\begin{equation}W(LC_+)\cap W(LC_-)=\{0\}.\end{equation}\label{LCorbits}
We can also represent the real and imaginary roots as
 $$\Phi^{re}=\Big\{ \alpha \in\Phi \ \mid \det(\alpha)=-1\Big\} \ \text{  and  }  \ \Phi^{im}=\Big\{  \alpha \in\Phi \ \mid \det(\alpha)\geq 0\Big\}.$$ 

%%%
The formulas for the simple reflections determine their actions on $\alpha=\left[\begin{array}{cc}
	a & b \\
	b & c \\
		\end{array}\right]:$
$$w_1\alpha = \left[\begin{array}{cc}
	a & -b \\
	-b & c \\
		\end{array}\right], \quad 
w_2\alpha = \left[\begin{array}{cc}
	a-2b+c & c-b \\
	c-b & c \\
		\end{array}\right], \quad 
w_3\alpha = \left[\begin{array}{cc}
	c & b \\
	b & a \\
		\end{array}\right].$$
Defining the three matrices	
$$M_1=  \left[\begin{array}{cc}
	1 & 0 \\
	0 & -1 \\
		\end{array}\right], \quad \ M_2 =  \left[\begin{array}{cc}
	-1 & 1 \\
	0 & 1 \\
		\end{array}\right] , \quad \ M_3 =  \left[\begin{array}{cc}
	0 & 1 \\
	1 & 0 \\
		\end{array}\right],$$
we see that for $i=1,2,3$, $w_i\alpha=M_i\alpha M_i^t.$
If $w=w_{i_1}\cdots w_{i_s}\in W$ then there is a corresponding $M=M_{i_1}\cdots M_{i_s} \in PGL_2(\Z)$ such that $w(\alpha)=M\alpha M^t$.

%%% Slide 15:
Using
$$A^{-1} = \ds\left(\begin{array}{ccc}
	-\frac{3}{2} & -2 & -1 \\
	-2 & -2 & -1\\
	-1 & -1 & 0\\
		\end{array}\right),$$

%%%Slide 17:
we may write the fundamental weights in terms of the simple roots as follows:
$$\omega_1=-\frac{3}{2}\alpha_1 -2\alpha_2 -\alpha_3 = \left[\begin{array}{cc}
	1 & \frac{1}{2} \\
	\frac{1}{2} & 1 \\		\end{array}\right],\quad \omega_2=-2\alpha_1 -2\alpha_2 -\alpha_3 = \left[\begin{array}{cc}
	1 & 0 \\
	0 & 1 \\		\end{array}\right],$$
$$\text{and} \quad \omega_3=-\alpha_1 -\alpha_2 = \left[\begin{array}{cc}
	1 & 0 \\
	0 & 0 \\		\end{array}\right].$$
Thus for $\alpha\in\Phi$, we have
\begin{equation}\label{eq:c>0}\alpha\in\Phi_\pm \ \text{  if and only if  } \ \la  \alpha, \pm \omega_3 \ra = \mp c >0.\end{equation}
We also observe that if $\omega = \left[\begin{array}{cc}
	a & b \\
	b & c \\
		\end{array}\right] \in P^+$ then $a, 2b, c \in\Z$ and by definition \ref{fund},
$$\la\alpha_1, \omega \ra = 2b \geq 0, \quad 
\la \alpha_2, \omega \ra = -2b+c \geq 0, \quad 
\la \alpha_3, \omega \ra = a-c \geq 0,$$			
giving us $P^+$ as the intersection of $P$ with the three real half-spaces defined by the $\alpha^\perp_i$,
$$P^+= \Bigg\{ \omega = \left[\begin{array}{cc}
	a & b \\
	b & c \\		\end{array}\right] \ \Bigg| \ a, 2b, c \in\Z , \ a\ge c \ge 2b \ge 0 \Bigg\}.$$
%Slide 18:
Note that $P^+ \subset LC_-$ since for $\omega\in P^+$,
$$\det(\omega)=ac-b^2\geq (2b)(2b)-b^2=3b^2\geq0.$$
The (positive) Tits cone $T_+$ is defined to be the strict interior of the forward light-cone $LC_+$, namely,
$$T_+=\{\alpha\in LC_+ \ | \ (\alpha, \alpha ) < 0 \}.$$
%\cite{K2} Proposition 3.12 
Define $P^-=-P^+$. Then a result of Kac (\cite{K2}) shows that $P^-$ is a fundamental domain for the action of the Weyl group on $P \cap LC_+$, the weight lattice points which lie within the forward light-cone $LC_+$ (including points on the null cone). 
\begin{figure}[h!]
\centering
\includegraphics[trim = 0cm 0cm 0cm .5cm, clip=true, width=3.5 in]{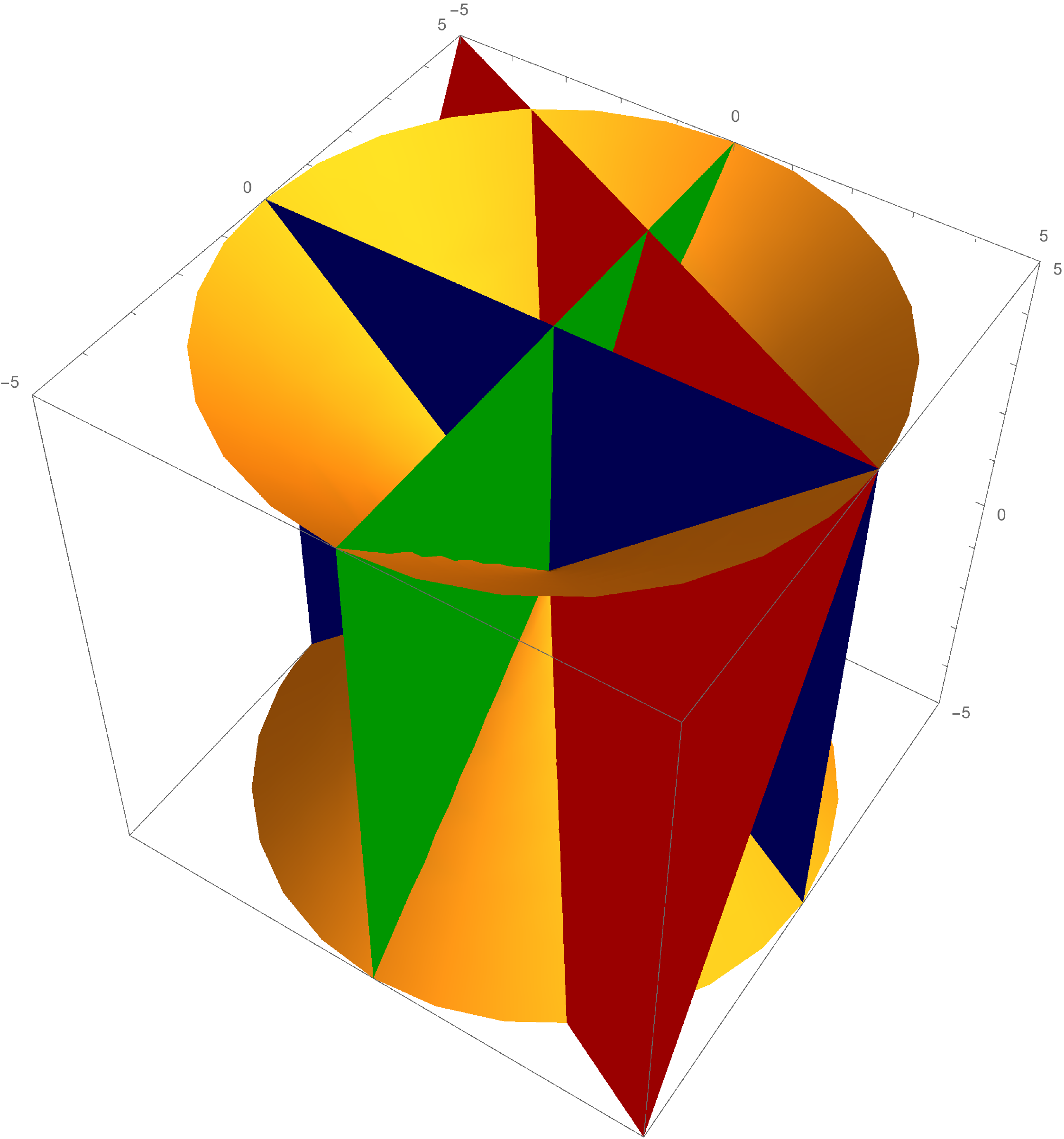}
\caption{$P^-$ is bounded by the reflecting planes $\alpha_1^{\perp}, \alpha_2^{\perp}, \alpha_3^{\perp}$. The null-cone is shown in yellow.}
\label{fig:tits}
\end{figure}

%\subsection{Rank 2 subalgebras and their root systems}\label{subsec:rank2} \
\section{Feingold-Frenkel decomposition of $\FF$ with respect to $A^{(1)}_1$}\label{sec:aff} \
In \cite{FF}, Feingold and Frenkel investigated the decomposition of $\FF$ with respect to a rank 2 affine subalgebra. Consider the generalized Cartan matrix 
$$\ds \left(\begin{array}{cc}
	2 & -2  \\
	-2 & 2 \\
		\end{array}\right),$$
and consider the subalgebra of $\FF$ generated by $\{e_i, f_i, h_i \ | \ i=1,2 \}$. The bilinear form on $\FF$ restricted to this subspace is degenerate, however, so the corresponding simple roots $\alpha_1, \alpha_2$ are no longer linearly independent. This subalgebra is isomorphic to the central extension of the loop algebra,
$$\mathfrak{sl}_2(\C)\otimes \C[t,t^{-1}]\oplus \C c,$$
where $\mathfrak{sl}_2(\C)$ is the simple finite-dimensional Lie algebra with basis $\{e, f, h\}$. To endow this algebra with a symmetric invariant nondegenerate bilinear form, the usual approach is to extend the Cartan subalgebra by adding the derivation $d=-t\frac{d}{dt}.$ The resulting \textit{affinization} of $\mathfrak{sl}_2(\C)$,
$$\widehat{\mathfrak{sl}_2(\C)}= \mathfrak{sl}_2(\C)\otimes \C[t,t^{-1}]\oplus \C c \oplus  \C d$$
is an affine KM algebra. Feingold and Frenkel then give an embedding of $\widehat{\mathfrak{sl}_2(\C)}$ in $\FF$ defined by
$$h\otimes1\mapsto h_1, \quad c \mapsto h_1+h_2, \quad d \mapsto h_1+h_2+h_3,$$
$$e\otimes1\mapsto e_1, \quad f\otimes1\mapsto f_1, \quad f\otimes t\mapsto e_2, \quad e\otimes t^{-1}\mapsto f_2.$$
We refer to this subalgebra as $\Aff$. Note the Cartans of $\Aff$ and $\FF$ are the same. 

The simple roots $\Pi_\Aff= \{\alpha_1, \alpha_2\}$ immediately give us

\begin{itemize}
\item the affine plane $\R\alpha_1\oplus\R\alpha_2 \subset \HH^*_\R$,
\item the affine root sublattice $Q_\Aff = \Z\alpha_1+ \Z\alpha_2,$
\item the affine Weyl group $W_\Aff= \la w_1, w_2 \ra < W$,
\item the affine null root $\delta=\alpha_1+\alpha_2$,
\item the affine root subsystem $\Phi_\Aff=\{n\delta \mid 0\neq n\in \Z \} \cup \{m\delta \pm \alpha_1 \mid  m\in \Z\}\subset\Phi$,
\item the affine real roots, $\Phi^{re}_\Aff=W_\Aff\Pi_\Aff=\{m\delta \pm \alpha_1 \mid  m\in \Z\}=\{\alpha\in Q_\Aff \ | \ ||\alpha||^2=2 \}, $
\item the affine imaginary roots $\Phi^{im}_\Aff=\{n\delta \mid 0\neq n\in \Z \}=\{\alpha\in Q_\Aff \ | \ ||\alpha||^2= 0 \} $,
\item the affine fundamental weights $\{\omega_1, \omega_2\}$,
\item the affine weight sublattice $P_\Aff=\Z\omega_1+\Z\omega_2\subset P$, and
\item the affine fundamental domain $P^\pm_\Aff=P^\pm\cap (\R\alpha_1\oplus\R\alpha_2)$.
\end{itemize}
%%%%%%%%%%  

%%%%%
Feingold and Frenkel showed that $\FF$ can be decomposed into an infinite sum of $\Aff$-modules, graded by level. The levels of $\FF$ with respect to $\Aff$ were determined as follows. The affine plane is level 0, and it contains roots of the form 
$$n_1\alpha_1 + n_2\alpha_2 = \left[\begin{array}{cc}
	-n_2 & n_1-n_2 \\
	n_1-n_2 & 0 \\
		\end{array}\right].$$ 
The non-zero levels were determined by adding multiples of the root $\gamma=\alpha_1+\alpha_2+\alpha_3=\left[\begin{array}{cc}
	0 & 0  \\
	0 & -1 \\
	\end{array}\right]$ to the affine plane, so the ``level $m$ affine plane'' is $\R\alpha_1\oplus\R\alpha_2 + m\gamma,$
and the root system is partitioned as $\Phi =\ds\bigcup_{m\in\Z} \Delta_m,$
where
$$\Delta_m=\Delta_{\Aff,m}=\Phi \cap (\R\alpha_1\oplus\R\alpha_2+ m\gamma) = \Bigg\{\beta\in\Phi \ \Bigg| \ \beta=\left[\begin{array}{cc}
	a & b  \\
	b & -m \\
	\end{array}\right], \ a, b, m\in\Z \Bigg\}$$
is the set of all affine level $m$ roots. The inner product gives a useful method for determining the level of any root in $\Phi$. Let $\delta=(\alpha_1+\alpha_2)=\left[\begin{array}{cc}
	-1 & 0  \\
	0 & 0 \\
		\end{array}\right]$. Then 
$$\la -\delta, \alpha \ra = \Bigg\langle \left[\begin{array}{cc}
	1 & 0  \\
	0 & 0 \\
		\end{array}\right], \left[\begin{array}{cc}
	a & b  \\
	b & -m \\
		\end{array}\right] \Bigg\rangle = m.$$
		
		We refer to the root spaces of level $m$ as $\Aff(m)$, thus giving the $\Z$-grading mentioned above as
$$\FF=\bigoplus_{m\in\Z} \Aff(m).$$
For each $m\in\Z$, the action of $\Aff$ preserves $\Aff(m)$, making $\Aff(m)$ an $\Aff$-module. By Kac \cite{K2}, each slice $\Aff(m)$ therefore has a decomposition into a direct sum of irreducible $\Aff$-modules. Additionally, each irreducible module $V$ has a weight space decomposition
$$V=\bigoplus_{\mu\in \Delta_m} V_\mu, \quad \text{where}\quad V_\mu=V\cap \FF_\mu.$$
The adjoint representation is the only non-standard $\Aff$-module occurring in this decomposition of $\FF$. (We will see more examples of non-standard modules in the decomposition with respect to $\Fib$ in later chapters.) Besides the adjoint representation, each irreducible integrable module $V$ of the decomposition of $\FF$ with respect to $\Aff$ is standard. 

%Modules of the first type are well-studied; along with Verma modules they are included in `Category $\mathcal{O}$' (cf. Chapter 10 of \cite{K2}). Type II modules are equally well understood as they are dual to type I modules with respect to the Cartan involution. 		
		
Kac showed that if $\mathfrak{g}$ is an affine KM algebra of rank $l+1$ then the multiplicity of every imaginary root of $\mathfrak{g}$ is $l$ \cite{K2}.  Therefore the dimension of every root space in $\Aff(0)$ is 1. 
\begin{figure}[h!]
\centering
\includegraphics[trim = 3cm 2cm 3cm 10cm, clip=true, width=2.5 in]{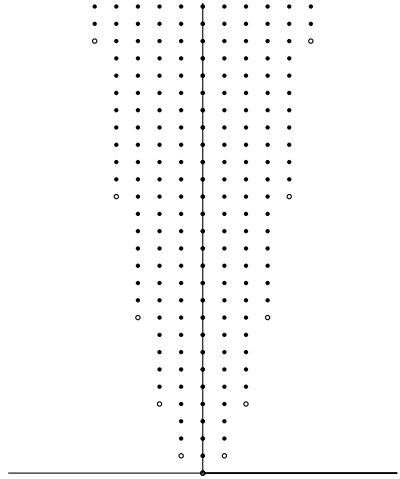}
\caption{$\Aff(1)$. Weights of the fundamental domain lie on the center line.}\label{fig:aff(1)}
\end{figure}
Dimensions of the weight spaces in $\Aff(1)$, whose dominant integral weights are of the form $\left(\begin{array}{cc}
	n & 0  \\
	0 & 1 \\
		\end{array}\right)$
where $n\geq -1$ (see Figure \ref{fig:aff(1)}), had already been discovered by Feingold and Lepowsky \cite{FL} to be 
$$ Mult\ds\left(\begin{array}{cc}
	n & 0  \\
	0 & 1 \\
		\end{array}\right) = p(n+1),$$ 	
where $p(n)$ is the classical partition function, with generating function $\ds\sum_{m\geq 0}p(m)q^m = \ds\prod_{k\geq 1}(1-q^k)^{-1}.$
The particular geometry of this module was the key to the proof, and the fact that all dominant integral weights occur on the central line simplified the Racah-Speiser recursion to Euler's recursion for the classical partition function. In addition, $\Aff(1)$ was found to be a single irreducible $\Aff$-module, thus the irreducible decompositions for levels $0$ and $\pm 1$ were completely determined.

Of course, as the level increases, so does the size of the dominant integral region and hence the complexity of the problem. In $\Aff(2)$, \cite{FF} found the fundamental domain to consist of two lines of roots, of the form $\left(\begin{array}{cc}
	n & 0  \\
	0 & 2 \\
		\end{array}\right)$ and $\left(\begin{array}{cc}
	n & 1  \\
	1 & 2 \\
		\end{array}\right)$
for $n\geq0$. Here, the vertex operator construction of $\Aff(1)$ from \cite{LW} was used to decompose the tensor product $\Aff(1)\otimes\Aff(1)$ into symmetric and antisymmetric tensors. Further analysis of the antisymmetric tensors that correspond to commutators in $[\Aff(1),\Aff(1)]$ resulted in a complete determination of $\Aff(2)$ root multiplicities,
$$\textrm{Mult}\ds\left(\begin{array}{cc}
	n & 0  \\
	0 & 2 \\
		\end{array}\right) = p'(2n+1) \qquad \textrm{and} \qquad 	\textrm{Mult}\ds\left(\begin{array}{cc}
	n & 1  \\
	1 & 2 \\
		\end{array}\right) = p'(2n)$$
where $p'(m)$ is a `modified partition function' whose generating function
		$$\sum_{m\geq 0}p'(m)q^m = (\prod_{k\geq 1}(1-q^k)^{-1})(1-q^{20}+q^{22}-\cdots )$$
agrees with the classical function for the first 20 terms.

This method of tensor decomposition of higher levels with respect to $\Aff$ was then used by Kang in several papers of the 1990s (e.g. \cite{Ka1}, \cite{Ka2}), eventually giving results up to level 5. The question for higher levels is still open, as is finding a closed form formula for the multiplicity of an arbitrary root of $\FF$.

%%%%%%%%%%%

\newpage
%--------------------------------------------------------%
%--------------------------------------------------------%
%--------------------- Chapter 2 --------------------%
%--------------------------------------------------------%
%--------------------------------------------------------%
\chapter{The Fibonacci subalgebra $\Fib$ of $\FF$}\label{ch:fib}
%--------------------------------------------------------%
%---------------------Section 2.1---------------------%
%--------------------------------------------------------%
% QUESTION: On any given hyperboloid of two sheets of fixed negative square-length, find an algorithm for determining how many imaginary roots lie inside the fundamental domain. Use the conditions on $a, b, c$.
% QUESTION: Benkart and Misra (?) generalized the notion of using levels 1 and -1 to generate a freer algebra larger than F which if you mod out by some ideal, you obtain F. 1993?
% QUESTION: Put in the table with data showing $F_1$ and $F_2$ and a cropped diagram for $\alpha_3$ (level 2) and do the next 4 calculations. More data for now: show vectors even if we know the dimension of the root space is less than the total. Add to the table and use Wolfram to determine independence of vectors going as high as 11.
As an alternative to finding a decomposition of $\FF$ with respect to $\Aff$, it is possible to decompose $\FF$ with respect to other rank 2 subalgebras inside $\FF$. We wish to explore how $\FF$ decomposes with respect to a particular rank 2 subalgebra of hyperbolic type. We will show that one of the crucial differences is the emergence of several non-standard modules, which are irreducible and integrable like the adjoint representation, but are not generated by either a highest- or lowest-weight vector. The affine decomposition of $\FF$ does not contain non-standard modules, except for the adjoint. We will discuss non-standard modules briefly in this section, and explore them in greater depth in Chapter \ref{ch:nonstd}.

Feingold and Nicolai \cite{FN} showed that all of the rank 2 hyperbolic KM Lie algebras, whose GCM are given by
$$\ds B(n)=\left(\begin{array}{cc}
	2 & -n  \\
	-n & 2 \\
		\end{array}\right)$$	
for $n\geq 3$ and which we denote by  $\mathcal{H}(n)$, are contained in $\FF$.  In other words, there exist roots $\beta_1, \beta_2 \in \Phi^{re}$ such that $\ds B(n)=(\la \beta_i, \beta_j \ra)$. Also, there exist Serre generators
$$E_i=E_i(n)\in\FF_{\beta_i}, \qquad F_i=F_i(n)\in\FF_{-\beta_i}, \qquad \textrm{and} \qquad H_i=H_i(n)=[E_i,F_i]\in\HH$$ 
of an algebra isomorphic to $\mathcal{H}(n)$, whose Cartan subalgebra is $\HH(n)=\C H_1 \oplus \C H_2$. We use the notation $\HH_\R(n)=\R H_1 \oplus \R H_2$. (Recall from the previous section that $\HH_\Aff=\HH$, while $\HH(n)\subsetneq \HH$.) Furthermore, the dual $\HH(n)_\R^*$ has inner product $( \cdot, \cdot )_n$ determined by $B(n)$ which agrees with the inner product on $\HH^*_\R$ determined by $A$. 

The simplest case is when $n=3$, giving the Cartan matrix
$$\ds B(3) = (b_{ij}) = \left(\begin{array}{cc}
	2 & -3  \\
	-3 & 2 \\
		\end{array}\right)$$
which corresponds  to $\mathcal{H}(3)$, the simplest rank 2 KM algebra of hyperbolic type. Feingold \cite{F1} discovered  that the Fibonacci numbers occur in an interesting way in the Weyl-Kac denominator formula for this algebra. We therefore refer to  $\mathcal{H}(3)$ as the \textit{Fibonacci hyperbolic}.

In order to decompose $\FF$ with respect to a subalgebra isomorphic to $\mathcal{H}(3)$, we must first find a realization of $\cH (3)$ inside $\FF$. The resulting subalgebra will be denoted by $\Fib$.  First, we find two positive real root vectors whose corresponding roots $\beta_1,\beta_2\in\Phi^{re}$ are simple roots of $\cH (3)$. Following \cite{FN} we could choose $\beta_1$ and $\beta_2$ to be
$$\left[\begin{array}{cc}
	1 & 0 \\
	0 & -1 \\
		\end{array}\right] = \alpha_3 \quad \text{and} \quad 
\left[\begin{array}{cc}
	-3 & \pm1 \\
	\pm1 & 0 \\
		\end{array}\right],$$
		respectively, so  $\beta_2= 2\alpha_1+3\alpha_2$ or $2\alpha_1+4\alpha_2.$  However, there is a choice whereby $\beta_1$ and $\beta_2$ have lower combined height, which will be helpful for multibracket and vertex algebra calculations in later sections. Similar to \cite{FN} we choose one root to be in $\Pi$. We cannot choose $\alpha_1$ since for any other $\alpha\in \Phi^{re}$ we get
$$\langle \alpha_1, \alpha \rangle = \Bigg\langle\ \left[\begin{array}{cc}
	0 & 1 \\
	1 & 0 \\
		\end{array}\right], \left[\begin{array}{cc}
	a & b \\
	b & c \\
		\end{array}\right]\Bigg\rangle = 2b\neq -3,$$
since $b\in \Z$. Choosing $\alpha_2$ gives the condition 
$$\langle \alpha_2, \alpha \rangle =\Bigg\langle \left[\begin{array}{cc}
	-1 & -1 \\
	-1 & 0 \\
		\end{array}\right], \left[\begin{array}{cc}
	a & b \\
	b & c \\
		\end{array}\right]\Bigg\rangle = -2b+c=-3,$$
and  $\alpha\in\Phi^{re}$ gives the additional condition $ac-b^2=-1$.  Setting $a=0$ and $b=1$ yields $c=-1$, hence $\alpha=\left[\begin{array}{cc}
	0 & 1 \\
	1 & -1 \\
		\end{array}\right]$ is a valid candidate for the second simple root. 
		
In the following proposition, we use the multibracket notation from Definition \ref{multinot}.

\begin{proposition} Let
$\beta_1=\left[\begin{array}{cc}
	0 & 1 \\
	1 & -1 \\
		\end{array}\right] = 2\alpha_1+\alpha_2+\alpha_3, \ \beta_2=\left[\begin{array}{cc}
	-1 & -1 \\
	-1 & 0 \\
		\end{array}\right] = \alpha_2\in\Phi.$
Then $\beta_1, \beta_2$ are simple roots of a subalgebra of $\FF$, denoted by $\Fib$, which is isomorphic to $\mathcal{H}(3)$.  The elements 
$$E_1 = \frac{1}{2}e_{1123}\in\FF_{\beta_1}, \ F_1 = -\frac{1}{2}f_{1123}\in\FF_{-\beta_1}, \  E_2=e_2\in\FF_{\beta_2}, \ F_2=f_2\in\FF_{-\beta_2},$$ 
$$H_1 = 2h_1+h_2+h_3, \ H_2=h_2\in\HH$$
are Serre generators of $\Fib$. We denote the $\C$-span of $H_1, H_2$ by $\HH_\Fib$, the $\R$-span of $H_1, H_2$ by $(\HH_\Fib)_\R$, and the $\Fib$ subroot system of $\Phi$ by $\Delta=\Delta_\Fib$. Also, $\nu$ restricts to the Cartan involution of $\Fib$, which we also denote by $\nu$, in other words, $\nu(E_i)=-F_i$ and $\nu(H_i)=-H_i$ for $i=1,2$.
\end{proposition}
\begin{proof}
We show these generators satisfy the Serre relations of $\mathcal{H}(3)$.

1) Clearly $H_1$ and $H_2$ commute.

2) We have $[E_1, F_2]=-\frac{1}{2}[f_2,e_{1123}]=0$,  $[E_2, F_1] = \frac{1}{2}[e_2, f_{1123}]=0$, and $[E_2,F_2]=[e_2,f_2]=h_2=H_2$. The next calculation shows that $[E_1,F_1]=H_1$, and makes use of the multibracket identities that follow from Theorem \ref{multibracket}, listed in Appendix \ref{appendix:B}.
\begin{align*}
[E_1&,F_1] = -\frac{1}{4}[e_{1123},f_{1123}] =-\frac{1}{4}\Big([[e_1,f_{1123}],e_{123}]+[e_1,[e_{123},f_{1123}]]\Big)
\\ &=-\frac{1}{4}\Big([2f_{123},e_{123}]+[e_1,\big[\Big[e_{123},f_1\Big], f_{123}\big]+\big[f_1,\Big[e_{123},f_{123}\Big]\big]]\Big)
\\ &=\frac{1}{4}\Big(2[e_1,[e_{23},f_{123}]] + 2[[e_1,f_{123}],e_{23}] + 2[e_1,[e_{23}, f_{123}]]  
\\ &\qquad \qquad - [e_1, \big[f_1,  \big[e_1,\Big[e_{23},f_{123}\Big]\big] + \big[\Big[e_1,f_{123}\Big],e_{23}\big]\big]]\Big)
\\ &=\frac{1}{4}\Big( 4[e_1,\big[e_2, \Big[e_3,f_{123}\Big]\big] + \big[\cancel{\Big[e_2,f_{123}\Big]},e_3\big]] -  4[e_{23},f_{23}] \\ &\qquad \qquad - [e_1, [f_1, \big[e_1, \big[e_2, \Big[e_3,f_{123}\Big]\big] + \big[\cancel{\Big[e_2,f_{123}\Big]},e_3\big]\big]]] - 2[e_1, [f_1, [f_{23},e_{23}]]] \Big)
\\ &=\frac{1}{4}\Big(-4[e_1,[e_2,f_{12}]] + 4[f_2,[f_3,e_{23}]] + 4[[f_2,e_{23}],f_3]
\\ &\qquad \qquad  + [e_1,[f_1,[e_1,[e_2,f_{12}]]]] -2[e_1,[f_1,[f_2,-e_2]+[e_3,f_3]]]\Big)  
\\ &=\frac{1}{4}\Big(-4[e_1,-2f_1] + 4[f_2,-e_2] +4[e_3,f_3]+ [e_1,[f_1,[e_1,-2f_1]]] -2[e_1,[f_1,h_2 + h_3]]\Big)
\\ &=\frac{1}{4}\Big(8h_1 + 4h_2 +4h_3+ [e_1,[f_1,-2h_1]] - 2[e_1,-2f_1 + 0]\Big)
\\ &=\frac{1}{4}\Big(8h_1 + 4h_2 +4h_3 - \cancel{[e_1,-4f_1 + 4f_1]}\Big)=2h_1 + h_2 +h_3,
\end{align*}
which is exactly $H_1$.  We also have $[E_2, F_2]=[e_2,f_2]=h_2=H_2.$

3) The following calculations show that $[H_i,E_j] = b_{ji} E_j$ for $i,j=1,2$.
$$[H_1,E_1]=[2h_1+h_2+h_3,E_1]= \beta_1(2h_1+h_2+h_3)E_1 =
(2\alpha_1+\alpha_2+\alpha_3)(2h_1+h_2+h_3)E_1 = 2E_1,$$ 
$$[H_1,E_2]=[2h_1+h_2+h_3,e_2]=\beta_2(2h_1+h_2+h_3)e_2 = \alpha_2(2h_1+h_2+h_3)e_2 = -3e_2 = -3E_2,$$
$$[H_2,E_1]= [h_2,E_1] = \beta_1(h_2)E_1 = (2\alpha_1+\alpha_2+\alpha_3)(h_2)E_1 = -3E_1, \quad \text{and} \quad [H_2,E_2]=2E_2.$$

4) It can easily be shown based on (3) above that $[H_i,F_j] = -b_{ji} F_j$.

5) We now show $(ad_{E_j})^{1-b_{ij}} E_j = 0$ for $1\leq i\neq j\leq 2$. First, we have that 
$$[E_2,E_1]\in\FF_{2\alpha_1+2\alpha_2+\alpha_3},\quad [E_2,[E_2,E_1]] \in\FF_{2\alpha_1+3\alpha_2+\alpha_3}, \quad [E_2,[E_2,[E_2,E_1]]]\in\FF_{2\alpha_1+4\alpha_2+\alpha_3}$$
are all root space vectors, since the squared-lengths of their corresponding roots are:
$$||2\alpha_1+2\alpha_2+ \alpha_3|| =2(4-8+4-2+1)=-2, \quad 
||2\alpha_1+3\alpha_2+\alpha_3|| =2(4-12+9-3+1)=-2,$$
$$\text{and}\quad ||2\alpha_1+4\alpha_2+\alpha_3|| =2(4-16+16-4+1)=2.$$ 
However, $[E_2,[E_2,[E_2,[E_2,E_1]]]]$ is not a root space vector since  $||2\alpha_1+5\alpha_2+\alpha_3|| =2(4-20+25-5+1)=10,$
therefore $2\alpha_1+5\alpha_2+\alpha_3 \notin \Delta$ and $(ad_{E_2})^{4} E_1 = 0$.

Similarly, 
$$[E_1,E_2]\in\FF_{2\alpha_1+2\alpha_2+\alpha_3},\quad[E_1,[E_1,E_2]] \in\FF_{4\alpha_1+3\alpha_2+2\alpha_3},\quad[E_1,[E_1,[E_1,E_2]]]\in\FF_{6\alpha_1+4\alpha_2+3\alpha_3}$$
are root space vectors since
$$||2\alpha_1+2\alpha_2+ \ \alpha_3|| =2(4-8+4-2+1)=-2, 
||4\alpha_1+3\alpha_2+2\alpha_3|| =2(16-24+9-6+4)=-2,$$
$$\text{and} \quad ||6\alpha_1+4\alpha_2+3\alpha_3|| =2(36-48+16-12+9)=2,$$
but $[E_1,[E_1,[E_1,[E_1,E_2]]]] \in\FF_{8\alpha_1+5\alpha_2+4\alpha_3}$ is not since  $||8\alpha_1+5\alpha_2+4\alpha_3|| =2(64-80+25-20+16)=10,$
giving us $(ad_{E_1})^{4} E_2 = 0$.

6) Multibracket calculations involving $F_1$ and $F_2$ similar to (5) above will show that $(ad_{F_i})^{-a_{ij}+1}(F_j)=0,\ i\neq j$.

The rest is obvious.
\end{proof}

The  proposition allows us to identify the dual Cartan subalgebra of $\mathcal{H}(3)$ with the dual Cartan $\HH_\Fib^*$ of $\Fib$. As before, we call $(\HH^*_\Fib)_\R = \R\beta_1\oplus \R\beta_2$ the  $\Fib$\textit{-plane}. The roots of $\Fib$, $\Delta$, lie in the $\Fib$ root lattice, $Q_\Fib=\Z\beta_1 + \Z\beta_2$.

\begin{figure}[h!]
\begin{center}
\includegraphics[trim = .5cm 2.5cm .1cm .9cm, clip=true, width=3 in]{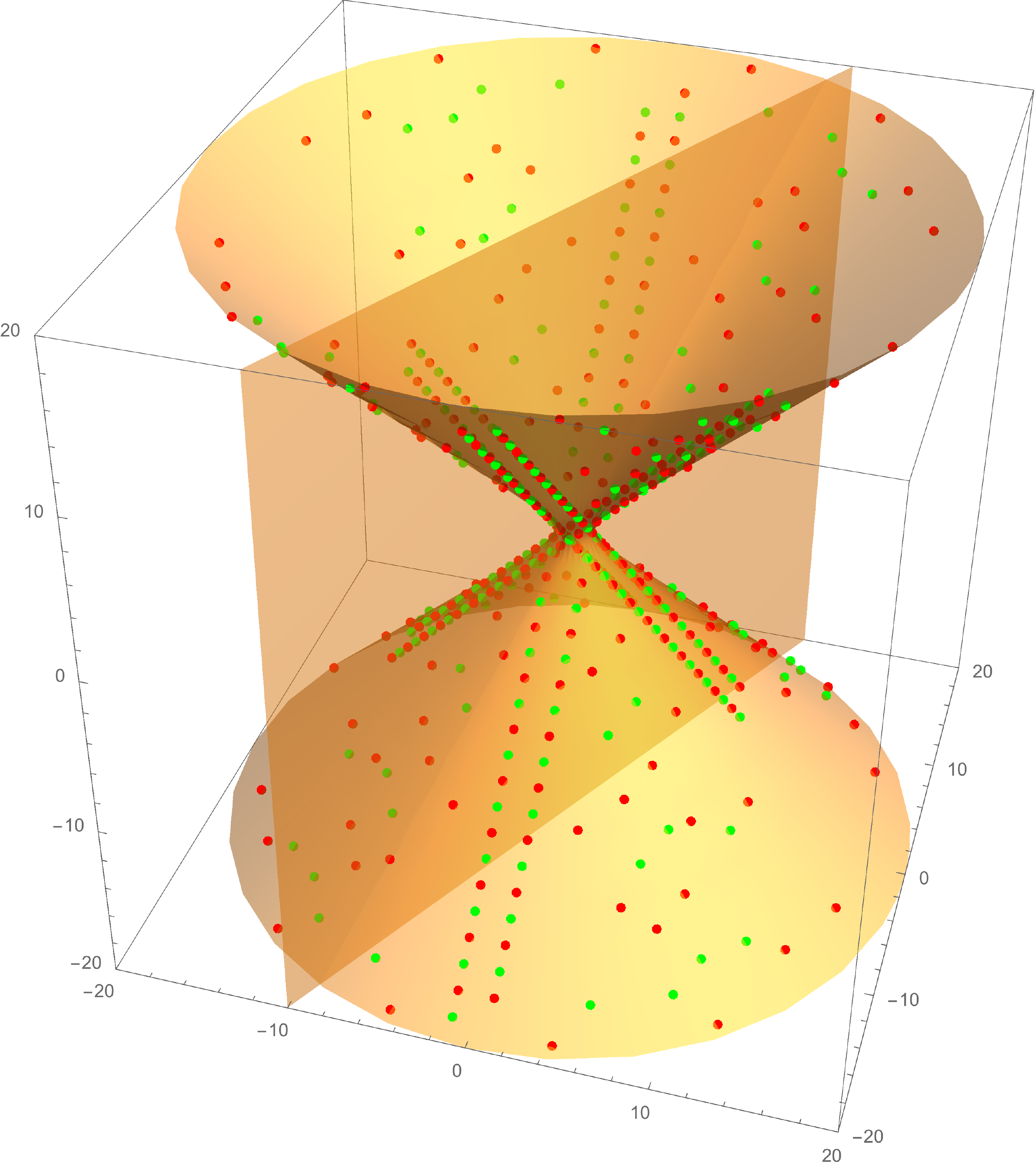} 
\caption{$\Fib$-plane $\R\beta_1\oplus\R\beta_2$ `slicing' through $S_2$, the real root hyperboloid of $\FF$. Shown are real roots of $\FF$. The two colors correspond to the two Weyl orbits of $\Phi^{re}$. }
%\label{fig:tits}
\end{center}\end{figure}

\begin{proposition} We have $Q_\Fib=Q\cap(\R\beta_1\oplus \R\beta_2).$
\end{proposition}
\begin{proof}
The containment $\subseteq$ is obvious. For the reverse containment, let $\alpha\in Q\cap(\R\beta_1\oplus \R\beta_2)$. Then $\alpha=n_1\alpha_1+n_2\alpha_2+n_3\alpha_3=c_1\beta_1+c_2\beta_2$, where each $n_i\in\Z$ and each $c_i\in\R$. Then
$$n_1\alpha_1+n_2\alpha_2+n_3\alpha_3 = 2c_1\alpha_1+(c_1+c_2)\alpha_2+c_1\alpha_3$$
implies that $c_1,c_2\in\Z$, so $\alpha\in Q_\Fib$.
\end{proof}
Since $(\HH^*_\Fib)_\R$ inherits the bilinear form from $\HH_\R^*$, we have
$$\Delta^{re}=\{\beta\in Q_\Fib \ | \ ||\beta||^2=2 \} = \Phi^{re}\cap(\R\beta_1\oplus \R\beta_2),$$
$$\Delta^{im}=\{\beta\in Q_\Fib \ | \ ||\beta||^2\leq 0 \} = \Phi^{im}\cap(\R\beta_1\oplus \R\beta_2),$$
and
$$\Delta=\Phi\cap(\R\beta_1\oplus \R\beta_2).$$
For any root $\beta\in\Delta,$
$$\beta= n_1\beta_1 + n_2\beta_2 = \left[\begin{array}{cc}
	-n_2 & n_1-n_2 \\
	n_1-n_2 & -n_1 \\
		\end{array}\right] =\left[\begin{array}{cc}
	a & b \\
	b & c \\
		\end{array}\right]$$
for some integers $n_1, n_2$ both non-negative or non-positive, so $a = b + c$
describes the $\Fib$-plane (for $a,b,c\in \R$).% In fact, this relation also holds for $\beta\in \R\beta_1\oplus \R\beta_2$.

Define $\lambda_1, \lambda_2 \in \R\beta_1\oplus\R\beta_2$ to be fundamental weights for $\Fib$, in other words, 
\begin{equation}\label{eq:fund}\la \beta_i, \lambda_j \ra = \lambda_j(H_i) = \delta_{ij}.\end{equation}
Since $B(3)^{-1} = -\frac{1}{5} \ds\left(\begin{array}{cc}
	2 & 3  \\
	3 & 2 \\
	\end{array}\right),$ we have
$$\lambda_1=-\frac{1}{5}(2\beta_1+3\beta_2) = \frac{1}{5}\left[\begin{array}{cc}
	3 & 1  \\
	1 & 2 \\
	\end{array}\right] \qquad \hbox{and}\qquad \lambda_2=-\frac{1}{5}(3\beta_1+2\beta_2) = \frac{1}{5}\left[\begin{array}{cc}
	2 & -1  \\
	-1 & 3 \\
	\end{array}\right].$$
Denote the weight lattice determined by these fundamental weights by
$$P_{\Fib}=\Z\lambda_1+\Z\lambda_2.$$

We have the simple reflections $r_1=w_1w_2w_3w_2w_1$ and $r_2=w_2$, since
\begin{align*}r_1\beta_1&=w_1w_2w_3w_2w_1(2\alpha_1+\alpha_2+\alpha_3) = w_1w_2w_3w_2(\alpha_2+\alpha_3)=w_1w_2w_3(\alpha_3)\\
&=w_1w_2(-\alpha_3)=-w_1(\alpha_2+\alpha_3)=-(2\alpha_1+\alpha_2+\alpha_3) =-\beta_1 =\beta_1-\la \beta_1,\beta_1\ra \beta_1,
\\
r_1\beta_2&=w_1w_2w_3w_2w_1\alpha_2 = w_1w_2w_3w_2(2\alpha_1+\alpha_2)=w_1w_2w_3(2\alpha_1+3\alpha_2)\\
&=w_1w_2(2\alpha_1+3\alpha_2+3\alpha_3)=w_1(2\alpha_1+4\alpha_2+3\alpha_3)=(6\alpha_1+4\alpha_2+3\alpha_3)\\
&=3\beta_1+\beta_2 =\beta_2-\la \beta_1,\beta_2\ra \beta_1,
\\
r_2\beta_2&=w_2\alpha_2=-\alpha_2=-\beta_2=\beta_2-\la \beta_2,\beta_2\ra \beta_2,
\\
r_2\beta_1&=w_2(2\alpha_1+\alpha_2+\alpha_3) = 2\alpha_1+4\alpha_2+\alpha_3=\beta_1+3\beta_2 =\beta_1-\la \beta_2,\beta_1\ra \beta_2. 
\end{align*}
Thus the matrices corresponding to $r_1, r_2$ (as in section \ref{sec:F}) are
$$M_{r_1} = M_1M_2M_3M_2M_1 =  \left[\begin{array}{cc}
	-1 & 0 \\
	1 & 1 \\
		\end{array}\right], \quad  \text{ and } \quad  M_{r_2} = M_2 = \left[\begin{array}{cc}
	-1 & 1 \\
	0 & 1 \\
		\end{array}\right].$$
		
%For convenience, we do not distinguish between a weight in $P$ and its restriction to the subspaces of lower dimension, for example, $\lambda\in P$ is also considered as an element of $P_\Fib$.  

The Weyl group $W_\Fib=\la r_1, r_2 \ra < W$ permutes the roots $\Delta$ which is a subset of the $\Fib$-plane $\R\beta_1\oplus \R \beta_2\simeq \R^{(1,1)}\subset\R^{(2,1)}\simeq\HH^*_\R$. We also have that $\Delta^{re}=W_\Fib\Pi_\Fib$ and $\Delta^{im}=\Delta \backslash \Delta^{re}$.

We have the sets of real fixed points in $\HH_\R^*$ for $r_1, r_2$,
$$\beta_i^\perp = \{ \beta \in \HH^*_\R \ | \ (\beta, \beta_i) = 0 \}$$ 
whose intersection with the $\Fib$-plane are the reflecting lines for $W_\Fib$,
$$\beta_1^\perp\cap(\R\beta_1\oplus\R\beta_2)=\R\lambda_2 \quad \text{and} \quad \beta_2^\perp\cap(\R\beta_1\oplus\R\beta_2)=\R\lambda_1.$$ 
%= \{ \beta \in \HH^*_\R \ | \ (\beta, \beta_i) = 0 \}$$ 
We also have that $\beta_2^\perp = \alpha_2^\perp$. The reflecting planes $\alpha_1^\perp$ and $\alpha_3^\perp$ intersect the $\Fib$-plane in the same line, 
$\R(\lambda_1+\lambda_2) =\alpha_1^\perp\cap\alpha_3^\perp$.

\begin{figure}[h!]
\begin{center}
\includegraphics[width=3.5 in]{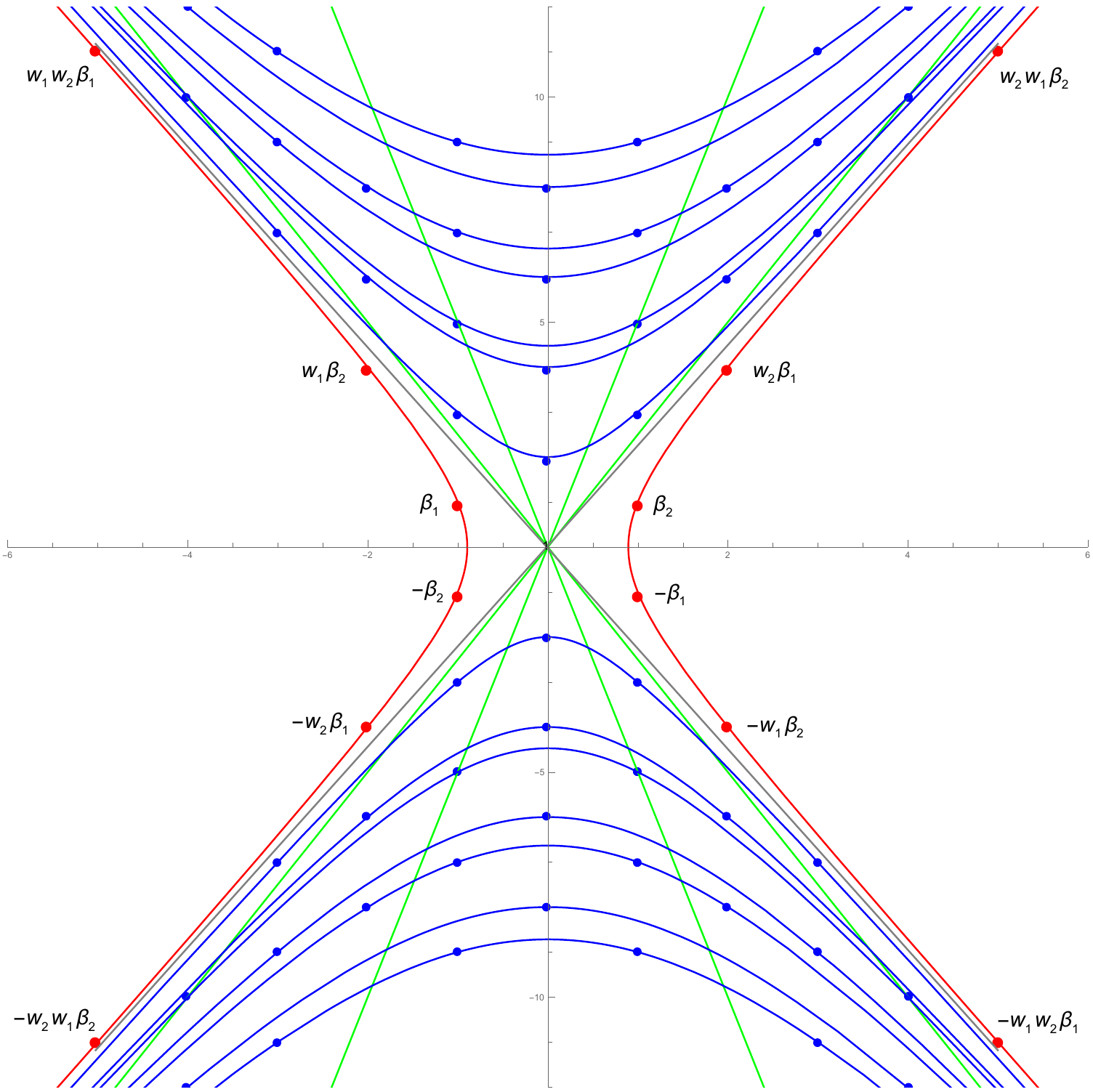}\caption{The $\Fib$ root system $\Delta$. Real roots lie on the red hyperbola, imaginary roots lie on the blue hyperbolas. The inner green lines show the reflecting planes $\R\lambda_1$, $\R\lambda_2$. The gray lines are the $\Fib$ null-cone.}
\label{fig:fibplane}
\end{center}\end{figure}

The curves of constant squared-length in the $\Fib$-plane are
$$S_{\Fib,c}=S_c\cap(\R\beta_1\oplus\R\beta_2).$$

$S_{\Fib,0}$ is the $\Fib$ \textit{null-cone}, a pair of lines that intersect at the origin. If $\mu =c_1\beta_1+c_2\beta_2 \in S_{\Fib,0}$, then $\la \mu, \mu \ra = 2c_1^2-3c_1c_2+2c_2^2 =0$. Dividing both sides by $c_1^2$ gives $(\frac{c_2}{c_1})^2 -3 (\frac{c_2}{c_1}) +1 = 0$. The solution $\frac{c_2}{c_1} = \frac{3 \pm \sqrt{5}}{2}$ shows that the null-cone lines have irrational slope, therefore no roots of $\Fib$ lie on the null-cone.

For each even $n\leq 2$, $n\neq 0$, $S_{\Fib,n}$ is a disjoint union of two branches of a hyperbola. Each branch of the \textit{real hyperbola} $S_{\Fib,2}$ contains roots of both orbits of $\Delta^{re}=W_\Fib\{\beta_1\} \cup W_\Fib\{\beta_2\}$ (\cite{F1}). We have the $\Fib$ \textit{light-cone }
$$LC_\Fib = LC\cap(\R\beta_1\oplus\R\beta_2)$$
which partitions into the \textit{forward} and \textit{backward} $\Fib$ \textit{light-cones},
$$LC_{\Fib,\pm}=LC_\pm\cap(\R\beta_1\oplus\R\beta_2),$$
and the positive $\Fib$ Tits cone,
$$T_{\Fib,+}=T_+\cap (\R\beta_1\oplus\R\beta_2).$$
Figure \ref{fig:fibplane} indicates, for each even $n> 0$, there is a ``positive'' branch of $S_{\Fib, -n}$ that lies in $LC_{\Fib,+}$ and a ``negative'' branch that lies in $LC_{\Fib,-}$. Since each branch is $W_\Fib$-invariant, we have that $W_\Fib(LC_{\Fib,\pm})=LC_{\Fib,\pm}$, and
\begin{equation}\label{eq:Winv}W_\Fib(LC_{\Fib,+})\cap W_\Fib(LC_{\Fib,-})=\{0\}.\end{equation}
%-------------
%-------------
\chapter{$\Fib$-modules in $\FF$}\label{ch:modules}
Our main goal is to study the decomposition of $\FF$ with respect to $\Fib$ that is analogous to the Feingold-Frenkel decomposition of $\FF$ into $\Aff$-modules outlined in Section \ref{sec:aff}. We first introduce the similar notion of \textit{Fibonacci level} by slicing $\Phi$ into planes parallel to the $\Fib$-plane, and show that $\FF$ has a decomposition into $\Fib$-modules, graded by Fibonacci level. From this point onwards ``level'' shall refer to the Fibonacci level.  The next two chapters will then explore the decompositions of certain levels in greater detail. 

\section{The level $m$ $\Fib$-module $\Fib(m)$, $m\in\Z$}\label{sec:fiblevel}

The  \textit{level }$m$ $\Fib$\textit{-plane} is $\R\beta_1\oplus\R\beta_2 + m\gamma_\Fib$
where $\gamma_\Fib=-\alpha_1-\alpha_2=\left[\begin{array}{cc}
	1 & 0  \\
	0 & 0 \\
	\end{array}\right]$. Then for each $m\in\Z$, the level $m$ roots are 
$$\Delta_m=\Delta_{m}=\Bigg\{\beta\in\Phi \ \Bigg| \ \beta=\left[\begin{array}{cc}
	b+c+m & b  \\
	b & c \\
	\end{array}\right], \ b, c, m\in\Z \Bigg\}.$$
Moreover, given any root $\alpha=\left[\begin{array}{cc}
	a & b  \\
	b & c \\
	\end{array}\right] \in \Phi$, we have $\alpha\in\Delta_{a-b-c},$
thus we have the disjoint union
$$\Phi=\bigcup_{m\in\Z}\Delta_m.$$

It is clear that if $\beta\in\Delta_m$ then $-\beta\in\Delta_{-m}$ so that $\Delta_{-m}=-\Delta_m.$

As in Section \ref{sec:aff}, we use the inner product to determine the level of any root in $\Phi$. Let $\delta_\Fib=\alpha_3-\frac{1}{2}\alpha_1=\left[\begin{array}{cc}
	1 & -\frac{1}{2}  \\
	-\frac{1}{2} & -1 \\
		\end{array}\right]\in P$. Then for any $\beta\in\Delta_m$,
$$\la \delta_\Fib, \beta \ra = \Bigg\langle \left[\begin{array}{cc}
        1 & -\frac{1}{2}  \\
	-\frac{1}{2} & -1 \\
		\end{array}\right], \left[\begin{array}{cc}
	b+c+m & b  \\
	b & c \\
		\end{array}\right] \Bigg\rangle = 2\Big(-\frac{1}{2}\Big)b - c -(-1)(b+c+m) = m.$$
We define $\Fib$\textit{-level }$m$ to be 
\begin{equation}\label{eq:fibm}\Fib(m)= \bigoplus_{\beta\in\Delta_m}\FF_\beta.\end{equation}
The levels provide a $\Z$-grading of $\FF$ according to planes parallel to the $\Fib$-plane, 
\begin{equation}\label{eq:grading}\FF=\bigoplus_{m\in\Z}  \Fib(m).\end{equation}

 \begin{figure}[!ht]
            \hfill
    \subfloat[Level -3\label{subfig-2:dummy}]{%
      \includegraphics[trim = 0.17cm 0.1cm 1.5cm 0cm, clip=true,width=0.32\textwidth]{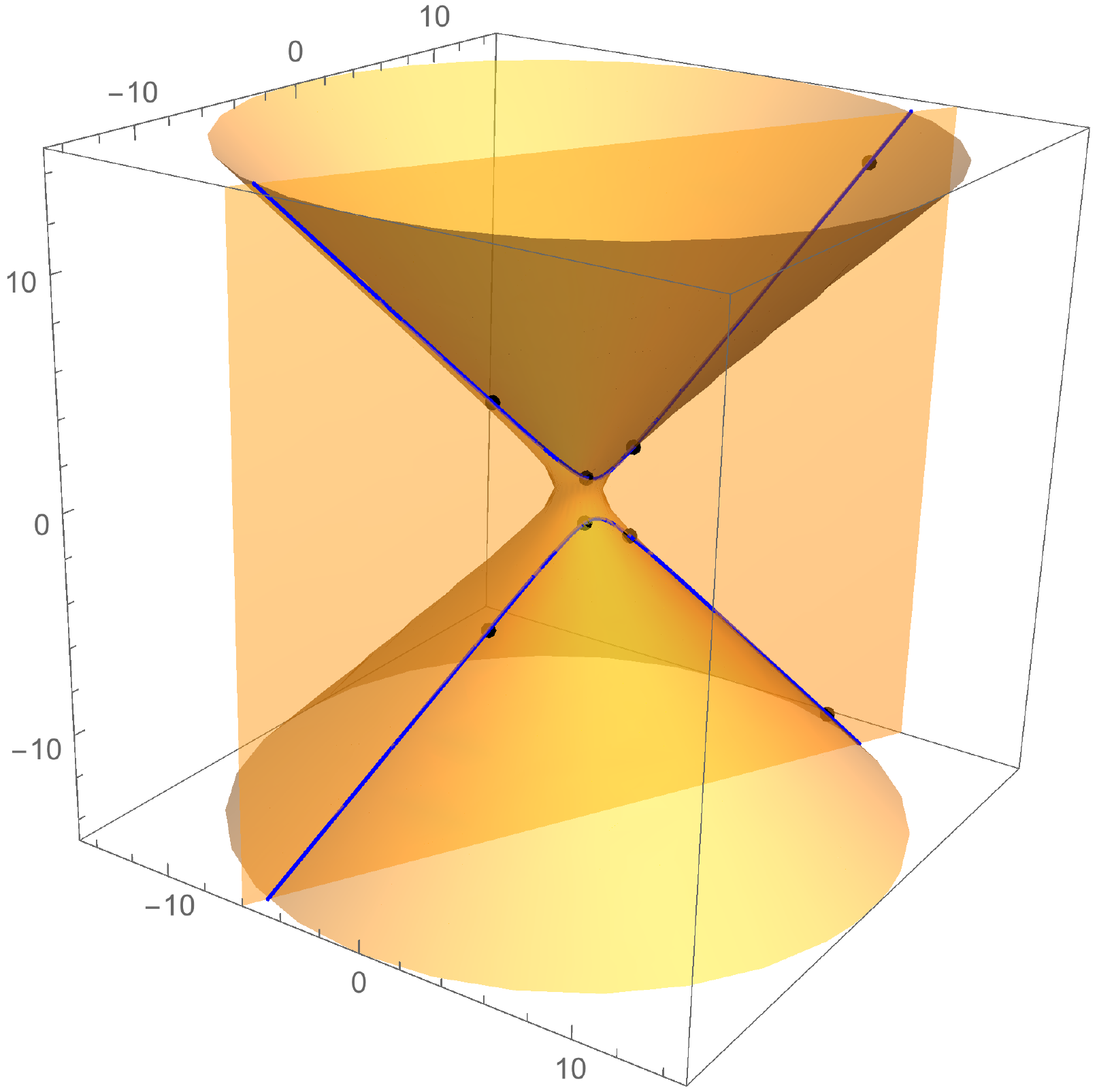}
    } 
        \hfill
    \subfloat[Level -2\label{subfig-2:}]{%
      \includegraphics[trim = 0.17cm 0.1cm 1.5cm 0cm, clip=true,width=0.32\textwidth]{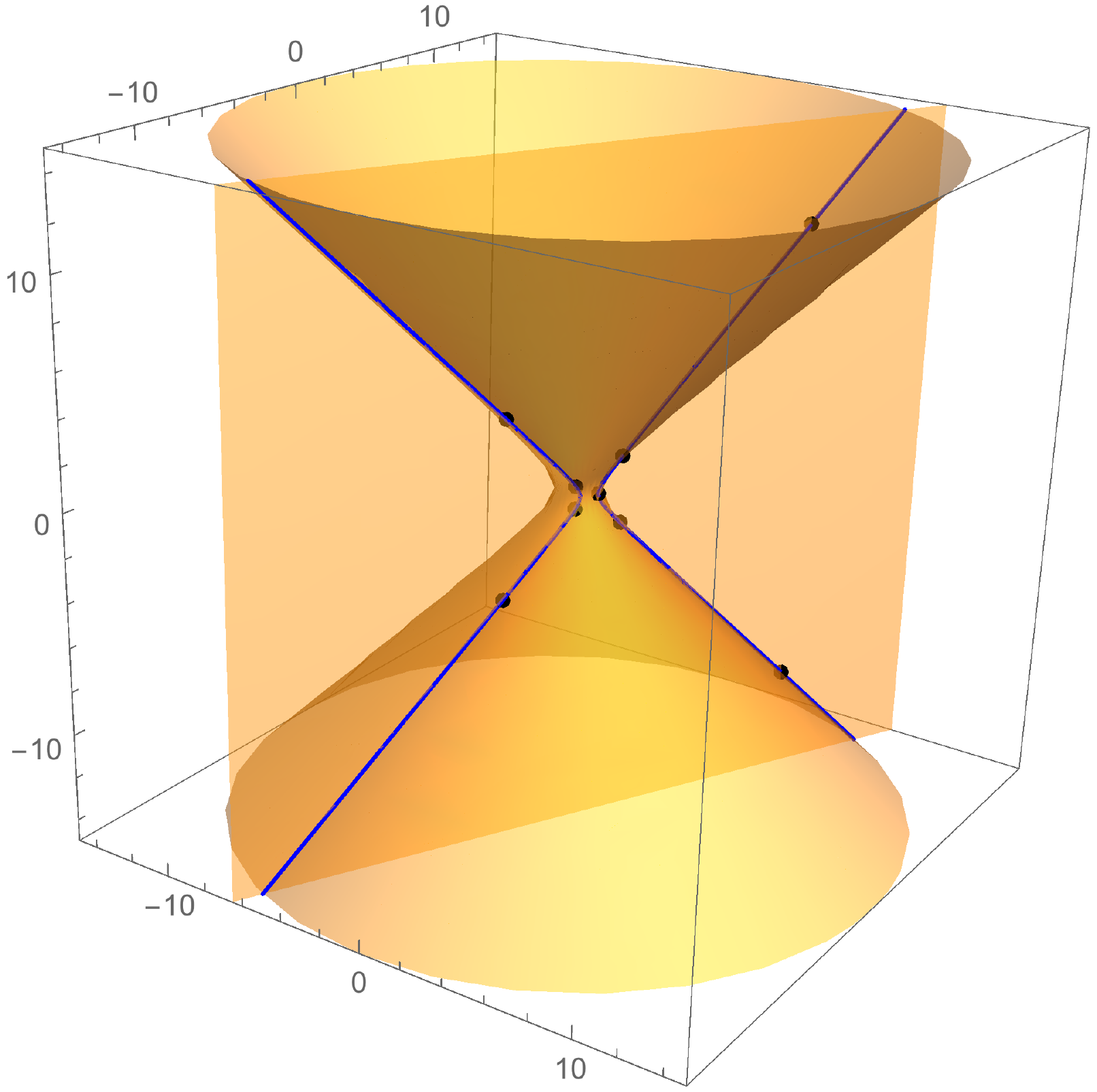}
    } 
        \hfill
    \subfloat[Level -1\label{subfig-2:slice1}]{%
      \includegraphics[trim = 0.17cm 0.1cm 1.5cm 0cm, clip=true,width=0.32\textwidth]{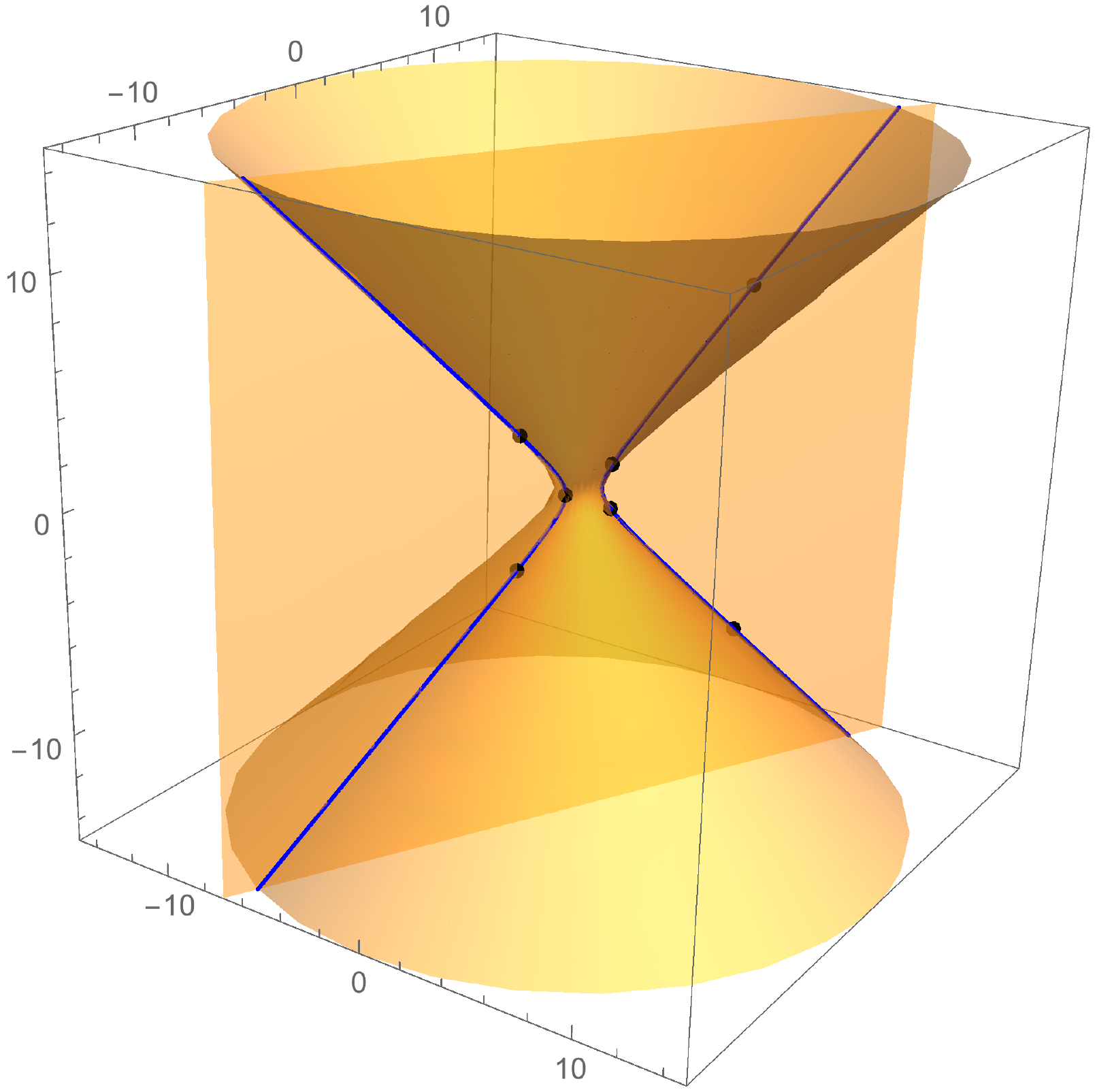}
    } 
    \\
        \hfill
    \subfloat[Level 0\label{subfig-2:slice0}]{%
      \includegraphics[trim = 0.17cm 0.1cm 1.5cm 0cm, clip=true,width=0.32\textwidth]{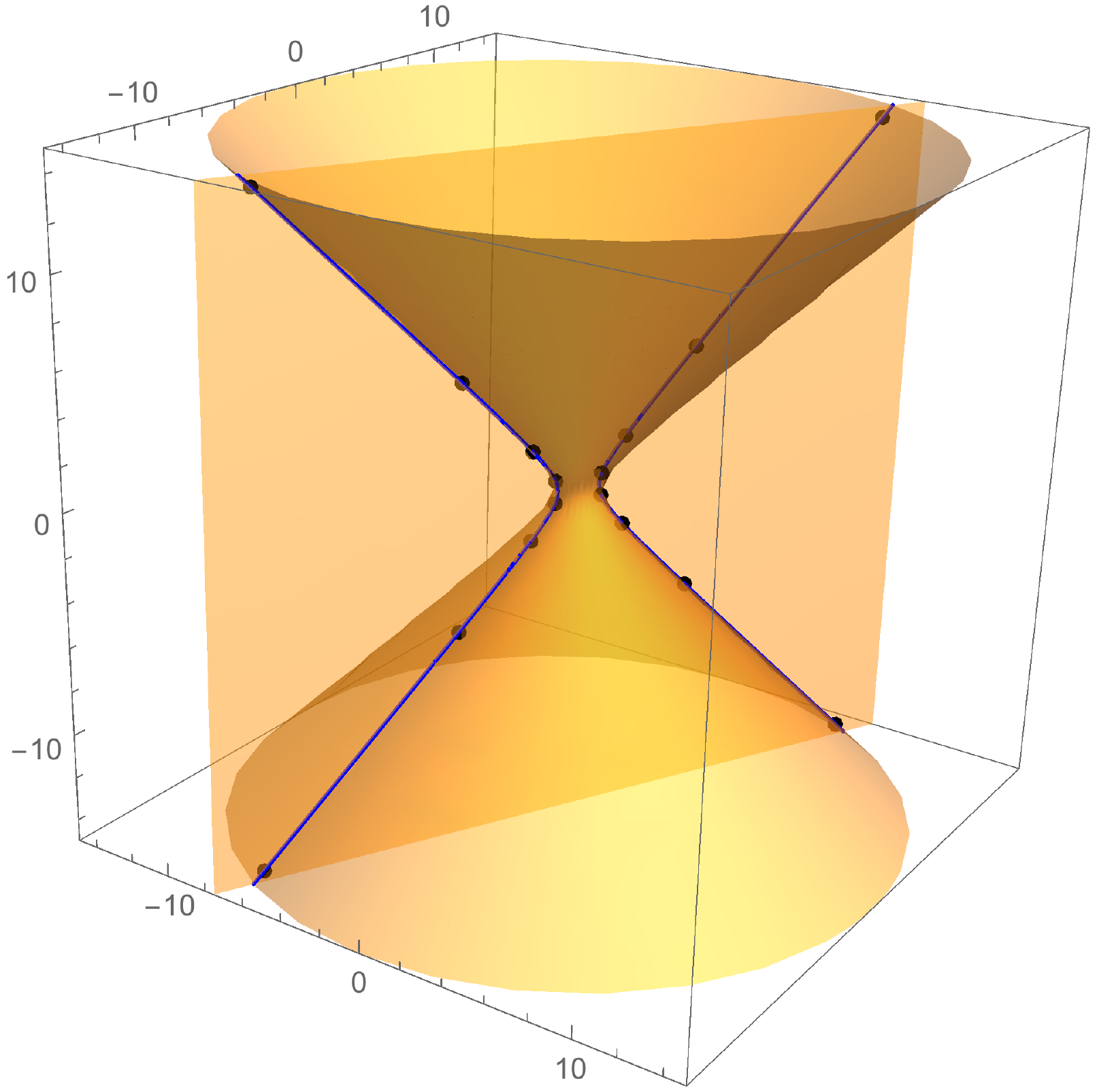}
    } 
    \hfill
    \subfloat[Level 1\label{subfig-2:slice-1}]{%
      \includegraphics[trim = 0.17cm 0.1cm 1.5cm 0cm, clip=true,width=0.32\textwidth]{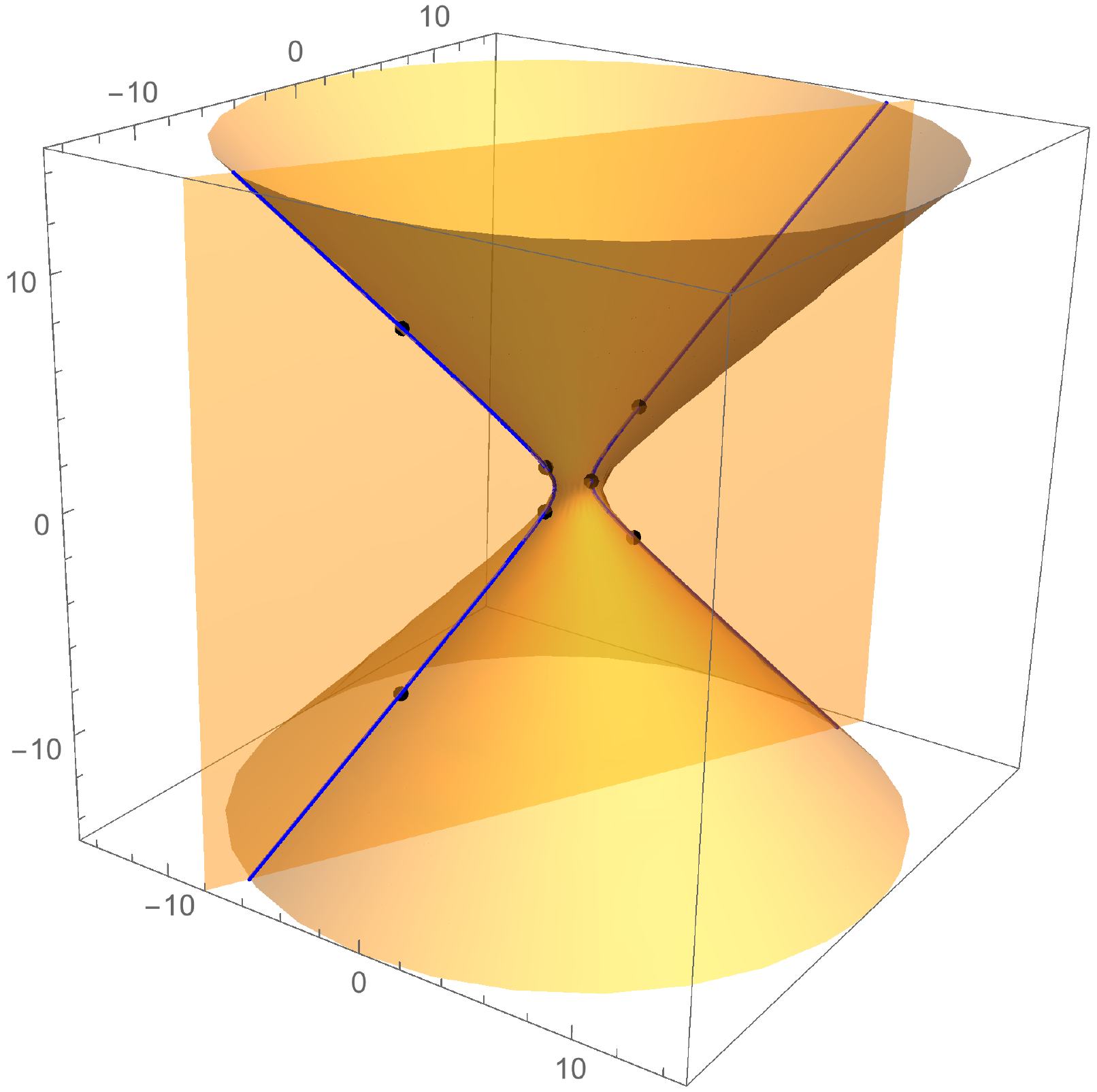}
    } 
    \subfloat[Level 2 \label{subfig-1:slice-2}]{%
      \includegraphics[trim = 0.17cm 0.1cm 1.5cm 0cm, clip=true,width=0.32\textwidth]{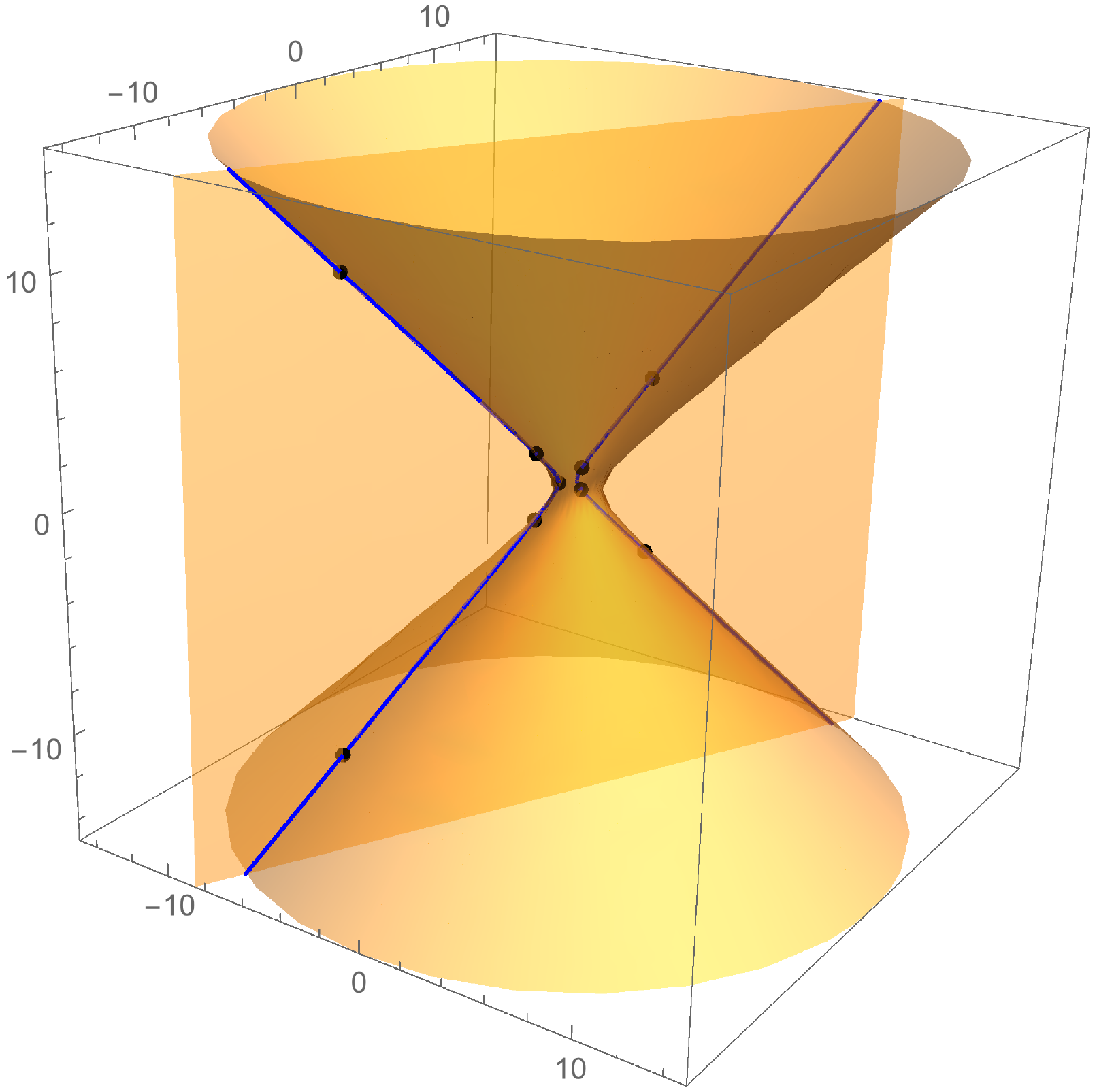}
    }
    \caption{$\Fib$ levels -3 through 2 slicing through the hyperboloid $S_2$. Imaginary roots on each level are not shown. Real roots lie on the blue hyperbolas $S_2\cap (\R\beta_1\oplus\R\beta_2+m\gamma)$.}
    \label{fig:slices}
  \end{figure}

Let $\pi:\HH_\R^*\rightarrow \R\beta_1\oplus \R\beta_2$ be the projection map given by
\begin{equation}\label{eq:proj}\pi(\mu) = \mu(H_1)\lambda_1 + \mu(H_2)\lambda_2.\end{equation}
Note that for $H\in(\HH_\Fib)_\R$, $\pi(\mu)(H)=\mu(H)$, and also, $W_\Fib$ and $\pi$ commute.

\begin{proposition}\label{fibmintegrable}  For $m\in\Z$, $\Fib(m)$ is an $\HH_\Fib$-diagonalizable, integrable $\Fib$-module where the action is the restriction to $\Fib$ of the adjoint action of $\FF$.  In particular, 
\begin{itemize}
\item[1)]  the weights of $\Fib(m)$ are $P(\Fib(m))= \pi(\Delta_m)$,
\item[2)] if $\beta\in\Delta_m$ and $x\in\FF_\beta$, then $ad_{H} x = \pi(\beta)(H) x$ for $H\in\HH_\Fib$,
\item[3)]  $\Delta_m$ and $P(\Fib(m))$ are $W_\Fib$-invariant, 
\item[4)]  $\pi(\beta_i^\perp)=\R\lambda_{3-i} = \beta_i^\perp\cap(\R\beta_1\oplus\R\beta_2)$ for $i=1,2$,
\item[5)]  $\pi(P^\pm)\subset P^\pm_\Fib$.
\end{itemize}
\end{proposition}

\begin{proof} 
Let $\beta= \left[\begin{array}{cc}
	b+c+m & b  \\
	b & c \\
		\end{array}\right]\in\Delta_m$ and $x\in\FF_\beta$. Then
$$ad_{E_i} x = [E_i, x ] \in \FF_{\beta_i+\beta}\subset \Fib(m) \quad \text{ and } \quad ad_{F_i} x = [F_i, x ] \in \FF_{-\beta_i+\beta}\subset \Fib(m),$$
and the action of each $E_i$ and $F_i$ for $i=1,2$ is locally nilpotent on $\Fib(m)$ since it is a multibracket of the Serre generators of $\FF$, which have locally nilpotent action on all of $\FF$. Moreover, $\Fib(m)$ is $\HH_\Fib$-diagonalizable since for $\beta\in\Delta_m$ and $x\in\FF_\beta$,
$$ad_{H} x = [H, x ] =\beta(H)x = \pi(\beta)(H)x \in \FF_{\beta} \quad \text{ for  } H\in\HH_\Fib,$$
which proves $1)$ and $2)$. 

We have
$$r_1\beta =  \left[\begin{array}{cc}
	-1 & 0 \\
	1 & 1 \\
		\end{array}\right] \left[\begin{array}{cc}
	b+c+m & b  \\
	b & c \\
		\end{array}\right] \left[\begin{array}{cc}
	-1 & 1 \\
	0 & 1 \\
		\end{array}\right] =  \left[\begin{array}{cc}
	b+c+m & -2b-c-m  \\
	-2b-c-m & 3b+2c+m \\
		\end{array}\right],$$
thus $r_1\beta \in \Delta_m$. Similarly, $r_2\beta\in\Delta_m$, since
$$r_2\beta =  \left[\begin{array}{cc}
	-1 & 1 \\
	0 & 1 \\
		\end{array}\right] \left[\begin{array}{cc}
	b+c+m & b  \\
	b & c \\
		\end{array}\right] \left[\begin{array}{cc}
	-1 & 0 \\
	1 & 1 \\
		\end{array}\right] =  \left[\begin{array}{cc}
	2c-b+m & c-b \\
	c-b & c \\
		\end{array}\right],$$
		and since $W_\Fib$ and $\pi$ commute, we have $W_\Fib P(\Fib(m))=P(\Fib(m))$, proving 3).

To show the projection map preserves the reflecting planes of $W_\Fib$, let $\mu\in\beta_i^\perp$ for $i=1$ or 2, so $\mu(H_i)=0$. Then $\pi(\mu) = \mu(H_1)\lambda_1 + \mu(H_2) \lambda_2 = \mu(H_{3-i})\lambda_{3-i}$ so $\pi(\mu) \in \beta^\perp_i$.

For $5)$, let $\mu\in P^-$. Then $\pi(\mu)=\mu(H_1)\lambda_1+\mu(H_2)\lambda_2 = \la \beta_1,\mu\ra \lambda_1 + \la \beta_2,\mu \ra \lambda_2$.  Since $\la \alpha_i, \mu \ra =  \frac{2(\alpha_i,\mu)}{(\alpha_i,\alpha_i)} = (\alpha_i,\mu) <0$ it follows that $\la \beta_i, \mu \ra = \frac{2(\beta_i,\mu)}{(\beta_i,\beta_i)} = (\beta_i,\mu) <0$.		
		
\end{proof}

When viewing $\Fib(m)$ as a $\Fib$-module, we write the weight-space decomposition as
\begin{equation}\label{eq:fibmdecomp}\Fib(m)=\ds\bigoplus_{\lambda\in P(\Fib(m))} \Fib(m)_\lambda.\end{equation}
We also have the multiplicity of $\lambda$ in $\Fib(m)$,
\begin{equation}\label{eq:multm}Mult_m(\lambda) = Mult_{\Fib(m)}(\lambda) = \dim_\FF (\FF_\beta),\end{equation}
where $\beta\in\Delta_m$ is such that $\pi(\beta)=\lambda$.
\begin{definition}
If $V^\lambda$ is a standard or non-standard $\Fib$-module with set of weights $P(V^\lambda)$, then for $\mu\in P(V^\lambda)$, the \textbf{inner multiplicity} of $\mu$ in $V^\lambda$ is
$$Mult_\lambda(\mu) = \dim_{V^\lambda}(V^\lambda_\mu).$$
\end{definition}

%%%Recall for any $X\in\FF_\mu$ we have $[H_i,X]=\mu(H_i)X = (a_1\lambda_1+a_2\lambda_2)(H_i)X$ which implies that $a_i=\mu(H_i)$ since $\lambda_i(H_j)=\delta_{ij}$.

%---Where is the result of Kac? Closest I found is section 11.9 in \cite{K2}

\section{Symmetries and cosets}\label{sec:symcos}
The Cartan involution of $\Fib$ is the restriction to $\Fib$ of $\nu:\FF\rightarrow\FF$ determined by $\nu(e_i)=-f_i, \ \nu(f_i)=-e_i$ for $1\leq i \leq 3$, and $\nu(h)=-h$ for $h\in \HH$.  Since for $\alpha\in\Phi$, $\nu(\FF_\alpha)= \FF_{-\alpha}$ we also let $\nu(\alpha)=-\alpha$. Note that $Mult_\FF(\alpha)=Mult_\FF(-\alpha)$.

\begin{rem}\label{contra}It is sufficient to investigate the $\Fib(m)$ decomposition for $m\geq 0$, since for all $m\in\Z,  \ \nu(\Fib(m))=\Fib(-m)$. Moreover, 
\begin{enumerate}[1)]
\item $V^\lambda\subset \Fib(m)$ is an irreducible HW $\Fib$-module with HWV $v_\lambda$ if and only if $V^{-\lambda}\subset\Fib(-m)$ is a irreducible LW $\Fib$-module with LWV $\nu(v_\lambda) = v_{-\lambda}.$ 
\item  If $V^\lambda\subset \Fib(m)$ is a non-standard irreducible $\Fib$-module with a generating vector $v_\lambda$ of weight $\lambda$, then $V^{-\lambda}\subset \Fib(-m)$ is a non-standard irreducible $\Fib$-module in $\Fib(-m)$ with generating vector $\nu(v_\lambda)$ of weight $-\lambda$.
\end{enumerate}
\end{rem}

The following lemma describes a symmetry in each $\Fib$-level that allows us to consider only positive weights in the decomposition of $\Fib(m)$ for each $m\in\Z$.

\begin{lemma}
Let $\psi=w_1w_3\nu: \Phi \rightarrow \Phi$. Then $\psi$ is an involution of $\Phi$ such that
\begin{enumerate}[1)]
\item $\psi(\Delta_m)=\Delta_m$,
\item $\psi$ fixes $\pm\alpha_1$ and $\pm \alpha_3$,
\item If $|m| \neq 1, 2$, \  $\psi(\Delta_m\cap\Phi_+) = \Delta_m\cap\Phi_-$, 
\item $\psi(\Delta_1\cap\Phi_+\backslash\{\alpha_1\}) = \Delta_1\cap\Phi_-\backslash\{-\alpha_1\}$,
\item $\psi(\Delta_2\cap\Phi_+\backslash\{\alpha_3\}) = \Delta_2\cap\Phi_-\backslash\{-\alpha_3\}$,
\item $Mult_\FF(\beta)=Mult_\FF(\psi(\beta))$,
\item $\psi$ commutes with $\pi$.
\end{enumerate}
\end{lemma}
\begin{proof}
1) We have $M_{w_1w_3} = \left[\begin{array}{cc}
	1 & 0 \\
	0 & -1 \\
		\end{array}\right] \left[\begin{array}{cc}
	0 & 1 \\
	1 & 0 \\
		\end{array}\right] = \left[\begin{array}{cc}
	0 & 1 \\
	-1 & 0 \\
		\end{array}\right]$. Then
$$\psi(\beta) = -\left[\begin{array}{cc}
	0 & 1 \\
	-1 & 0 \\
		\end{array}\right] \left[\begin{array}{cc}
	a & b \\
	b & c \\
		\end{array}\right] \left[\begin{array}{cc}
	0 & -1 \\
	1 & 0 \\
		\end{array}\right] = \left[\begin{array}{cc}
	-c & b \\
	b & -a \\
		\end{array}\right]\in\Delta_{a-b-c}.$$
In particular, if $a=b+c+m$ then $\psi(\beta)\in\Delta_m$ if $\beta\in\Delta_m$.  

$2-5)$ Also, $||\beta||=ac-b^2\geq -1 \Leftrightarrow ac \geq b^2-1$, which gives us three cases: 
\begin{enumerate}[I.]
\item If $|b|\geq 2$, then $ac>0$, and we have $\beta\in\Phi_+ \Leftrightarrow c<0 \Leftrightarrow -a>0 \Leftrightarrow \psi(\beta)\in\Phi_-$.  
\item If $b=0$, then $ac\geq -1$. Assume $ac=-1$. Then $\beta=\pm \alpha_3$, and $\psi(\pm \alpha_3) = w_1w_3(\mp \alpha_3) = \pm \alpha_3$. Now assume $ac=0$. Then for some $n>0$, $\beta=\pm n(\alpha_1+\alpha_2)\in\Phi_\pm$ or $\mp n(\alpha_2+\alpha_2+\alpha_3)\in\Phi_\mp$, and an easy computation shows that $\psi$ permutes these two sets of roots. 
\item If $|b|=1$, then $ac \geq 0$. If $ac=0$, $\beta=\pm\alpha_1$, and $\psi(\beta)=w_1w_3(\mp\alpha_1)=\pm\alpha_1$. If $ac>0$ then $\beta\in \Phi^{im}$, and the argument from Case I still holds, so $\beta\in\Phi_+\Leftrightarrow \psi(\beta)\in\Phi_-$.
\end{enumerate}
$6)$ follows from $W$- and $\nu$-invariance of $Mult_\FF(\cdot)$ and 
$7)$ follows because both $W_\Fib$ and $\nu$ commute with $\pi$.

\end{proof}
\begin{rem}
For $m\in\Z$, $\psi$ corresponds to a reflection across the horizontal axis in the weight diagram for $\Fib(m)$ (when viewed as a $\Fib$-module). Moreover for each weight $\lambda\in\pi(\Delta_m)$, we have 
$$Mult_m(\lambda)=Mult_m(\psi(\lambda)).$$  
Thus, when determining the decomposition of $\Fib(m)$ into irreducible $\Fib$-modules, it is enough to consider only positive weights.
\end{rem}

The $\Fib$ root lattice $Q_{\Fib}$ is an index 5 sublattice of $P_{\Fib}$, so the quotient module $P_{\Fib}/Q_{\Fib}$ consists of 5 cosets.  We denote the $i^{th}$ coset of $P_{\Fib}/Q_{\Fib}$ by $K_i$ where $i\in \{-2,-1,0,1,2\}$. Then 
\begin{equation}\label{eq:coset}K_i=\{ \Lambda_i + n_1\beta_1 + n_2\beta_2 \mid n_1, n_2 \in \Z \},\end{equation}
where the representative weight $\Lambda_i\in \pi(\Delta_i)$ is from Table \ref{tab:cosetreps}.
 
 \begin{table}[h!]
 \begin{center}
\begin{tabular}{| c | c | c | c |}
  \hline                       
    $m$ (level) & $\Lambda_m$ & weight & projection \\
\hline
  $-2$ & $\Lambda_{-2}$ &  $-(\lambda_1-\lambda_2)$ & $=\pi(-\alpha_3)$ \\
    \hline      
  $-1$ &$\Lambda_{-1}$  &  $2(\lambda_1-\lambda_2) $ & $= \pi(\alpha_1)$ \\
    \hline      
  $0$ &  $\Lambda_0$  & $0$ & $0$ \\
      \hline      
   $1$ &$\Lambda_1$ & $-2(\lambda_1-\lambda_2) $ & $ =\pi(-\alpha_1)$ \\
     \hline      
  $2$ &$\Lambda_2$  & $\lambda_1-\lambda_2 $ & $ =\pi(\alpha_3)$\\
  \hline  
\end{tabular}
\caption{Coset representative $\Lambda_m=\pi(\beta)$ for some $\beta\in \Delta_m^{re}$,  $|m|\leq 2$.}\label{tab:cosetreps}
\end{center}
\end{table}
\begin{figure}[h!]
\begin{center}
\includegraphics[width=4 in]{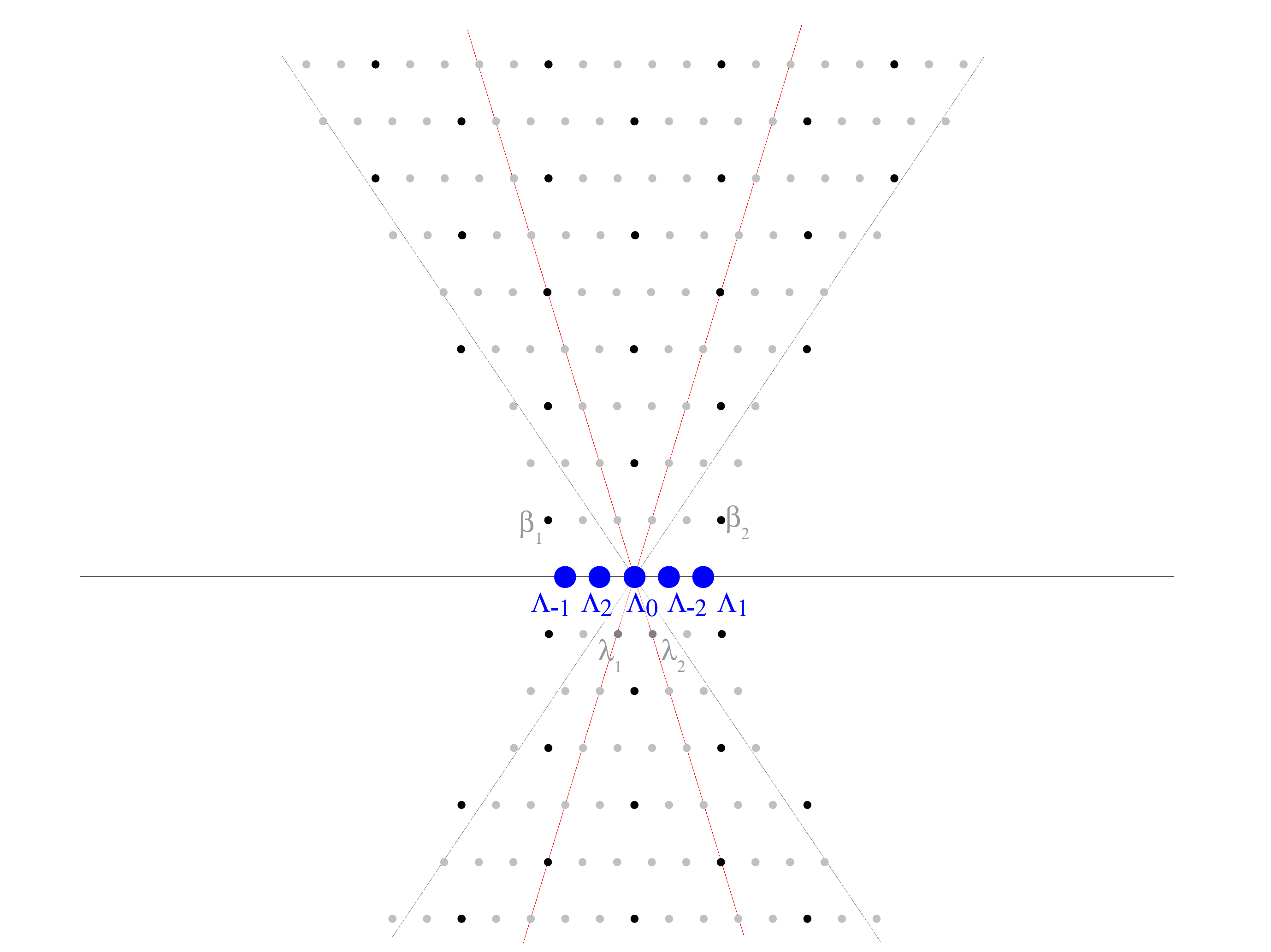}\caption{$Q_{\Fib}$ is an index 5 sublattice of $P_{\Fib}$ in the $\Fib$-plane. This figure shows only those lattice points corresponding to real or imaginary weights of $\Fib$.}
\label{fig:cosets}
\end{center}\end{figure}
Each root of $\Phi$ projects to a weight in one of the five cosets $K_i$ (each invariant under $W_\Fib$), and all roots from the same level project to weights in the same coset. In proposition \ref{mod} we use $a \ mod \ 5$ to mean the unique integer $b\in\{0, 1, 2, 3, 4\}$ such that $b\equiv a \ mod \ 5$.

\begin{proposition}\label{mod}
Let $m\in\Z$. If $\beta\in\Delta_m$, then $\pi(\beta)\in K_{(m+2)(mod \ 5)-2}$.
\end{proposition}
\begin{proof}
We have for any $m\in\Z$,
$$\beta=\left[\begin{array}{cc}
	b+c+m & b  \\
	b & c \\
	\end{array}\right]=\left[\begin{array}{cc}
	b+c & b  \\
	b & c \\
	\end{array}\right]+\left[\begin{array}{cc}
	m & 0  \\
	0 & 0\\
	\end{array}\right] = \beta' +m\lambda_1, \quad \text{where} \ \beta' = \left[\begin{array}{cc}
	b+c & b  \\
	b & c \\
	\end{array}\right]\in\Delta,$$
giving the projection $\pi(\beta) = \beta'+m\lambda_1 \in Q_{\Fib}+m\lambda_1$. Thus $\pi(\beta)$ and $m\lambda_1$ lie in the same coset $K_{i(m)}$, where $i(m)\in\{-2,-1,0,1,2\}$.  Note also that $K_{i(m)}$ is invariant under addition by $\Z$-linear combinations of the simple roots, so that 
$$m\lambda_1+m\beta_2=-2m(\lambda_1-\lambda_2)$$ 
also lies in $K_{i(m)}$. Using Table \ref{tab:cosetreps}, we have $i(m)=(m+2)(mod \ 5)-2$. 
\end{proof}
It can be seen that the projection map partitions the levels with respect to $\Fib$. Let $\mathfrak{L}=\{ \Delta_m \ | \ m\in\Z \}$. For $|m|\leq2$, the pullback $\pi^{-1}(K_{m})=\{\Delta_j \ | \ j \equiv m(mod \ 5)\}$ is an equivalence class of levels, giving the partition
$\mathfrak{L}=\ds\bigsqcup_{|m|\leq 2} \pi^{-1}(K_{m}).$

%----------------------------------------------------------------------------------------------------
%----------------------------------------------------------------------------------------------------
%---------3.3:  Is $\Fib(m)$ completely reducible for all $m\in\Z$?-----------------
%----------------------------------------------------------------------------------------------------
%----------------------------------------------------------------------------------------------------
\section{Is $\Fib(m)$ completely reducible for all $m\in\Z$?}\label{sec:comred}
We now return to the investigation of the decomposition of $\FF$ into a direct sum of irreducible $\Fib$-modules. Recall that in \cite{FF}, it was shown that $\Aff(0)$ was the single irreducible adjoint representation, and each $\Aff(m)$ for $m\neq 0$ was a direct sum of modules from either category $\cO$ (if $m>0$) or $\cO^{op}$ (if $m<0)$, and so was completely reducible by Theorem \ref{reducible}.  It is therefore not unreasonable to conjecture the following for the modules $\Fib(m)$:

\begin{conjecture}\label{conjecture}
For each $m\in\Z, \ \Fib(m)$ completely reduces into a direct sum of one trivial module (on level 0), one non-standard module (on levels $|m|\leq 2$), and standard modules (highest and lowest) on all levels. 
\end{conjecture}

\begin{proposition}\label{wheretrivial} The only trivial $\Fib$-modules in $\FF$ are in $\HH_\FF$.
\end{proposition}
\begin{proof}
Let $V$ be a trivial $\Fib$-module with set of weights $P(V)$ and weight-space decomposition 
$$V=\bigoplus_{\lambda\in P(V)} V_\lambda.$$
Let $\alpha=\ds\sum_{i=1}^3 n_i\alpha_i \in \Phi\cup\{0\}$ such that $\pi(\alpha)=\lambda\in P(V)$. We will show that $\alpha$ is necessarily $0$, and therefore $P(V)=\{0\}$ and $V\subset \HH_\FF$. 

For all $X_\alpha\in\FF_\alpha$, we have $[H_i, X_\alpha]=\alpha(H_i)X_\alpha = 0$ for $i=1,2$. This gives us

\begin{center}
 \systeme{
-2n_1+2n_2-n_3=0, 
 2n_1-3n_2+n_3=0},
 \end{center}
which has solutions $\alpha=c(\alpha_1-2\alpha_3)$ where $c\in\Z$. However, if $c\neq 0$ then $\alpha\notin\Phi$, therefore $\alpha$ must be $0$. \end{proof}

The following two lemmas will help to prove that for $|m|>2$, $\Fib(m)$ is a direct sum of irreducible standard modules from both category $\cO$ and $\cO^{op}$. The cases of $|m|\leq2$ are not as clear, and will be explored afterwards.

\begin{lemma}
If $\beta\in\Delta_m$ for $m\in\Z$, then  $||\pi(\beta)||^2=||\beta||^2-\ds\frac{2m^2}{5}$.
\end{lemma}
\begin{proof}
Write $\beta=\left[\begin{array}{cc}
	b+c+m & b  \\
	b & c \\
	\end{array}\right]=\left[\begin{array}{cc}
	b+c & b  \\
	b & c \\
	\end{array}\right]+\left[\begin{array}{cc}
	m & 0  \\
	0 & 0\\
	\end{array}\right],$ so that $\beta=(-c)\beta_1+(-b-c)\beta_2+m(-\alpha_1-\alpha_2)$. Then the projection 
$$\pi(\beta)=\beta(H_1)\lambda_1+\beta(H_2)\lambda_2=(3b+c+m)\lambda_1+(c-2b)\lambda_2,$$
and
\begin{align*}
||\pi(\beta)||^2 &=\Big\langle (3b+c+m)\lambda_1+(c-2b)\lambda_2, (3b+c+m)\lambda_1+(c-2b)\lambda_2\Big\rangle\\
	&= -\frac{2}{5} \Big[(3b+c+m)^2 + (c-2b)^2\Big] -2\Big(\frac{3}{5}\Big)(3b+c+m)(c-2b)  \\	
	&= 2b^2-2c^2-2(b+m)c-\frac{2m^2}{5} =||\beta||^2-\frac{2m^2}{5}.
\end{align*}\end{proof}
This immediately gives the following.
\begin{lemma}\label{>2} We have:
\begin{enumerate}[a)]
\item If $|m|>2$ and $\beta\in\Delta_m$, then $||\pi(\beta)||^2<0$. In other words, $\pi(\Delta_m)\subset LC_\Fib$. 
\item If $|m|\leq 2$ and $\beta\in\Delta_m$, then $||\pi(\beta)||^2>0$ if and only if $\beta$ is real $(||\beta||^2=2)$.
\end{enumerate}
\end{lemma}
%\dnote{Comment: possible to determine through explicit calculations the vectors found in the kernel of F1 and F2 are really zero in F through identities.}
\begin{proposition}\label{fibmred}
If $|m|>2$, then $\Fib(m)$ is a sum of integrable highest-weight $\Fib$-modules in category $\cO$ (which are therefore completely reducible) and integrable lowest-weight $\Fib$-modules in category $\cO^{op}$ (which are also completely reducible), hence $\Fib(m)$ is completely reducible.
\end{proposition}
\begin{proof}
We have the set of weights $P(\Fib(m))=\pi(\Delta_m)=P(\Fib(m))_-\cup P(\Fib(m))_+$, and we consider $\Fib(m)_\pm$ (as in  Definition \ref{Vpm}). By Lemma \ref{>2}, if $|m|>2$ then $P(\Fib(m))$ contains only imaginary weights inside the light-cone on level 0, so $P(\Fib(m))_\pm\subset LC_\mp$. Furthermore, since $W_\Fib(LC_+)\cap W_\Fib(LC_-) = \{0\}$, we have that $\Fib(m)_+$ and $\Fib(m)_-$ are integrable $\Fib$-modules from category $\cO$ and $\cO^{op}$, respectively. By Theorem \ref{reducible} and Remark \ref{contra}, $\Fib(m)_+$ (resp. $\Fib(m)_-$) completely reduces as a direct sum of irreducible highest-weight (resp. lowest-weight) $\Fib$-modules. Thus, $\Fib(m)=\Fib(m)_-\oplus \Fib(m)_+$ is completely reducible.
\end{proof}
%The symmetries in Section \ref{sec:symcos} allow one to determine the decomposition of these levels by focusing on the decomposition of $\Fib(m)_-$ for $m>2$ into irreducible lowest-weight $\Fib$-modules. The next chapter will prescribe a method for exploring such a decomposition.  

\begin{lemma}\label{3.12}If $V\subseteq \Fib(m)$ is a standard highest-weight $\Fib$-module, then $P(V)\subset LC_{\Fib,-}$. If $V\subseteq \Fib(m)$ is a standard lowest-weight $\Fib$-module, then $P(V)\subset LC_{\Fib,+}$. If $V\subseteq \Fib(m)$ is a non-standard module, then a $\lambda\in P(V)$ satisfying Definition \ref{nstd} (ii) has positive norm, and $W_\Fib\lambda\subset P(V)\subseteq \{\mu\in K_m\mid ||\mu||\leq ||\lambda|| \}$. 
\end{lemma}
\begin{proof}
Let $m\in\Z$ and let $V\subseteq \Fib(m)$ be a standard module.  Assume $V$ is a highest-weight module with highest weight $\lambda\in P_\Fib^+$, and consider $\lambda\geq \mu\in P(V)$. Let $\mu'\in P_\Fib^+$ be a $W_\Fib$-conjugate of $\mu$. Then since $\mu'\leq \lambda$ and $\mu, \mu'$ lie on a branch of the hyperbola $S_{\Fib, ||\mu||^2}$ that lies below the hyperbola containing $\lambda$ (cf. Figure \ref{fig:fibplane}), we have that $||\mu||^2\leq ||\lambda||^2<0$. Since $P(V)$ is $W_\Fib$-invariant and $W_\Fib$ preserves squared-length, $P(V)\subset LC_{\Fib,+}$. Similarly, if $V$ is a lowest-weight module with lowest weight $\lambda\in P_\Fib^-$, then $P(V)\subset LC_{\Fib,-}$.

Assume $V$ is a non-standard module but that $P(V)\subset LC_\Fib$. Then by Definition \ref{nstd}  (ii) there exists $\lambda\in P(V)$ such that $\lambda\notin P_\Fib^\pm$, $||\lambda||$ is maximal in $\{||\mu||\mid \mu\in P(V)\}$, $|wt(\lambda)|$ is minimal among those, and $V=\cU(\Fib)\cdot v_\lambda$ for $0\neq v_\lambda\in V_\lambda$. Then $\lambda$ is $W_\Fib$-conjugate to a unique weight $\mu\in P_\Fib^\pm$, and since $V$ is irreducible, $V=\cU(\Fib)\cdot v_\mu$ for any $0\neq v_\mu\in V_\mu$. Moreover, since $\lambda, \mu$ lie on the same branch of a hyperbola inside the light-cone of constant squared-length $||\lambda||^2$, we have $|wt(\mu)|<|wt(\lambda)|$, contradicting the minimality of $|wt(\lambda)|$. Therefore, $P(V)$ must contain weights of positive norm, including any $\lambda$ from Definition \ref{nstd} (ii). Since $P(V)$ is $W_\Fib$-invariant and saturated (cf. Definition \ref{string}), for any $w\in W_\Fib$, $\beta\in \Delta$, $S_{\beta}(w\lambda)$ contains weights inside $LC_\Fib$. So $P(V)\subseteq \{\mu\in K_m\mid ||\mu||\leq ||\lambda|| \}$.
\end{proof}

%----------REVISION TO THEOREM 3.13 STARTS HERE
\begin{lemma}\label{oneorbit} We have
\begin{enumerate}[i)]
\item $\pi(\Delta_0^{re}) = \Delta^{re} = W_\Fib \beta_1 \bigsqcup W_\Fib \beta_2$,
\item If $|m|=1,2, \ \pi(\Delta_m^{re}) = W_\Fib(\Lambda_m)$.
\end{enumerate}
\end{lemma}
\begin{proof}
Let $|m|\leq 2$. We have the ``real'' hyperbola $\pi(S_2\cap \R\beta_1\oplus \R\beta_2+m\gamma)=S(m)^1\cup S(m)^2$, where $S(m)^i$ for $i=1,2$ is a branch of a hyperbola defined by 
$$S(m)^i=\Big\{\mu \in (\HH^*_\Fib)_\R \ \mid \ ||\mu||^2 = 2-\frac{2m^2}{5}, \ \ \la \mu, \beta_i \ra >0 \ \  \text{and} \ \  \la \mu, \beta_{3-i} \ra <0\Big\}.$$
(In all weight diagrams in the present work, $S(m)^1$ is on the left and $S(m)^2$ is on the right.) If $\beta\in\Delta_m^{re}$, then $\pi(\beta)\in S(m)^i$ for $i=1$ or $i=2$. 

First we make the following general observations. If $\lambda=c_1\lambda_1+ c_2\lambda_2\in S(m)^i$, then $c_i>0, \ c_{3-i}<0$ and 
$$r_i\lambda=\lambda-\la \lambda, \beta_i\ra \beta_i = c_1\lambda_1+c_2\lambda_2-c_i(2\lambda_i-3\lambda_{3-i}),$$
and since $wt(r_i\lambda)=c_1+c_2+c_i$ and $wt(\lambda)=c_1+c_2$, we have 
$$wt(r_i\lambda)-wt(\lambda)=c_i>0 \qquad \text{and} \qquad wt(r_{3-i}\lambda)-wt(\lambda)=c_{3-i}<0.$$
This gives us
\begin{equation}\label{eq:branch1}\lambda\in S(m)^i \qquad \Rightarrow \qquad wt(r_i\lambda)>wt(\lambda)\quad \text{and} \quad wt(r_{3-i}\lambda)<wt(\lambda).\end{equation}
%Observe that $P(m)_+=\psi(P(m)_-$, and since $\psi$ preserves $|wt|$, $wt(\mu)$ is maximal in $P(m)_-$ if and only if $wt(\psi(\mu))$ is minimal in $P(m)_+$. 
%Now let $V\subseteq\Fib(m)$ be a $\Fib$-submodule. 
Note that $\pi(\Delta_m)$ partitions as
$$\pi(\Delta_m) = \ P(m)_- \ \bigcup \ \{\Lambda_m\} \ \bigcup \ P(m)_+,$$
where 
$$P(m)_-=\{\mu\in\pi(\Delta_m) \mid wt(\mu)<0\}\qquad \text{and}\qquad P(m)_+=\{\mu\in\pi(\Delta_m) \mid wt(\mu)>0\}.$$
%are partially ordered by $wt$. 
Let $P(m)^{re}_\pm=P(m)_\pm\cap\pi(\Delta_m^{re})$, and let $\beta\in P(m)^{re}_-$  (by $\psi$-symmetry, it is enough to consider only negative weights). The set $W_\Fib\pi(\beta)\cap P(m)^{re}_-$ is graded by $wt$, is bounded above by 0, so it contains an element $\lambda(m)^i=d_1\lambda_1+d_2\lambda_2\in S(m)^i$ for $i=1$ or $i=2$ such that $wt(\lambda(m)^i)=d_1+d_2<0$ is maximal. %Since $\lambda(m)^i$ is a $W_\Fib$-conjugate of $\pi(\beta),$ $\lambda(m)^i\in S(m)^i$ for $i=1$ or $i=2$.  
By the implications \eqref{eq:branch1}, we have $wt(r_i\lambda(m)^i)>wt(\lambda(m)^i)$. But $wt(\lambda(m)^i)$ is maximal in $W_\Fib\pi(\beta)\cap P(m)^{re}_-$, so if $m\neq 0$, then either  $r_i\lambda(m)^i\in P(m)^{re}_+$ or $r_i\lambda(m)^i=\Lambda_m$, and if $m=0$, then $r_i\lambda(0)^i\in P(0)^{re}_+$ (since $\Lambda_0=0$ is not a $W_\Fib$-conjugate of the projection of a real root).

Each of the following weights $\mu(m)^i\in S(m)^i\cap P(m)^{re}_-$ satisfies $0>wt(\mu(m)^i)\geq-2$, and is of maximal $wt$ in $P(m)^{re}_-$ (see Figure \ref{fig:cosets}):
$$\mu(0)^1 = 2\lambda_1-3\lambda_2=\beta_1,  \qquad \mu(1)^1 =\pi(-\alpha_1) = 2\lambda_1-4\lambda_2, \qquad \mu(2)^2 =\pi(\alpha_3) = -2\lambda_1+\lambda_2,$$
$$\mu(0)^2 =  -3\lambda_1-2\lambda_2=\beta_2, \qquad \mu(-1)^2 = \pi(\alpha_1) =-4\lambda_1+2\lambda_2, \qquad \mu(-2)^1 = \pi(-\alpha_3) = \lambda_1-2\lambda_2.$$
Then $\mu(m)^i$ is a candidate for $\lambda(m)^i$. We see that if $m=0$, then $\pi(\beta)$ is a $W_\Fib$-conjugate of either $\beta_1$ or $\beta_2$, and if $m\neq 0$ we have:
\begin{itemize}
\item if $m=-2$, then $r_1(\mu(-2)^1)=r_1(\lambda_1-2\lambda_2) = -\lambda_1+\lambda_2=\Lambda_{-2},$
\item  if $m=-1$,   then $r_2(\mu(-1)^2)=r_2(-4\lambda_1+2\lambda_2) = 2\lambda_1-2\lambda_2=\Lambda_{-1},$
\item  if $m=1$,  then $r_1(\mu(1)^1)=r_1(2\lambda_1-4\lambda_2) = -2\lambda_1+2\lambda_2=\Lambda_{1}, $
\item  if $m=2$,  then $r_2(\mu(2)^2)=r_2(-2\lambda_1+\lambda_2) = \lambda_1-\lambda_2=\Lambda_{2},$
\end{itemize}
completing the proof.
\end{proof}

\begin{theorem}\label{onereal} Let $|m|\leq 2$, and let $V\subseteq \Fib(m)$ be a $\Fib$-submodule. Then the following are equivalent:
\begin{enumerate}[i)]
\item $\pi(\beta)\in P(V)$ for some $\beta\in\Delta^{re}_m$,
\item $\Lambda_m\in P(V)$ and $V$ is non-trivial,
\item $\pi(\Delta^{re}_m)\subset P(V),$
\item $P(V)=P(\Fib(m))$.
\end{enumerate}
\end{theorem}
\begin{proof}

%If $r_i\lambda(m)^i\in P(m)^{re}_+$, then 
%$$wt(r_i\lambda(m)^i)=d_1+d_2+d_i= \left\{\begin{array}{cl}
%	2d_1+d_2>0 & \text{if } i=1, \\
%	d_1+2d_2>0 & \text{if } i=2, \\
%	\end{array}\right.$$ 
%which implies $d_2>-2d_1$ if $i=1$, or $d_1>-2d_2$ if $i=2.$ If $m=0$ then $\lambda(m)^i=\beta_i$, because $\beta_i$ is the only weight of maximal $wt$ satisfying the corresponding inequality. If $m\neq 0$ there is no such weight satisfying either condition. 

Assume $\pi(\beta)\in P(V)$ for some $\beta\in\Delta^{re}_m$. By Proposition \ref{wheretrivial}, $P(V)\neq \{0\}$ so $V$ is not a trivial module. If $|m|=1,2$, then $\Lambda_m\in W_\Fib \pi(\beta)\subset P(V)$ by Lemma \ref{oneorbit}(ii) and the fact that $P(V)$ is $W_\Fib$-invariant. If $m=0$, then $\beta_i\in W_\Fib \pi(\beta)$ for $i=1$ or $i=2$ by Lemma \ref{oneorbit}(i). Since $P(V)$ is a saturated set, we have $S_{\beta_i}(\beta_i)=\{\beta_i, \Lambda_0, -\beta_i\}\subset P(V)$ (cf. \eqref{eq:saturated}), so  $(i)\Rightarrow (ii)$. 

If $|m|=1,2$, then $\pi(\Delta^{re}_m) = W_\Fib\Lambda_m\subset P(V)$ by part $(ii)$ and Lemma \ref{oneorbit}(ii). If $m=0$, we have that $P(V)$ is a saturated set that contains $\Lambda_0$, and since $V$ is not trivial, $S_{\beta_i}(\beta_i)\subset P(V)$ for $i=1$ and $i=2$, so $P(V)^{re}=W_\Fib\{\beta_1\}\bigsqcup W_\Fib\{\beta_2\}=\pi(\Delta^{re})  \subset P(V)$, so $(ii)\Rightarrow(iii)$.

Thus $P(\Fib(m))$ and $P(V)$ contain the same set of maximal-norm weights $\pi(\Delta_m^{re})$, and each is a saturated, $W_\Fib$-invariant set. Therefore, if weight strings $S_{\beta_i}(\mu)\subset P(\Fib(m))$ for some $\mu\in \pi(\Delta_m)$ and $i=1,2$, then also $S_{\beta_i}(\mu) \subset P(V)$. Thus $P(\Fib(m))=P(V)$, so $(iii)\Rightarrow(iv)$. 

Lastly, $(iv)\Rightarrow (i)$ is obviously true, so all four conditions are equivalent.

\end{proof}

If $|m|\leq2$, then Lemma \ref{>2} shows that all real roots of $\Delta_m$ project outside the $\Fib$ light-cone. Since $\Fib(m)$ is a $\Fib$-module, these projected weights are part of a $\Fib$-module which is neither in category $\cO$ nor in $\cO^{op}$. A theorem analogous to Theorem \ref{reducible} on the complete reducibility of general integrable modules has not been found in the literature. The irreducible adjoint representation for $\Fib$ is integrable but neither in category $\cO$ nor $\cO^{op}$. We write this non-standard $\Fib$-module as $V^{\pi(\beta_i)}$ for either $i=1$ or $i=2$. The next results will show that $\Fib$ is the only non-standard module on level 0. 
\begin{theorem}\label{onlyone}
Let $|m|\leq 2$. Then $\Fib(m)$ has a unique irreducible non-standard quotient whose set of weights is the same as $P(\Fib(m))$. For the case $m=0$, this quotient is the adjoint representation of $\Fib$.
\end{theorem}

\begin{proof}
Let $|m|\leq2$, and define $V(m)=\cU(\Fib)\cdot v_\beta$ where $\beta$ and $v_\beta$ are chosen from Table \ref{tab:nonstd}. Since $\beta$ is a real root, we have that $Mult_\FF(\beta)=1$, $(v_\beta)$ is a basis for $\FF_\beta$, and by Theorem \ref{onereal}, $P(V(m))=P(\Fib(m))$.
\begin{table}[h!]
 \begin{center}
\begin{tabular}{| c | c | c | c |}
  \hline                       
    $m$ (level) & $\beta\in\Delta^{re}_m$ & $\pi(\beta)$ & $v_\beta\in\FF_\beta$\\
\hline
  $-2$ & $-\alpha_3$ & $\Lambda_{-2}$ & $f_3$\\
    \hline      
  $-1$   & $\alpha_1$ & $\Lambda_{-1}$ & $e_1$\\
    \hline      
  $0$  & $\beta_2$ & $\beta_2$ &  $E_2=e_2$\\
      \hline      
   $1$ & $-\alpha_1$ & $\Lambda_1$ & $f_1$\\
     \hline      
  $2$   & $\alpha_3$ & $\Lambda_2$& $e_3$\\
  \hline  
\end{tabular}
\caption{A choice of real root $\beta\in\Delta^{re}_m$ for $|m|\leq 2$ and a vector $v_\beta\in\FF_\beta$ that generates the $\Fib$-module $V(m)=\cU(\Fib)\cdot v_\beta$.}
\label{tab:nonstd}
\end{center}
\end{table}
Also, $\pi(\beta)\notin LC_{\Fib}$ so, by Lemma \ref{3.12}, $V(m)$ is not a sum of standard modules.% Assume $V(m)$ is a standard irreducible module, so $V(m)=V^\lambda$ for some $\lambda\in P(V(m))$. Since all $\Fib$-submodules of $\Fib(m)$ are integrable by Proposition \ref{fibmintegrable}, we have by Lemma \ref{dominant} that $\lambda\in P^-_\Fib$ or $P^+_\Fib$, so $||\lambda||^2 <0$. The weight diagram of $V^\lambda$ is $W_\Fib$-invariant, and for all $\mu\in P(V^\lambda)$, $||\mu||^2\leq ||\lambda||^2$, so $P(V^\lambda)\subset LC_\Fib$. Thus $\pi(\beta)\in LC_\Fib$, which contradicts Lemma \ref{>2}. Therefore, $V(m)$ is not a standard module.

By Borcherds, there exists a maximal proper $\Fib$-submodule $U(m)\subset V(m)$ not containing $v_\beta$, and the quotient $V(m)\Big/U(m)$ is irreducible (\cite{B}). Since $V(m)\Big/U(m)=\cU(\Fib)\cdot\overline{v_\beta}$ and $\pi(\beta)$ meets the criteria in Definition \ref{nstd} (ii), this quotient module is non-standard, and we may write it as $V^{\beta_2}$ if $m=0$, as $V^{\Lambda_m}$ if $|m|=1,2$, or more generally as $V^{\pi(\beta)}$. %Assume there exists another irreducible quotient $V'=V(m)\Big/U(m).$ Then we have for each $\mu\in V(m)$,

Let  $Y(m)=\Fib(m)\Big/V(m)$. Then for each $\mu\in \pi(\Delta_m)$,
$$Mult_m(\mu)=Mult_{V(m)}(\mu)+Mult_{Y(m)}(\mu).$$
Since $Mult_m(\pi(\beta)) = Mult_{V(m)}(\pi(\beta))=1$, we have $Mult_{Y(m)}(\pi(\beta))=0$. Thus $\pi(\beta)\notin P(Y(m))$. By Theorem \ref{onereal}, $P(Y(m))\cap\pi(\Delta_m^{re})=\emptyset$, hence  $P(Y(m))\subset LC_\Fib$, so $Y(m)$ cannot contain any other non-standard modules. Therefore $V^{\Lambda_m}$ is the unique irreducible non-standard module in $\Fib(m)$.

In the case $m=0$ we have $V(0)=\cU(\Fib)\cdot E_2 = \Fib$, and since the adjoint representation of $\Fib$ contains no proper submodules, we have that $U(0)=\{0\}$, hence $V(m)=\Fib$.
%$$Mult_{V(m)}(\beta)=0+1+1=2,$$
%which contradicts $Mult_{V(m)}(\beta)\leq Mult_\FF(\beta)=1$, hence $V'$ does not exist, and $V^{\pi(\beta)}$ is unique.
\end{proof}

\begin{proposition}\label{Y(m)}
For $0<|m|\leq 2$, let $Y(m)=\Fib(m)\Big/ V(m)$ as in Proposition \ref{onlyone}. Then 
\begin{enumerate}[i)]
\item $P(Y(m))$ consists only of weights $\mu$ where $||\mu||^2< 0$, and
\item $Y(m)$ is a sum of integrable highest-weight $\Fib$-modules in category $\cO$ (which are therefore completely reducible) and integrable lowest-weight $\Fib$-modules in category $\cO^{op}$ (which are also completely reducible), hence $Y(m)$ is completely reducible.
\end{enumerate}
 %For $m=0$, let $Y(m)$ = \Fib(0)\Big/(V^0\oplus V^{\Lambda_m})$.
\end{proposition}
\begin{proof}
$(i): \ $
Let $0<|m|\leq 2$. Then for all $\beta\in \Delta_m^{re}$, we have $Mult_\FF(\beta)=1$, so
$$Mult_{Y(m)}\big(\pi(\beta)\big) = Mult_{\Fib(m)}\big(\pi(\beta)\big) - Mult_{V(m)}\big(\pi(\beta)\big)=0,$$
so $\pi(\beta)\notin P(Y(m))$. By Lemma \ref{>2}, $P(Y(m))$ contains only imaginary weights that are projections of imaginary roots in $\Delta_m$.

$(ii): \ $ Since $P(Y(m))=\pi(\Delta^{im}_m)$ lies strictly inside the light-cone $LC_\Fib$ for $|m|=1,2$, we have $P(Y(m))_\pm \subset LC_{\Fib,\mp}$ (cf. Definition \ref{Vpm}), and $P(Y(m))_- = \psi\Big(P(Y(m))_+\Big).$

Since $W_\Fib(LC_{\Fib,+})\cap W_\Fib(LC_{\Fib,-})=\{0\}$, we have that $Y(m)_+$ and $Y(m)_-$ are integrable $\Fib$-modules from category $\cO$ and $\cO^{op}$, respectively, and so by Theorem \ref{reducible}, are completely reducible. Thus, $Y(m)=Y(m)_-\oplus Y(m)_+$ is also completely reducible.
\end{proof}

The case $m=0$ (excluded from the proposition) is slightly complicated by the presence of a trivial module  $V^0$ for $\Fib$ on level 0. In Chapter \ref{ch:decomp0}, we will prove the existence of $V^0$, and we will then show that
$$\Fib(0)=V^0\oplus \Fib \oplus Y(0)$$
is completely reducible. This follows from the fact that $V(m)=\Fib$ is simple, so $U(0)=\{0\}$. However, for $|m|=1,2$,  $U(m)$ is not necessarily trivial. The third isomorphism theorem gives us
$$\Big(\Fib(m)/U(m)\Big)\Big/\Big(V(m)/U(m)\Big) \cong \Fib(m)\Big/ V(m),$$
and since $V(m)/U(m)=V^{\Lambda_m} $ is irreducible, it is a direct summand of $\Fib(m)/U(m)$, so knowledge of the inner multiplicities of $V^{\Lambda_m}$ will give outer multiplicities of other irreducible modules in the decomposition of $\Fib(m)/U(m)$, not $\Fib(m)$ as desired. In the present work we assume the following conjecture is true.  
\begin{conjecture}
For $0<|m|\leq 2$, the maximal submodule $U(m)$ from Theorem \ref{onlyone} is $\{0\}$, and $V(m)=V^{\Lambda_m}$ is a direct summand of $\Fib(m)$, so $\Fib(m) = V^{\Lambda_m} \oplus Y(m)_- \oplus Y(m)_+.$
\end{conjecture}
Note that if this conjecture can be proven true, then this will complete the proof of Conjecture \ref{conjecture}. 

\section{Determining inner multiplicities of irreducible submodules}\label{sec:inner}
We now outline procedures for determining the inner multiplicities of any irreducible $\Fib$-module $U$. Since $\nu(\Fib(m))=\Fib(-m)$ for all $m\in\Z$, we may assume $m\geq 0$, and let $U\subset\Fib(m)$ be an irreducible $\Fib$-submodule. The Weyl subgroup $W_\Fib$ preserves $P(U)$, and for all $\mu\in P(U)$, 
$Mult_U(w\mu) = Mult_U(\mu)$
for each $w \in W_\Fib$. Thus, determining the inner multiplicities of weights in $P(U)$ reduces to calculating multiplicities of only those weights inside the fundamental domain.

If $U=V^\lambda$ is a standard lowest-weight module with lowest weight $\lambda\in P^-$, the weight diagram $P(U)$ labeled with inner multiplicities can be found in the following recursive way. For an example, we refer the reader to \ref{subfig-2:rhomult} in Chapter \ref{ch:decomp0}. First, we determine the weights of $P(U)$. Starting with $\lambda$, find all weight strings that result from simple $W_\Fib$ reflections of $\lambda$ (by reflecting across the lines $\R\lambda_1, \R\lambda_2$), and add all the weights from those weight strings to the diagram. Then, reflect all of the previously found weights across each reflecting line, add those weight strings to the diagram, and repeat. Next we find the inner multiplicities of weights in the diagram, which follow the Racah-Speiser recursion,
$$Mult_\lambda(\mu) = -\sum_{1 \neq w \in W_\Fib} \det(w) \ Mult_\lambda\Big(\mu +(w\rho - \rho)\Big)$$  
where $\rho=\lambda_1+\lambda_2$. This author used the geometric approach outlined by Feingold in \cite{F3}, which we now describe for $U=V^\lambda$. The weights of $P(U)$ are partially ordered by height, whereby $\mu>\lambda$ if and only if $ht(\mu-\lambda)>0$. We start with the fact that for the lowest weight $\lambda$, $Mult_\lambda(\lambda)=1$. Now let $\lambda<\mu\in P(U)$, and assume that for all $\lambda\leq\gamma<\mu,$ $Mult_\lambda(\gamma)$ is already known. Refer to Figure \ref{subfig-1:racah1} for a diagram showing the $W_\Fib$ conjugates of $\rho$. For all $1\neq w\in W_\Fib$, note that $w\rho-\rho$ is in $\Phi_{\Fib,-}$. The Racah-Speiser recursion states that the multiplicity of $\mu$ is a sum of signed multiplicities of weights in the diagram of $P(U)$ lower than $\mu$, where the sign is determined by $-\det(w)=(-1)^{\ell(w)+1}$ where $\ell:W_\Fib\longrightarrow \Z$ is the standard length function on $W_\Fib$ with respect to $r_1, r_2$. 
\begin{figure}[!ht]
    \subfloat[Transparency diagram of Racah-Speiser recursion\label{subfig-1:racah1}]{%
      \includegraphics[trim = .5cm 0cm .2cm 1cm, clip=true, width=3 in]{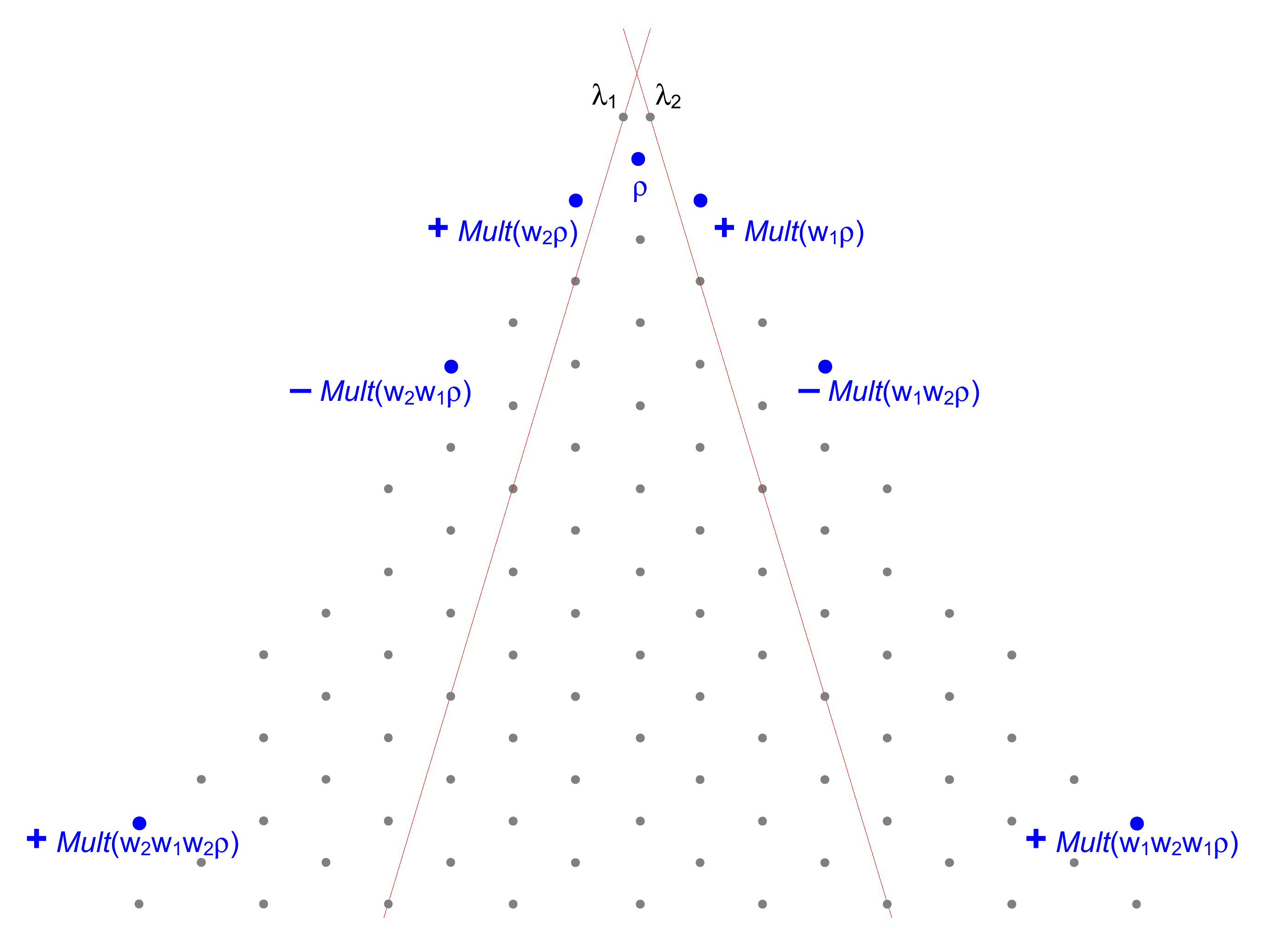}
    }
    \hfill
    \subfloat[Computation of $Mult_{-\rho}(\mu)$. $Mult(\gamma)$ for each $\gamma < \mu$ was previously determined using Racah-Speiser, and is shown in green. Signed multiplicities relevant to recursion for $\mu$ are shown in blue. \label{subfig-2:racah2}]{%
      \includegraphics[trim = 3cm 1cm 2cm 0cm, clip=true, width=2.5 in]{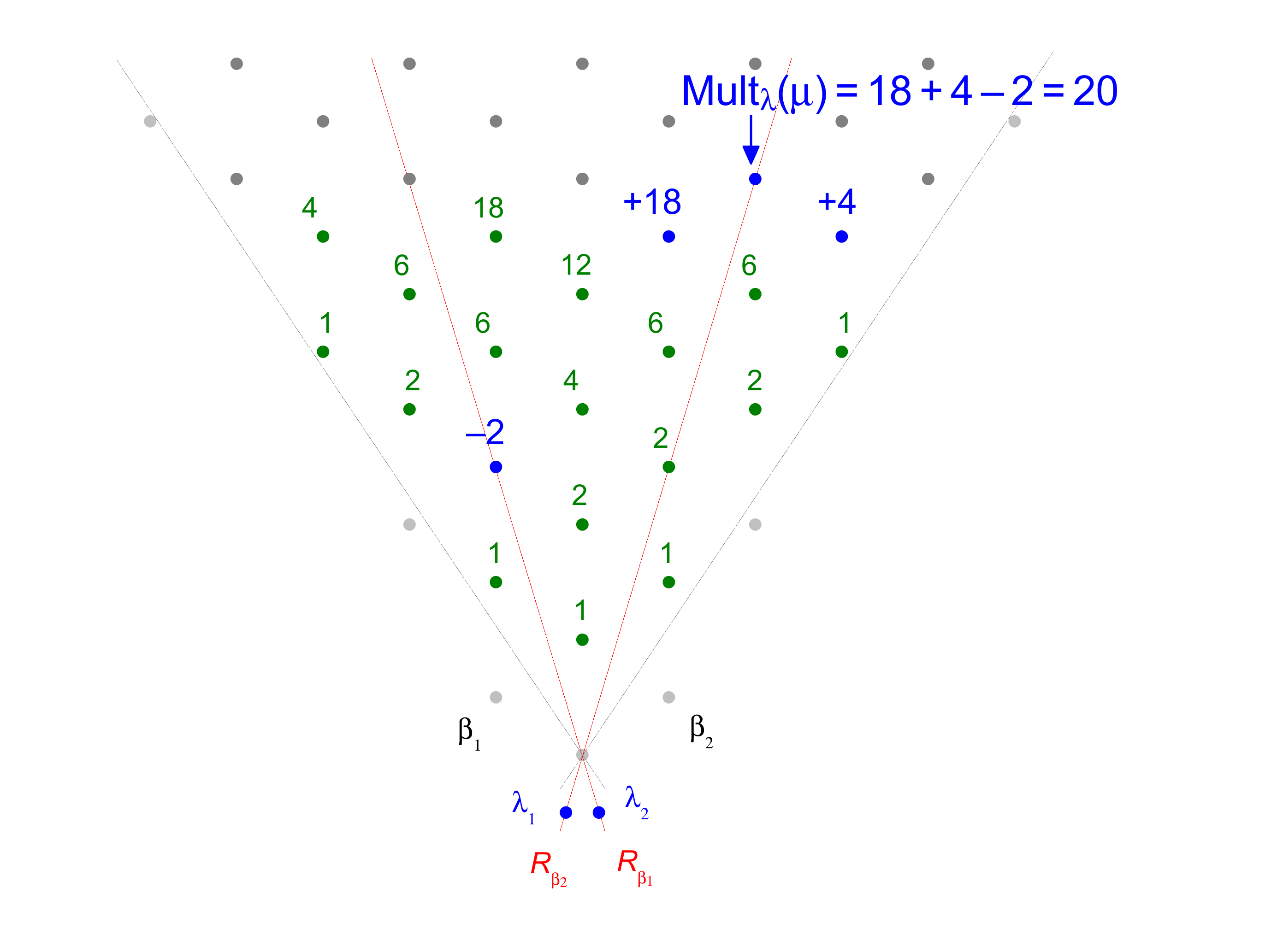}
    } 
    \caption{Explanation of Racah-Speiser recursion for the example of $V^{-\rho}\subset\Fib(0)$. The red lines are the reflecting lines $\beta_1^\perp, \beta_2^\perp$. }
    \label{fig:racah}
  \end{figure}
Feingold noted that one may make a transparency of the $W_\Fib$-conjugates of $\rho$ (with the same lattice spacing as $P(U))$, label it with signs to denote which multiplicities get added or subtracted, and overlay this transparency onto the weight diagram of $U$, being sure to line up $\mu$ with $\rho$. Then, one may quickly compute $Mult_\lambda(\mu)$ by adding the signed multiplicities of weights referred to in the Racah-Speiser formula. See Figure \ref{subfig-2:racah2} for an example (note that Figure \ref{subfig-2:rhomult} shows the completed diagram up to height 12).

If $U=\Fib$ is the adjoint representation on level 0, the weight multiplicities follow a Kac-Peterson recursion and have already been determined by Kac (Chapter 11 of \cite{K2}).

If $U$ is a non-standard module on levels $\pm1$ or $\pm 2$, there is no recursion found in the literature for determining the weight multiplicities of $U$. The approach for finding these multiplicities (which was briefly outlined in Section \ref{sec:kacmoody}) will be elaborated upon in Chapter \ref{ch:nonstd} in the cases of $\Fib(1)$ and $\Fib(2)$.

\section{Determining outer multiplicities of standard $\Fib$-submodules}\label{sec:outer}
We will use the following shorthand notation for the outer multiplicity of a highest or lowest weight $\lambda$ in $\Fib(m)$, 
$$M_m(\lambda)=M_{\Fib(m)}(\lambda).$$

%By Theorem \ref{onlyone}, there is only one non-standard $\Fib$-module in $\Fib(m)$ for $|m|\leq 2$. So if $\lambda=\Lambda_m$ (for $|m|=1,2$) or $\beta_2$ (if $m=0$), then $Mult_m(\lambda)=1$. Additionally, we will see in Chapter \ref{ch:decomp0} that, for the trivial module on level 0, $Mult_0(0)=1$.

For $|m|>2$, Proposition \ref{fibmred} gives us that $\Fib(m)_+$ and $\Fib(m)_-$ are integrable $\Fib$-modules from category $\cO$ and $\cO^{op}$, respectively, so
$$\Fib(m)=\Bigg(\bigoplus_{\mu\in P'(\Fib(m))_+} M_m(\mu) V^{\mu}\Bigg)\oplus \Bigg(\bigoplus_{\lambda\in P'(\Fib(m))_-} M_m(\lambda) V^\lambda \Bigg).$$
By the $\psi$-symmetry of $\Fib(m)$, it is enough to find the outer multiplicities of only positive weights, since for all $\mu\in P'(\Fib(m))_+$, $\mu=\psi(\lambda)$ for some $\lambda\in P'(\Fib(m))_-$, and moreover, $M_m(\lambda)=M_m(\psi(\lambda))=M_m(\mu)$, hence
$$\Fib(m)=\bigoplus_{\lambda\in P'(\Fib(m))_-} M_m(\lambda)\Big(V^\lambda \oplus V^{\psi(\lambda)}\Big).$$

For $0<|m|\leq 2$, Proposition \ref{Y(m)} and the assumption that $V^{\Lambda_m}$ is a direct summand of $\Fib(m)$ gives us a similar decomposition,
$$\Fib(m)=V^{\Lambda_m}\oplus \bigoplus_{\lambda\in P'(\Fib(m))_-} M_m(\lambda)\Big(V^\lambda \oplus V^{\psi(\lambda)}\Big).$$

We wish to find the outer multiplicities $M_m(\lambda)$ for all highest and lowest weights of $\Fib$ in $P(\Fib(m))$. If we know the root multiplicities of $\FF$, and if we know the inner multiplicities of 1) standard modules on each level (by Racah-Speiser), 2) the adjoint representation (by Kac-Peterson), and 3) the non-standard modules (by the recursion which will be detailed in Chapter \ref{ch:nonstd}), then the above decompositions hint at a recursive algorithm for finding  the outer multiplicities of any level. This algorithm will be detailed in the next chapter, using the example of decomposing $\Fib(0)$.

%--------------------------------------------------------%
%--------------------------------------------------------%
%--------------------- Chapter 4 --------------------%
%--------------------------------------------------------%
%--------------------------------------------------------%
\chapter{Finding decomposition data for level 0}\label{ch:decomp0}
%--------------------------------------------------------%
%---------------------Section  4.1--------------------%
%--------------------------------------------------------%

\section{Trivial and adjoint representations of $\Fib$}\label{sec:trivial} 
We wish to express each $\Fib(m)$ as a direct sum of irreducible $\Fib$-modules, starting with level 0. In \cite{FF}, $\Aff$-level 0 consisted only of $\Aff$. Similarly, we showed in Theorem \ref{onlyone} that the adjoint representation of $\Fib$ is the unique irreducible non-standard module in $\Fib(0)$. We now find other $\Fib$-modules in $\Fib(0)$.
%\begin{lemma}\label{fibonlevel0} $\Fib\subset\Fib(0)$.  
%\end{lemma}
%\begin{proof}
%Clearly $\Delta\subset\Delta_0$. In fact, since $\Delta=\{\beta\in \Z\beta_1\oplus\Z\beta_2 \ | \ ||\beta||^2\leq2 \}=\Phi\cap(\Z\beta_1\oplus\Z\beta_2)$, and $\Z\beta_1\oplus\Z\beta_2=Q\cap(\R\beta_1\oplus\R\beta_2)$, we have that $\Delta=\Phi\cap(\R\beta_1\oplus\R\beta_2)$. Therefore $\Delta$ is exactly $\Delta_0$. 
%\end{proof}

First we note that the Cartan of $\Fib$ is two-dimensional while the Cartan of $\FF$ is three-dimensional.  We now find a vector in $\HH$ that commutes with all of $\Fib$, that is, the vector generates a one-dimensional trivial representation of $\Fib$. 

\begin{theorem}\label{trivial}
The element $C:=-\frac{1}{2}(h_1-2h_3)\in\HH$ generates a trivial one-dimensional $\Fib$-module $V^0=\C  C$ in $\Fib(0)$, that is, $[C,\Fib]=0$, and $\HH=\HH_\Fib\oplus V^0$. Moreover, 
$$[C, X]=mX$$
for all $X\in\Fib(m)$, so the grading of $\FF$ by $\Fib$ level,
$$\FF=\bigoplus_{m\in\Z}\Fib(m),$$
is the eigenspace decomposition of $\FF$ with respect to $ad_C$, and $[C,\Fib(0)]=0$.% The involution $\nu$ exchanges $\Fib(m)$ and $\Fib(-m)$.
\end{theorem}
\begin{proof}

Let $C$ be a generator of a trivial $\Fib$-module. Proposition \ref{wheretrivial} showed that $C$ must be a Cartan element, and moreover, $\Z(\alpha_1-2\alpha_3)$ are the only solutions to $\alpha(H_i)=0$ for $i=1,2$.  By Remark \ref{aij=aji}, since the Cartan matrices $A$ and $B(3)$ are both symmetric, we have for $1\leq i,j\leq 3$, $\alpha_i(h_j)=\alpha_j(h_i)$, and for $1\leq k,l\leq 2$, $\beta_k(H_l)=\beta_l(H_k)$. Extending linearly, this gives us $\alpha(H_i)=0 \Leftrightarrow \beta_i(H_\alpha) = 0,$ where $H_\alpha\in\R(h_1-2h_3)$. We choose $C=-\frac{1}{2}(h_1-2h_3)$.

Now suppose $\alpha=x\alpha_1+y\alpha_2+z\alpha_3\in\Delta_m$.  Consider $X\in\FF_\alpha$. Then we have
$$[C,X] = (x\alpha_1(-\frac{1}{2}(h_1-2h_3))+ y\alpha_2(-\frac{1}{2}(h_1-2h_3)) + z\alpha_3(-\frac{1}{2}(h_1-2h_3)))X$$
$$=-\frac{1}{2}(x(2)+y(-2-2(-1))+z(-2(2)))X =-(x-2z)X$$
Using the conversions $x=b-a-c$, $z=-c$, and $m=a-b-c$ (since $\alpha\in\Delta_m$), it follows that $[C,X] =-(b-a+c)X=mX.$\end{proof}

\section{Highest and lowest-weight modules of $\Fib$ in level 0}\label{sec:low} 
We may now ``complete'' Proposition \ref{Y(m)} to include the case $m=0$, and show that $Y(0)=\Fib(0)\Big/\Big(V^0 \oplus \Fib \Big)$ is a direct sum of highest- and lowest-weight $\Fib$-modules.

\begin{proposition}[cf. Proposition \ref{Y(m)}]\label{Y(0)}
Let $Y(0)=\Fib(0)\Big/\Big(V^0 \oplus \Fib\Big)$. Then 
\begin{enumerate}[i)]
\item $P(Y(0))$ consists only of weights $\mu$ where $||\mu||^2< 0$, and
\item $Y(0)$ is a sum of integrable highest-weight $\Fib$-modules in category $\cO$ (which are therefore completely reducible) and integrable lowest-weight $\Fib$-modules in category $\cO^{op}$ (which are also completely reducible), hence $Y(0)$ is completely reducible.
\end{enumerate}
 %For $m=0$, let $Y(m)$ = \Fib(0)\Big/(V^0\oplus V^{\Lambda_m})$.
\end{proposition}
\begin{proof}
Since $V^0\cap\Fib=\{0\}$, and $\dim(V^0)+\dim(\HH_\Fib) = 3=\dim(\FF_0),$ we have that $0\notin P(Y(0))$. The rest of the proof is identical to the proof of Proposition \ref{Y(m)}, only setting $m=0$.
\end{proof}

\begin{proposition}\label{direct}$\Fib$ is a direct summand of $\Fib(0)$.
\end{proposition}
\begin{proof}
If $\Fib$ is not a direct summand of $\Fib(0)$, then there necessarily exists a non-zero vector $v\in \Fib(0)$ and $\Fib$-module $\Fib\neq M\subset \Fib(0)$ such that $v\in \Fib\cap M$. Proposition \ref{Y(0)} shows that $M$ is either the trivial module or a direct sum of standard modules. Since $v\in\Fib$ and $\Fib$ is irreducible, $v$ is a generator of $\Fib$. If $v\in V^0$ then $v$ is a generator of $V^0$, but since $\Fib$ is irreducible, we have $\Fib=V^0$, a contradiction. Likewise, if $v\in M$ then it is a generator of some irreducible standard module $V^\lambda\subset M$ where $||\lambda||^2<0$, hence $\Fib=V^\lambda$, which is also a contradiction. Hence, $v$ does not exist, and $\Fib$ is a direct summand of $\Fib(0)$.
\end{proof}
Thus we have proven the following theorem, which proves the case of $m=0$ for Conjecture \ref{conjecture}:
\begin{theorem}\label{level0}
$$\Fib(0) = V^0 \oplus \Fib \oplus Y(0)_-\oplus Y(0)_+,$$
where $Y(0)_-$ and $Y(0)_+$ are direct sums of highest-weight and lowest-weight modules, respectively. 
\end{theorem}
Thus finding the decomposition for $\Fib(0)$ is equivalent to finding the decomposition
$$Y(0) = \bigoplus_{\lambda\in P'(Y(0))_-}M_0(\lambda) (V^{\lambda}\oplus V^{\psi(\lambda)}),$$
where $V^\lambda$ and $V^{\psi(\lambda)}$ are irreducible lowest-weight and highest-weight modules, respectively, and 
$$P'(Y(0))_- = \{ \mu\in \Delta^{im}_+ \mid M_0(\lambda)>0 \} = \Delta\cap P^-_\Fib,$$
the set of negative dominant integral roots of $\Fib$. Thus for each $\lambda\in P'(Y(0))_-$, we must find the weight diagram for $V^\lambda$, its inner multiplicities, and outer multiplicity $M_0(\lambda)$. Section \ref{sec:inner} details how to find weight diagram and inner multiplicities for $V^\lambda$. We now focus on finding $M_0(\lambda)$.

First note that $\mu\in P'(Y(0))_-$ if and only if
$$Mult_{Y(0)_-}(\mu) = Mult_0(\mu) - Mult_\Fib(\mu)  - \delta_{0,\mu}>0,$$
and in general, 
$$Mult_0(\mu)\geq Mult_\Fib(\mu),$$
with equality only occurring when $\mu\in\Delta^{re}$. This inequality becomes evident when one compares Figures \ref{subfig-2:Fib(0)mult} and \ref{subfig-1:fibmult}, which show the multiplicities of the positive roots of $\Fib(0)$ up to height 12, and their multiplicities in $\Fib$, respectively \cite{K2}.

 \begin{figure}[!ht]
   \subfloat[Multiplicities of weights in $\Fib(0)$\label{subfig-2:Fib(0)mult}]{%
      \includegraphics[trim = 2.3cm 1cm 4cm 0cm, clip=true, width=3 in]{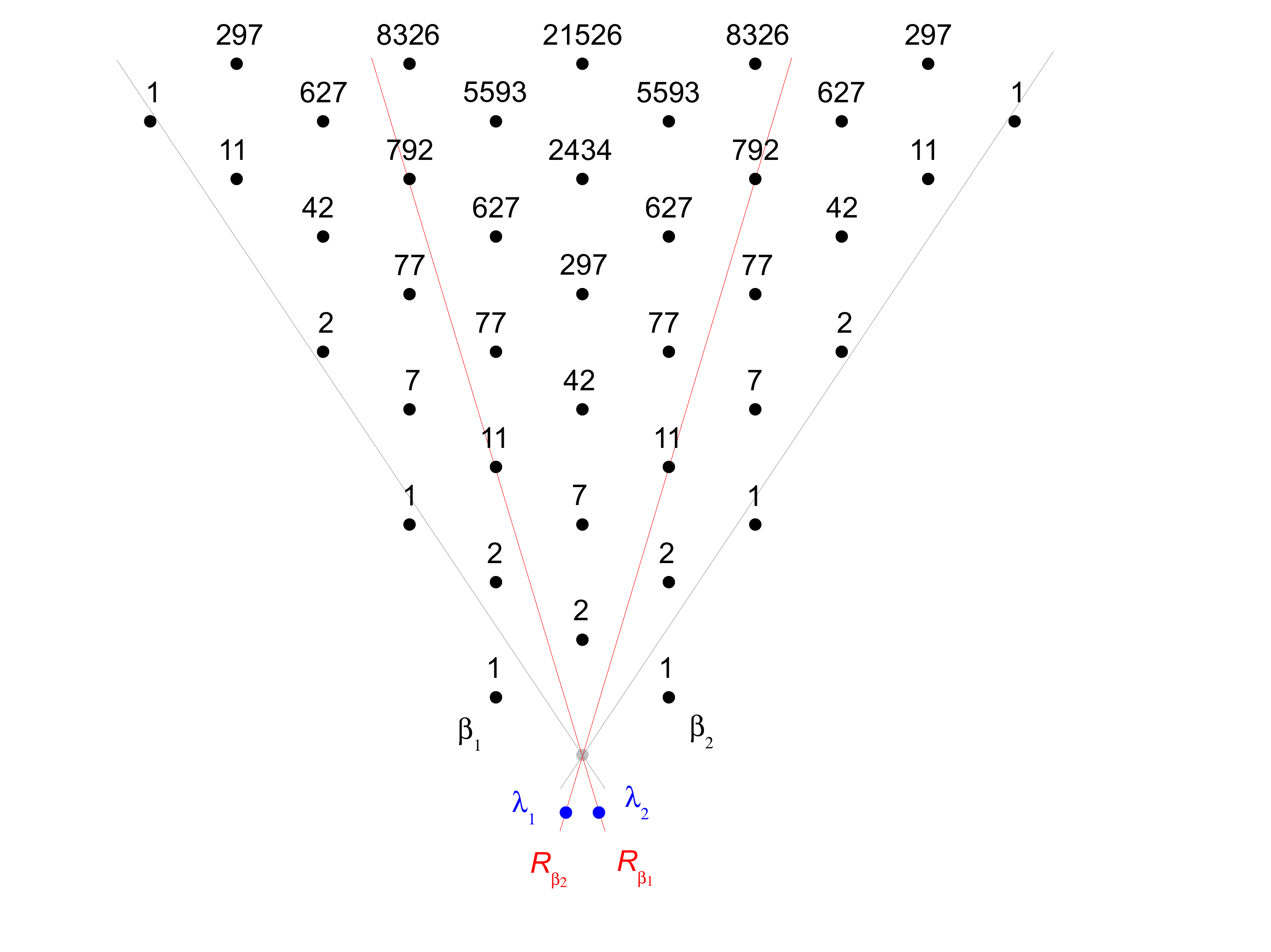}
    } 
    \hfill
     \subfloat[Multiplicities of roots in $\Fib$ \label{subfig-1:fibmult}]{%
      \includegraphics[trim = 2.3cm 1cm 4cm 0cm, clip=true, width=3 in]{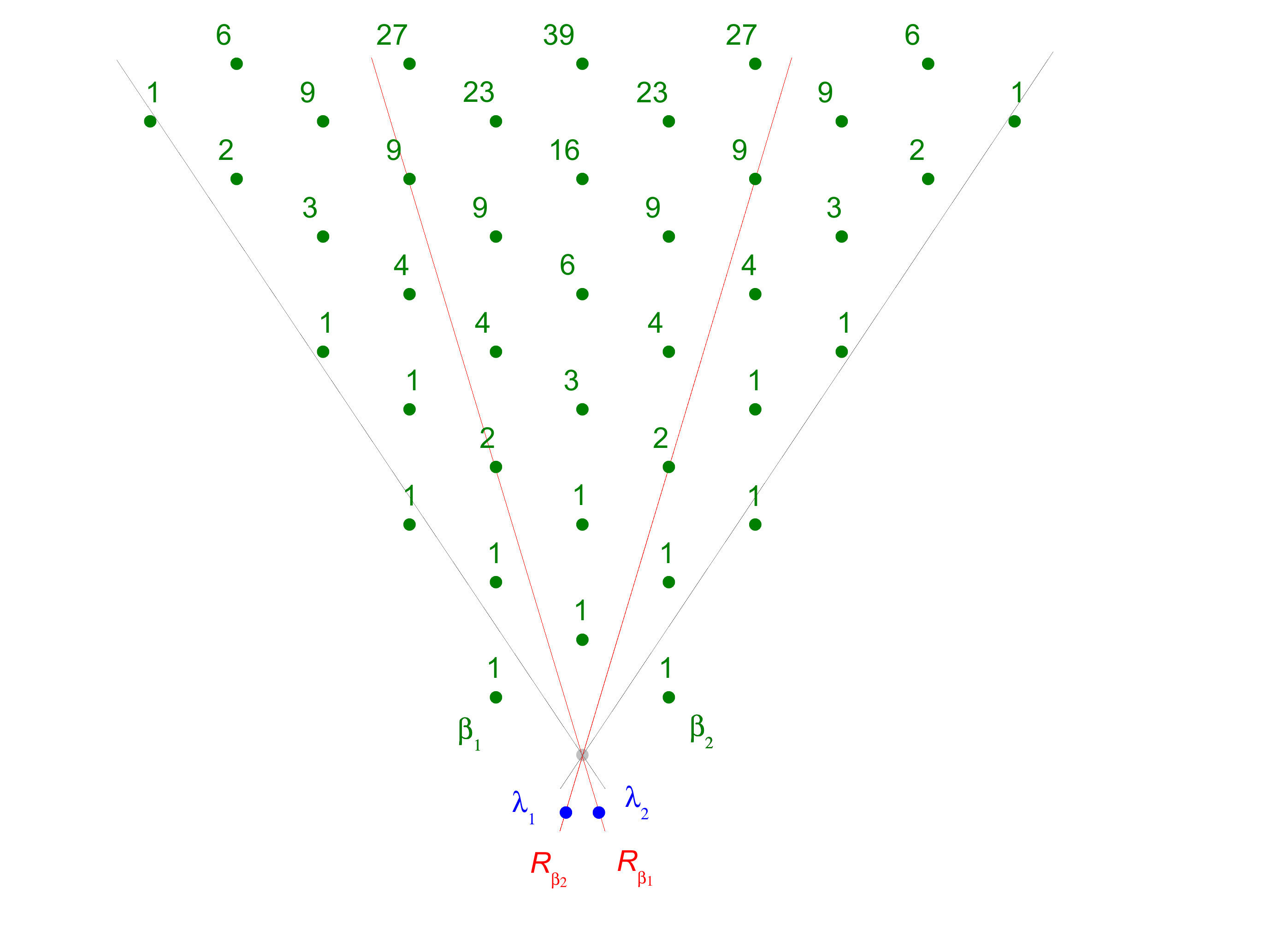}
    }    \caption{A comparison of the $\FF$- and $\Fib$-multiplicities of positive roots on level 0.}
    \label{fig:fibcompare}
  \end{figure}

Starting with the weight diagram $P(\Fib(0))=\Delta$, one may construct the weight diagram $P(Y(0))_-$ by deleting any weight $\mu$ for which $Mult_{Y(0)_-}(\mu)=0$ (i.e., the projections of real roots, and $0$). Then using the formula above one may label the remaining weights with their corresponding multiplicities in $Y(0)_-$. The result is shown in Figure \ref{subfig-1:Fib0-Fibmult}.
\begin{figure}[!ht]
    \subfloat[Multiplicities of negative weights of $Y(0)_-$ \label{subfig-1:Fib0-Fibmult}]{%
      \includegraphics[trim = 2.3cm 1cm 4cm 0cm, clip=true, width=3 in]{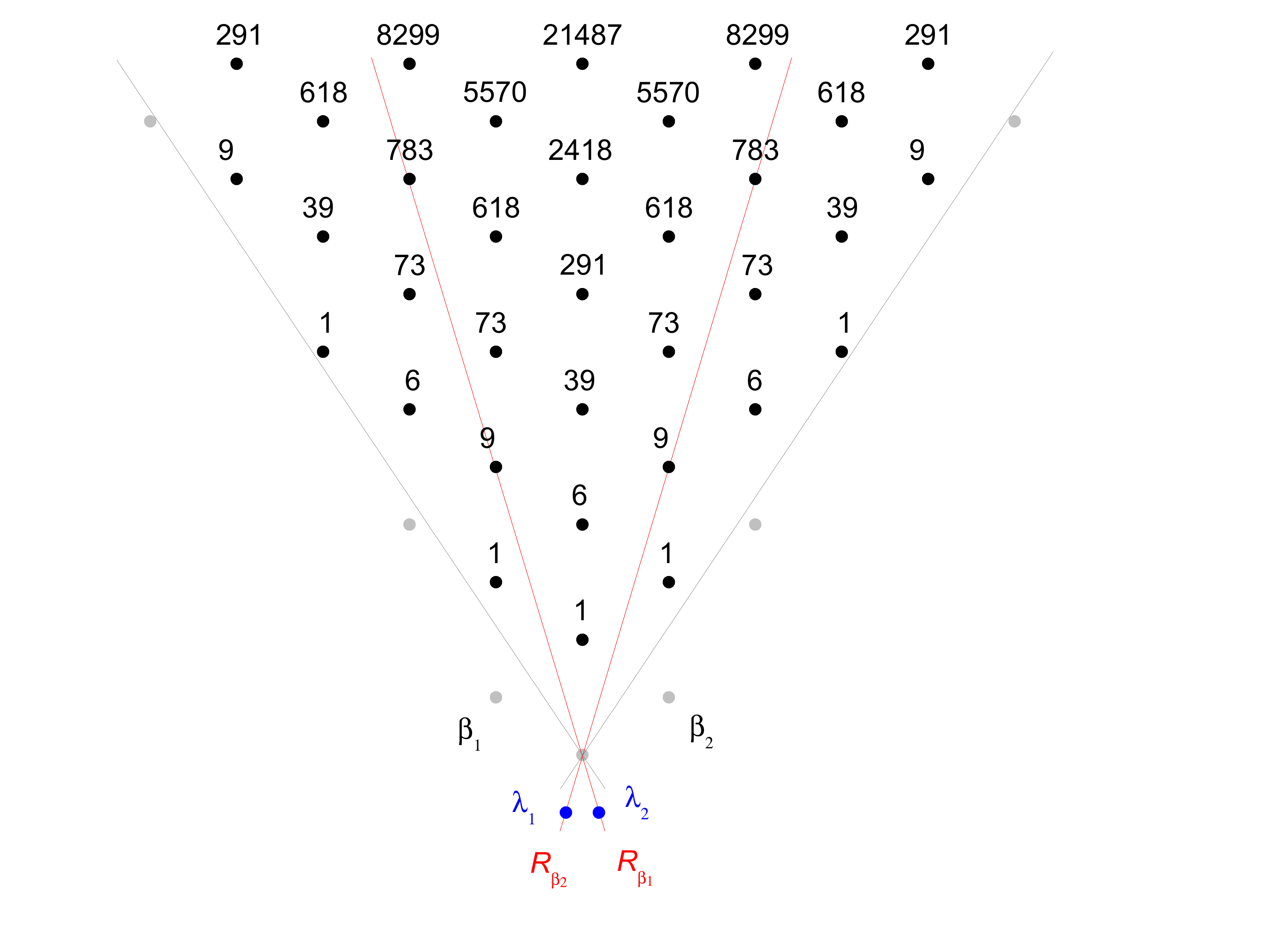} }
    \hfill
    \subfloat[Multiplicities of negative weights in $V^{-\rho}$ \label{subfig-2:rhomult}]{%
      \includegraphics[trim = 2.3cm 1cm 4cm 0cm, clip=true, width=3 in]{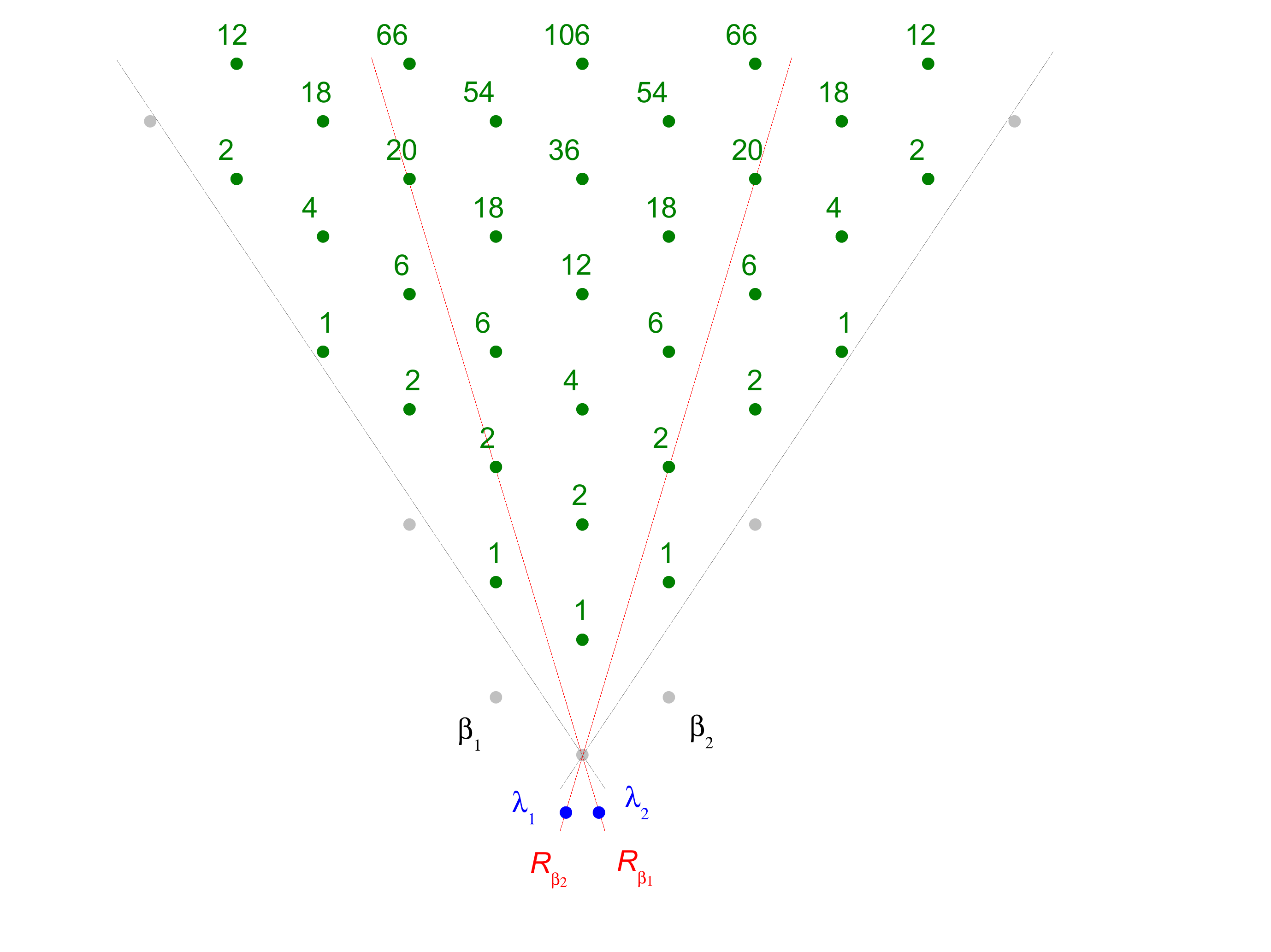}
    }
  
    \caption{A comparison of the multiplicities of negative weights of $Y(0)_-$ and $P(V^{-\rho})$.}
    \label{fig:deleteFib}
  \end{figure}

We now define a total order ``$\succ$'' on $P(Y(0))_-$ based on the partial order by height ``$>$'' in the following way.  For $\lambda=n_1\beta_1+ n_2\beta_2, \  \mu=m_1\beta_1+m_2\beta_2\in P(Y(0))_-$ (so $n_i, m_i\geq 0$ for $i=1,2$), define $\lambda\succ\mu$ to be true if and only if either 
\begin{itemize}
\item $\lambda>\mu$, or
\item $ht(\lambda-\mu)=0$ and $n_1>m_1$.
\end{itemize}
It is then clear from Figure \ref{subfig-1:Fib0-Fibmult} that
$$-\rho = -\rho_{\Fib} = -(\lambda_1+\lambda_2) = \frac{1}{5}\left[\begin{array}{cc}
	5 & 0  \\
	0 & 5 \\
	\end{array}\right]=\left[\begin{array}{cc}
	1 & 0  \\
	0 & 1 \\
	\end{array}\right]= \beta_1+\beta_2$$
is the minimal height in $P(Y(0))_-$ with respect to $\succ$, and $M_0(-\rho) = Mult_{Y(0)_-}(-\rho)=1.$

It would appear then that $\Fib_{\pm\rho}$ are the first examples of root spaces which contain extremal vectors for $\Fib$ in $\Fib(0)$. To verify this claim, we seek a basis for $Low_{\Fib(0)}(-\rho)$ that generates a $\Fib$-module $V^{-\rho}$. A basis for $\FF_{-\rho}$ is $\{[E_2, 2E_1] = e_{21123}, \ e_{12123}\}$, where the first basis element is in $\Fib$.  

\begin{theorem}\label{vrho}
The set  $\{v_{-\rho}=-3e_{12123}+e_{21123}\}$ is a basis for $Low_{\Fib(0)}(-\rho)$, and generates $V^{-\rho}\subset\Fib(0)$, with outer multiplicity $M_0(-\rho)=1$.
\end{theorem}
\begin{proof}
We find $v=a \ e_{12123} + b \ e_{21123}\neq 0$ such that $F_1\cdot v = 0 = F_2 \cdot v.$ We have
\begin{align*}
F_2 \cdot v &= [f_2, a \ e_{12123} + b \ e_{21123} ] = a [f_2, e_{12123} ] + b [f_2, e_{21123} ] 
\\ &= a (e_{1213} + e_{1123}) + b( e_{2113}+ 3 e_{1123}) = a ( e_{1123}) + b(  3 e_{1123})
\\ &= (a+3b) e_{1123},
\end{align*}
so $F_2\cdot v=0$ if and only if $a+3b=0.$  Now,
\begin{align*}
F_1 \cdot v &= [-\frac{1}{2}f_{1123}, \ a \ e_{12123}] +  [F_1, \ b \ e_{21123} ] = -\frac{a}{2}[f_{1123}, e_{12123}] + b [F_1, [E_2, 2E_1]]
\\ &= -\frac{a}{2}  [f_{1123}, e_{12123} ] + 2b\Big( [E_1,\cancel{[E_2,F_1]}]+[E_2,-H_1]\Big)
\\&=-\frac{a}{2}\Big(\big[e_{2123},[e_1,f_{1123}]\big]+\big[e_1,\cancel{[f_{1123},e_{2123}]}\big]\Big)-6bE_2
\\&=-\frac{a}{2}[e_{2123},2f_{123}]-6bE_2 =-a\Big(\big[f_{23},\cancel{[f_1,e_{2123}]}\big]+\big[f_1,[e_{2123},f_{23}]\big]\Big)-6bE_2
\\&=-a\Big[f_1,[f_3,[f_2,e_{2123}]]+[f_2,[e_{2123},f_3]]\Big]-6bE_2
\\&=-a\Big[f_1,[f_3,e_{123}]+[f_2,e_{212}]\Big]-6bE_2 =-a([f_1,-e_{12}-2e_{21}])-6bE_2.
\\&=-2(a+3b)E_2,
\end{align*}
and once again, $F_1\cdot v=0$ if and only if $a+3b=0.$ Since the basis contains only one vector, we have $M_0(-\rho)=1$.
\end{proof}
%After $-\rho_\Fib$, the positive root of next lowest height in $\Delta'$ is $2\beta_1+2\beta_2=-2\rho_\Fib$.  The root space $\FF_{2\beta_1+2\beta_2}$ has dimension 7 inside of $\FF$ and dimension $1+2=3$ inside of $\Fib \oplus \Fib^0 \oplus \Fib^{-\rho}$, so the remaining dimension is 4.  
Thus we have begun to construct the decomposition of $\Fib(0)$,
$$\Fib(0) = V^0\oplus \Fib\oplus \Big(V^{-\rho}\oplus V^\rho \Big)\oplus \bigoplus_{\substack{\mu\in P'(Y(0))_- \\ \mu\succ -\rho}}M_0(\mu) (V^{\mu}\oplus V^{\psi(\mu)}).$$
Let 
$$Y^1(0)= Y(0)\Big/ V^{-\rho} = \Fib(0)\Big/ \Big(V^0\oplus \Fib\oplus V^{-\rho}\oplus V^\rho\Big),$$
so that $P(Y^1(0))_- = P'(Y(0))_-\backslash\{-\rho\}$.  Then for all $\mu\in P'(Y^1(0))_-$,
$$Mult_{Y^1(0)_-}(\mu) = Mult_{Y(0)_-}(\mu) - Mult_{-\rho}(\mu).$$
As before, the comparison in Figure \ref{fig:deleteFib} indicates that for all $\mu\in P(Y^1(0)_-)$,
$$Mult_{Y(0)_-}(\mu) \geq Mult_{-\rho}(\mu).$$
\begin{figure}[!ht]
    \subfloat[Multiplicities of negative weights of $Y^1(0)_-$ \label{subfig-1:Y1}]{%
      \includegraphics[trim = 2.3cm 1cm 4cm 0cm, clip=true, width=3 in]{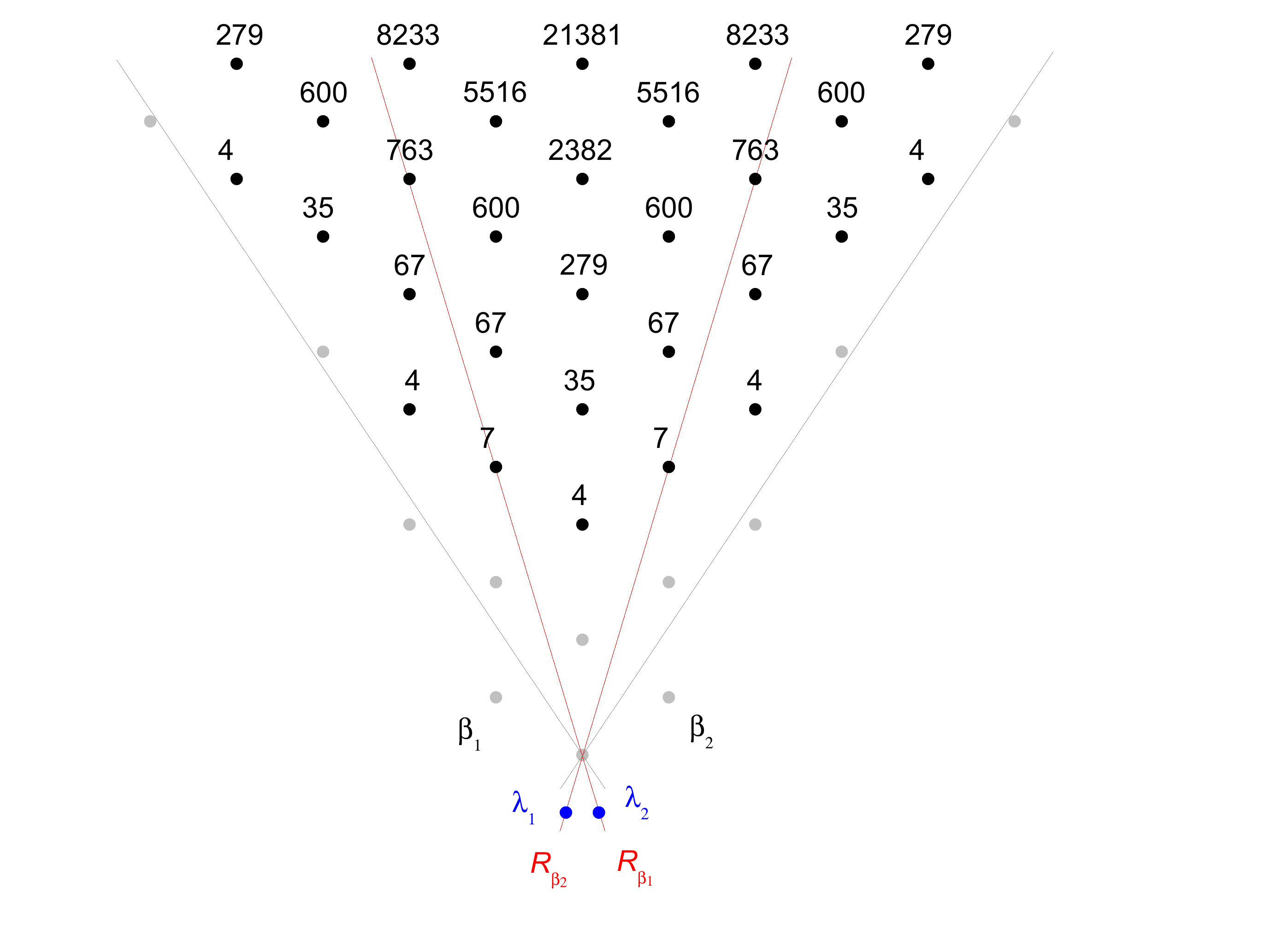} }
    \hfill
    \subfloat[Multiplicities of negative weights in $V^{-2\rho}$ \label{subfig-2:2rho}]{%
      \includegraphics[trim = 2.3cm 1cm 4cm 0cm, clip=true, width=3 in]{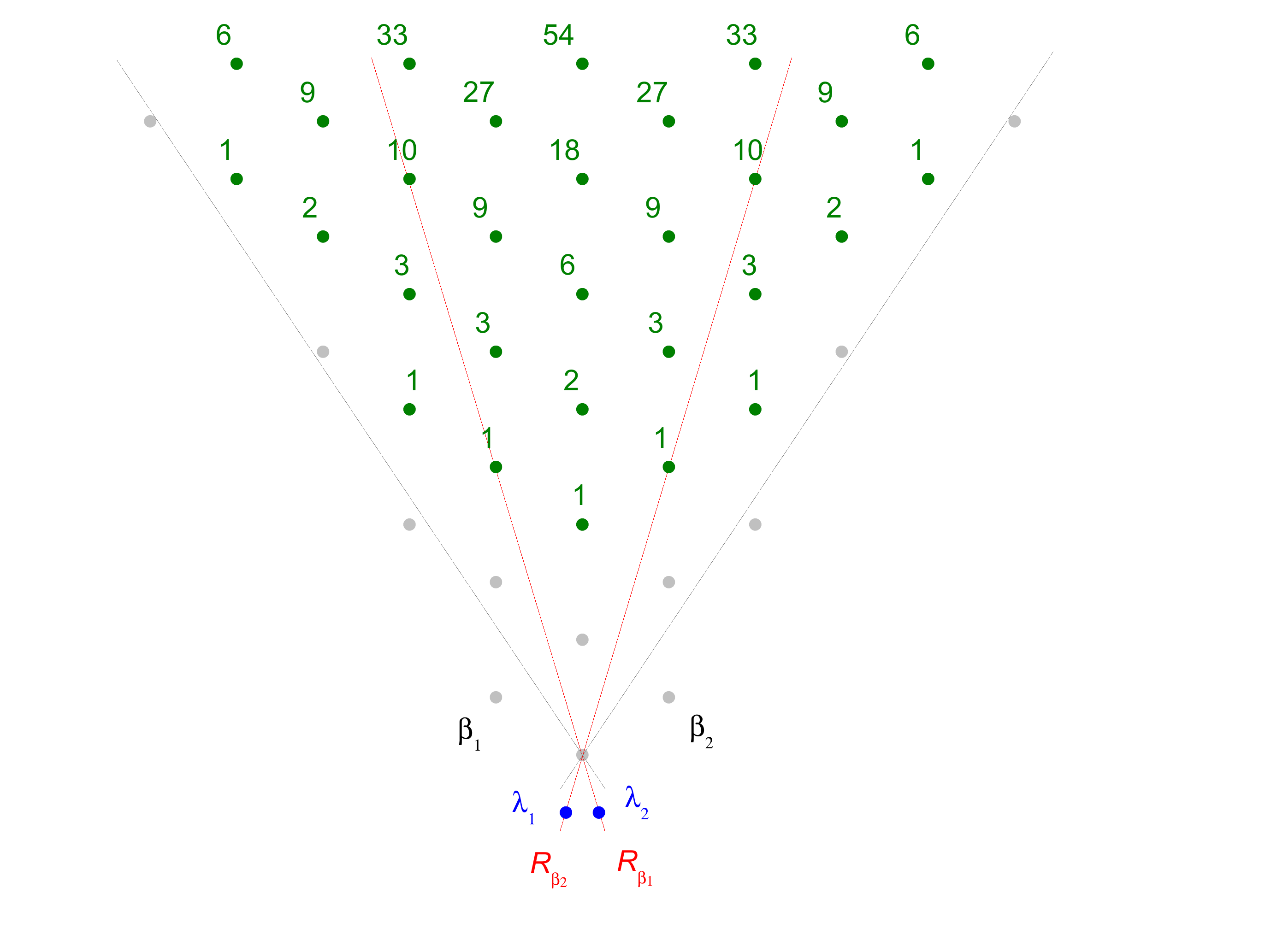}
    }
    \caption{A comparison of the multiplicities of negative weights of $Y^1(0)_-$ and $P(V^{-2\rho})$.}
    \label{fig:deleterho}
  \end{figure}
  
The weight diagram for $Y^1(0)_-$ (Figure  \ref{subfig-1:Y1}) was found by deleting any $\mu\in P(Y(0))_-$ for which equality holds (i.e., all $W_\Fib$-conjugates of $-\rho$). The next weight in $P'(Y(0))_-$ with respect to $\succ$ is $-2\rho$. Figure \ref{subfig-1:Y1} shows that $M_0(-2\rho) = Mult_{Y^1(0)_-}(-2\rho)=4,$
so
$$\Fib(0) = V^0\oplus \Fib\oplus \Big(V^{-\rho}\oplus V^\rho \Big)\oplus 4 \ \Big(V^{-2\rho}\oplus V^{2\rho} \Big) \oplus \bigoplus_{\mu\in P'(Y^2(0))_-}M_0(\mu) (V^{\mu}\oplus V^{\psi(\mu)}),$$
where  
$$Y^2(0)= Y^1(0)\Big/ 4\Big(V^{-2\rho}\oplus V^{2\rho}\Big) = \Fib(0)\Big/ \Big(V^0\oplus \Fib\oplus V^{-\rho}\oplus V^\rho\oplus 4V^{-2\rho}\oplus 4V^{2\rho}\Big),$$
so that $P(Y^2(0))_- = P'(Y^1(0)_-)\backslash\{-2\rho\}$.  Then for all $\mu\in P'(Y^2(0))_-$,
$$Mult_{Y^2(0)_-}(\mu) = Mult_{Y^1(0)_-}(\mu) - M_0(-2\rho) Mult_{-2\rho}(\mu).$$

Figure \ref{subfig-2:2rho} shows the labeled weight diagram for $P(V^{-2\rho})$.

In general, for the $n^{th}$ weight $\mu_n$ in the total ordering of $P'(Y(0))_-$, if we define
$$Y^n(0) = Y^{n-1}(0)\Big/ M_0(\mu_n)\Big(V^{\mu_n}\oplus V^{\psi(\mu_n)}\Big) = \Fib(0)\Big/ \Big(V^0\oplus \Fib\oplus \bigoplus_{i=1}^n M_0(\mu_i)\Big(V^{\mu_i}\oplus V^{\psi(\mu_i)}\Big)\Big),$$
then for $\mu\in P(Y(0))_-$,
$$Mult_{Y^n(0)_-}(\mu) = Mult_0(\mu)-\delta_{0,\mu} - Mult_\Fib(\mu) - \ds \sum_{i=1}^n M_0(\mu_i) Mult_{\mu_i}(\mu),$$
and we have the decomposition of level 0,
$$\Fib(0) = V^0 \oplus \Fib \oplus \bigoplus_{\mu \in P'(Y(0))_-} M_0(\mu) \Big(V^\mu \oplus V^{\psi(\mu)}\Big).$$
\begin {table}[h!]
\begin{center}
\begin{tabular}{ | c | c |}           
 \hline
   $\lambda \in P'(Y(0)_-)$ & $M_0(\lambda)$ \\
\hline
$-\rho$ & 1 \\
$5 \lambda_1$ and $5 \lambda_2$ & 3 \\
$-2 \rho$ & 4 \\
$-3\rho$ & 21 \\
$6\lambda_1+\lambda_2$ and $\lambda_1+6 \lambda_2$ & 28 \\
$10 \lambda_1$ and $10 \lambda_2$ & 135 \\
$ -4\rho$ & 145 \\
$ 7\lambda_1+2\lambda_2$ and $2\lambda_1+7\lambda_2$ & 254 \\
$ -5\rho$ & 1182 \\
$11\lambda_1+\lambda_2$ and $\lambda_1+11 \lambda_2$ & 2184 \\
$ 8\lambda_1+3\lambda_2$ and $3\lambda_1+8\lambda_2$  & 2375 \\
$ -6\rho$ & 10349 \\
  \hline 
\end{tabular}\caption {The sequence of outer multiplicities of irreducible LW $\Fib$-modules in level 0, ordered by increasing $M_0(\lambda)$}\label{tab:outer0}
\end{center}
\end{table}

Continuing in this fashion will lead to more data on the location of lowest-weight vectors for $\Fib$, and the dimensions of the corresponding subspaces in $\Fib(0)$. However, as the last theorem suggests, calculating actual bases of extremal vectors using linear algebra will get increasingly difficult as the dimensions increase, and in the end this method will not provide any insight into how these extremal vectors arise, nor does it seem to hint at a simpler method of generating them. Ignoring for now the determination of bases of extremal vectors, and focusing instead on locating lowest weights and quotienting the appropriate number of copies of the corresponding modules in the way described above, our investigation has produced data on outer multiplicities of highest and lowest weights on level 0 shown in Table \ref{tab:outer0}. The dimensions of the extremal weight spaces themselves appear to not follow any recognizable pattern. 

%--------------------------------------------------------%
%--------------------------------------------------------%
%--------------------- Chapter 5 --------------------%
%--------------------------------------------------------%
%--------------------------------------------------------%
\chapter{Finding decomposition data for levels $\pm1, \pm 2$}\label{ch:nonstd}

%\chapter{Decomposition of $\FF$ with respect to $\Fib$: \\ 
%Levels $\pm1$ and $\pm2$}\label{ch:nonstd}

  \begin{figure}[!ht]
    \subfloat[Level 1 \label{subfig-1:level1}]{%
    \includegraphics[trim = 5.5cm 0cm 5.5cm 0cm, clip=true, width=3 in]{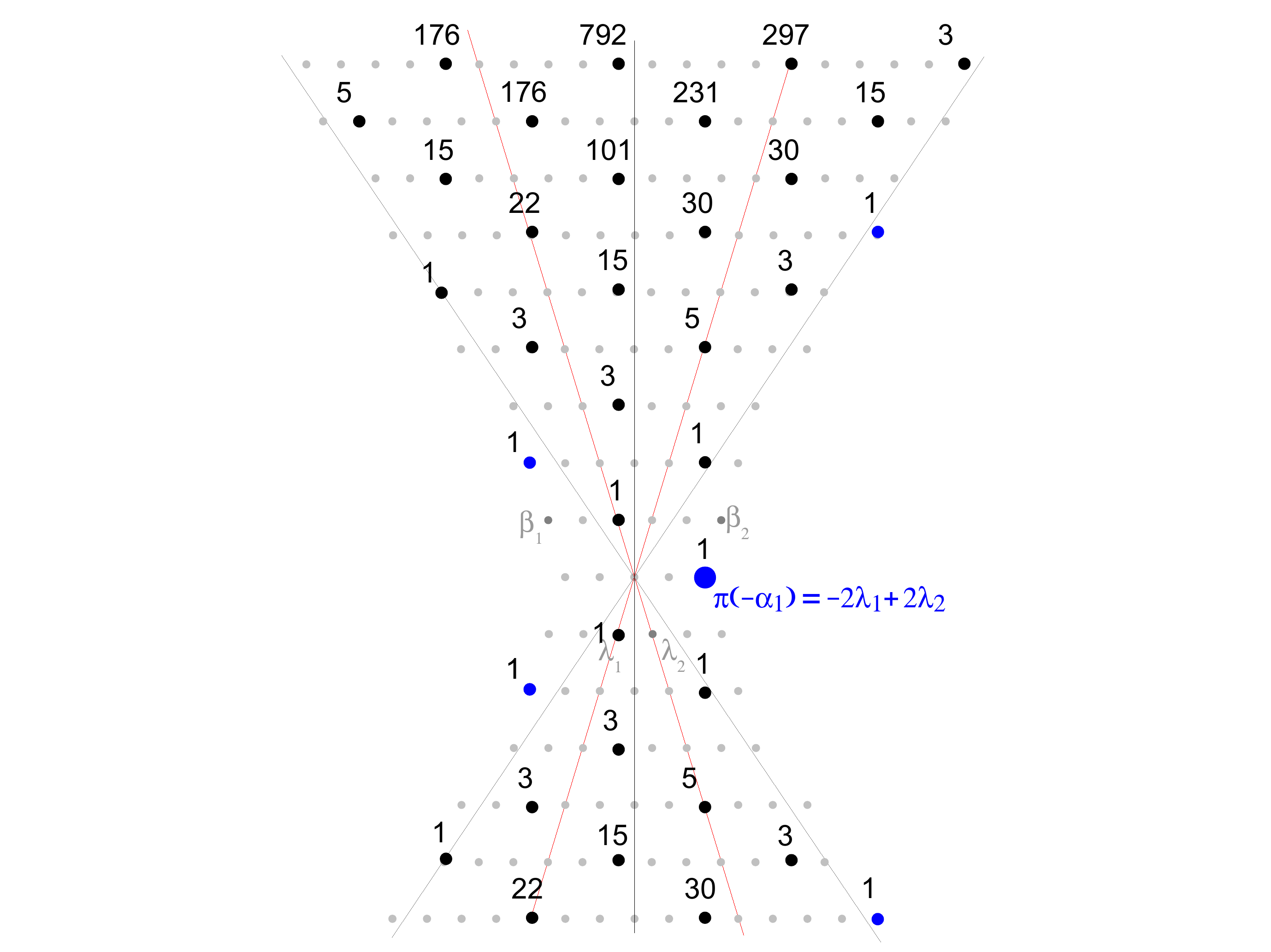} }
    \hfill
    \subfloat[Level 2 \label{subfig-2:level2}]{%
 \includegraphics[trim = 5.5cm 0cm 5.5cm 0cm, clip=true, width=3 in]{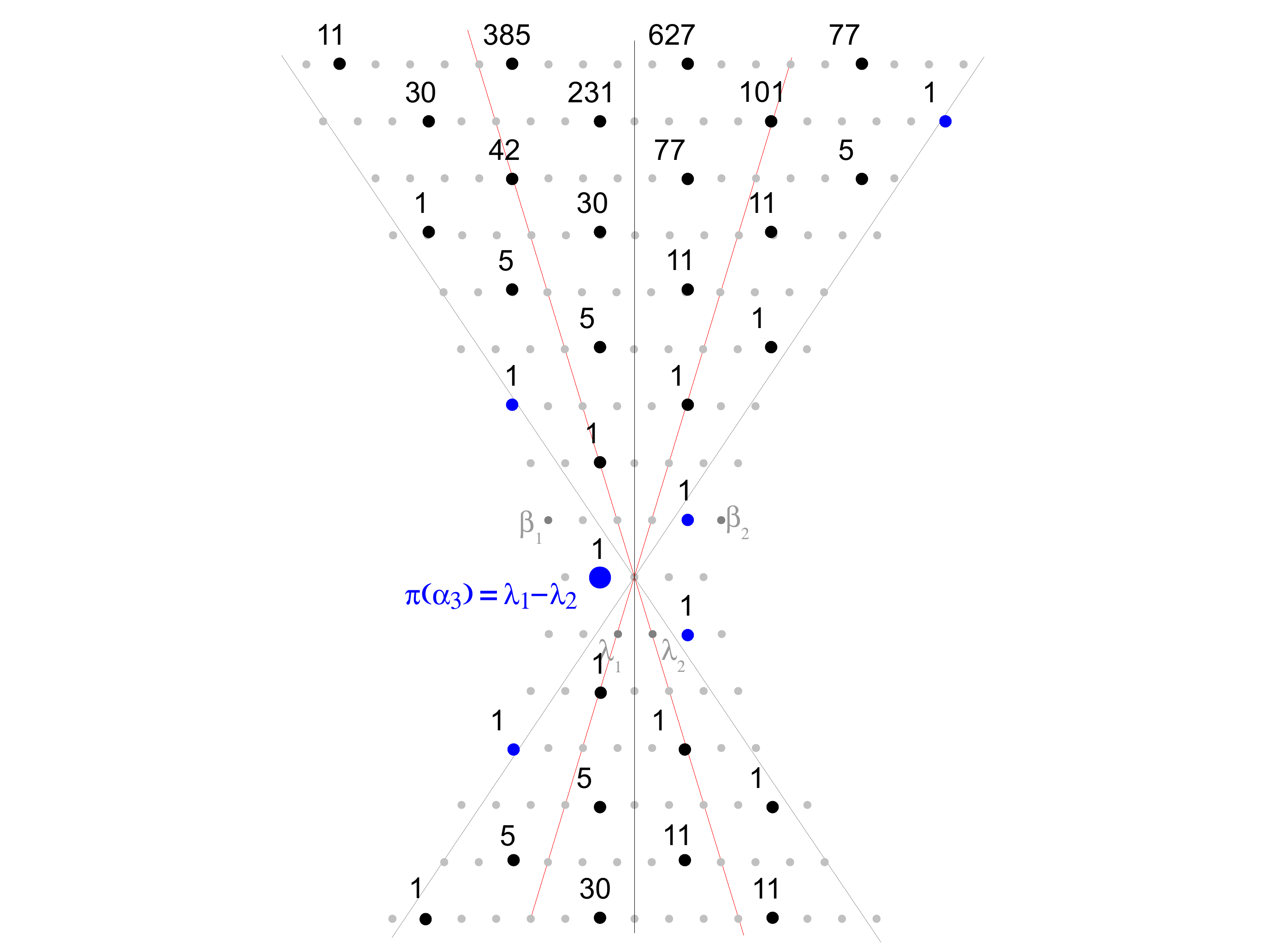} }
 \hfill
      \caption{Partial weight diagrams for $\Fib(m), m=1,2$. The multiplicities shown are $Mult_m(\mu)=\dim_\FF(\FF_\mu)$  for weights $\mu$ such that $-9 \leq wt(\mu)\leq 6$, and are upper bounds for the inner multiplicities of the irreducible non-standard module $V^{\Lambda_m}$.}
    \label{fig:levels1and2}
  \end{figure}

\section{Inner multiplicities of the non-standard $\Fib$-modules on Levels 1, 2}\label{sec:level1}
Figure \ref{fig:levels1and2} shows the weight diagrams for levels 1 and 2, for weights $\mu$ such that $-9 \leq wt(\mu)\leq 6$ (cf. also Figure \ref{fig:slices}). Both diagrams include weight multiplicities in $\FF$, computed by Kac using the Kac-Peterson recursion (see Appendix \ref{appendix:B}) and listed in Table $H_3$ in Ch. 11 of \cite{K2}. Recall that since $\nu(\Fib(m))=\Fib(-m)$, we only consider $m>0$. We know from Proposition \ref{onlyone} that there is one non-standard quotient module $V(m)\Big/U(m)$ on levels $m=\pm1, \pm2$. Assuming $U(m)=\{0\}$ for each of these levels, we have 
$$V^{\Lambda_1}=\cU(\Fib)\cdot f_1, \quad V^{\Lambda_{-1}}=\cU(\Fib)\cdot e_1 , \quad V^{\Lambda_2}=\cU(\Fib)\cdot e_3, \quad \text{ and } \quad V^{\Lambda_{-2}}=\cU(\Fib)\cdot f_3,$$
and by Proposition \ref{Y(m)} we have that
$$Y(\pm1)=\Fib(\pm1)\Big/V^{\Lambda_{\pm1}} \quad \text{ and } \quad  Y(\pm2)=\Fib(\pm1)\Big/V^{\Lambda_{\pm2}}$$
are completely reducible. Using the same method prescribed in Section \ref{sec:low}, we can locate extremal weights for $\Fib$ on levels $\pm 1$ and $ \pm 2$, and determine their outer multiplicities. However, this procedure assumes knowledge of the inner multiplicities for $V^{\Lambda_{\pm m}}$, $|m|=1,2$, which we do not yet have. The Racah-Spesier recursion can only be applied to standard modules, and although the Kac-Peterson recursion works for the adjoint representation of any KM algebra (which is non-standard) there is no reason to believe it can apply to any nonstandard module.

On page \pageref{basis} in Section \ref{sec:kacmoody}, we gave a brief outline for how one may determine weight multiplicities of an irreducible highest-weight $\LL$-module $V^\lambda$ by recursively computing bases for its weight spaces. This method, in principal, will work for any irreducible module, including non-standard ones. We introduce the following definitions which will be used in describing the method.

%RIn \cite {K2}, Kac provided a table of $\FF$-multiplicities for roots in the fundamental domain for $W_\FF$  (computed using the Kac-Peterson recursion). Using this data we determined the multiplicity $Mult_1(\mu)=\dim_\FF(\FF_\mu)$ for each $\mu\in P(\Fib(1))$. Figure \ref{subfig-1:level1} includes these multiplicities in $\FF$. 
\begin{definition}\label{()}
Let $m=1,2$. If $\mu=\Lambda_m+n_1\beta_2+n_2\beta_2\in P(V^{\Lambda_m})$, then $\mu$ is uniquely determined by the coefficients of $\beta_1, \beta_2$, so we may write $\mu=(n_1,n_2)_m.$
We suppress the subscript if there is no ambiguity as to the level in question.
\end{definition}

As in Chapter \ref{ch:decomp0} we utilize $\psi$-symmetry and consider only negative weights $P(V^{\Lambda_m}_-)$, which has a total order similar to that of $P(Y(0))_-$. %For all $\lambda\in P(\Fib(1))$, we have that $\lambda=\Lambda_1+n_1\beta_1 + n_2\beta_2=\pi(-\alpha_1+n_1\beta_1 + n_2\beta_2)$ (cf. Section \ref{sec:symcos}). 
\begin{definition}
For $\lambda=(n_1, n_2)$ and $\mu=(m_1, m_2)$,  define $\lambda\succ\mu$, to be true if and only if either 
\begin{itemize}
\item $ht(\lambda-\mu)>0$, or
\item $ht(\lambda-\mu)=0$ and $n_1>m_1$.
\end{itemize}
\end{definition}

\begin{definition}\label{def:121} Let $V^\lambda = \mathcal{U}(\Fib) v^\lambda$ be a non-standard $\Fib$-module with generating vector $v^\lambda\in V^\lambda_\lambda$. The notation $[i_1 i_2 \cdots i_n]_\lambda$ stands for $E_{i_1}E_{i_2}\cdots E_{i_n}v^\lambda = E_{i_1}\cdot(E_{i_2}\cdot(\cdots (E_{i_n}\cdot v^\lambda)\cdots ))$, and $v^\lambda$ is denoted by $[]_\lambda$. (The $\lambda$ in the subscript will be suppressed unless more than one projected level is being discussed.)
\end{definition}

We now elaborate on the procedure for determining inner multiplicities of an irreducible module $V^{\Lambda_m}$, using the example for $m=1$. By Proposition \ref{onlyone} we know that $P(V^{\Lambda_1}) =P(\Fib(1))$. Figure \ref{subfig-1:nonstd1} shows the partial weight diagram, labeled with computed inner multiplicities. 
 The lowest weight in the total order $\prec$ is $\Lambda_1=\pi(-\alpha_1)$. Since $-\alpha_1$ is a real root, $Mult_{\Lambda_1}(\Lambda_1) = 1$, so $V^{\Lambda_1}_{\Lambda_1}$ has basis consisting of the single vector $v_{\Lambda_1} = f_1$.  

In general, let $\Lambda_1\prec\mu =(n_1,n_2)\in P(V^{\Lambda_1})$. Assume we have previously determined a basis $\cB_{\mu-\beta_1}$ for $V^{\Lambda_1}_{\mu-\beta_1}$ and a basis $\cB_{\mu-\beta_2}$ for $V^{\Lambda_1}_{\mu-\beta_2}$. If for some $i=1$ or $2$, $\mu-\beta_i\notin P(V^{\Lambda_1})$, then $\cB_{\mu-\beta_i}=\emptyset$. Obtain a spanning set for $V^{\Lambda_1}_{\mu}$,
$$\cS_\mu=\cS_{(n_1,n_2)} = \Big(E_1\cdot \cB_{\mu-\beta_1} \Big)\bigcup \Big(E_2\cdot \cB_{\mu-\beta_2}\Big) = \{v_1, \ldots v_k\}$$%= \{ E_1\cdot v_1,\  \ldots , E_1\cdot v_n, \ E_2\cdot u_1, \ \ldots , E_2\cdot u_m\}$$
of size $k=|\cB_{\mu-\beta_1}|+|\cB_{\mu-\beta_2}|$. 
%If $\mu+\beta_i \notin P(V^{\Lambda_1})$, then $E_i\cdot \cB_{\mu-\beta_1} = 0$.   Why is this here?
Determine if there are any linear dependence relations on the vectors in $\cS_\mu$  by solving the homogeneous system of linear equations
\begin{equation}\label{eq:basis}
F_i \cdot \sum_{j=1}^k c_jv_j = \sum_{j=1}^k c_j(F_i\cdot v_j) = 0 \qquad \text{ for } i=1,2.
\end{equation}
Since $V^{\Lambda_1}_\mu$ cannot contain any lowest weight vectors, any nontrivial solutions will yield dependence relations on the vectors in $\cS_\mu$. Then choose vectors to delete from the spanning set to obtain a basis $\cB_\mu$ for $V^\lambda_\mu$. Finally, we have $Mult_\lambda(\mu)=|\cB(\mu)|$.

\begin{rem}\label{table}
If $v_j \in \cB_{\mu-\beta_j}$ for $j=1,2$, then for $i=1,2$ the Jacobi identity gives us
\begin{align*}
F_i\cdot (E_j\cdot v_j) &= E_j \cdot (F_i \cdot v_j )- [E_j, F_i ] \cdot v_j =  E_j \cdot (F_i \cdot v_j) - \delta_{ij}H_j \cdot v_j \\
&= E_j \cdot (F_i \cdot v_j) - \delta_{ij}(\mu-\beta_j)(H_j) v_j.
 \end{align*}
Thus we observe that for $i,j=1,2,$ \ $F_i \cdot (E_j \cdot v_j)$ is determined by $(\mu-\beta_j)(H_j)=\mu(\beta_j)-2$ and our recursive knowledge of $F_i\cdot v_j$ for $v_j\in \cB_{\mu-\beta_j}$. 
 \end{rem}

Remark \ref{table} reveals a useful recursive approach to solving the system (\eqref{eq:basis}), which involve complicated multibracket computations. In Table \ref{tab:tab1} of Appendix \ref{appendix:C}, we present the following data: \label{explain}
\begin{itemize} 
\item The first column shows $\mu=(n_1, n_2)\in P(\Lambda_1)$. The table is sorted by this first column, ordered by $\prec$.
\item The second column lists, for each $\mu\in P(\Lambda_1)$, an \textit{ordered} spanning set 
\begin{equation}\label{eq:spanning}\cS_\mu = (E_2\cdot u_1, \ldots, E_2\cdot u_m, \ E_1\cdot v_1, \ldots,  E_1\cdot v_n)\end{equation}
where $\cB_{\mu-\beta_2}=(u_1, \ldots, u_m)$ and $\cB_{\mu-\beta_1}=(v_1, \ldots, v_n)$ are ordered bases. (There is some justification for this ordering of the set $\cS_\mu$, which will be discussed shortly.) The vectors are written using the shorthand notation from Definition \ref{def:121}. 
\item The third column contains a `*' if there is a dependence relation found on $\cS_\mu$. The vectors deleted from $S_\mu$ to form the basis $\cB_\mu$ are then shown in red, and their coordinates with respect to $\cB_\mu$ are given.
\item The fourth and fifth columns show, for each $u\in \cS_\mu$ in the second column, the result of the computation $F_1\cdot u$ (in column 4) and $F_2\cdot u$ (in column 5), written as coordinate vectors $(c_1, c_2, \ldots , c_k)_{\mu-\beta_i}$ with respect to the basis $\cB_{\mu-\beta_i}$ (so $k=n$ if $i=1$ and $k=m$ if $i=2$).
\item For $u\in \cS_\mu$, the sixth column shows $\mu(H_1)$, and the seventh column shows $\mu(H_2)$.
\end{itemize}
 
We now demonstrate the recursive algorithm for determining bases for weight spaces in $V^{\Lambda_1}$ by showing how the first several rows of Table \ref{tab:tab1} were computed.  First, we have
$$\cS_{(0,0)}=\{v_{\Lambda_1} = []\}=\cB_{(0,0)},$$
and $Mult_{\Lambda_1}(\Lambda_1)=1$. Furthermore, the weight diagram shows that $F_1\cdot[]=0,$ and 
$$\Lambda_1(H_1) = (-2\lambda_1+ 2\lambda_2)(H_1)=-2, \quad \text{and} \quad \Lambda_1(H_2) = (-2\lambda_1+ 2\lambda_2)(H_2)=2.$$
($F_2\cdot[]$ will not be needed for the recursion, as we will see.) Since $V^{\Lambda_1}_{(0,0)}$ is a one-dimensional space there are no dependence relations, and we have completed row 1 of Table \ref{tab:tab1}.
For row 2, we have
$$\cS_{(1,0)}=\{E_1\cdot v_{\Lambda_1}\}=\{[1]\}=\cB_{(1,0)}$$
and $Mult_{\Lambda_1}((1,0))=1$, since $E_2\cdot v_{\Lambda_1}=0$ according to the weight diagram. Then using Remark \ref{table} and the data in row 1 of Table \ref{tab:tab1}, we have
$$F_1E_1\cdot v_{\Lambda_1} = (E_1F_1-H_1)\cdot v_{\Lambda_1} = -H_1\cdot v_{\Lambda_1} = -\Lambda_1(H_1)v_{\Lambda_1}=2v_{\Lambda_1}.$$
We write this vector with respect to the basis $\cB_{(0,0)}=(v_{\Lambda_1})$, entering it into the table as $(2)_{(0,0)}.$
Note that the weight diagram also reveals that $F_2E_2 \cdot  v_{\Lambda_1} = 0$, confirming there was no need to compute $F_2\cdot[]$. We complete row 2 of Table \ref{tab:tab1} by computing
$$(\beta_1+\Lambda_1)(H_1) = \beta_1(H_1)+\Lambda_1(H_1) = 0 \quad \text{and} \quad (\beta_1+\Lambda_1)(H_2) = \beta_1(H_2)+\Lambda_1(H_2) = -1.$$
We then have 
$$\cS_{(2,0)}=\{E_1E_1\cdot v_{\Lambda_1}\} = \{[11]\} = \cB_{(2,0)} \quad \text{ and } \quad \cS_{(1,1)}=\{E_2E_1\cdot v_{\Lambda_1}\}=\{[21]\} = \cB_{(1,1)},$$
so $Mult_{\Lambda_1}((2,0))=Mult_{\Lambda_1}((1,1))=1$. For $\cB_{(2,0)}$ we have
$$F_1[11] = F_1E_1[1]=(E_1F_1-H_1)[1] = E_1(2[]) - (0)[1] = 2[1] =(2)_{(1,0)},$$
$$F_2[11]= 0 \text{ (immediately follows from the weight diagram)},$$
$$(2\beta_1+\Lambda_1)(H_1) = 2\beta_1(H_1)+\Lambda_1(H_1) = 2, \text{ and} $$
$$(2\beta_1+\Lambda_1)(H_2) = 2\beta_1(H_2)+\Lambda_1(H_2) = -4.$$
For $\cB_{(1,1)}$ we have
$$F_1[21] = 0 \text{ (immediately follows from the weight diagram)},$$
$$F_2[21]= F_2E_2[1]=(E_2F_2-H_2)[1] = E_2(0) - (-1)[1] = [1] =(1)_{(1,0)},$$
$$\beta_2(H_1)+(\beta_1+\Lambda_1)(H_1) = -3, \text{ and}$$
$$\beta_2(H_2)+(\beta_1+\Lambda_1)(H_2) =1.$$
Next we have $$\cS_{(2,1)}=\Big(E_2\cdot \cB_{(2,0)} \Big)\bigcup \Big(E_1\cdot \cB_{(1,1)}\Big) = \{[211], \ [121] \},$$
and
$$F_1[211] = E_2F_1[11] = E_2(2[1]) = 2[21] = (2)_{(1,1)}$$
$$F_2[211]=(E_2F_2-H_2)[11] = E_2(0) - (-4)[11] = (4)_{(2,0)},$$
and
$$F_1[121] = (E_1F_1-H_1)[21] = E_1(0) -(-3)[21]= 3[21] = (3)_{(1,1)}$$
$$F_2[121]=(E_1F_2)[21] = E_1([1]) = (1)_{(2,0)},$$
and $(2\beta_1+\beta_2)(H_1)=-1$, $(2\beta_1+\beta_2)(H_2)=-2$ for both vectors. Then, the solution to the system of equations
$$c_1 F_1[211] + c_2 F_1[121] = 0,$$
$$c_1 F_2[211] + c_2 F_2[121] = 0,$$
is exactly the null space of the matrix $\bm 2 & 3 \\ 4 & 1 \ebm$, which is trivial. No dependence relations exist, so $\cB_{(2,1)}=\cS_{(2,1)}$ and $Mult_{\Lambda_1}((2,1))=2$. Observe that this matrix is the transpose of the matrix below formed from the cells in Table \ref{tab:tab1} corresponding to $F_i\cdot \cS_{(2,1)}$ for $i=1,2$: 
$$A(\mu)=A(2,1)=\bm (F_1[211])_{(1,1)} & (F_2[211])_{(2,0)} \\  (F_1[121])_{(1,1)} & (F_2[211])_{(2,0)}  \ebm =  \bm 2 & 3 \\ 4 & 1 \ebm^T,$$
and linear dependence relations on $V^{\Lambda_1}_{(2,1)}$ are defined by basis vectors in $Null(A(2,1)^T)$. In fact, this generalizes for all $\mu=(n_1,n_2)\in P(V^{\Lambda_1})$, and it the reason we choose a consistent ordering for the spanning sets $\cS_\mu$. In general, we have the block-form $(m+n)\times (m+n)$ matrix (recall from (\eqref{eq:spanning}) that $m=\dim V^{\Lambda_1}_{\mu-\beta_2}, \ n=\dim V^{\Lambda_1}_{\mu-\beta_1}$),
\begin{align*}
A_\mu &=\bm \Big(F_1(E_2\cB_{\mu-\beta_2})\Big)_{\mu-\beta_1}^T & \Big(F_2(E_2\cB_{\mu-\beta_2})\Big)_{\mu-\beta_2}^T \\ 
\Big(F_1(E_1\cB_{\mu-\beta_1})\Big)_{\mu-\beta_1}^T & \Big(F_2(E_1\cB_{\mu-\beta_1})\Big)_{\mu-\beta_2}^T \ebm \\
&= \bm \Big((E_2F_1)\cB_{\mu-\beta_2}\Big)_{\mu-\beta_1}^T & \Big((E_2F_2-H_2)\cB_{\mu-\beta_2} \Big)_{\mu-\beta_2}^T \\ 
\Big((E_1F_1-H_1)\cB_{\mu-\beta_1}\Big)_{\mu-\beta_1}^T & \Big((E_1F_2)\cB_{\mu-\beta_1}\Big)_{\mu-\beta_2}^T \ebm, 
\end{align*}
where the rows of each block are coordinates of vectors  $F_i \cdot (E_j \cdot \cB_{\mu-\beta_j})$ for $i,j=1,2$ with respect to the basis $\cB_{\mu-\beta_i}$. Then 
\begin{align*}
A_\mu^T &= \bm \Big((E_2F_1)\cB_{\mu-\beta_2}\Big)_{\mu-\beta_1} & \Big((E_1F_1-H_1)\cB_{\mu-\beta_1}\Big)_{\mu-\beta_1} \\
 \Big((E_2F_2-H_2)\cB_{\mu-\beta_2} \Big)_{\mu-\beta_2} & \Big((E_1F_2)\cB_{\mu-\beta_1}\Big)_{\mu-\beta_2} \ebm .
\end{align*}
%is the matrix transformation corresponding to the actions of $F_1$, $F_2$ on $E_2 \cB_{\mu-\beta_2}$, $E_1 \cB_{\mu-\beta_1}$.

The author has found additional interesting recursive patterns relating the blocks of $A_\mu^T$ to the blocks of $A_{\mu-\beta_i}^T$ for each $i=1,2$, that further reduced computation time. These recursions can then be used to automate the algorithm described above using Mathematica. 

If a dependence relation is found on $\cS_\mu$, then in column 2 of Table \ref{tab:tab1} we color-code in red all vectors deleted from $\cS_\mu$ in the formation of basis $\cB_\mu$, and indicate the associated dependence relations on $\cS_\mu$ by writing these deleted vectors as coordinate vectors with respect to $\cB_\mu$. The first example encountered in the table is for $\mu=(4,1)$, where
$$\cS_{(4,1)}=\{[11211],[11121]\}.$$
Since $\mu-\beta_2\notin P(V^{\Lambda_1}_+)$, we have $(\cB_{(4,0)})_{3,1}=\emptyset$, 
so 
$$A(4,1)^T = \bm  0 & 0 \\ 2 & 3 \ebm ,$$
which gives us the linear dependence relation $-\frac{3}{2}[11211]+[11121]=0$. We then let $\cB_{(4,1)}=\cS_{(4,1)}-\{[11211]\}$, and we write $[11211]=\frac{2}{3}[11121]$ in the second column.

  \begin{figure}[!ht]
    \subfloat[Level 1 \label{subfig-1:nonstd1}]{%
    \includegraphics[trim = 5.5cm 0cm 5.5cm 0cm, clip=true, width=3 in]{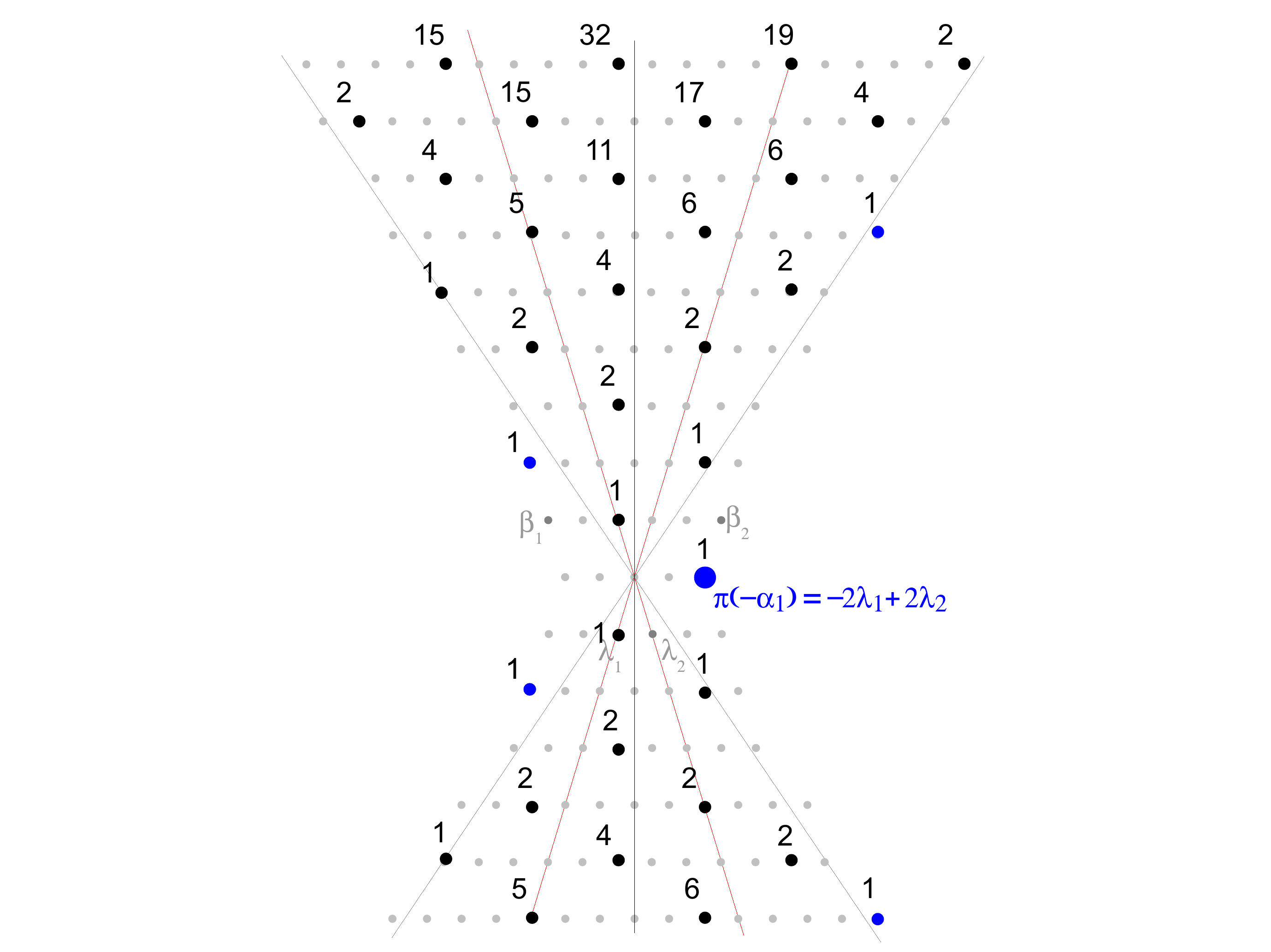} }
    \hfill
    \subfloat[Level 2 \label{subfig-2:nonstd2}]{%
 \includegraphics[trim = 5.5cm 0cm 5.5cm 0cm, clip=true, width=3 in]{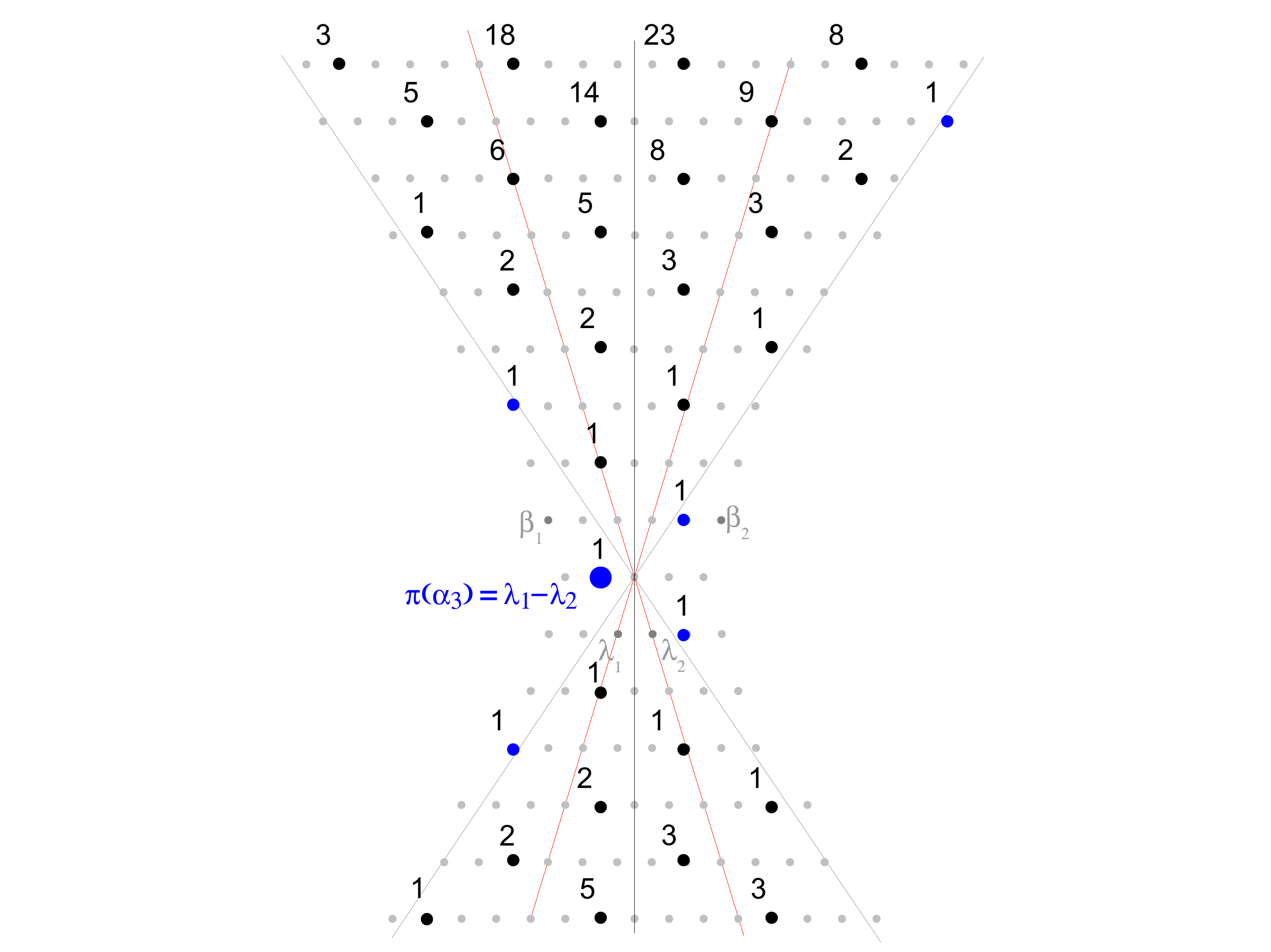} }
 \hfill
      \caption{Partial weight diagrams for non-standard $\Fib$-modules $V^{\Lambda_m}$ for $m=1$ and $2$. The inner multiplicities shown are $Mult_{\Lambda_m}(\mu)=\dim_{V^{\Lambda_1}}(V^{\Lambda_1}_\mu)$ for weights $\mu$ such that $-9\leq wt(\mu)\leq 6$ and were calculated using the recursive algorithm presented in Section \ref{sec:level1} (cf. Tables \ref{tab:tab1} and  \ref{tab:tab2}).}
    \label{fig:nonstdwts}
  \end{figure}

This algorithm was repeated for all $\mu\in P(V^{\Lambda_1})$  shown in Figure \ref{subfig-1:nonstd1}, though in principle it can be continued indefinitely to gain more inner multiplicity data. The algorithm was then applied to the non-standard module $V^{\Lambda_2}$, and the results are recorded in Table \ref{tab:tab2}. The data in Tables \ref{tab:tab1} and \ref{tab:tab2} can then be used to determine inner multiplicities  since $Mult_{\Lambda_m}(\mu)=\dim_{V^{\Lambda_m}}(V^{\Lambda_m}_\mu) = |\cB_\mu|$. Figure \ref{fig:nonstdwts} shows the resulting weight diagrams for non-standard modules on levels 1 and 2,  labeled with inner multiplicities. 
  
We observe that the weights in Figure  \ref{fig:nonstdwts}  do not follow the Kac-Peterson recursion (cf. Appendix \ref{appendix:A}), but instead follow a Racah-Speiser recursion (cf. Figure \ref{fig:racah}). 

\begin{conjecture}
For $m=\pm 1, \pm 2$, the weights of $P(V^{\Lambda_m}_-)$ follow the Racah-Speiser recursion
$$Mult_\lambda(\mu) = -\sum_{1 \neq w \in W_\Fib} \det(w) \ Mult_\lambda\Big(\mu +(w\rho - \rho)\Big),$$  
and the weights of $P(V^{\Lambda_m}_+)$ follow the Racah-Speiser recursion
$$Mult_\lambda(\mu) = -\sum_{1 \neq w \in W_\Fib} \det(w) \ Mult_\lambda\Big(\mu -(w\rho - \rho)\Big),$$  
where $\rho=\lambda_1+\lambda_2$.
\end{conjecture}

If this conjecture is true, then the non-standard modules on levels $\pm1, \pm2$ have more in common with highest- and lowest-weight modules than with the non-standard module (adjoint representation) on level 0. 

\section{Outer multiplicities of $\Fib(\pm1)$ and $\Fib(\pm2)$}\label{sec:outer1}
\vspace{-12pt}
\begin{figure}[!ht]
    \subfloat[Partial weight diagram for \newline $Y(1) = \Fib(1)\Big/ V^{\Lambda_1}$. \label{subfig-1:y1}]{%
      \includegraphics[trim = 6.5cm 1.7cm 6.5cm 0cm, clip=true, width=1.9 in]{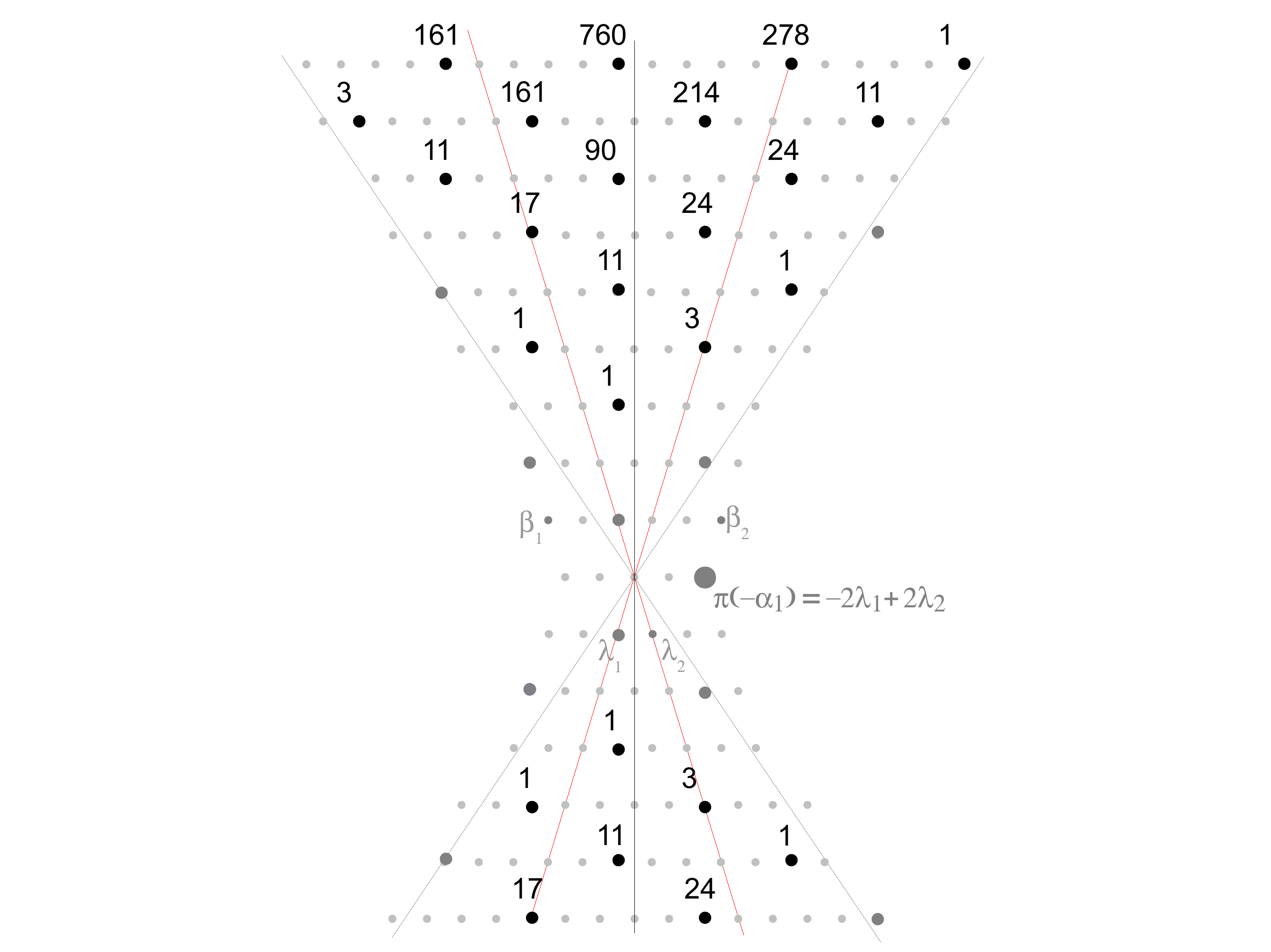}
    }
    \hfill
    \subfloat[Partial weight diagram for the irreducible module $V^{(2,1)}$. \label{subfig-2:21}]{%
      \includegraphics[trim = 6.5cm 1.7cm 6.5cm 0cm, clip=true, width=1.9 in]{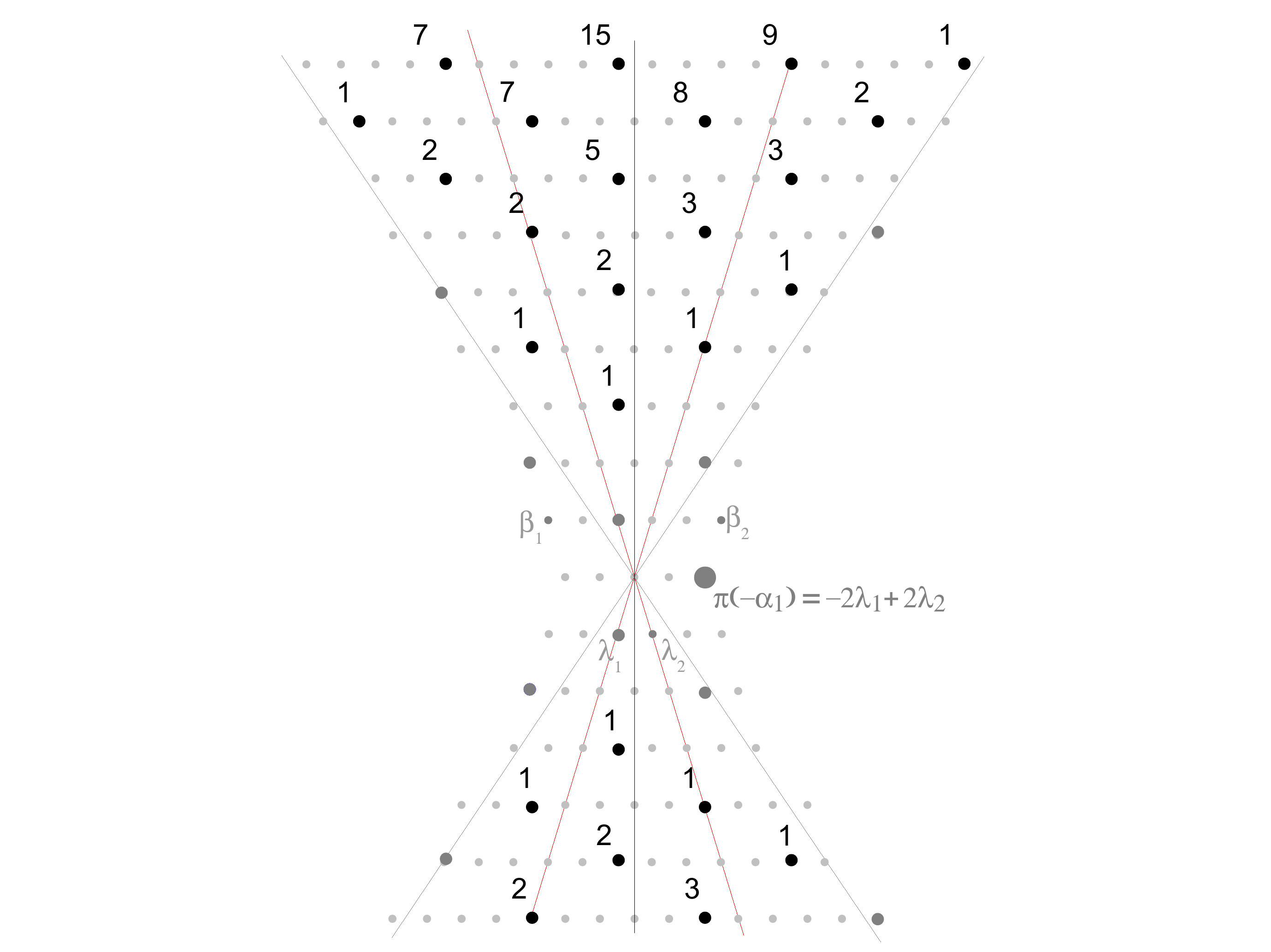}
    }
        \hfill
    \subfloat[Partial weight diagram for  $Y^1(1)=$ $\Fib(1)\Big/\Big(V^{\Lambda_{1}}\oplus V^{(2,1)}\oplus V^{\psi(2,1)}\Big).$  \label{subfig-3:y1-21}]{%
      \includegraphics[trim = 6.5cm 1.7cm 6.5cm 0cm, clip=true, width=1.9 in]{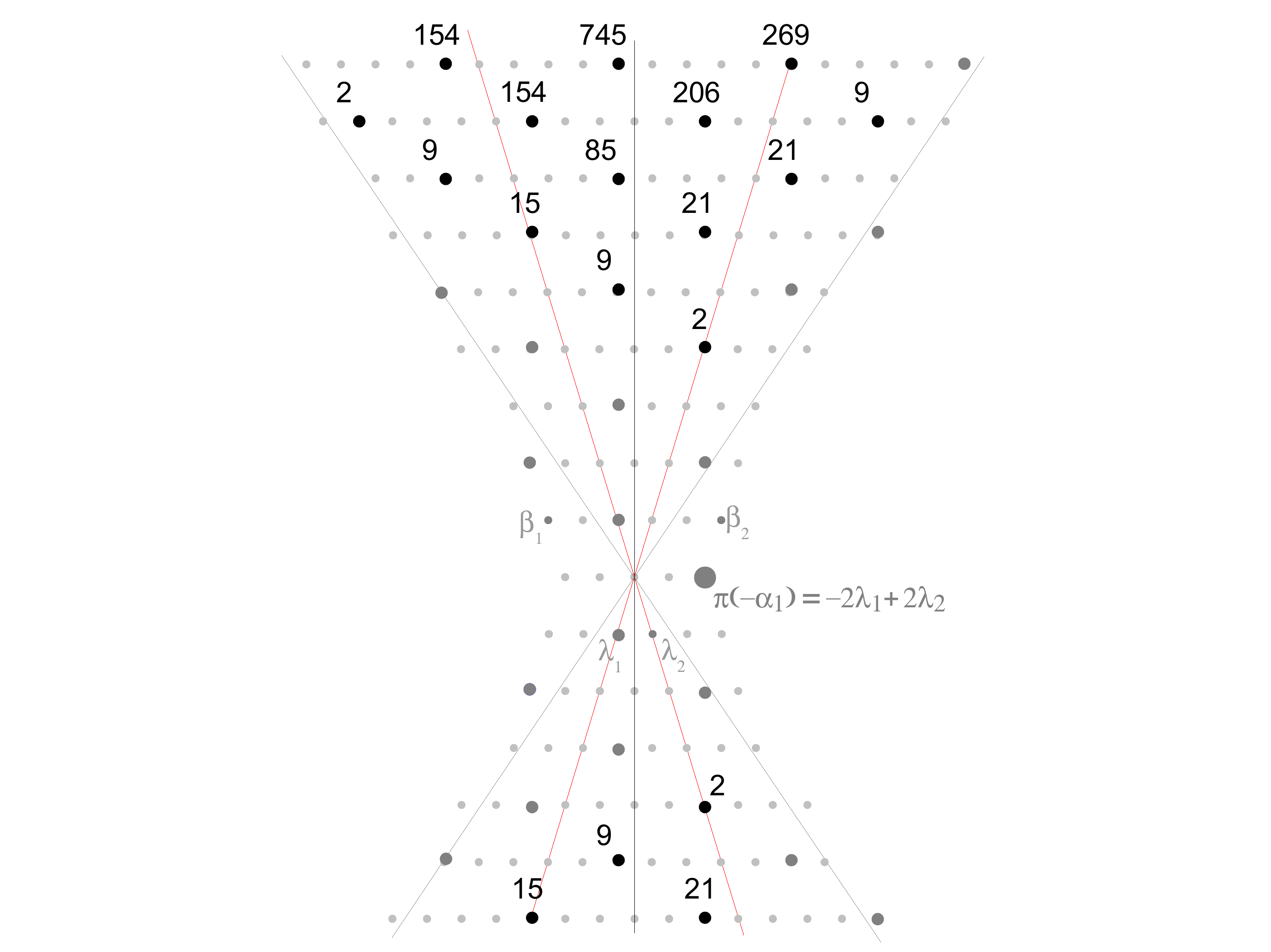}
    }

    \caption{Partial weight diagrams for three $\Fib$-modules on level 1 labeled with their multiplicities. Notation $(n_1,n_2)$ refers to weight $\Lambda_1+n_1\beta_1+n_2\beta_2$, since these diagrams are all in $\Fib(1)$.}
    \label{fig:3moduleslevel1}
  \end{figure}
We may now use the method described in Section \ref{sec:low} to locate extremal weights for $\Fib$ on levels $\pm1, \pm2$, and determine their outer multiplicities. For example, Figure \ref{subfig-1:y1} shows a portion of the weight diagram of the (completely reducible) quotient module $Y(1)=\Fib\Big/V^{\Lambda_{1}},$ with multiplicities $Mult_{Y(1)}(\mu)=Mult_1(\mu)-Mult_{\Lambda_1}(\mu).$  We observe from this diagram that $(2,1)$ is a lowest weight for $\Fib$, and using Racah-Speiser we find the labeled weight diagram for $V^{(2,1)}$ shown in Figure \ref{subfig-2:21}. These two diagrams then give us the labeled weight diagram for the quotient module  
$$Y^1(1)=\Fib(1)\Big/\Big(V^{\Lambda_{1}}\oplus V^{(2,1)}\oplus V^{\psi(2,1)}\Big)$$
shown in Figure \ref{subfig-3:y1-21}. We then observe that $M_1((2,2))=2$. 

Continuing in this way gives some data on outer multiplicities for level 1, presented in Table \ref{tab:outer1}. The table is ordered by increasing outer multiplicity $M_1(\lambda)$.
  
\begin {table}[!ht]
\hfill
    \subfloat[Outer multiplicity data for level 1.\label{tab:outer1}]{%
 \begin{tabular}{ | c | c |}           
 \hline
   $\lambda=(n_1,n_2)_1$ & $M_1(\lambda)$ \\
\hline
$(2,1)$ & 1 \\
$(2,2)$ & 2 \\
$(4,2)$ & 6 \\
$(3,2)$ & 7 \\
$(3,3)$ & 12 \\
$(4,3)$ & 49 \\
$(4,5)$ & 54 \\
$(5,3)$ & 67 \\
$(4,4)$ & 100 \\
$(5,4)$ & 385 \\
  \hline 
\end{tabular}}%
\hfill
    \subfloat[Outer multiplicity data for level 2.\label{tab:outer2}]{%
\begin{tabular}{ | c | c |}           
 \hline
   $\lambda=(n_1,n_2)_2$ & $M_2(\lambda)$ \\
\hline
$(2,2)$ & 3 \\
$(2,3)$ & 5 \\
$(3,3)$ & 14 \\
$(4,3)$ & 16 \\
$(3,5)$ & 20 \\
$(3,4)$ & 36 \\
$(4,4)$ & 107 \\
$(4,5)$ & 295 \\
  \hline 
\end{tabular}
}%
\hspace{80pt}

    \caption{The sequence of outer multiplicities of irreducible LW $\Fib$-modules in levels 1 and 2, where notation $(n_1,n_2)_m$ is as defined in Definition \ref{()}.}
    \label{fig:outer1and2}
\end{table}
\begin{figure}[!ht]
    \subfloat[Partial weight diagram for \newline $Y(2) = \Fib(2)\Big/ V^{\Lambda_2}$. \label{subfig-1:y2}]{%
      \includegraphics[trim = 6.5cm 1.6cm 6.5cm 0cm, clip=true, width=1.9 in]{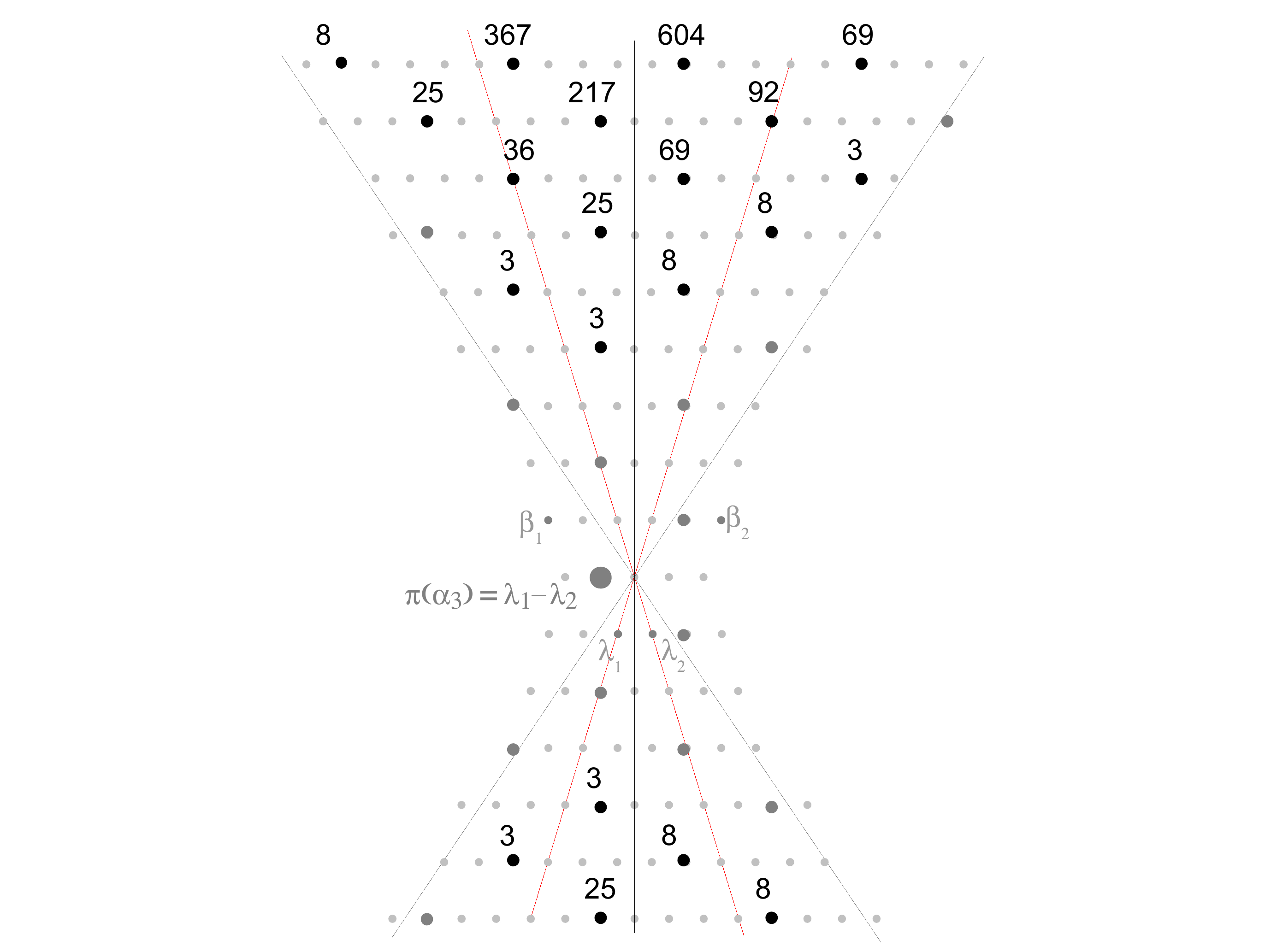}
    }
    \hfill
    \subfloat[Partial weight diagram for the irreducible module $V^{(2,2)}$. \label{subfig-2:22}]{%
      \includegraphics[trim = 6.5cm 1.6cm 6.5cm 0cm, clip=true, width=1.9 in]{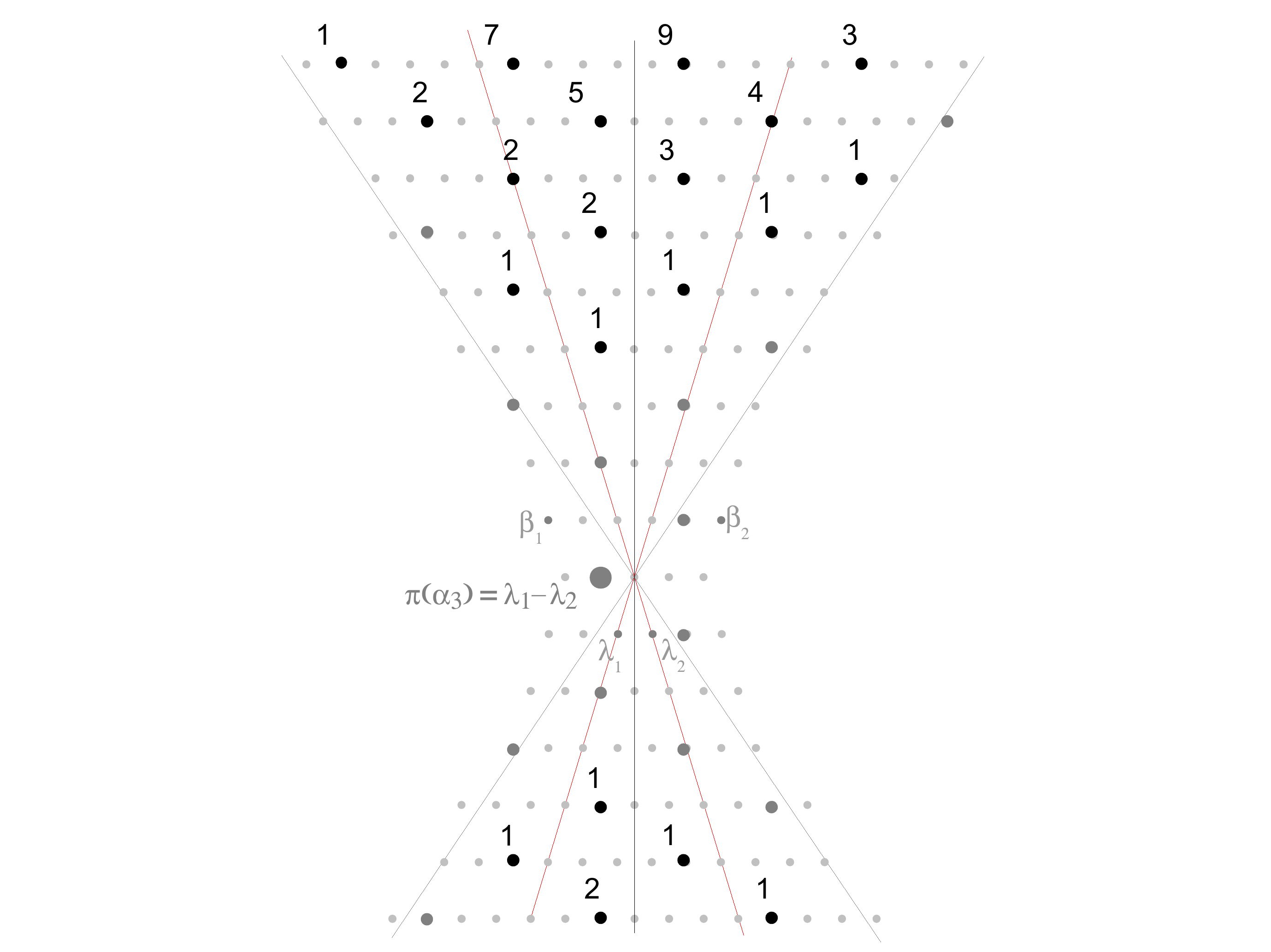}
    }
        \hfill
    \subfloat[Partial weight diagram for  $Y^1(2)=$ $\Fib(2)\Big/\Big(V^{\Lambda_{2}}\oplus V^{(2,2)}\oplus V^{\psi(2,2)}\Big).$  \label{subfig-3:y2-22}]{%
      \includegraphics[trim = 6.5cm 1.6cm 6.5cm 0cm, clip=true, width=1.9 in]{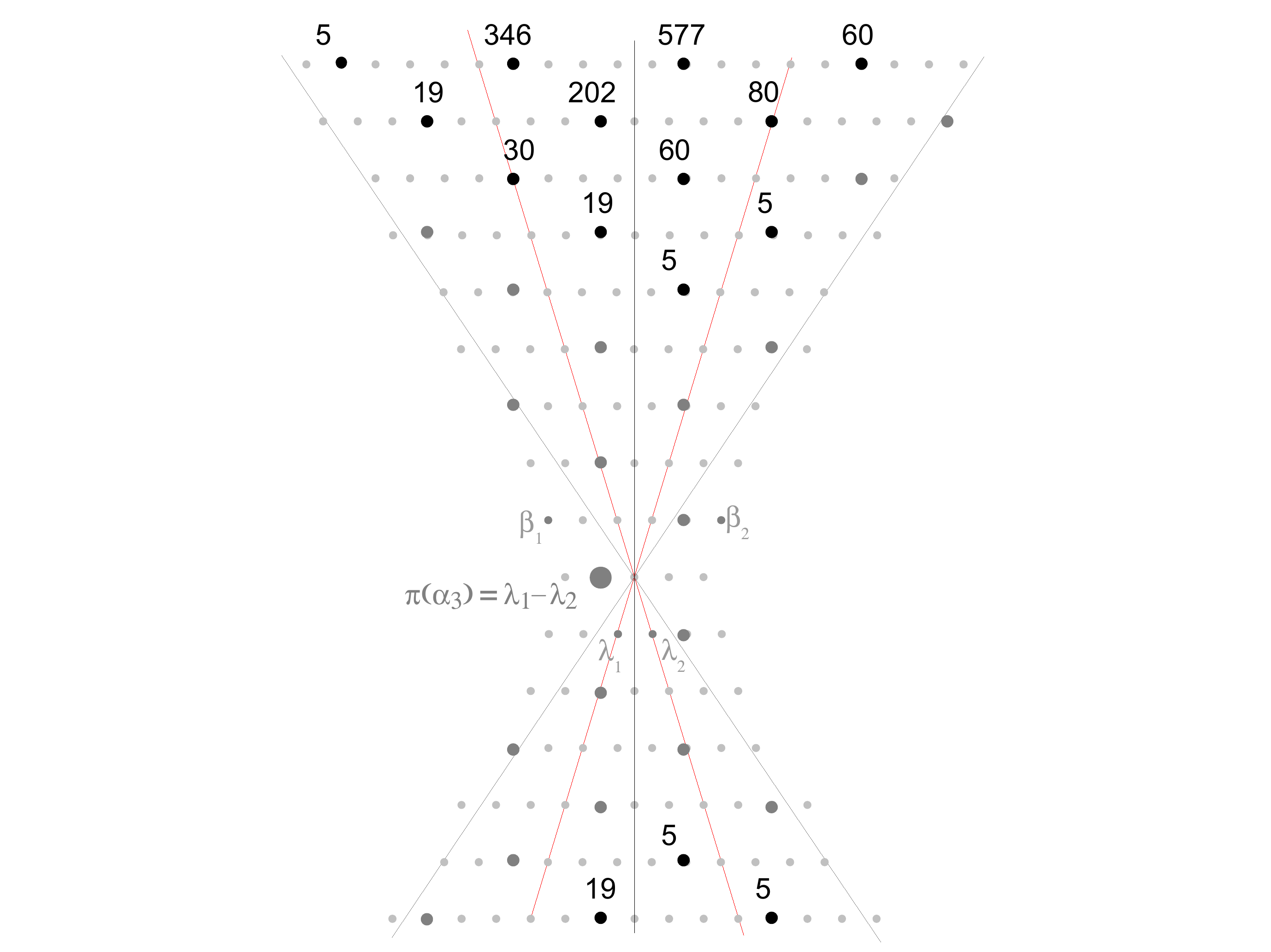}
    }
    \caption{Partial weight diagrams for three $\Fib$-modules on level 2 labeled with their multiplicities. Notation $(n_1,n_2)$ refers to weight $\Lambda_2+n_1\beta_1+n_2\beta_2$, since these diagrams are all in $\Fib(2)$.}
    \label{fig:3moduleslevel2}
  \end{figure}

Similarly, for level 2 we have Figure \ref{fig:3moduleslevel2}, which shows portions of the weight diagrams for the (completely reducible) quotient module $Y(2)=\Fib\Big/V^{\Lambda_{2}},$
the irreducible lowest-weight module $V^{(2,2)}$ (which has inner multiplicity 5), and the (completely reducible) quotient module 
$$Y^1(2)=\Fib(2)\Big/\Big(V^{\Lambda_{1}}\oplus V^{(2,2)}\oplus V^{\psi(2,2)}\Big).$$
We then observe that $(2,2)$ is a lowest weight for $\Fib$ with outer multiplicity $M_2((2,2))=3$. As before, we may continue the procedure indefinitely, obtaining more data on extremal vectors and their outer multiplicities. Table \ref{tab:outer2} presents the sequence of outer multiplicities for extremal weights on level 2, ordered by increasing outer multiplicity $M_2(\lambda)$.

%--------------------------------------------------------%
%--------------------------------------------------------%
%--------------------- Chapter 6 --------------------%
%--------------------------------------------------------%
%--------------------------------------------------------%
\chapter{The Vertex algebra approach}\label{ch:vertex}
%--------------------------------------------------------%
%---------------------Section 6.1--------------------%
%--------------------------------------------------------%

We observe that the algorithmic approach to finding extremal vectors for $\Fib$ described in Chapters \ref{ch:decomp0} and \ref{ch:nonstd} is time-intensive, and the sequences of outer multiplicities produced do not appear to follow any recognizeable pattern. We therefore seek an alternative approach using the theory of vertex algebras, which may give some insight into the decomposition of $\FF$ with respect to $\Fib$. 

\section{Definitions and the vertex algebra $V_\Fib$}\label{sec:vfib} 
In \cite{B} Borcherds gave a prescription for constructing a vertex algebra from any lattice, including indefinite lattices. In this section, we apply Borcherds' method to the indefinite root lattice $Q_\Fib=\Z\beta_1\oplus\Z\beta_2$, and in Section \ref{sec:VF} we apply it to $Q_\FF=\Z \alpha_1 \oplus\Z\alpha_2\oplus \Z\alpha_3$. 

Consider the Fock space $S(\hat{\HH}^-_\Fib)$, the algebra of symmetric polynomials in the commuting variables $\{\beta_i(-m), \mid i=1,2, \ 0< m\in \Z \}$ (see \cite{FLM}), which is a representation of the infinite-dimensional Heisenberg algebra with basis $\{ \mathbbm{1}, \beta_i(m) \ | \ i=1,2, \ m\in\Z \}$ and relations
$$[\beta_i(m),\beta_j(n)]=m(\beta_i,\beta_j)\delta_{m,-n}\mathbbm{1} \quad \text{and} \quad [\mathbbm{1},\beta_i(m)]=0.$$
We refer to $m$ in the operator $\beta_i(m)$ as the \textit{mode number} of $\beta_i(m)$. We define
$$V_{\Fib}=S(\hat{\HH}_\Fib^-)\otimes \C[Q_\Fib],$$
where $\C[Q_\Fib]$ is the group algebra of $Q_\Fib$ with basis $\{e^\beta \ | \ \beta\in Q_\Fib\}$ and multiplication given by $e^\beta e^\gamma = e^{\beta+\gamma}.$  Note that $\beta_i(0)$ is central since $[\beta_i(0),\beta_j(n)]=0(\beta_i,\beta_j)\delta_{0,-n}\mathbbm{1}=0,$ and there is a diagonal action of $\beta_i(0)$ on $\C[Q_\Fib]$ by 
\begin{equation}\label{eq:beta(0)}\beta_i(0)e^\lambda=(\beta_i,\lambda)e^\lambda.\end{equation}

For \textit{homogeneous} vectors in $V_\Fib$, that is, vectors of the form $\beta_{i_1}(-m_1)\cdots\beta_{i_r}(-m_r)\mathbbm{1}\otimes e^\beta$ where $0\leq m_i\in\Z$ and $i_j\in \{1,2\}$, $1\leq j\leq r$, define the $\Z$-valued weight function given by
\begin{equation}\label{eq:wt(v)}wt\Big(\beta_{i_1}(-m_1)\cdots\beta_{i_r}(-m_r)\mathbbm{1}\otimes e^\beta\Big) = \sum_{j=1}^r m_j + \frac{(\beta,\beta)}{2}.\end{equation}
As a vector space, $V_\Fib$ is graded by weight:
$$V_\Fib=\bigoplus_{m\in\Z}V_{\Fib, m},$$
where 
$V_{\Fib, m}=Span(\{\vect v\in V_\Fib \ | \ wt(\vect v)=m \}).$ 
There is a compatible grading by $Q_\Fib$,
$$V_\Fib=\bigoplus_{\beta\in Q_\Fib}V_\Fib^\beta,$$
where 
$V_\Fib^\beta=S(\hat{\HH}_\Fib^-)\otimes e^\beta,$
that is, $V_{\Fib,m}^\beta= V_\Fib^\beta \cap V_{\Fib, m}.$ Since the lattice $Q_\Fib$ is indefinite, both $V_\Fib^\beta$ and $V_{\Fib, m}$ are infinite dimensional. However, note that 
$$\dim(V_{\Fib, m}^\beta)=p\Big(m-\frac{(\beta,\beta)}{2}\Big)<\infty,$$
where $p(n)$ is the partition function (\cite{B}).
 
Later we will give the definition of vertex operators $Y(\cdot,z):V_\Fib\rightarrow End(V_\Fib)[[z, z^{-1}]]$ whose components are given by the notations
$$Y(\vect v,z)=\sum_{m\in\Z}\vect v_mz^{-m-1} = \sum_{n\in\Z}Y_n(\vect v) z^{-n-wt(\vect v)},$$
so
$$\vect v_{n+wt(\vect v)-1}=Y_n(\vect v)\in End(V_\Fib).$$
The operator $\vect v_m$ is computed by
\begin{equation}\label{eq:vm}\vect v_m=\oint Y(\vect v,z)z^m \ dz = Res_{z=0}(Y(\vect v,z)z^m),\end{equation}
which is the coefficient of the $z^{-1}$ term, so
$$Y_n(\vect v)=Res_{z=0}(Y(\vect v,z)z^{n+wt(\vect v)-1}).$$

We define the {\it vertex operator} $Y(\vect v,z)$ for the following choices of $\vect v\in V_\Fib$:

For $\vect v=\mathbbm{1}\otimes e^0$ (called the {\it vacuum vector}),
$$Y(\mathbbm{1}\otimes e^0,z)=I_{V_\Fib}.$$
We also have
$$Y\Big(\beta(-1)\mathbbm{1}\otimes e^0,z\Big) = \beta(z) =  \sum_{n\in\Z}\beta (n) z^{-n-1},$$
and for general $m\geq 0$,
$$Y\Big(\beta(-m-1)\mathbbm{1}\otimes e^0,z\Big) = \beta^{(m)}(z) = \frac{1}{m!}\Big(\frac{d}{dz}\Big)^m \sum_{n\in\Z}\beta (n) z^{-n-1} =  \frac{1}{m!}\Big(\frac{d}{dz}\Big)^m \beta(z).$$
$$Y\Big(\mathbbm{1}\otimes e^\beta,z\Big) = \exp\Big(\sum_{k>0} \frac{\beta(-k)}{k} z^k\Big)\exp\Big(\sum_{k>0} \frac{\beta(k)}{-k} z^{-k}\Big)e^\beta z^{\beta(0)}\varepsilon_\beta,$$
%Reminder: When you need it (like for computation with normal ordering) write the definition $\beta(z)=\beta^{(0)}(z).
where $z^{\beta(0)}$ and $\varepsilon_\beta$ act on $\vect v=w\otimes e^\lambda\in V_\Fib^\lambda$ by
$$z^{\beta(0)}(w\otimes e^\lambda)=z^{(\beta,\lambda)}(w\otimes e^\lambda),$$
$$\varepsilon_\beta (w\otimes e^\lambda) =\varepsilon(\beta,\lambda) (w\otimes e^\lambda),$$
and $\varepsilon(\beta,\lambda)=\pm1$ is a 2-cocycle (which may be chosen bilinear) subject to the conditions
$$\frac{\varepsilon(\beta, \lambda)}{\varepsilon(\lambda, \beta)}=(-1)^{(\beta, \lambda)}$$
for any $\beta, \lambda\in Q_\Fib$. We fix such a bilinear 2-cocycle determined by the matrix of values on the simple roots, 
\begin{equation}\label{Fibcocycle}\Big(\varepsilon(\beta_i,\beta_j)\Big)=\left(\begin{array}{rr}
	-1 & 1  \\
	-1 & -1  \\
	\end{array}\right),\end{equation}
	for $i,j=1,2$. It is clear that for all $n\in\Z$, $\varepsilon(n\beta,\gamma)=\varepsilon(\beta,n\gamma)=\varepsilon(\beta,\gamma)^{n}.$

We sometimes write $e^\beta = \mathbbm{1}\otimes e^\beta.$ 

For $m_{i_1}, \ldots,  m_{i_r} \geq 0,$ and $i_1, \ldots, i_r =1,2$,
$$Y\Big(\beta_{i_1}(-m_{i_1}-1)\cdots \beta_{i_r}(-m_{i_r} -1)\mathbbm{1}\otimes e^\beta,z\Big)= \ :\beta^{(m_1)}_{i_1}(z)\cdots\beta^{(m_r)}_{i_r}(z)Y(\mathbbm{1}\otimes e^\beta,z):,$$
where $: \cdot :$ is the bosonic normal ordering that places Heisenberg operators with positive mode numbers to the right and negative mode numbers to the left. 

It can be shown that $(V_\Fib, Y(\cdot,z), \vect \omega = \vect \omega_\Fib, \mathbbm{1}\otimes e^0)$ is a vertex algebra, where
$$\vect \omega=\frac{1}{2}\sum_{i=1}^2\beta_i(-1)\lambda_i(-1)\mathbbm{1}\otimes e^0$$
(so $wt(\vect \omega)=2$), and $\{\lambda_1,\lambda_2\}$ are the fundamental weights of $\Fib$ (see Chapter \ref{ch:fib}). We call $\vect \omega$ the {\it conformal vector}, and 
$$Y(\vect \omega, z)=\sum_{n\in\Z} L_n z^{-n-2}=\sum_{n\in\Z} Y_n(\vect \omega) z^{-n-2}=\sum_{m\in\Z} \vect \omega_m z^{-m-1},$$ 
where the operators $L_n$ represent the Virasoro algebra with central charge $c_{vir}=\text{rank}(Q_\Fib)=2$
and Lie brackets given by
$$[L_m,L_n] =(m-n)L_{m+n}+ c_{vir} \delta_{m,-n}(m^3-m)/12 \quad \text{and}\quad [c_{vir},L_n]=0.$$
It immediately follows that $L_n=\vect \omega_{n+1}$. The Virasoro operators act only on the Fock space component of each vector. The operator $L_0$ acts diagonalizably on $V_\Fib$, such that for $\vect v\in V_{\Fib,m}$, $L_0\cdot \vect v = m \vect v=wt(\vect v)\vect v.$ We also have that $L_{-1}=\vect \omega_0$ acts as a derivation,
% Check this in FLM $$[L_{-1},Y(u,z)] =\oint \ds\frac{d}{dz} \ Y(u,z) \ dz.$$
\begin{equation}\label{eq:l-1}[L_{-1},Y(\vect u,z)] = Y(L_{-1}\vect u, z) =\ds\frac{d}{dz} \ Y(\vect u,z).\end{equation}
which will prove useful in Section \ref{sec:LWVinVF}.
\begin{lemma}[\cite{FLM}]\label{flm} Let $V$ be a vertex algebra constructed from root lattice $Q$ with Virasoro operators $L_i$ and vacuum vector $\mathbbm{1}$. Then for $\alpha\in Q$,
\begin{itemize}
\item[(i)]  $[L_m,\alpha(-n)]=n\alpha(m-n)$ for $m,n\in\Z$, %[FLM] Equation 8.7.13 on p. 218
\item[(ii)] $L_m\mathbbm{1}=0$ for $0\leq m \in \Z$.
\end{itemize}
\end{lemma}
If $\vect u,\vect v \in V$ then the weight of vector $\vect u_n(\vect v)$ is given by the formula
$$wt(\vect u_n(\vect v))=wt(\vect u)+wt(\vect v)-n-1,$$
which shows that if $\vect u\in V_1$ and $\vect v\in V_m$, $wt(\vect u_0(\vect v))=1+m-0-1=m,$
so $\vect u_0(V_{\Fib, m})\subset V_{\Fib,m}$. 

%Does L(n) commute with all of Fib? If so, $L(n)(u_0(v))=u_0(L(n)(v))$\ Yes, Borcherds states the algebra $P_1/L{_1}(P_0)$ is a Lie algebra that commutes with the Virasoro algebra.

\section{The vertex algebra $V_\FF$}\label{sec:VF} 
We now apply Borcherd's method to the root lattice $Q_\FF=\Z\alpha_1\oplus \Z\alpha_2\oplus\Z\alpha_3$. Similar to the previous construction we have a Fock space $S(\hat{\HH}^-_\FF)$, the algebra of symmetric polynomials in the commuting variables $\{\alpha_i(m) \mid 0<m\in \Z, \ 1\leq i \leq 3 \}$, which is a representation of the Heisenberg algebra with basis $\{ \mathbbm{1}, \alpha_i(m)  \mid m\in\Z, \ 1\leq i \leq 3 \}$, and relations
$$[\alpha_i(m),\alpha_j(n)]=m(\alpha_i,\alpha_j)\delta_{m,-n}\mathbbm{1}, \quad \text{ and }\quad [\mathbbm{1},\alpha_i(m)]=0.$$
We define
$$V_{\FF}=S(\hat{\HH}_\FF^-)\otimes \C[Q_\FF].$$
The group algebra $\C[Q_\FF]$, $\Z$-valued weight function $wt$ on homogeneous vectors in $V_\FF$, gradings of the vector space $V_\FF$, vertex operators  $Y(\cdot,z):V_\FF\rightarrow End(V_\FF)[[z, z^{-1}]]$, are also defined as before with a new choice of 2-cocycle (see below).

We thus have the vertex algebra $(V_\FF, Y(\cdot,z), \vect \omega_\FF, \mathbbm{1}\otimes e^0)$ where
$$\vect \omega_\FF=\frac{1}{2}\sum_{i=1}^3\alpha_i(-1)\omega_i(-1)\mathbbm{1}\otimes e^0,$$
is the conformal vector for $V_\FF$, and $\{\omega_1,\omega_2, \omega_3\}$ are the fundamental weights for $\FF$ (see Section \ref{sec:F}). Since $Q_\Fib\subset Q_\FF$, we get $V_\Fib\subset V_\FF$, where compatibility of the vertex algebra structures requires that the choice of 2-cocycle for $\Fib$ be the restriction to $Q_\Fib$ of the 2-cocycle of $\FF$. We fix such a bilinear 2-cocycle, determined by the matrix of values on the simple roots, 
\begin{equation}\Big(\varepsilon(\alpha_i,\alpha_j)\Big)=\left(\begin{array}{rrr}
	-1 & 1 & 1 \\
	1 & -1 & 1 \\
	1 & -1 & -1 \\
	\end{array}\right).\end{equation}\label{Fcocycle}
	Note that the 2-cocycles for $\Fib$ and $\FF$ can be distinguished by their domains, so we use the same notation for both. Using the notation
$$\varepsilon_{ij} = \varepsilon(\alpha_i,\alpha_j),$$
we have
$$\varepsilon(\beta_1,\beta_1)=\varepsilon(2\alpha_1+\alpha_2+\alpha_3, 2\alpha_1+\alpha_2+\alpha_3)=\varepsilon_{11}^4\varepsilon_{12}^2\varepsilon_{13}^2\varepsilon_{21}^2\varepsilon_{22}\varepsilon_{23}\varepsilon_{31}^2\varepsilon_{32}\varepsilon_{33}=-1,$$
$$\varepsilon(\beta_1,\beta_2)=\varepsilon(2\alpha_1+\alpha_2+\alpha_3, \alpha_2)=\varepsilon_{12}^2\varepsilon_{22}\varepsilon_{32}=1,$$
$$\varepsilon(\beta_2,\beta_1)=\varepsilon(\alpha_2, 2\alpha_1+\alpha_2+\alpha_3)=\varepsilon_{21}^2\varepsilon_{22}\varepsilon_{23}=-1,$$
$$\varepsilon(\beta_2,\beta_2)=\varepsilon(\alpha_2, \alpha_2)=\varepsilon_{22}=-1,$$
which shows that the 2-cocycles for $V_\FF$ and $V_\Fib$ are compatible. 

\section{Representation of $\Fib$ acting on $V_\FF$}\label{sec:Fib in VF} 
%Since vectors that generate $\Fib$-modules are contained in $\FF$, we will start with a representation of $\FF$ inside $V_\FF$.  As $Q_\Fib$ is a sublattice of $Q_\FF$, the restriction of this representation to $\Fib$ will give a representation of $\Fib$ inside $V_\Fib$.

In \cite{FLM}, it was shown that if $V_L$ is a vertex operator algebra constructed from a positive-definite even lattice $L$, then the subspace of $End(V_L)$ spanned by the operators $\{ \vect u_m \mid  \vect u\in V_L, m\in\Z\}$ is a Lie subalgebra. In \cite{B}, Borcherds showed that this result extends to indefinite lattices, and in both cases, one has the commutator identity
$$[\vect u_m,\vect v_n]=\sum_{i\in \N} {m \choose i} (\vect u_i(\vect v))_{m+n-i}.$$ 
In special cases we have 
$$[\vect u_0,\vect v_n]=(\vect u_0(\vect v))_n \quad \text{and} \quad [\vect u_0,\vect v_0]=(\vect u_0(\vect v))_0,$$
so the space $End(V_L)_0$ of operators spanned by $\{ \vect u_0 \mid  \vect u\in V_L \}$ is a Lie subalgebra of $End(V_L)$, and the space $End(V_L)_n$ spanned by $\{ \vect v_n \mid  \vect v\in V_L\}$ is a module for $End(V_L)_0$. Furthermore, the span of $\{\vect u_0 \mid \vect u\in V_L, \ wt(\vect u)=1 \}$ is a Lie subalgebra of $End(V_L)_0$ since $wt(\vect u_0(\vect v))=wt(\vect u)+wt(\vect v)-0-1=1.$  We focus on the following special weight-1 vectors.
%We will therefore represent both $\FF$ and $\Fib$ by zero-mode operators of weight 1 vectors.
\begin{definition} Let $\LL$ be a KM algebra, $V_\LL$ the vertex algebra constructed from the root lattice of $\LL$, with Virasoro operators $L^\LL_n, n\in \Z$. For $i\in \Z$, the \textbf{physical i-space} of $V_\LL$ is
$$P^\LL_i =\{\vect v\in V_\LL \ | \ L^\LL_0 \vect v = i\vect v,  \ L^\LL_m \vect v=0 \text{ if } m > 0\}.$$
\end{definition}

From now on the reader may assume that $P_n$ means $P^\FF_n$ and $L_n$ means $L^\FF_n$. If a vector $\vect v\in P_1$, then $L_0 \vect v=wt(\vect v)\vect v=\vect v$, so $P_1\subset V_{\FF,1}$. For $wt(\vect v)$ arbitrary and any $n\in \Z$ we have
$$wt(L_n(\vect v))=wt(\vect \omega_{n+1}(\vect v))=wt(\vect \omega)+wt(\vect v)-(n+1)-1 = wt(\vect v)-n,$$
and in particular, if $wt(\vect v)=0$, then $wt(L_{-1}(\vect v))=1$, so $L_{-1}(P_0)\subset P_1$.

By a result of Borcherds, the quotient space $\ovP=P_1\Big/L_{-1}(P_0)$ is a Lie algebra, where the Lie algebra bracket is given by 
\begin{equation}\label{eq:bracket}[\ovect u,\ovect v] = \ovect u_0(\ovect v) \quad \text{for} \quad \ovect u, \ovect v \in \ovP\end{equation}
(\cite{B} $\S 5$). This formula, which defines the adjoint representation of $\ovP$, also defines the action of $\ovP$ on itself as a $\ovP$-module. In general, $V_\FF$ is also a $\ovP$-module with the action $\ovect u\cdot \vect v = \ovect u_0(\vect v)$ for $v\in V_\FF$. Note that if $\vect u\in L_{-1}(P_0)$, then $\ovect u_0 = \ovect 0_0 = 0$. 

Borcherds gives us the Lie algebra representation $\pi_\FF: \FF\rightarrow \ovP$ defined by
$$\pi_\FF(e_i)=\mathbbm{1}\otimes e^{\alpha_i} , \qquad  \pi_\FF(f_i)=-\mathbbm{1}\otimes e^{-\alpha_i} , \qquad \text{and} \qquad \pi_\FF(h_i)=\alpha_i(-1)\mathbbm{1}\otimes e^0$$
for  $i=1,2,3$ (for convenience we sometimes suppress the bars for vectors in $\ovP$). In Appendix \ref{sec:piF}, we verify that these elements indeed satisfy the Serre relations of $\FF$. He also gives us $\pi_\Fib: \Fib\rightarrow \ovP^\Fib = P^\Fib_1\Big/ L^\Fib_{-1}(P^\Fib_0)$ given by 
$$\pi_\Fib(E_i)=\mathbbm{1}\otimes e^{\beta_i}, \quad \pi_\Fib(F_i)=-\mathbbm{1}\otimes e^{-\beta_i}, \quad \text{and} \quad \pi_\Fib(H_i)=\beta_i(-1)\mathbbm{1}\otimes e^0,$$
for $i=1,2$ (we have again suppressed the bar notation, this time for vectors in $\ovP^\Fib$).

We also prove in Appendix \ref{sec:piFib}  that restricting $\pi_\FF$ to $\Fib$ gives a representation $\pi_\FF|_\Fib$ of $\Fib$ acting on $V_\FF$ that is compatible with $\Fib\subset \FF$, i.e., $\pi_\FF|_\Fib=\pi_\Fib$. This is done by showing for $i=1,2,$
$$\pi_\FF(E_i)=\mathbbm{1}\otimes e^{\beta_i}, \quad \pi_\FF(F_i)=-\mathbbm{1}\otimes e^{-\beta_i}, \quad \text{and} \quad \pi_\FF(H_i)=\beta_i(-1)\mathbbm{1}\otimes e^0,$$
where $E_1=\frac{1}{2}e_{1123}, \ E_2=e_2, \ F_1=-\frac{1}{2}f_{1123}, \ F_2=f_2, \ H_1=2h_1+h_2+h_3,$ and $H_2=h_2$. 

The action of $\pi_\FF(\FF)$ on $\ovP$ is defined by the adjoint representation of $\ovP$ and the bracket formula \eqref{eq:bracket}, and is compatible with the action of $\FF$ on itself, since for all $x,y\in\FF$,
$$\pi_\FF(x) \cdot \pi_\FF(y) = [\pi_\FF(x), \pi_\FF(y)]  = \pi_\FF([x,y]) = \pi_\FF(x\cdot y),$$
where the action on the left-hand side is on $\ovP$ and the right-hand side is on $\FF$.

The restricted representation $\pi_\Fib$ makes $\ovP$ also a $\Fib$-module, where the action of $X\in\Fib$ on $\ovect v\in \ovP$ is given by
\begin{equation}\label{eq:action}\pi_\Fib(X) \cdot \ovect v = \big(\pi_\Fib(X)\big)_0(\ovect v).\end{equation}
%Note that $\pi_\Fib(\Fib)$ commutes with $L_{-1}(P_0)$, so for all $\ovect v \in \ovP$, we have $F_i\cdot \ovect v =\ovect 0$.
%Formerly 6.5

\section{Finding lowest-weight vectors for $\Fib$ in $\ovP^{-\rho}$} \label{sec:LWVinVF}

We look for lowest-weight vectors for $\Fib$ in $\ovP$, starting with the weight space $\ovP^{-\rho}$ where $-\rho=-\rho_\Fib=\beta_1+\beta_2$. In Theorem \ref{vrho} we found that the space of lowest-weight vectors $Low_{\Fib(0)}(-\rho)$ has dimension 1, and has basis $(v_{-\rho}=-3e_{12123}+2E_{21})$. Then 
$$\ovect x = \pi_\FF(v_{-\rho}) = \pi_\FF(-3e_{12123}+2E_{21}) \in \ovP^{-\rho}$$
is aso a lowest-weight vector for $\Fib$, since for $i=1,2$, 
$$\pi_\Fib(F_i)\cdot \ovect x = [\pi_\Fib(F_i),\ovect x]= [\pi_\FF(F_i),\pi_\FF(v_{-\rho})]= \pi_\FF([F_i, v_{-\rho}]) = \ovect 0.$$

Therefore $\ovect 0 \neq \ovect x \in Low_{\ovP}(-\rho)$, however, we will now show that $\dim Low_{\ovP}(-\rho)>1$. 

\begin{theorem}\label{ovPbasis}A basis for $\ovP^{-\rho}$ is $B_\ovP^{-\rho}=(\ovect p_1, \ \ovect p_2, \ \ovect p_3),$ where
$$\vect p_1=\Big(\alpha_2(-2)+\alpha_2(-1)^2\Big)\vac \otimes e^{-\rho},$$
$$\vect p_2=\Big(-\alpha_1(-1)^2+\alpha_3(-1)^2\Big)\vac \otimes e^{-\rho},  \text{  and}$$
$$\vect p_3=\alpha_1(-1)\alpha_3(-1)\vac \otimes e^{-\rho}.$$
\end{theorem}
\begin{proof}
All vectors in $\ovP^{-\rho}$ are weight 1 and have a degree-2 Fock space polynomial, since $||-\rho \ ||^2=-2$ (cf. equation \eqref{eq:wt(v)}). We have the ordered basis for the subspace of degree-2 polynomials in $S(\hat{\HH}^-_\FF)$,
$$S =(u_i \mid 1\leq i \leq 9) = \Big(\alpha_1(-2), \ \alpha_2(-2), \ \alpha_3(-2), \ \alpha_1(-1)^2,  \ \alpha_2(-1)^2, \ \alpha_3(-1)^2,$$
 $$\qquad \qquad \qquad \alpha_1(-1)\alpha_2(-1), \ \alpha_2(-1)\alpha_3(-1), \ \alpha_1(-1)\alpha_3(-1) \Big),$$
which gives us the basis for the 9-dimensional space  $V^{-\rho}_{\FF,1}$,
$$B_{\FF,1}^{-\rho} = \Big( u_i \vac \ox e^{-\rho} \mid \ 1\leq i \leq 9\Big).$$
Let
$$\vect v=\Big(\sum_{i=1}^9 c_i  u_i\Big)\mathbbm{1}\otimes e^{-\rho}, \text{  where  } c_i\in \C$$  
be a general vector in $V_{\FF,1}^{-\rho}$. Then $\vect v\in P_1^{-\rho}$ if and only if $L_m(\vect v)=\vect 0$ for $m=1,2$. Note that $L_0(\vect v)=\vect v$ is immediately true since $wt(\vect v)=1$. We therefore find $P_1^{-\rho}=Ker(L_1)\cap Ker(L_2)$. Using Proposition \ref{flm}, we have for $i=1,2$ and $j=1,2,3,$
\begin{align*}
L_i(\alpha_j(-2)\mathbbm{1}\otimes e^{-\rho}) &= \Big([L_i, \alpha_j(-2)]+\alpha_j(-2)L_1\Big)\mathbbm{1}\otimes e^{-\rho} \\
&=2\alpha_j(i-2)\mathbbm{1}\otimes e^{-\rho}, 
\end{align*}
so
$$L_1(\alpha_j(-2)\mathbbm{1}\otimes e^{-\rho})=2\alpha_j(-1)\mathbbm{1}\otimes e^{-\rho}\quad \text{and}\quad L_2(\alpha_j(-2)\mathbbm{1}\otimes e^{-\rho})=-2\delta_{j2}\mathbbm{1}\otimes e^{-\rho}.$$
For $(j,k)=(1,1), (2,2),(3,3),(1,2),(2,3), (1,3),$ we have
\begin{align*}
L_i(\alpha_j(-1)\alpha_k(-1)\mathbbm{1}\otimes e^{-\rho}) &= \Big([L_i, \alpha_j(-1)]+\alpha_j(-1)L_i\Big)\alpha_k(-1)\mathbbm{1}\otimes e^{-\rho} \\
&=\Big(\alpha_j(i-1)\alpha_k(-1)+ \alpha_j(-1)\alpha_k(i-1)\Big) \mathbbm{1}\otimes e^{-\rho}.
\end{align*}
We now have formulas for computing $L_i(u_t\vac\ox e^{-\rho})$ for $u_t\vac\ox e^{-\rho}\in\cB_{\FF,1}^{-\rho}$, $1\leq t \leq 9$, and $i=1,2$: 
$$L_1\big(\alpha_j(-1)\alpha_k(-1)\mathbbm{1}\otimes e^{-\rho}\big) =-\delta_{j2}\alpha_k(-1)-\delta_{k2}\alpha_j(-1)\mathbbm{1}\otimes e^{-\rho},$$
$$L_1\big(\alpha_j(-1)^2\mathbbm{1}\otimes e^{-\rho}\big) = -2\delta_{j2}\alpha_j(-1)\mathbbm{1}\otimes e^{-\rho},$$
$$L_2\big(\alpha_j(-1)\alpha_k(-1)\mathbbm{1}\otimes e^{-\rho}\big) =(\alpha_j,\alpha_k)\mathbbm{1}\otimes e^{-\rho},\quad \text{and}$$
$$L_2\big(\alpha_j(-1)^2\mathbbm{1}\otimes e^{-\rho}\big) =(\alpha_j,\alpha_j)\mathbbm{1}\otimes e^{-\rho}.$$
Applying these formulas we have
$$L_1(\vect v)=\Big(\alpha_1(-1)(2c_1-c_7) + \alpha_2(-1)(2c_2-2c_5)+\alpha_3(-1)(2c_3-c_8)\Big)\mathbbm{1}\otimes e^{-\rho}=\vect 0$$
and
$$L_2(\vect v)=(-2c_2+2c_4+2c_5+2c_6-2c_7-c_8)\mathbbm{1}\otimes e^{-\rho}=\vect 0.$$
The resulting system of linear equations
\begin{center}
 \systeme{
2c_1-c_7=0,
2c_2-2c_5=0,
2c_3-c_8=0,
-2c_2+2c_4+2c_5+2c_6-2c_7-c_8=0}
 \end{center}
has a $5$-dimensional solution space
$P_1^{-\rho}= Ker(L_1)\cap Ker(L_2)=Span(\vect p_1, \vect p_2, \vect p_3, \vect p_4, \vect p_5)$,
where 
$$\vect p_1=\Big(\alpha_2(-2)+\alpha_2(-1)^2\Big)\vac\ox e^{-\rho}=\bm 0 & 1 & 0 & 0 & 1 & 0 & 0 & 0 & 0 \ebm ^T,$$
$$\vect p_2=\Big(-\alpha_1(-1)^2+\alpha_3(-1)^2\Big)\vac\ox e^{-\rho}=\bm 0 & 0 & 0 & -1 & 0 & 1 & 0 & 0 & 0 \ebm ^T,$$
$$\vect p_3=\alpha_1(-1)\alpha_3(-1)\vac\ox e^{-\rho}=\bm 0 & 0 & 0 & 0 & 0 & 0 & 0 & 0 & 1 \ebm ^T,$$
$$\vect p_4=\Big(\frac{1}{2}\alpha_1(-2)+\alpha_1(-1)^2\Big)\vac\ox e^{-\rho} = \bm \frac{1}{2} & 0 & 0 & 1 & 0 & 0 & 1 & 0 & 0 \ebm ^T,$$
$$\vect p_5=\Big(\frac{1}{2}\alpha_3(-2)+\frac{1}{2}\alpha_1(-1)^2+\alpha_2(-1)\alpha_3(-1)\Big)\vac\ox e^{-\rho}=\bm 0 & 0 & \frac{1}{2}  & \frac{1}{2}  & 0 & 0 & 0 & 1 & 0 \ebm ^T,$$
where vectors are written with respect to the basis $B_{\FF,1}^{-\rho}$.

Let $Q:P_1^{-\rho}\rightarrow \ovP^{-\rho}$ be the quotient map given by $Q(\vect v) = \ovect v$ where $\ovect v$ is the equivalence class of $\vect v$ in the quotient space $\ovP^{-\rho}=P_1^{-\rho}\Big/L_{-1}(P_0^{-\rho})$, so if $\vect v\in L_{-1}(P_0^{-\rho})$ then $Q(\vect v)=\ovect 0$. We find a basis for $Ker(Q)\subset P_1^{-\rho}.$ Let $\vect v=L_{-1}(\vect u)$ where $\vect u\in P_0^{-\rho}$. Then $\vect v_0=0$, so by equations \eqref{eq:vm} and \eqref{eq:l-1},
$$\vect v_0=(L_{-1}(\vect u))_0=\oint Y\Big(L_{-1}(\vect u), z\Big) \ dz = \oint \frac{d}{dz}Y\Big(\vect u, z\Big) \ dz=0.$$
Note that since $wt(\vect u)=0$ and $(-\rho,-\rho)=-2$, the Fock space component of $\vect u$ has degree 1. Writing $\vect u=\alpha(-1)\otimes e^{-\rho}$ we have 
\begin{align*}
0 &= \oint \frac{d}{dz}Y\Big(\alpha(-1)\otimes e^{-\rho}, z\Big) \ dz \\
&= \oint \frac{d}{dz} : \alpha^{(0)}(z)Y\Big(e^{-\rho}, z \Big) : dz \\
&= \oint :\alpha^{(1)}(z)Y\Big(e^{-\rho}, z \Big): - :\alpha^{(0)}(z)\rho^{(0)}(z)Y\Big( e^{-\rho}, z\Big): dz \\
&= \oint Y\Big(\alpha(-2)\mathbbm{1} \otimes  e^{-\rho}, z \Big) - Y\Big(\alpha(-1)\rho(-1) \mathbbm{1} \otimes e^{-\rho}, z\Big) dz \\
&= \oint Y\Big(\big(\alpha(-2)-\alpha(-1)\rho(-1)\big) \mathbbm{1} \otimes  e^{-\rho}, z \Big) dz ,
\end{align*}
which holds if
\begin{equation}\label{eq:kerq}\bigg(\Big(\alpha(-2)-\alpha(-1)\rho(-1)\Big)\mathbbm{1}\otimes e^{-\rho}\bigg)_0=0,\end{equation}
or equivalently, if $\Big(\alpha(-2)-\alpha(-1)\rho(-1)\Big)\mathbbm{1}\otimes e^{-\rho}\in Ker(Q)$.
Furthermore, observe that Lemma \ref{flm} and equation \eqref{eq:beta(0)} give us
$$L_2 (\alpha(-1)\vac\ox e^{-\rho})= ([L_2, \alpha(-1)]+\alpha(-1)L_2)\vac\ox e^{-\rho} = 0$$
and
$$L_1 (\alpha(-1)\vac\ox e^{-\rho})= ([L_1, \alpha(-1)]+\alpha(-1)L_1)\vac\ox e^{-\rho} = \alpha(0)\vac\ox e^{-\rho} = (\alpha, -\rho)\vac\ox e^{-\rho}.$$
Thus if $u=(a\alpha_1(-1)+b\alpha_2(-1)+c\alpha_3(-1))\vac\ox e^{-\rho}\in P_0^{-\rho}$ where $a,b,c\in \C$,  we have
$$L_1(\vect u)= (a\cancel{(\alpha_1, -\rho)}+ b(\alpha_2, -\rho)+ c\cancel{(\alpha_3, -\rho)}) \vac\ox e^{-\rho} = -b \vac\ox e^{-\rho} = \vect 0,$$
hence $b=0$ and $\vect u \in Span(\alpha_1(-1)\vac\ox e^{-\rho}, \alpha_3(-1)\vac\ox e^{-\rho})$. Substituting $\alpha=c_1\alpha_1+c_2\alpha_3$ for some $c_1,c_2\in\C$ and $-\rho=2\alpha_1+2\alpha_2+\alpha_3$ into equation \eqref{eq:kerq} gives us
\begin{align*}
0 &=\bigg(\Big( c_1\alpha_1(-2)+c_2\alpha_3(-1) +\Big( c_1\alpha_1(-1)+c_2\alpha_3(-1) \Big) \\
&\qquad \qquad\cdot (2\alpha_1+2\alpha_2+\alpha_3)(-1)\Big)\mathbbm{1}\otimes e^{-\rho}\bigg)_0\\
&= \bigg(\Big(c_1\Big(\alpha_1(-2) + 2\alpha_1(-1)^2+2\alpha_1(-1)\alpha_2(-1)+\alpha_1(-1)\alpha_3(-1)\Big) \\
& \qquad \qquad + c_2\Big(\alpha_3(-2)+2\alpha_2(-1)\alpha_3(-1)+\alpha_3(-1)^2+2\alpha_1(-1)\alpha_3(-1)\Big)\Big)\mathbbm{1}\otimes e^{-\rho}\bigg)_0. 
\end{align*}
Thus we have $L_{-1}(P_0)^{-\rho} = Ker(Q) = Span (\vect q_1, \vect q_2)$, where 
\begin{align*}
\vect q_1&=\Big(\alpha_1(-2) + 2\alpha_1(-1)^2+2\alpha_1(-1)\alpha_2(-1)+\alpha_1(-1)\alpha_3(-1)\Big)\mathbbm{1}\otimes e^{-\rho}\\
&\quad =\bm 1 & 0 & 0 & 2 & 0 & 0 & 2 & 0 & 1 \ebm ^T,\\
%\vect q_2&=\Big(\alpha_2(-2)+2\alpha_1(-1)\alpha_2(-1)+2\alpha_2(-1)^2+\alpha_2(-1)\alpha_3(-1)\Big)\mathbbm{1}\otimes e^{-\rho}\\
%&\quad =\bm 0 & 1 & 0 & 0 & 2 & 0 & 2 & 1 & 0 \ebm ^T,\\
\vect q_2&=\Big(\alpha_3(-2)+2\alpha_2(-1)\alpha_3(-1)+\alpha_3(-1)^2+2\alpha_1(-1)\alpha_3(-1)\Big)\mathbbm{1}\otimes e^{-\rho}\\
&\quad =\bm 0 & 0 & 1 & 0 & 0 & 1 & 0 & 2 & 2 \ebm ^T.
\end{align*}
where vectors are written with respect to the basis $B_{\FF,1}^{-\rho}$. Observe that for each $i=1,2$, $\ovect q_i=\ovect 0$ gives a congruence relation on the vectors in $P_1^{-\rho}$. 

We now compute $\ovP^{-\rho}=Ker(Q)\cap Ker(L_1) \cap Ker(L_2)$ by finding the null space of
$$\bm  \vect q_1 & \vect q_2 & \vect p_1 & \vect p_2 & \vect p_3 &\vect p_4 &\vect p_5 \ebm \ = \  \left[\begin{array}{cccccccc}
       1 & 0  & 0 & 0 & 0 & \frac{1}{2} & 0 \\
       0 & 0  & 1 & 0 & 0 & 0 & 0 \\
       0 & 1  & 0 & 0 & 0 & 0 & \frac{1}{2} \\
       2 & 0  & 0 & -1 & 0 & 1 & \frac{1}{2} \\
       0 & 0  & 1 & 0 & 0 & 0 & 0 \\
       0 & 1  & 0 & 1 & 0 & 0 & 0 \\
       2 & 0  & 0 & 0 & 0 & 1 & 0 \\
       0 & 2  & 0 & 0 & 0 & 0 & 1 \\
       1 & 2  & 0 & 0 & 1 & 0 & 0 \\
                \end{array}\right].$$
The solution gives the dependence relations:
$$\vect p_4=\frac{1}{2}(\vect q_1-\vect p_3) \quad \text{and} \quad \vect p_5=\frac{1}{2}\vect q_2 -\frac{1}{2}\vect p_2-\vect p_3,$$
so  $2\ovect p_4=\ovect p_3$ and $2\ovect p_5=-\ovect p_2-2\ovect p_3$.  Thus $\ovP^{-\rho}$ is 3-dimensional, with basis
$$B_\ovP^{-\rho}=(\ovect p_1, \ \ovect p_2, \ \ovect p_3).$$  
\end{proof}
The theorem demonstrates that $\pi_\FF$ is not surjective, since 
$$Mult_\FF(-\rho)=2 < 3 = Mult_\ovP(-\rho).$$
\begin{theorem}\label{lowbasis}A basis for $Low_\ovP(-\rho)$ is $B_{Low}^{-\rho}=(\ovect x_1, \ \ovect x_2)$, where
$$\ovect x_1 = \ovect p_1 + 2\ovect p_2,$$
$$\ovect x_2 = -\ovect p_1 + 3\ovect p_3.$$
\end{theorem}
\begin{proof}                       
If $\ovect x \in Low_\ovP(-\rho)$ then $F_i \cdot \ovect x =\ovect 0$ for $i=1,2$, so we compute $Ker(F_1)\cap Ker(F_2)\subset \ovP^{-\rho}$.  As before, we find general formulas for $F_i\cdot  u_t\vac\ox e^{-\rho}$ where $u_t\vac\ox e^{-\rho}\in B_{\FF,1}^{-\rho}, \ 1\leq t \leq 9$, to help compute the action of $F_i$ on the basis elements $\ovect p_j\in B_\ovP^{-\rho}$. We also make use of Lemma \ref{vertexlemma} in Appendix \ref{appendix:B}.

For $i=1,2$ and $j=1,2,3$, we have
\begin{align*} 
F_i \cdot &\alpha_j(-2)\mathbbm{1}\otimes e^{\beta_1+\beta_2} =e^{-\beta_i}_0(\alpha_j(-2)\mathbbm{1}\otimes e^{\beta_1+\beta_2}) \\
	&= Res_{z=0}(Y(e^{-\beta_i},z))\alpha_j(-2)\mathbbm{1}\otimes e^{\beta_1+\beta_2} \\
	&= Res_{z=0}\Big(exp\Big(\sum_{k>0} \frac{-\beta_i(-k)}{k} z^k\Big)exp\Big(\sum_{k>0} \frac{-\beta_i(k)}{-k} z^{-k}\Big)e^{-\beta_i}z^{-\beta_i(0)}\varepsilon_{-\beta_i}\Big)\alpha_j(-2)\mathbbm{1}\otimes e^{\beta_1+\beta_2}. 
\end{align*}
Note that equation \eqref{Fibcocycle} gives us 
$$\varepsilon(-\beta_i,\beta_1+\beta_2)=\varepsilon(\beta_i,\beta_1)^{-1}\varepsilon(\beta_i,\beta_2)^{-1}=(-1)^i \quad \text{and}\quad z^{-\beta_i(0)}e^{\beta_1+\beta_2}= z^{(-\beta_i,\beta_1+\beta_2)}=z$$ for $i=1,2$. Also, we observe that any operator in the second formal series above whose mode number or polynomial degree is greater than 2 will annihilate $\alpha_j(-2)\mathbbm{1}$. The surviving terms of this formal series are shown below in the continuation of the computation:
\begin{align*}
F_i \cdot  &\alpha_j(-2)\mathbbm{1}\otimes e^{\beta_1+\beta_2}  = (-1)^iRes_{z=0}\Big( exp\Big(\sum_{k>0} \frac{-\beta_i(-k)}{k} z^k\Big)\Big(I+\frac{1}{1!}\Big(\frac{\beta_i(1)}{1}z^{-1}+\frac{\beta_i(2)}{2}z^{-2}\Big) \\
&\qquad \qquad +\frac{1}{2!}\Big(\frac{\beta_i(1)}{1}z^{-1}\Big)^2 \Big)  z \Big) \alpha_j(-2)\mathbbm{1}\otimes e^{\beta_{3-i}} \\
&= (-1)^iRes_{z=0}\Big(exp\Big(\sum_{k>0} \frac{-\beta_i(-k)}{k} z^k\Big)\Big(\alpha_j(-2)z+\cancel{\beta_i(1)\alpha_j(-2)}+\frac{1}{2}\beta_i(2)\alpha_j(-2)z^{-1} + \\
& \qquad \qquad + \frac{1}{2}\cancel{\beta_i(1)^2\alpha_j(-2)}z^{-1}\Big)\mathbbm{1}\otimes e^{\beta_{3-i}} \\
&= (-1)^iRes_{z=0}\Big(exp\Big(I+h.o.t. \Big)\Big(\alpha_j(-2)z+(\beta_i,\alpha_j)z^{-1} \Big)\mathbbm{1}\otimes e^{\beta_{3-i}} \\
&= (-1)^i (\beta_i,\alpha_j)\mathbbm{1}\otimes e^{\beta_{3-i}}. 
\end{align*}

The remaining basis vectors of $B_{\FF,1}^{-\rho}$ are of the form $\alpha_j(-1)\alpha_k(-1)\mathbbm{1}\otimes e^{\beta_1+\beta_2} $ for $1 \leq j,k \leq 3$, so we compute $F_i \cdot \alpha_j(-1)\alpha_k(-1)\mathbbm{1}\otimes e^{\beta_1+\beta_2} $:
\begin{align*}
F_i \cdot \alpha_j(-1)&\alpha_k(-1)\mathbbm{1}\otimes e^{\beta_1+\beta_2} =e^{-\beta_i}_0\Big(\alpha_j(-1)\alpha_k(-1) \mathbbm{1}\otimes e^{\beta_1+\beta_2}\Big) \\
	&= Res_{z=0}\Big(Y(e^{-\beta_i},z)\Big)\alpha_j(-1)\alpha_k(-1) \mathbbm{1}\otimes e^{\beta_1+\beta_2}   \\	
	&= Res_{z=0}\Big(exp\Big(\sum_{k>0} \frac{-\beta_i(-k)}{k} z^k\Big)exp\Big(\sum_{k>0} \frac{-\beta_i(k)}{-k} z^{-k}\Big)e^{-\beta_i}z^{-\beta_i(0)}\varepsilon_{-\beta_i}\Big) \\
	&\qquad \cdot \alpha_j(-1)\alpha_k(-1) \mathbbm{1}\otimes e^{\beta_1+\beta_2}. 
\end{align*}
Now any operator in the second formal series above whose mode number is greater than 1 or whose polynomial degree is greater than 2 will annihilate $\alpha_j(-1)\alpha_k(-1) \mathbbm{1}$. Continuing the computation,
\begin{align*}		
F_i &\cdot \alpha_j(-1)\alpha_k(-1)\mathbbm{1}\otimes e^{\beta_1+\beta_2} = (-1)^i Res_{z=0}\Big(exp\Big(\sum_{k>0} \frac{-\beta_i(-k)}{k} z^k\Big)\\
&\qquad \qquad \cdot \Big(I+\frac{1}{1!}\Big(\frac{\beta_i(1)}{1}\Big)z^{-1}+\frac{1}{2!}\Big(\frac{\beta_i(1)}{1}z^{-1}\Big)^2 \Big)z\Big) \alpha_j(-1)\alpha_k(-1)\mathbbm{1}\otimes e^{\beta_{3-i}} \\
&= (-1)^i Res_{z=0}\Big( exp\Big(\sum_{k>0} \frac{-\beta_i(-k)}{k} z^k\Big)\Big(\alpha_j(-1)\alpha_k(-1) z+\beta_i(1)\alpha_j(-1)\alpha_k(-1) z^0 \\
&\qquad \qquad +\frac{1}{2}\beta_i(1)^2\alpha_j(-1)\alpha_k(-1) z^{-1}\Big) \Big)\mathbbm{1}\otimes e^{\beta_{3-i}} \\
&= (-1)^i Res_{z=0}\Big(\Big(I+h.o.t. \Big)\Big(\alpha_j(-1)\alpha_k(-1) z+(\beta_i,\alpha_j)\alpha_k(-1) + (\beta_i,\alpha_j)(\beta_i,\alpha_k) z^{-1}\Big)\Big)\\
&\qquad \qquad \cdot \mathbbm{1}\otimes e^{\beta_{3-i}} \\	
 	&= (-1)^i(\beta_i,\alpha_j)(\beta_i,\alpha_k)  \mathbbm{1}\otimes e^{\beta_{3-i}}, 	
\end{align*}
which when $j=k$ gives us $F_i \cdot \alpha_j(-1)^2\mathbbm{1}\otimes e^{\beta_1+\beta_2} =  (-1)^i(\beta_i,\alpha_j)^2  \mathbbm{1}\otimes e^{\beta_{3-i}}.$

Let $\ovect x=a\ovect p_1 + b \ovect p_2 +c\ovect p_3\in Ker(F_1)\cap Ker(F_2)$ where $a,b,c\in\C$. Then
\begin{align*}
\ovect 0 = F_i \cdot \ovect x &=a F_i\cdot \Big(\alpha_2(-2)+\alpha_2(-1)^2\Big)  + b F_i\cdot \Big(-\alpha_1(-1)^2+\alpha_3(-1)^2\Big)\\
&\qquad   + c F_i\cdot \Big(\alpha_1(-1)\alpha_3(-1)\Big)\ox \vac e^{-\rho} \\
&= (-1)^i\Big(a((\beta_i,\alpha_2)+(\beta_i,\alpha_2)^2)+b(-(\beta_i,\alpha_1)^2+(\beta_i,\alpha_3)^2)\\
&\qquad   +c((\beta_i,\alpha_1)(\beta_i,\alpha_3)) \vac \ox e^{-\rho}\Big)\\
&= \left\{\begin{array}{cl}
	(-6a+3b-2c)\vac \ox e^{-\rho} & \text{if } i=1, \\
	(6a-3b+2c)\vac \ox e^{-\rho} & \text{if } i=2. \\
	\end{array}\right.
\end{align*}
Thus we have $Low_\ovP(-\rho) = Span(\ovect x_1, \ \ovect x_2)$, where
$$\ovect x_1 = \ovect p_1 + 2\ovect p_2$$
$$\ovect x_2 = -\ovect p_1 + 3\ovect p_3$$
\end{proof}
Now we show that $v_{-\rho}$ is represented in $Low_\ovP(-\rho)$.

\begin{proposition}\label{operator} $\pi_\FF(v_{-\rho})= \ovect x_2$, where $v_{-\rho}=-3e_{12123}+[E_2,2E_1]$ is from Theorem \ref{vrho}.
\end{proposition}
\begin{proof}
Since $\pi_\FF$ is a Lie algebra representation, we have that  $\pi_\FF(v_{-\rho})=-3\pi_\FF(e_{12123})+2\pi_\FF([E_2,E_1]).$
We may use the restriction map $\pi_\Fib$ for the second term, since the vector $[E_2,E_1]\in \Fib$. First we compute
\begin{align*}
\pi_\Fib&([E_2,E_1])=[\pi_\Fib(E_2),\pi_\Fib(E_1)] = (\mathbbm{1}\otimes e^{\beta_2})_0(\mathbbm{1}\otimes e^{\beta_1})\\
&=Res_{z=0}\Big(\exp\Big(\sum_{k>0} \frac{\beta_2(-k)}{k} z^k\Big)\exp\Big(\sum_{k>0} \frac{\beta_2(k)}{-k} z^{-k}\Big)e^{\beta_2} z^{\beta_2(0)}\varepsilon_{\beta_2}\Big)(e^{\beta_1})\\
&=\varepsilon(\beta_2,\beta_1)Res_{z=0}\Big(\exp\Big(\sum_{k>0} \frac{\beta_2(-k)}{k} z^k\Big)(I+ h.o.t.)z^{-3}\mathbbm{1}\otimes e^{\beta_1+\beta_2}\Big)\\
&=-Res_{z=0}\Big(I + \beta_2(-1)z + \frac{1}{2}\Big(\beta_2(-2)+\beta_2(-1)^2\Big)z^2+ h.o.t.\Big)z^{-3}\mathbbm{1}\otimes e^{\beta_1+\beta_2}\Big)\\
&= -\frac{1}{2}\Big(\beta_2(-2)+\beta_2(-1)^2\Big)\mathbbm{1}\otimes e^{\beta_1+\beta_2}.
\end{align*}
Then, we compute $\pi_\FF(e_{12123})$ in steps. Using equation \eqref{eq:useful}, we have 
\begin{align*}
\pi_\FF&(e_{2123}) = e^{\alpha_2}_0(\alpha_1(-1)\mathbbm{1}\otimes e^{\alpha_1+\alpha_2+\alpha_3})\\
&= Res_{z=0}\Big(\exp\Big(\sum_{k>0} \frac{\alpha_2(-k)}{k} z^k\Big)\exp\Big(\sum_{k>0} \frac{\alpha_2(k)}{-k} z^{-k}\Big)e^{\alpha_2} z^{\alpha_2(0)}\varepsilon_{\alpha_2}\Big)\alpha_1(-1)\mathbbm{1}\otimes e^{\alpha_1 +\alpha_2+\alpha_3}\\
&= \varepsilon_{21}\varepsilon_{22}\varepsilon_{23}Res_{z=0}\Big(\exp\Big(\sum_{k>0} \frac{\alpha_2(-k)}{k} z^k \Big)\Big(I - \alpha_2(1)z^{-1} + h.o.t.\Big)z^{-1}\Big)\alpha_1(-1)\mathbbm{1}\otimes e^{\alpha_1+2\alpha_2+\alpha_3}  \\
&= (1)(-1)(1)Res_{z=0}\Big( \Big(I+\alpha_2(-1)z+h.o.t.\Big)\Big(\alpha_1(-1) -(-2)z^{-1}\Big)z^{-1}\Big)\mathbbm{1}\otimes e^{\alpha_1+2\alpha_2+\alpha_3}   \\
&=  -\Big(\alpha_1(-1)+2\alpha_2(-1)\Big)\mathbbm{1}\otimes e^{\alpha_1+2\alpha_2+\alpha_3},  
\end{align*} 
which gives us
\begin{align*}
\pi_\FF(e_{12123}) &= e^{\alpha_1}_0\Big(-(\alpha_1(-1)+2\alpha_2(-1))\mathbbm{1}\otimes e^{\alpha_1+2\alpha_2+\alpha_3}\Big) \\
&= -Res_{z=0}\Big(\exp\Big(\sum_{k>0} \frac{\alpha_1(-k)}{k} z^k\Big)\exp\Big(\sum_{k>0} \frac{\alpha_1(k)}{-k} z^{-k}\Big)e^{\alpha_1} z^{\alpha_1(0)}\varepsilon_{\alpha_1}\Big)\\
& \qquad \qquad \cdot (\alpha_1(-1)+2\alpha_2(-1))\mathbbm{1}\otimes e^{\alpha_1+2\alpha_2+\alpha_3}\\
&= -\varepsilon_{11}(\varepsilon_{12})^2\varepsilon_{13}Res_{z=0}\Big(\exp\Big(\sum_{k>0} \frac{\alpha_1(-k)}{k} z^k \Big)\Big(I - \alpha_1(1)z^{-1} + h.o.t.\Big)z^{-2}\Big)\\
& \qquad \qquad \cdot ((\alpha_1+2\alpha_2)(-1))\mathbbm{1}\otimes e^{2\alpha_1+2\alpha_2+\alpha_3}   \\
&= -(-1)(1)(1)Res_{z=0}\Big( \Big(I+\alpha_1(-1)z+ \frac{1}{2}(\alpha_1(-2)+\alpha_1(-1)^2)z^2  +h.o.t.\Big)\\
& \qquad \qquad \cdot \Big((\alpha_1+2\alpha_2)(-1) -(2+2(-2))z^{-1} \Big)z^{-2}\Big)\mathbbm{1}\otimes e^{2\alpha_1+2\alpha_2+\alpha_3}   \\
&= \Big(\alpha_1(-1)(\alpha_1+2\alpha_2)(-1) +2\Big(\frac{1}{2}(\alpha_1(-2)+\alpha_1(-1)^2)\Big)\Big)\mathbbm{1}\otimes e^{2\alpha_1+2\alpha_2+\alpha_3} \\
&= \Big(2\alpha_1(-1)^2+2\alpha_1(-1)\alpha_2(-1) +\alpha_1(-2)\Big)\mathbbm{1}\otimes e^{2\alpha_1+2\alpha_2+\alpha_3}. \\
&=-\alpha_1(-1)\alpha_3(-1)\mathbbm{1}\otimes e^{\beta_1+\beta_2} + q_1.
\end{align*}       
Thus, we have
\begin{align*}
\pi_\FF(v_{-\rho}) &= -3\pi_\FF(e_{12123})+2\pi_\FF([E_2,E_1]) \\
&= -3\Big(-\alpha_1(-1)\alpha_3(-1)\mathbbm{1}\otimes e^{\beta_1+\beta_2}\Big)+2\Big( -\frac{1}{2}\Big(\beta_2(-2)+\beta_2(-1)^2\Big)\mathbbm{1}\otimes e^{\beta_1+\beta_2}\Big)\\
&= \Big(3\alpha_1(-1)\alpha_3(-1)-\beta_2(-2)-\beta_2(-1)^2\Big)\mathbbm{1}\otimes e^{\beta_1+\beta_2} \\
&= \Big(3\alpha_1(-1)\alpha_3(-1)-\alpha_2(-2)-\alpha_2(-1)^2\Big)\mathbbm{1}\otimes e^{\beta_1+\beta_2} \\
&= \ovect x_2.
\end{align*}
\end{proof}
Theorem \ref{lowbasis} and Proposition \ref{operator} show that $\ovect x_1$ is a lowest-weight vector for $\Fib$ in $\ovP^{-\rho}$, but $\ovect x_1\notin\pi_\FF(\FF)$. Thus, the outer multiplicity data produced by this approach would only show how $\ovP$, not $\FF$, decomposes with respect to $\pi_\FF(\Fib)$. In this sense, the embedding of $\FF$ in $\ovP$ does not help in analyzing how $\FF$ decomposes with respect to $\Fib$. In addition, it is also clear that this type of vertex algebra computation becomes more complex when examining $\pi_\FF(\FF_\mu)\subset \ovP^\mu$ as $|wt(\mu)|$ increases, rendering this approach less useful than anticipated. 

However, Theorem \ref{operator} leads us to conjecture the existence of an algorithm based on the Schur polynomials (\cite {B}) of vectors in $\ovP$ which will allow one to determine which vectors are representations of vectors in $\FF$, and/or which vectors are extremal with respect to $\Fib$. If such ``recognition algorithms'' are found, then a pattern might emerge within the set of extremal vectors for $\Fib$ in $\ovP$ which could shed light on the decomposition of $\FF$ with respect to $\Fib$. We therefore consider this approach still to be valuable, and the research is ongoing.

\begin{appendix}

%\addcontentsline{toc}{chapter}{Appendix A: Kac-Peterson formulas}
%--------------------------------------------------------%
%--------------------------------------------------------%
%--------------------- Appendix A---------------------%
%--------------------------------------------------------%
%--------------------------------------------------------%
%\chapter*{Appendix A: Kac-Peterson formulas}\label{ch:appA}
\chapter{Kac-Peterson formulas}\label{appendix:A}
%--------------------------------------------------------%
%---------------------Section 6.1--------------------%
%--------------------------------------------------------%

The method of Kac and Peterson (provided as exercise in \cite{K2}) is a recursive algorithm for determining root multiplicities of the adjoint representation of any KM algebra $\LL$ with simple roots $\beta_i, 1\leq i \leq \ell$. Define for any $\beta\in Q_+$,
$$c_\beta=\sum_{d|\beta} \frac{1}{d}Mult\bigg(\frac{\beta}{d}\bigg).$$
As a simple example, we observe that if $\beta$ is primitive (that is, $\beta = n_1\beta_1 + n_2\beta_2$ where $n_1$ and $n_2$ are coprime), then $c_\beta=Mult(\beta)$, giving the starting values $c_{\beta_i} = Mult(\beta_i)= 1 \text{ for } i=1,2.$ It can be shown that the  $c_\beta$'s follow the recursive relation
$$(\beta | \beta-2\rho) c_\beta = \sum_{\substack{\beta=\beta'+\beta'' \\ \beta',\beta''\in Q_+}}( \beta' | \beta'' ) c_{\beta'}c_{\beta''}.$$
We further observe that $c_\beta$ is defined in such a way that a M\"obius inversion might be possible, which could give a similar expression for multiplicities in terms of $c_\beta$'s. First note that, for any $\beta \in Q_+$, $\beta=n\gamma$, where $n\geq 1$ and $\gamma$ is primitive. Let 
$$F_\gamma(m)=m c_{m\gamma}=\sum_{d|m} \frac{m}{d}Mult\bigg(\frac{m\gamma}{d}\bigg).$$
The above sum can be rewritten to range over the divisors $d'=\frac{m}{d}$ :
$$F_\gamma (m)= \sum_{d'|m} d' Mult\big(d' \gamma \big).$$
If we define  $f_\gamma(d) =  d \ Mult (d\gamma)$,  then we have the M\"obius inversion
$$m \ Mult(m\gamma) = f_\gamma(m) = \sum_{d|m}\mu(d)F_\gamma\bigg(\frac{m}{d}\bigg) = \sum_{d|m}\mu(d)\frac{m}{d}c_{\frac{m\gamma}{d}},$$
which after dividing both sides by $m$ gives
$$Mult(m\gamma) = \sum_{d|m}\mu(d)\frac{1}{d}c_{\frac{m\gamma}{d}}.$$
Thus we have a recursive method for finding the multiplicity of $\beta$ in terms of the $c_\gamma$ where $\gamma | \beta$:
$$Mult(\beta) = Mult(n\gamma) = \sum_{d|n}\mu(d)\frac{1}{d}c_{\frac{\beta}{d}}.$$

%\addcontentsline{toc}{chapter}{Appendix B: Computing multibrackets}
%--------------------------------------------------------%
%--------------------------------------------------------%
%--------------------- Appendix B---------------------%
%--------------------------------------------------------%
%--------------------------------------------------------%
\chapter{Supplementary results}\label{appendix:B}
Some of the proofs in Chapters \ref{ch:fib} and \ref{ch:vertex} required lengthy computations which are either shown here in their entirety, or were facilitated through the use of the results and identities below.

\section{Multibracket theorem and identities used in Chapter 2}\label{ch3}
\begin{theorem}\label{multibracket}
Let $e_i, \in \FF_{\alpha_i}$, $f_i \in \FF_{-\alpha_i}$, $e_{j_n \cdots j_1} \in \FF_\alpha$, and $f_{j_n \cdots j_1}\in \FF_{-\alpha}$ where $\alpha=\ds\sum_{k=1}^n \alpha_{j_k}$ and $n \geq 2$. Then
$$[e_i, f_{j_n \cdots j_1}] = \delta_{ij_1}a_{j_1j_2}f_{j_n \cdots j_2} - \ds\sum^n_{m=2}\delta_{ij_m}\Big(\sum^{m-1}_{k=1}a_{j_m j_k}\Big)f_{j_n \cdots \hat{j_m} \cdots j_1},$$
and
$$[f_i, e_{j_n \cdots j_1}] = \delta_{ij_1}a_{j_1j_2}e_{j_n \cdots j_2} - \ds\sum^n_{m=2}\delta_{ij_m}\Big(\sum^{m-1}_{k=1}a_{j_m j_k}\Big)e_{j_n \cdots \hat{j_m} \cdots j_1},$$
where $x_{j_n \cdots \hat{j_m} \cdots j_1}=x_{j_n \cdots \ j_{m+1} j_{m-1} \cdots j_1}$ for $x \in \{e, f\}$.
\end{theorem}
\begin{proof}
We prove the first identity by induction on the length of the multibracket, $n$.  Fix $1\leq i \leq 3$. To show the base case is true let $n=2$.  Then we have
\begin{align*}
[e_i,f_{j_2j_1}] &= [[e_i,f_{j_2}],f_{j_1}] + [f_{j_2},[e_i,f_{j_1}]]
\\ &= \delta_{ij_2}[h_{j_2},f_{j_1}] + \delta_{ij_1}[f_{j_2},h_{j_1}]
\\ &= - \ \delta_{ij_2}a_{j_2j_1}f_{j_1} + \delta_{ij_1}a_{j_1j_2}f_{j_2}.
\end{align*}
Now assume that the statement is true for all $2\leq n\leq N$ for some large integer $N$. Then
\begin{align*}
[e_i,f_{j_{N+1}\cdots j_1}] &= [e_i, [f_{j_{N+1}},f_{j_N \cdots j_1}]]
\\ &= [[e_i,f_{j_{N+1}}],f_{j_N \cdots j_1}] + [f_{j_{N+1}},[e_i,f_{j_N \cdots j_1}]].
\intertext{We now invoke the induction hypothesis for the bracket inside the second term.}
&= \delta_{ij_{N+1}}[h_{j_{N+1}},f_{j_N \cdots j_1}] 
\\ & \qquad + [f_{j_{N+1}}, \delta_{ij_1}a_{j_1j_2}f_{j_N \cdots j_2} - \ds\sum^N_{m=2}\delta_{ij_m}\Big(\sum^{m-1}_{k=1}a_{j_m j_k}\Big)f_{j_N \cdots \hat{j_m} \cdots j_1}]
\\ &= - \ \delta_{ij_{N+1}}\Big(\sum^{N}_{k=1}a_{j_{N+1}j_k}\Big)f_{j_N \cdots j_1} 
\\ & \qquad + \delta_{ij_1}a_{j_1j_2}f_{j_{N+1} \cdots j_2} - \sum^N_{m=2}\delta_{j_m j_m}\Big(\sum^{m-1}_{k=1}a_{ij_k}\Big)f_{j_{N+1} \cdots \hat{j_m} \cdots j_1}.
\intertext{Observe that the first term can be absorbed into the last term's outer sum as the case $m=N+1$ (since $f_{j_N \cdots j_1} = f_{\hat j_{N+1} j_N \cdots j_1}$), giving us}
&= \delta_{ij_1}a_{j_1j_2}f_{j_{N+1} \cdots j_2} - \sum^{N+1}_{m=2}\delta_{j_m j_m}\Big(\sum^{m-1}_{k=1}a_{ij_k}\Big)f_{j_{N+1} \cdots \hat{j_m} \cdots j_1}.
\end{align*}

The second identity follows from the first via the Cartan involution $\nu$. 

\end{proof}
The following identities were obtained using Theorem \ref{multibracket}, and are used in the proofs of Chapter \ref{ch:fib}.
$$[e_1,f_{1123}]= a_{32}(0) - \Big((a_{13}+a_{12})f_{123}+(a_{13}+a_{12}+a_{11})f_{123}\Big)  = 2 f_{123},$$
$$[e_1,f_{123}]= a_{32}(0) - (a_{13}+a_{12})f_{23} = 2 f_{23},$$
$$[f_1,e_{123}]= 2 e_{23}, $$  
$$[e_2,f_{123}]= a_{32}(0) - (a_{23})f_{13} = f_{13} = 0,$$
$$[e_3,f_{123}]= (a_{32})f_{12}-(0) = - f_{12},$$
$$[e_2,f_{23}]= a_{32}(0) - (a_{23})f_3 = f_3,$$
$$[e_3,f_{23}]= a_{32}f_2 - (0) = -f_2,$$
$$[e_2,f_{12}]= a_{21}f_1 = -2 f_1.$$

We also have the following immediate consequence of the theorem. 
\begin{corollary}
Let $1\leq i \leq \ell$, and let $X$ be a multibracket $f_{i_1i_2\ldots i_n}$ (resp. $e_{i_1i_2\ldots i_n}$) such that $i_j \not= i$ for all $1\leq j \leq n$. Then $[e_i,X]=0$ (resp. $[f_i,X]=0$).
\end{corollary}

In Chapter \ref{ch:vertex} the following Lemma aided in vertex operator computations.
\begin{lemma}\label{vertexlemma}
Let $\beta, \gamma\in Q_\Fib$. Then for $k\geq 1$ and $n>0$,

\medskip
\quad (i) \ $\beta(n)\gamma(-n)^k\mathbbm{1}=nk(\beta, \gamma) \gamma(-n)^{k-1}\mathbbm{1}$,

\medskip
\quad (ii) \ $\beta(n)^m\gamma(-n)^k\mathbbm{1}=\frac{k!}{(k-m)!}n^m(\beta, \gamma)^m \gamma(-n)^{k-m}\mathbbm{1}$, for $k\geq m>0$.

\end{lemma}
\begin{proof} (i) Induct on $k$. Base case: 
$$\beta(n)\gamma(-n)\mathbbm{1}=[\beta(n),\gamma(-n)]\mathbbm{1}+\gamma(-n)\cancel{\beta(n)\mathbbm{1}}=n(\beta, \gamma)\mathbbm{1}.$$
Assume inductive hypothesis for all $k\leq K$ for some $K\in\Z_{>0}$. Then
\begin{align*}
\beta(n)\gamma(-n)^{K+1}\mathbbm{1} &= \Big(\beta(n)\gamma(-n)\Big)\gamma(-n)^K\mathbbm{1} \\
 &= \Big([\beta(n),\gamma(n)]+\gamma(-n)\beta(n)\Big) \gamma^K(-n)\mathbbm{1} \\
 &= n(\beta,\gamma)\gamma(-n)^K\mathbbm{1}+ \gamma(-n) \Big( \beta(n)\gamma(-n)^K\mathbbm{1}\Big)\\
&= n(\beta,\gamma)\gamma(-n)^K\mathbbm{1}+  \gamma(-n)\Big(nK(\beta, \gamma) \gamma(-n)^{K-1} \mathbbm{1} \Big) \\
&= n(K+1)(\beta, \gamma) \gamma(-n)^K\mathbbm{1} .
\end{align*}
The induction hypothesis was applied in the fourth equality. Thus the statement is true for all $k\in\Z$. 

(ii) Induct on $m$ for fixed $k\in\Z_{\geq0}$. Base case: When $m=1$ and $k\geq 0$ we have part $(i)$. Assume that
$$\beta(n)^m\gamma(-n)^k\mathbbm{1}=\frac{k!}{(k-m)!}n^m(\beta, \gamma)^m \gamma(-n)^{k-m}\mathbbm{1}$$
is true for all $m\leq M$ for some $M\in\Z_{>0}$. Then
\begin{align*}
\beta(n)^{M+1}\gamma(-n)^k\mathbbm{1} &= \beta(n)\Big(\beta(n)^M\gamma(-n)^k\mathbbm{1} \Big) \\
&= \beta(n)\Big(\frac{k!}{(k-M)!}n^M(\beta, \gamma)^M \gamma(-n)^{k-M}\mathbbm{1} \Big)\\
&= \frac{k!}{(k-M)!} n^M(\beta, \gamma)^M \Big(\beta(n) \gamma(-n)^{k-M}\mathbbm{1}\Big)\\
&=\frac{k!}{(k-M)!} n^M(\beta, \gamma)^M\Big((k-M)n(\beta,\gamma)\gamma(-n)^{k-M-1}\mathbbm{1} \Big)\\
&=\frac{k!}{(k-(M+1))!} n^{M+1}(\beta, \gamma)^{M+1}\gamma(-n)^{k-(M+1)}\mathbbm{1}. 
\end{align*}
\end{proof}

%
%
%----------APPENDIX B.2 - The Lie algebra representation pi_\FF
%
%
\section{The Lie algebra representation $\pi_\FF: \FF \rightarrow \ovP$}\label{sec:piF}
We now show that the representation $\pi_\FF: \FF \rightarrow \ovP$ from Section \ref{sec:Fib in VF} is a Lie algebra representation (i.e., the representations of the Serre generators satisfy the Serre relations in $\FF$ given in Definition \ref{def:kml}). Note that since we are working with the representation of $\FF$ in $V_\FF$, the 2-cocycle is $\varepsilon_{ij}^\FF$, but we will suppress the superscript for brevity. We verify the first four Serre relations as follows:
\begin{align*} 
[\pi_\FF&(h_i), \pi_\FF(h_j)] =\big(\alpha_i(-1)\mathbbm{1}\otimes e^0\big)_0\big(\alpha_j(-1)\mathbbm{1}\otimes e^0\big) \\
&=Res_{z=0}(\alpha_i(z))\alpha_j(-1)\mathbbm{1}\otimes e^0 =Res_{z=0}(\sum_{n\in\Z}\alpha_i(n)z^{-n-1})\alpha_j(-1)\mathbbm{1}\otimes e^0\\
&= \alpha_i(0)\alpha_j(-1)\mathbbm{1}\otimes e^0 =(\alpha_i,0)\alpha_j(-1)\mathbbm{1}\otimes e^0 =0,
\end{align*}
and
\begin{align*} 
[\pi_\FF&(h_i), \pi_\FF(e_j)] =(\alpha_i(-1)\mathbbm{1}\otimes e^0)_0 (\mathbbm{1}\otimes e^{\alpha_j}) =Res_{z=0}(\alpha_i(z))(\mathbbm{1}\otimes e^{\alpha_j}) \\
&=\alpha_i(0)\mathbbm{1}\otimes e^{\alpha_j} =(\alpha_i,\alpha_j)\mathbbm{1}\otimes e^{\alpha_j} =a_{ji}\mathbbm{1}\otimes e^{\alpha_j}=a_{ji}\pi_\FF(e_j),
\end{align*}
and
\begin{align*} 
[\pi_\FF&(h_i), \pi_\FF(f_j)] =( \alpha_i(-1)\mathbbm{1}\otimes e^0)_0( -\mathbbm{1}\otimes e^{-\alpha_j}) =-Res_{z=0}(\alpha_i(z))(\mathbbm{1}\otimes e^{-\alpha_j}) \\
&=-\alpha_i(0)\mathbbm{1}\otimes e^{-\alpha_j} =-(\alpha_i,-\alpha_j)\mathbbm{1}\otimes e^{-\alpha_j} =a_{ji}\mathbbm{1}\otimes e^{-\alpha_j}=-a_{ji}\pi_\FF(f_j),
\end{align*}
and
\begin{align*} 
[\pi_\FF&(e_i), \pi_\FF(f_j)] =- (\mathbbm{1}\otimes e^{\alpha_i})_0(\mathbbm{1}\otimes e^{-\alpha_j}) = -Y_0\Big(\mathbbm{1}\otimes e^{\alpha_i},z\Big)(\mathbbm{1}\otimes e^{-\alpha_j}) \\
&= -Res_{z=0}\Big(\exp\Big(\sum_{k>0} \frac{\alpha_i(-k)}{k} z^k\Big)\exp\Big(\sum_{k>0} \frac{\alpha_i(k)}{-k} z^{-k}\Big)e^{\alpha_i} z^{\alpha_i(0)}\varepsilon_{\alpha_i}\Big)(\mathbbm{1}\otimes e^{-\alpha_j})\\
&=- \varepsilon(\alpha_i,-\alpha_j) Res_{z=0}\Big(\exp\Big(\sum_{k>0} \frac{\alpha_i(-k)}{k} z^k\Big)\exp\Big(\sum_{k>0} \frac{\alpha_i(k)}{-k} z^{-k}\Big) z^{(\alpha_i,-\alpha_j)}\Big)(\mathbbm{1}\otimes e^{\alpha_i-\alpha_j})\\
&= -\varepsilon_{ij}^{-1} Res_{z=0}\Big(\exp\Big(\sum_{k>0} \frac{\alpha_i(-k)}{k} z^k\Big)\Big(I+h.o.t.\Big) z^{-a_{ji}}\Big)(\mathbbm{1}\otimes e^{\alpha_i-\alpha_j})\\
&= -\varepsilon_{ij}^{-1} Res_{z=0}\Big( I + \alpha_i(-1)z +h.o.t.\Big)z^{-a_{ji}}\Big)(\mathbbm{1}\otimes e^{\alpha_i-\alpha_j})=-\delta_{ij}\varepsilon_{ij}^{-1}\alpha_i(-1)\mathbbm{1}\otimes e^0\\
&=-\delta_{ij}\varepsilon_{ij}^{-1}\pi_\FF(h_i)=\delta_{ij}\pi_\FF(h_i).
\end{align*}
Note in the fourth to last equality above if $a_{ji}\leq0$ then the residue is 0 since there would only be positive powers of $z$ in the resulting series. If $a_{ji}>0$ then $a_{ji}=2$ and $i=j$, so the result is $-\varepsilon_{ii}^{-1}\alpha_i(-1)\mathbbm{1}\otimes e^0=\alpha_i(-1)\mathbbm{1}\otimes e^0$, since we have chosen $\varepsilon_{ii}=-1$ for all $i=1,2,3$.

Lastly we need to check that the Serre relations $(ad_{e_i})^{-a_{ij}+1}(e_j)=0$ and $(ad_{f_i})^{-a_{ij}+1}(f_j)=0$ for all $i,j=1,2,3$ where $i\neq j$ are correctly represented in $\ovP$.  

%------ad_(e_i)'s

The next page of calculations will show that $\pi_\FF\big((ad_{e_i})^{-a_{ij}+1}(e_j)\big)=0$. We start with $\pi_\FF\big([e_i, e_j]\big)$ for general $1\leq i,j\leq 3$, then substitute for each of the three cases $a_{ij}=0, -1,$ and $-2$:
\begin{align*} 
[\pi_\FF&(e_i), \pi_\FF(e_j)] =[ \mathbbm{1}\otimes e^{\alpha_i}, \mathbbm{1}\otimes e^{\alpha_j}] = (\mathbbm{1}\otimes e^{\alpha_i})_0(\mathbbm{1}\otimes e^{\alpha_j}) \\
&= Res_{z=0}\Big(\exp\Big(\sum_{k>0} \frac{\alpha_i(-k)}{k} z^k\Big)\exp\Big(\sum_{k>0} \frac{\alpha_i(k)}{-k} z^{-k}\Big)e^{\alpha_i} z^{\alpha_i(0)}\varepsilon_{\alpha_i}\Big)(\mathbbm{1}\otimes e^{\alpha_j})\\
&= \varepsilon_{ij} Res_{z=0}\Big(\exp\Big(\sum_{k>0} \frac{\alpha_i(-k)}{k} z^k\Big)\Big(I + h.o.t.\Big) z^{(\alpha_i,\alpha_j)}\Big)(\mathbbm{1}\otimes e^{\alpha_i+\alpha_j})\\
&= \varepsilon_{ij} Res_{z=0}\Big( I + \alpha_i(-1)z + \frac{1}{2}(\alpha_i(-1)^2+\alpha_i(-2))z^2+h.o.t.\Big)z^{a_{ji}}\Big)(\mathbbm{1}\otimes e^{\alpha_i+\alpha_j})\\
&=\left\{\begin{array}{cl}
	0 & \text{if } a_{ij}\geq 0, \\
	\varepsilon_{ij}\mathbbm{1}\otimes e^{\alpha_2+\alpha_3} & \text{if } a_{ij}=-1, \\
	\varepsilon_{ij}\alpha_i(-1)\mathbbm{1}\otimes e^{\alpha_1+\alpha_2} & \text{if } a_{ij}=-2. \\
	\end{array}\right.
\end{align*}
%The result, which is $\pi_\FF(e_{ij})$, is then
%$$[\pi_\FF(e_i), \pi_\FF(e_j)] = \left\{\begin{array}{cl}
%	0 & \text{if } a_{ij}\geq 0 \\
%	\varepsilon_{ij}\mathbbm{1}\otimes e^{\alpha_2+\alpha_3} & \text{if } a_{ij}=-1, \\
%	\varepsilon_{ij}\alpha_1(-1)\mathbbm{1}\otimes e^{\alpha_1+\alpha_2} & \text{if } a_{ij}=-2. \\
%	\end{array}\right.$$
The first case above verifies the Serre relation for $i=1,j=3$ (and similarly for $i=3, j=1$), where $a_{13}=0$.
Continuing to check the other two cases, we first determine $\pi_\FF((ad_{e_2})^2(e_3))$. We have
\begin{align*} 
[\pi_\FF&(e_2), [\pi_\FF(e_2), \pi_\FF(e_3)]] = (\mathbbm{1}\otimes e^{\alpha_2})_0(-\mathbbm{1}\otimes e^{\alpha_2+\alpha_3}) \\
&= -Res_{z=0}\Big(\exp\Big(\sum_{k>0} \frac{\alpha_2(-k)}{k} z^k\Big)\exp\Big(\sum_{k>0} \frac{\alpha_2(k)}{-k} z^{-k}\Big)e^{\alpha_2} z^{(\alpha_2,\alpha_2+\alpha_3)}\varepsilon_{\alpha_2}\Big)(\mathbbm{1}\otimes e^{\alpha_2+\alpha_3})\\
&= -\varepsilon_{22}\varepsilon_{23}Res_{z=0}\Big(\exp\Big(\sum_{k>0} \frac{\alpha_2(-k)}{k} z^k\Big)\Big(I + h.o.t.\Big) z^1\Big)(\mathbbm{1}\otimes e^{2\alpha_2+\alpha_3})= 0,
\end{align*}
since all of the negative powers of $z$ from the annihilator expansion have coefficients which kill $\mathbbm{1}$, leaving only positive powers of $z$ in the final expansion. Thus the residue is 0, and for $i=2, j=3$ (or for $i=3, j=2$) where $a_{ij}=-1$, the Serre relation $(ad_{e_2})^{-a_{23}+1}(e_3)=[e_2, [e_2,e_3]]=0$
is correctly represented in $\ovP$. Lastly we compute $\pi_\FF((ad_{e_1})^3(e_2))$,
\begin{align*} 
[&\pi_\FF(e_1), \pi_\FF(e_{12})] = (\mathbbm{1}\otimes e^{\alpha_1})_0(\alpha_1(-1)\mathbbm{1}\otimes e^{\alpha_1+\alpha_2}) \\
&= Res_{z=0}\Big(\exp\Big(\sum_{k>0} \frac{\alpha_1(-k)}{k} z^k\Big)\exp\Big(\sum_{k>0} \frac{\alpha_1(k)}{-k} z^{-k}\Big)e^{\alpha_1} z^{(\alpha_1,\alpha_1+\alpha_2)}\varepsilon_{\alpha_1}\Big)(\alpha_1(-1)\mathbbm{1}\otimes e^{\alpha_1+\alpha_2})\\
&= \varepsilon_{11}\varepsilon_{12}Res_{z=0}\Big(\exp\Big(\sum_{k>0} \frac{\alpha_1(-k)}{k} z^k\Big)\Big(I -\alpha_1(1)z^{-1} + h.o.t. \Big) z^0\Big)(\alpha_1(-1)\mathbbm{1}\otimes e^{2\alpha_1+\alpha_2})\\
&=  \varepsilon_{11}\varepsilon_{12}Res_{z=0}\Big(I + h.o.t. \Big)\Big(\alpha_1(-1) -2z^{-1} + h.o.t. \Big)\mathbbm{1}\otimes e^{2\alpha_1+\alpha_2}\\
&= -2\varepsilon_{11}\varepsilon_{12}\mathbbm{1}\otimes e^{2\alpha_1+\alpha_2},
\end{align*}
then
\begin{align*} 
[&\pi_\FF(e_1), \pi_\FF(e_{112})] = (\mathbbm{1}\otimes e^{\alpha_1})_0 (-2\varepsilon_{11}\varepsilon_{12}\mathbbm{1}\otimes e^{2\alpha_1+\alpha_2}) \\
&=-  \varepsilon_{11}\varepsilon_{12}Res_{z=0}\Big(\exp\Big(\sum_{k>0} \frac{\alpha_1(-k)}{k} z^k\Big)\exp\Big(\sum_{k>0} \frac{\alpha_1(k)}{-k} z^{-k}\Big)e^{\alpha_1} z^{(\alpha_1,2\alpha_1+\alpha_2)}\varepsilon_{\alpha_1}\Big)(2\mathbbm{1}\otimes e^{2\alpha_1+\alpha_2})\\
&=- (\varepsilon_{11})^3(\varepsilon_{12})^2Res_{z=0}\Big(\exp\Big(\sum_{k>0} \frac{\alpha_i(-k)}{k} z^k\Big)\Big(I  + h.o.t. \Big) z^2\Big)(2\mathbbm{1}\otimes e^{3\alpha_1+\alpha_2})= 0,
\end{align*}
since the expansion has no negative exponents, so the Serre relation $(ad_{e_1})^{-a_{12}+1}(e_2)=[e_1, [e_1, [e_1,e_2]]=0$
is correctly represented in $\ovP$. 

%-----ad_(f_i)'s

Using a similar approach, we now show that $\pi_\FF\big((ad_{f_i})^{-a_{ij}+1}(f_j)\big)=0$, starting as we did above with $\pi_\FF\big([f_i, f_j]\big)$ for general $1\leq i,j\leq 3$, then substituting for each of the three cases $a_{ij}=0, -1,$ and $-2$:
\begin{align*} 
[\pi_\FF&(f_i), \ \pi_\FF(f_j)] = (\mathbbm{1}\otimes e^{-\alpha_i})_0(\mathbbm{1}\otimes e^{-\alpha_j}) \\
&= Res_{z=0}\Big(\exp\Big(\sum_{k>0} \frac{-\alpha_i(-k)}{k} z^k\Big)\exp\Big(\sum_{k>0} \frac{-\alpha_i(k)}{-k} z^{-k}\Big)e^{-\alpha_i} z^{-\alpha_i(0)}\varepsilon_{-\alpha_i}\Big)(\mathbbm{1}\otimes e^{-\alpha_j})\\
&= \varepsilon(-\alpha_i,-\alpha_j) Res_{z=0}\Big(\exp\Big(\sum_{k>0} \frac{-\alpha_i(-k)}{k} z^k\Big)\Big(I + h.o.t.\Big) z^{(-\alpha_i,-\alpha_j)}\Big)(\mathbbm{1}\otimes e^{-\alpha_i-\alpha_j})\\
&= \varepsilon_{ij} Res_{z=0}\Big( I  - \alpha_i(-1)z + \frac{1}{2}(\alpha_i(-1)^2-\alpha_i(-2))z^2+h.o.t.\Big)z^{a_{ji}}\Big)(\mathbbm{1}\otimes e^{-\alpha_i-\alpha_j})\\
&= \left\{\begin{array}{cl}
	0 & \text{if } a_{ij}\geq 0, \\
	\varepsilon_{ij}\mathbbm{1}\otimes e^{-\alpha_2-\alpha_3} & \text{if } a_{ij}=-1, \\
	-\varepsilon_{ij}\alpha_i(-1)\mathbbm{1}\otimes e^{-\alpha_1-\alpha_2} & \text{if } a_{ij}=-2. \\
	\end{array}\right.
\end{align*}
%so
%$$[\pi_\FF(f_i), \pi_\FF(f_j)] = \left\{\begin{array}{cl}
%	0 & \text{if } a_{ij}\geq 0 \\
%	\varepsilon_{ij}\mathbbm{1}\otimes e^{-\alpha_2-\alpha_3} & \text{if } a_{ij}=-1, \\
%	-\varepsilon_{ij}\alpha_1(-1)\mathbbm{1}\otimes e^{-\alpha_1-\alpha_2} & \text{if } a_{ij}=-2. \\
%	\end{array}\right.$$
As before, the first case shows that the Serre relations $(ad_{f_1})^{-a_{13}+1}(f_3)=[f_1, f_3]=0$ and $(ad_{f_3})^{-a_{31}+1}(f_1)=[f_3, f_1]=0$
are correctly represented in $\ovP$.
Next, we have
\begin{align*} 
[\pi_\FF(f_2), & \pi_\FF(f_{23})] = (\mathbbm{1}\otimes e^{-\alpha_2})_0(-\mathbbm{1}\otimes e^{-\alpha_2-\alpha_3}) \\
&=- Res_{z=0}\Big(\exp\Big(\sum_{k>0} \frac{-\alpha_2(-k)}{k} z^k\Big)\exp\Big(\sum_{k>0} \frac{-\alpha_2(k)}{-k} z^{-k}\Big)e^{-\alpha_2} z^{(-\alpha_2,-\alpha_2-\alpha_3)}\varepsilon_{-\alpha_2}\Big)\\
& \qquad \cdot(\mathbbm{1}\otimes e^{-\alpha_2-\alpha_3})\\
&= -\varepsilon_{22}\varepsilon_{23} Res_{z=0}\Big(\exp\Big(\sum_{k>0} \frac{-\alpha_2(-k)}{k} z^k\Big)\Big(I + h.o.t.\Big) z^1\Big)(\mathbbm{1}\otimes e^{2\alpha_2+\alpha_3}) = 0,
\end{align*}
since all of the negative powers of $z$ from the annihilator expansion have coefficients which kill $\mathbbm{1}$, leaving only positive powers of $z$ in the final expansion. Thus the residue is 0, so the Serre relation 
$(ad_{f_2})^{-a_{23}+1}(f_3)=[f_2, [f_2,f_3]]=0$
is correctly represented in $\ovP$.
Lastly we compute $\pi_\FF((ad_{f_1})^3(f_2))$,
\begin{align*} 
[&\pi_\FF(f_1), \pi_\FF(f_{12})] = (\mathbbm{1}\otimes e^{-\alpha_1})_0(-\alpha_1(-1)\mathbbm{1}\otimes e^{-\alpha_1-\alpha_2}) \\
&= -Res_{z=0}\Big(\exp\Big(\sum_{k>0} \frac{-\alpha_1(-k)}{k} z^k\Big)\exp\Big(\sum_{k>0} \frac{-\alpha_1(k)}{-k} z^{-k}\Big)e^{-\alpha_1} z^{(-\alpha_1,-\alpha_1-\alpha_2)}\varepsilon_{-\alpha_1}\Big) \\
& \qquad \cdot (\alpha_1(-1)\mathbbm{1}\otimes e^{-\alpha_1-\alpha_2})\\
&= -\varepsilon_{11}\varepsilon_{12}Res_{z=0}\Big(\exp\Big(\sum_{k>0} \frac{-\alpha_1(-k)}{k} z^k\Big)\Big(I +\alpha_1(1)z^{-1} + h.o.t. \Big) z^0\Big)(\alpha_1(-1)\mathbbm{1}\otimes e^{-2\alpha_1-\alpha_2})\\
&= -\varepsilon_{11}\varepsilon_{12}Res_{z=0}\Big((I + h.o.t. )(\alpha_1(-1) -2z^{-1} + h.o.t. )\Big)\mathbbm{1}\otimes e^{-2\alpha_1-\alpha_2}= 2\varepsilon_{11}\varepsilon_{12}\mathbbm{1}\otimes e^{-2\alpha_1-\alpha_2},
\end{align*}
then
\begin{align*} 
[\pi_\FF(f_1), &\pi_\FF(f_{112})] = (\mathbbm{1}\otimes e^{-\alpha_1})_0( 2\varepsilon_{11}\varepsilon_{12}\mathbbm{1}\otimes e^{-2\alpha_1-\alpha_2}) \\
&= \varepsilon_{11}\varepsilon_{12}Res_{z=0}\Big(\exp\Big(\sum_{k>0} \frac{-\alpha_1(-k)}{k} z^k\Big)\exp\Big(\sum_{k>0} \frac{-\alpha_1(k)}{-k} z^{-k}\Big)e^{-\alpha_1} z^{(-\alpha_1,-2\alpha_1-\alpha_2)}\varepsilon_{-\alpha_1}\Big)\\
& \qquad \cdot(2\mathbbm{1}\otimes e^{-2\alpha_1-\alpha_2})\\
&= (\varepsilon_{11})^3(\varepsilon_{12})^2Res_{z=0}\Big(\exp\Big(\sum_{k>0} \frac{-\alpha_1(-k)}{k} z^k\Big)\Big(I  + h.o.t. \Big) z^2\Big)(2\mathbbm{1}\otimes e^{3\alpha_1+\alpha_2})= 0,
\end{align*}
since the expansion has no negative exponents, so $(ad_{f_1})^{-a_{12}+1}(f_2))=[f_1, [f_1, [f_1,f_2]]=0$ is correctly represented in $\ovP$.

%
%
%----------APPENDIX B.3 - Proving $\pi_\FF|_\Fib = \pi_\Fib$
%
%
\section{Proving $\pi_\FF|_\Fib = \pi_\Fib$}\label{sec:piFib}
Now we verify that for $i=1,2$, $\pi_\Fib=\pi_\FF|_\Fib$, that is,
$$\pi_\FF(E_i)=\mathbbm{1}\otimes e^{\beta_i}, \quad \pi_\FF(F_i)=-\mathbbm{1}\otimes e^{-\beta_i}, \quad \text{and} \quad \pi_\FF(H_i)=\beta_i(-1)\mathbbm{1}\otimes e^0,$$
where $E_1=\frac{1}{2}e_{1123}, \ E_2=e_2, \ F_1=-\frac{1}{2}f_{1123}, \ F_2=f_2, \ H_1=2h_1+h_2+h_3,$ and $H_2=h_2$. 

First, we have
\begin{align*}
\pi_\FF(H_1)&=\pi_\FF(2h_1+h_2+h_3) = 2\alpha_1(-1)\vac\ox e^0 + \alpha_2(-1)\vac\ox e^0+\alpha_3(-1)\vac\ox e^0 \\
&= (2\alpha_1+\alpha_2+\alpha_3)(-1)\vac\ox e^0 = \beta_1(-1)\vac\ox e^0\\
\intertext{and} 
\pi_\FF(H_2)&=\pi_\FF(h_2) = \alpha_2(-1)\vac\ox e^0= \beta_2(-1)\vac\ox e^0.
\end{align*}

Next, we have $\pi_\FF(E_1) = \pi_\FF\Big(\frac{1}{2}e_{1123}\Big)=\frac{1}{2}[\pi_\FF(e_1),\pi_\FF(e_{123})].$ Since it will prove useful in Proposition \ref{operator}, we first compute $\pi_\FF(e_{123})$, and we recall that $\pi_\FF(e_{23})=\varepsilon_{23}(e^{\alpha_2+\alpha_3})$.
\begin{align*} 
\pi_\FF&(e_{123})= [\pi_\FF(e_1), \pi_\FF(e_{23})] = e^{\alpha_1}_0(\varepsilon_{23}e^{\alpha_2+\alpha_3})\\
&=  Res_{z=0}\Big(\exp\Big(\sum_{k>0} \frac{\alpha_1(-k)}{k} z^k\Big)\exp\Big(\sum_{k>0} \frac{\alpha_1(k)}{-k} z^{-k}\Big)e^{\alpha_1} z^{(\alpha_1,\alpha_2+\alpha_3)}\varepsilon_{\alpha_1}\Big)(\varepsilon_{23}\mathbbm{1}\otimes e^{\alpha_2+\alpha_3})\\
&=  \varepsilon_{12}\varepsilon_{13}\varepsilon_{23}Res_{z=0}\Big(\exp\Big(\sum_{k>0} \frac{\alpha_1(-k)}{k} z^k\Big)\Big(I  + h.o.t. \Big) z^{-2}\Big)(\mathbbm{1}\otimes e^{\alpha_1+\alpha_2+\alpha_3})\\
&=  \varepsilon_{12}\varepsilon_{13}\varepsilon_{23}Res_{z=0}\Big(\Big( I + \alpha_1(-1)z + h.o.t. \Big) z^{-2}\Big)(\mathbbm{1}\otimes e^{\alpha_1+\alpha_2+\alpha_3})\\
&=\varepsilon_{12}\varepsilon_{13}\varepsilon_{23} \alpha_1(-1)\mathbbm{1}\otimes e^{\alpha_1+\alpha_2+\alpha_3}
\end{align*}
Since $\varepsilon_{12}=\varepsilon_{23}=\varepsilon_{13}=1$, we have
\begin{equation}\label{eq:useful}\pi_\FF(e_{123})= \alpha_1(-1)\mathbbm{1}\otimes e^{\alpha_1+\alpha_2+\alpha_3}.\end{equation}
Thus we have
%-------------\pi(E_2)---------------%
\begin{align*}
\pi_\FF&(E_1) = \frac{1}{2}e^{\alpha_1}_0\Big(\varepsilon_{12}\varepsilon_{13}\varepsilon_{23}\alpha_1(-1)\mathbbm{1}\otimes e^{\alpha_1+\alpha_2+\alpha_3}\Big)\\
&= \frac{1}{2}\varepsilon_{12}\varepsilon_{13}\varepsilon_{23} Res_{z=0}\Big(\exp\Big(\sum_{k>0} \frac{\alpha_1(-k)}{k} z^k\Big)\exp\Big(\sum_{k>0} \frac{\alpha_1(k)}{-k} z^{-k}\Big)e^{\alpha_1} z^{(\alpha_1,\alpha_1+\alpha_2+\alpha_3)}\varepsilon_{\alpha_1}\Big)\\
& \qquad \qquad \cdot \alpha_1(-1)\mathbbm{1}\otimes e^{\alpha_1+\alpha_2+\alpha_3}\\
&= \frac{1}{2}\varepsilon_{11}(\varepsilon_{12})^2(\varepsilon_{13})^2\varepsilon_{23}Res_{z=0}\Big(\exp\Big(\sum_{k>0} \frac{\alpha_1(-k)}{k} z^k\Big)\Big(I - \alpha_1(1)z^{-1}+h.o.t. \Big) z^0\Big) \\ 
&\qquad \qquad \cdot \alpha_1(-1)\mathbbm{1}\otimes e^{2\alpha_1+\alpha_2+\alpha_3}\\
&= \frac{1}{2}\varepsilon_{11}(\varepsilon_{12})^2(\varepsilon_{13})^2\varepsilon_{23} Res_{z=0}\Big(\Big(I + h.o.t. \Big)\Big(I - 2z^{-1}\Big)\Big) \mathbbm{1}\otimes e^{2\alpha_1+\alpha_2+\alpha_3}\\
&=- \varepsilon_{11}\varepsilon_{23} \mathbbm{1}\otimes e^{2\alpha_1+\alpha_2+\alpha_3}.
\end{align*}
Since $\varepsilon_{11}=-1$ and $\varepsilon_{23}=1$ we have
$$\pi_\FF(e_{1123})= \mathbbm{1}\otimes e^{2\alpha_1+\alpha_2+\alpha_3}=\mathbbm{1}\otimes e^{\beta_1}=\pi_\Fib(E_1),$$
and since $E_2=e_2$ and $\beta_2=\alpha_2,$
$$\pi_\FF(E_2)= \mathbbm{1}\otimes e^{\beta_2}.$$
Also,
%-------------\pi(F_2)---------------%
\begin{align*} 
&\pi_\FF(F_1) = \pi_\FF\Big(-\frac{1}{2}f_{1123}\Big)\\
&= -\frac{1}{2}[\pi_\FF(f_1),[\pi_\FF(f_1), \pi_\FF(f_{23})]]\\
&= -\frac{1}{2} e^{-\alpha_1}_0\Big(Res_{z=0}\Big(\exp\Big(\sum_{k>0} \frac{-\alpha_1(-k)}{k} z^k\Big)\exp\Big(\sum_{k>0} \frac{-\alpha_1(k)}{-k} z^{-k}\Big)e^{-\alpha_1} z^{(-\alpha_1,-\alpha_2-\alpha_3)}\varepsilon_{-\alpha_1}\Big)\\
& \qquad \cdot(\varepsilon_{23}\mathbbm{1}\otimes e^{-\alpha_2-\alpha_3})\\
&= -\frac{1}{2} (\varepsilon_{23})(\varepsilon_{12}\varepsilon_{13})e^{-\alpha_1}_0\Big(Res_{z=0}\Big(\exp\Big(\sum_{k>0} \frac{-\alpha_1(-k)}{k} z^k\Big)\Big(I  + h.o.t. \Big) z^{-2}\Big)(\mathbbm{1}\otimes e^{-\alpha_1-\alpha_2-\alpha_3})\Big)\\
&= -\frac{1}{2} \varepsilon_{23}\varepsilon_{12}\varepsilon_{13}e^{-\alpha_1}_0\Big( Res_{z=0}\Big(\Big( I -\alpha_1(-1)z + h.o.t. \Big) z^{-2}\Big)(\mathbbm{1}\otimes e^{-\alpha_1-\alpha_2-\alpha_3})\\
&= -\frac{1}{2} \varepsilon_{23}\varepsilon_{12}\varepsilon_{13}e^{-\alpha_1}_0\big(- \alpha_1(-1)\mathbbm{1}\otimes e^{-\alpha_1-\alpha_2-\alpha_3}\big)\\
&= \frac{1}{2} \varepsilon_{23}\varepsilon_{12}\varepsilon_{13} Res_{z=0}\Big(\exp\Big(\sum_{k>0} \frac{-\alpha_1(-k)}{k} z^k\Big)\exp\Big(\sum_{k>0} \frac{-\alpha_1(k)}{-k} z^{-k}\Big)e^{-\alpha_1} z^{(-\alpha_1,-\alpha_1-\alpha_2-\alpha_3)}\varepsilon_{-\alpha_1}\Big) \\
& \qquad \cdot ( \alpha_1(-1)\mathbbm{1}\otimes e^{-\alpha_1-\alpha_2-\alpha_3})\\
&= \frac{1}{2} (\varepsilon_{12})^2(\varepsilon_{13})^2\varepsilon_{11}\varepsilon_{23}Res_{z=0}\Big(\exp\Big(\sum_{k>0} \frac{\alpha_1(-k)}{k} z^k\Big)\Big(I + \alpha_1(1)z^{-1}+h.o.t. \Big) z^0\Big) \\
& \qquad \cdot ( \alpha_1(-1)\mathbbm{1}\otimes e^{-2\alpha_1-\alpha_2-\alpha_3})\\
&=\frac{1}{2} (\varepsilon_{12})^2(\varepsilon_{13})^2\varepsilon_{11}\varepsilon_{23} Res_{z=0}\Big(\Big(I + h.o.t. \Big)\Big(I + 2z^{-1}\Big)\Big)( \mathbbm{1}\otimes e^{-2\alpha_1-\alpha_2-\alpha_3})\\
&=  \frac{1}{2} (\varepsilon_{12})^2(\varepsilon_{13})^2\varepsilon_{11}\varepsilon_{23} (2 \mathbbm{1}\otimes e^{-2\alpha_1-\alpha_2-\alpha_3})= \varepsilon_{11}\varepsilon_{23} \mathbbm{1}\otimes e^{-2\alpha_1-\alpha_2-\alpha_3}.
\end{align*}
Thus we have
$$\pi_\FF(F_1)=-\mathbbm{1}\otimes e^{-\beta_1}$$
and since $F_2=f_2$ and $\beta_2=\alpha_2,$
$$\pi_\FF(F_2)=-\mathbbm{1}\otimes e^{-\beta_2}.$$

%\addcontentsline{toc}{chapter}{Appendix C: Weight multiplicities of non-standard $\Fib$-modules}
%--------------------------------------------------------%
%--------------------------------------------------------%
%--------------------- Appendix C---------------------%
%--------------------------------------------------------%
%--------------------------------------------------------%
%\chapter*{Appendix C: Multiplicities of weights in the non-standard $\Fib$-modules}\label{ch:appC}
\chapter{Dimension data for weight spaces in irreducible $\Fib$-modules}\label{appendix:C}
Chapter  \ref{ch:nonstd} details our algorithm for determining the inner multiplicities of non-standard modules on levels $\pm1$ and $\pm2$. For explanation of the notation and organization of the tables, we refer the reader to Remark \ref{table} and the bulleted text on page \pageref{table}.

Tables \ref{tab:tab1} and \ref{tab:tab2} show data for the non-standard modules on levels 1 and 2, respectively. 

Tables \ref{tab:tab3} and \ref{tab:tab4} show data for the adjoint representation and the $-\rho$-module on level 0. Though not referenced within the text, these data are provided to the reader as a demonstration of the claim in Chapter \ref{ch:nonstd} that this algorithmic approach can be used to determine inner multiplicities of any irreducible module, including non-standard ones.

\newpage

%------- LEVEL 1 TABLE-------
%------- LEVEL 1 TABLE-------
%------- LEVEL 1 TABLE-------
%------- LEVEL 1 TABLE-------
%------- LEVEL 1 TABLE-------

\begin{landscape}

%%%
\setcounter{table}{0}
\makeatletter 
\renewcommand{\thetable}{C\arabic{table}}

\begin{table}[!ht]
%\thisfloatpagestyle{empty}
\centering
\caption{Determination of bases of multibrackets in weight spaces $V^{\Lambda_1}_\mu$ where $\mu=n_1\beta_1+n_2\beta_2+\Lambda_1$. \\ See page \pageref{explain} for explanation of notation.}
\label{tab:tab1}
\begin{tabular}{| c | c | c | c | c | c | c |}
\hline

$n_1, n_2$ & $X_\mu$ & & $F_1X_\mu$ & $F_2X_\mu$ & $\mu(H_1)$ & $\mu(H_2)$\\
\hline

\hline

0,0 & $[]$ & & $0$ & N/A & $-2$ & $2$ \\
\hline

\hline

1,0 & $[1]$ & & $(2)_{0,0}$ & $0$ & $0$ & $-1$ \\
\hline

\hline
2,0 & $[11]$ & & $(2)_{1,0}$ & 0 & 2 & $-4$ \\
\hline

\hline
1,1 & $[21]$ & & $0$ & $(1)_{1,0}$ &  $-3$ & $1$ \\
\hline

\hline
\multirow{2}{*}{2,1} 
 & $[211]$ & & $(2)_{1,1}$ & $(4)_{2,0}$ &  \multirow{2}{*}{$-1$} &  \multirow{2}{*}{$-2$} \\
\hhline{~----}
 & $[121]$ & & $(3)_{1,1}$ & $(1)_{2,0}$ &  &  \\
\hline

\hline
\multirow{2}{*}{3,1} 
 & $[1211]$ & & $(1,2)_{2,1}$ & \multirow{2}{*}{$0$} &  \multirow{2}{*}{$1$} &  \multirow{2}{*}{$-5$} \\
\hline

\end{tabular}
\end{table}

%%%%%%%%

\setcounter{table}{0}

\makeatletter 
\renewcommand{\thetable}{C\arabic{table}}

\begin{table}[!ht]
%\thisfloatpagestyle{empty}
\centering
\begin{tabular}{| c | c | c | c | c | c | c |}
\hline

$n_1, n_2$ & $X_\mu$ & & $F_1X_\mu$ & $F_2X_\mu$ & $\mu(H_1)$ & $\mu(H_2)$\\
\hline

\hline
\multirow{2}{*}{2,4}  
 & $[222211]$   & \multirow{2}{*}{*} & \multirow{2}{*}{$0$} & $(4,0)_{2,3}$ & \multirow{2}{*}{$-10$} &  \multirow{2}{*}{$4$} \\
\hhline{~-~~-}
& \pbox{5cm}{$\color{red}[222121]=$ \\ $\color{red}\frac{1}{4}[222211]$}  & &  & $(1,0)_{2,3}$ & &\\
\hline

\hline
\multirow{5}{*}{5,2}   
 & \pbox{5cm}{\color{red}$[1211121]=$ \\ \color{red}$(0,\frac{3}{2},0,-1)$}  & \multirow{5}{*}{*}  & $(0,0,3,0,0)_{4,2}$ & \multirow{5}{*}{$0$} & \multirow{5}{*}{$2$} &  \multirow{5}{*}{$-9$} \\
\hhline{~-~-}
 & $[1121211]$  & & $(0,2,0,1,2)_{4,2}$ &  & &\\
\hhline{~-~-}
 & $[1121121]$  & & $(0,0,2,0,4)_{4,2}$ & & &\\
\hhline{~-~-}
 & $[1112211]$  & & $(0,0,0,6,0)_{4,2}$ &  & & \\
\hhline{~-~-}
  & $[1112121]$  & & $(0,0,0,0,6)_{4,2}$ &  & & \\
\hline

        \end{tabular}
\end{table}
%%%%%%%%

\setcounter{table}{0}
\makeatletter 
\renewcommand{\thetable}{C\arabic{table}}

\begin{table}[!ht]
%\thisfloatpagestyle{empty}
\centering
\begin{tabular}{| c | c | c | c | c | c | c |}
\hline

$n_1, n_2$ & $X_\mu$ & & $F_1X_\mu$ & $F_2X_\mu$ & $\mu(H_1)$ & $\mu(H_2)$\\
\hline

\hline
\multirow{15}{*}{5,3}   
  & $[21121211]$  & & $(0,2,0,1,2,0,0,0,0,0,0)_{4,3}$ & $(9,0,0,0)_{5,2}$  & &\\
\hhline{~----}
& $[21121121]$  &   & $(0,0,2,0,4,0,0,0,0,0,0)_{4,3}$ & $(0,9,0,0)_{5,2}$ & \multirow{15}{*}{$-1$} &  \multirow{15}{*}{$-7$} \\
\hhline{~----}
 & $[21112211]$  & & $(0,0,0,6,0,0,0,0,0,0,0)_{4,3}$ &  $(0,0,9,0)_{5,2}$& & \\
\hhline{~----}
  & $[21112121]$  & & $(0,0,0,0,6,0,0,0,0,0,0)_{4,3}$ & $(0,0,0,9)_{5,2}$ & & \\
\hhline{~----}
 & $[12211121]$  &  & $(3,0,0,0,0,0,3,0,0,0,0)_{4,3}$ & $(\color{red}14\color{black},0,0,0,0)_{4,2}^1=(0,\color{blue}21\color{black},0,\color{blue}-14\color{black})_{5,2}$ & & \\
\hhline{~----}
 & $[12121211]$  & &  $(0,3,0,0,0,2,0,1,2,0,0)_{4,3}$ & $(\color{magenta}\frac{10}{3}\color{black},6,0,0,0)_{4,2}^1=(6,\color{blue}5\color{black},0,\color{blue}-\frac{10}{3}\color{black})_{5,2}$  & &\\
\hhline{~----}
 & $[12121121]$ & & $(0,0,3,0,0,0,2,0,4,0,0)_{4,3}$ & $(\color{magenta}5\color{black},0,6,0,0)_{4,2}^1=(0,\color{blue}\frac{27}{2}\color{black},0,\color{blue}-5\color{black})_{5,2}$ & &\\
\hline

\end{tabular}
\end{table}

%%%%%%%%

\setcounter{table}{0}
\makeatletter 
\renewcommand{\thetable}{C\arabic{table}}

\begin{table}[!ht]
%\thisfloatpagestyle{empty} 
\centering
\begin{tabular}{| c | c | c | c | c | c | c |}
\hline

$n_1, n_2$ & $X_\mu$ & & $F_1X_\mu$ & $F_2X_\mu$ & $\mu(H_1)$ & $\mu(H_2)$\\
\hline

\hline

\multirow{6}{*}{3,5}   
& $[22221211]$  & \multirow{6}{*}{*} &  \multirow{6}{*}{$0$} & $(8,0,0,0,0,0)_{3,4}$ & \multirow{6}{*}{$-11$} &  \multirow{6}{*}{$3$} \\
\hhline{~-~~-}
 & $[22221121]$ & &  & $(0,8,0,0,0,0)_{3,4}$ & &\\
\hhline{~-~~-}
 & $[22212211]$  & &  & $(6,0,3,0,0,0)_{3,4}$ & & \\
\hhline{~-~~-}
 & $[22212121]$  & &  &  $(1,2,0,3,0,0)_{3,4}$ & & \\
\hhline{~-~~-}
  & \pbox{4.5cm}{\color{red}$[22122211]$\\ $= (-\frac{3}{2},0,2,0)$ } & &  &  $(0,0,6,0,0,0)_{3,4}$ & & \\
\hhline{~-~~-}
 &  \pbox{4.5cm}{\color{red}$[22122121]$\\ $= (-\frac{1}{3},-\frac{1}{6},\frac{1}{3},\frac{2}{3})$ } & &  &  $(0,0,1,2,0,0)_{3,4}$ & & \\
\hline

\hline
\multirow{4}{*}{3,6}   
& $[222221211]$  & \multirow{4}{*}{*} &  \multirow{4}{*}{$0$} & $(8,0,0,0)_{3,5}$  & \multirow{4}{*}{$-14$} &  \multirow{4}{*}{$5$} \\
\hline

\end{tabular}

\end{table}

%%%%%%%%%%%

\setcounter{table}{0}
\makeatletter 
\renewcommand{\thetable}{C\arabic{table}}

\begin{table}[!ht]
%\thisfloatpagestyle{empty}
 \centering
\begin{tabular}{| c | c | c | c | c | c | c |}
\hline

$n_1, n_2$ & $X_\mu$ & & $F_1X_\mu$ & $F_2X_\mu$ & $\mu(H_1)$ & $\mu(H_2)$\\
\hline

\hline
\multirow{17}{*}{4,5}   
 & $\scriptstyle[222211121]$  &  \multirow{21}{*}{*} & $\scriptstyle(0,3,0,0,\color{red}0,0\color{black})^2_{3,4}=(0,3,0,0)_{3,5}$ & $\scriptstyle(20,0,0,0,0,0,0,0,0,0,0,0,0,0,0,0,0)_{4,4}$ & \multirow{21}{*}{$-9$} &  \multirow{21}{*}{$0$} \\
\hhline{~-~--}
 & $\scriptstyle[222121211]$  & & $\scriptstyle(2,0,1,2,\color{red}0,0\color{black})^2_{3,4}=(2,0,1,2)_{3,5}$ & $\scriptstyle(\color{blue}\frac{10}{3}\color{black},12,0,0,0,0,0,0,0,0,0,0,0,0,0,0,0)_{4,4}$  & &\\
\hhline{~-~--}
 & $\scriptstyle[222121121]$ & & $\scriptstyle(0,2,0,4,\color{red}0,0\color{black})^2_{3,4}=(0,2,0,4)_{3,5}$ & $\scriptstyle(\color{blue}5\color{black},0,12,0,0,0,0,0,0,0,0,0,0,0,0,0,0)_{4,4}$ & &\\
\hhline{~-~--}
 & $\scriptstyle[222112211]$  & & $\scriptstyle(0,0,6,0,\color{red}0,0\color{black})^2_{3,4}=(0,0,6,0)_{3,5}$ & $\scriptstyle(\color{blue}4\color{black},0,0,12,0,0,0,0,0,0,0,0,0,0,0,0,0)_{4,4}$ & & \\
\hhline{~-~--}
  & $\scriptstyle[222112121]$  & & $\scriptstyle(0,0,0,6,\color{red}0,0\color{black})^2_{3,4}=(0,0,0,6)_{3,5}$ & $\scriptstyle(\color{blue}\frac{8}{3}\color{black},0,0,0,12,0,0,0,0,0,0,0,0,0,0,0,0)_{4,4}$ & & \\
\hhline{~-~--}
   & $\scriptstyle[221221211]$  & & $\scriptstyle(5,0,0,0,\color{red}1,2\color{black})^2_{3,4}=(\color{blue}\frac{17}{6},-\frac{1}{3},\frac{8}{3},\frac{4}{3}\color{black})_{3,5}$ & $\scriptstyle(0,8,0,0,0,6,0,0,0,0,0,0,0,0,0,0,0)_{4,4}$ & & \\
\hhline{~-~--}
  & $\scriptstyle[221221121]$  & & $\scriptstyle(0,5,0,0,\color{red}0,4\color{black})^2_{3,4}=(\color{blue}-\frac{4}{3},\frac{13}{3},\frac{4}{3},\frac{8}{3}\color{black})_{3,5}$ & $\scriptstyle(0,0,8,0,0,0,6,0,0,0,0,0,0,0,0,0,0)_{4,4}$ & & \\
\hline

\end{tabular}

\end{table}

%%% 5,4 (continued across page)

\setcounter{table}{0}
\makeatletter 
\renewcommand{\thetable}{C\arabic{table}}

\begin{table}[!ht]
%\thisfloatpagestyle{empty} 
\centering
\begin{tabular}{| c | c | c | c | c | c | c |}
\hline

$n_1, n_2$ & $X_\mu$ & & $F_1X_\mu$ & $F_2X_\mu$ & $\mu(H_1)$ & $\mu(H_2)$\\
\hline

\hline
\multirow{22}{*}{5,4 (continued)}   
  & $[211221121]$  & & $\scriptstyle(0,0,0,0,0,0,8,0,0,0,4,0,0,0,0,0,0)_{4,4}$  & $\scriptstyle(0,8,0,0,0,0,0,0,0,0,7,0,0,0,0)_{5,3}$ & \multirow{22}{*}{$-4$} &  \multirow{22}{*}{$-5$} \\

  \hhline{~----}
   & $[211212211]$  & & $\scriptstyle(0,0,0,0,0,0,0,8,0,4,0,0,0,0,0,0,0)_{4,4}$   & $\scriptstyle(6,0,3,0,0,0,0,0,0,0,0,7,0,0,0)_{5,3}$ & & \\
\hhline{~----}
  & $[211212121]$  & & $\scriptstyle(0,0,0,0,0,0,0,0,8,0,4,0,0,0,0,0,0)_{4,4}$  &  $\scriptstyle(1,2,0,3,0,0,0,0,0,0,0,0,7,0,0)_{5,3}$ & & \\
  \hhline{~----}
   & $[211122211]$  & & $\scriptstyle(0,0,0,0,0,0,0,0,0,15,0,0,0,0,0,0,0)_{4,4}$  & $\scriptstyle(0,0,6,0,0,0,0,0,0,0,0,0,0,7,0)_{5,3}$ & & \\
\hhline{~----}
  & $[211122121]$  & & $\scriptstyle(0,0,0,0,0,0,0,0,0,0,15,0,0,0,0,0,0)_{4,4}$  & $\scriptstyle(0,0,1,2,0,0,0,0,0,0,0,0,0,0,7)_{5,3}$ & & \\
\hhline{~----}
 & $[122211121]$  &  & $\scriptstyle(6,0,0,0,0,0,0,0,0,0,0,0,3,0,0,0,0)_{4,4}$ & $\scriptstyle(0,0,0,0,18,0,0,0,0,0,0,0,0,0,0)_{5,3}$ &  &   \\
\hhline{~----}
 & $[122121211]$  & & $\scriptstyle(0,6,0,0,0,0,0,0,0,0,0,2,0,1,2,0,0)_{4,4}$ & $\scriptstyle(0,0,0,0,\color{blue}\frac{10}{3}\color{black},10,0,0,0,0,0,0,0,0,0)_{5,3}$  & &\\
\hline

\end{tabular}

\end{table}

%------- LEVEL 2 TABLE-------
\setcounter{table}{1}
\makeatletter 
\renewcommand{\thetable}{C\arabic{table}}

\begin{table}[h!]
%\thisfloatpagestyle{empty} 
\centering
\caption{Determination of bases of multibrackets in weight spaces $V^{\Lambda_2}_\mu$ where $\beta=n_1\beta_1+n_2\beta_2+\Lambda_2$. \\ See page \pageref{explain} for explanation of notation.}
\label{tab:tab2}
\begin{tabular}{| c | c | c | c | c | c | c |}
\hline

$n_1, n_2$ & $X_\mu$ & & $F_1X_\mu$ & $F_2X_\mu$ & $\mu(H_1)$ & $\mu(H_2)$\\
\hline

\hline
0,1 & $[2]$ & & $0$ & $(1)_{0,0}$  & $-2$ & $1$ \\
\hline

\hline
1,1 & $[12]$ & &  $(2)_{0,1}$ & $0$ &  $0$ & $-2$ \\
\hline

\hline
2,1  & $[112]$ & & $(2)_{1,1}$ & $0$  &   $2$ &  $-5$ \\
\hline

\hline
1,2  & $[212]$ & & $0$ & $(2)_{1,1}$ &  $-3$ &  $0$ \\
\hline

\hline
\multirow{2}{*}{2,2} 
 & $[2112]$ & & $(2)_{1,2}$ & $(5)_{2,1}$  &  \multirow{2}{*}{$-1$} &  \multirow{2}{*}{$-3$} \\
\hhline{~----}
 & $[1212]$ & & $(3)_{1,2}$ & $(2)_{2,1}$  &  &  \\
\hline

\hline
1,3 & $[2212]$ & & $0$  & $(2)_{1,2}$  &  $-6$ & $2$ \\
\hline

\end{tabular}
\end{table}

%%%%%%

\setcounter{table}{0}
\makeatletter 
\renewcommand{\thetable}{C\arabic{table}}

\begin{table}[h!]
%\thisfloatpagestyle{empty}
 \centering
\begin{tabular}{| c | c | c | c | c | c | c |}
\hline

$n_1, n_2$ & $X_\mu$ & & $F_1X_\mu$ & $F_2X_\mu$ & $\mu(H_1)$ & $\mu(H_2)$\\
\hline

\hline
\multirow{8}{*}{3,4}   
 & $[2212112]$  & & $(1,2,0)_{2,4}$ & $(10,0,0,0,0)_{3,3}$ & \multirow{8}{*}{$-5$} &  \multirow{8}{*}{$-2$} \\
\hhline{~----}
 & $[2211212]$  & & $(0,4,0)_{2,4}$ & $(0,10,0,0,0)_{3,3}$  & &\\
\hhline{~----}
 & $[2122112]$  & & $(4,0,2)_{2,4}$ & $(8,0,4,0,0)_{3,3}$  & &\\
\hhline{~----}
 & $[2121212]$  & & $(0,4,3)_{2,4}$ & $(2,3,0,4,0)_{3,3}$  & &\\
 \hhline{~----}
 & $[2112212]$  & & $(0,0,10)_{2,4}$ & $(0,2,0,0,4)_{3,3}$  & &\\
 \hhline{~----}
 & $[1222112]$  & & $(7,0,0)_{2,4}$ & $(0,0,9,0,0)_{3,3}$  & &\\
 \hhline{~----}
 & $[1221212]$  & & $(0,7,0)_{2,4}$ & $(0,0,2,4,0)_{3,3}$  & &\\
\hline
\end{tabular}
\end{table}

%------------------------
\setcounter{table}{0}
\makeatletter 
\renewcommand{\thetable}{C\arabic{table}}

\begin{table}[h!]
%\thisfloatpagestyle{empty}
 \centering
\begin{tabular}{| c | c | c | c | c | c | c |}
\hline

$n_1, n_2$ & $X_\mu$ & & $F_1X_\mu$ & $F_2X_\mu$ & $\mu(H_1)$ & $\mu(H_2)$\\
\hline

\hline
\multirow{10}{*}{3,5}   
 & $[22212112]$  & \multirow{10}{*}{*}  & $(1,2)_{2,5}$ & $(12,0,0,0,0,0,0,0)_{3,4}$ & \multirow{10}{*}{$-8$} &  \multirow{10}{*}{$0$} \\
\hhline{~-~--}
 & $[22211212]$  & & $(0,4)_{2,5}$ & $(0,12,0,0,0,0,0,0)_{3,4}$  & &\\
\hhline{~-~--}
 & $[22122112]$  & & $(\frac{11}{3},\frac{4}{3})_{2,5}$ & $(8,0,6,0,0,0,0,0)_{3,4}$  & &\\
\hhline{~-~--}
 & $[22121212]$  & & $(-\frac{1}{2},6)_{2,5}$ & $(2,3,0,6,0,0,0,0)_{3,4}$  & &\\
\hhline{~-~--}
 & $[22112212]$  & & $(-\frac{5}{3},\frac{20}{3})_{2,5}$ & $(0,2,0,0,6,0,0,0)_{3,4}$  & &\\
\hhline{~-~--}
 & $[21222112]$  & & $(7,0)_{2,5}$ & $(0,0,9,0,0,2,0,0)_{3,4}$  & &\\
\hhline{~-~--}
 & $[21221212]$  & & $(0,7)_{2,5}$ & $(0,0,2,4,0,0,2,0)_{3,4}$  & &\\
\hline
\end{tabular}
\end{table}

%------------------------Continued across page
\setcounter{table}{0}
\makeatletter 
\renewcommand{\thetable}{C\arabic{table}}

\begin{table}[h!]
%\thisfloatpagestyle{empty} 
\centering
\begin{tabular}{| c | c | c | c | c | c | c |}
\hline

$n_1, n_2$ & $X_\mu$ & & $F_1X_\mu$ & $F_2X_\mu$ & $\mu(H_1)$ & $\mu(H_2)$\\
\hline

\hline
\multirow{7}{*}{5,4 (continued)}  & $[112211212]$  &  \multirow{7}{*}{*} & $\scriptstyle(0,0,0,0,0,0,0,8,0,0,0,0,4,0)_{4,4}$ & $(0,10,0,0,0)_{5,3}$ & \multirow{7}{*}{-1}  & \multirow{7}{*}{-8}  \\
\hhline{~-~--}
 & $[112122112]$  & & $\scriptstyle(0,0,0,0,0,0,0,0,8,0,0,4,0,2)_{4,4}$ & $(8,0,4,0,0)_{5,3}$  & &\\
\hhline{~-~--}
 & $[112121212]$  & & $\scriptstyle(0,0,0,0,0,0,0,0,0,8,0,0,4,3)_{4,4}$ & $(2,3,0,4,0)_{5,3}$  & &\\
\hhline{~-~--}
 & $[112112212]$  & & $\scriptstyle(0,0,0,0,0,0,0,0,0,0,8,0,0,10)_{4,4}$ & $(0,2,0,0,4)_{5,3}$  & &\\
\hhline{~-~--}
 & $[111222112]$  & & $\scriptstyle(0,0,0,0,0,0,0,0,0,0,0,15,0,0)_{4,4}$ & $(0,0,9,0,0)_{5,3}$  & &\\
\hhline{~-~--}
 & \pbox{4.5cm}{\color{red}$[111221212]=$ \\$\scriptstyle(0,0,0,0,\frac{1}{4},0,0,0,0,0,-1,0,0,0,0,\frac{3}{2},0,0)$} & & $\scriptstyle(0,0,0,0,0,0,0,0,0,0,0,0,15,0)_{4,4}$ & $(0,0,2,4,0)_{5,3}$  & &\\
\hhline{~-~--}
 & $[111212212]$  & & $\scriptstyle(0,0,0,0,0,0,0,0,0,0,0,0,0,15)_{4,4}$ & $(0,0,0,2,1)_{5,3}$  & &\\
\hline
\end{tabular}

\end{table}

%-----------
\setcounter{table}{0}
\makeatletter 
\renewcommand{\thetable}{C\arabic{table}}

\begin{table}[h!]
%\thisfloatpagestyle{empty} 
\centering
\begin{tabular}{| c | c | c | c | c | c | c |}
\hline

$n_1, n_2$ & $X_\mu$ & & $F_1X_\mu$ & $F_2X_\mu$ & $\mu(H_1)$ & $\mu(H_2)$\\
\hline

\hline
\multirow{10}{*}{3,6}   
 & $[222212112]$  & \multirow{10}{*}{*}  & $(\frac{9}{5})_{2,6}$ & $(12,0,0,0,0,0,0,0,0)_{3,5}$ & \multirow{10}{*}{$-11$} &  \multirow{10}{*}{$2$} \\
\hhline{~-~--}
 & $[222211212]$  & & $(\frac{8}{5})_{2,6}$ & $(0,12,0,0,0,0,0,0,0)_{3,5}$  & &\\
\hhline{~-~--}
 & $[222122112]$  & & $(\frac{21}{5})_{2,6}$ & $(8,0,6,0,0,0,0,0,0)_{3,5}$  & &\\
\hhline{~-~--}
 & $[222121212]$  & & $(-\frac{19}{10},)_{2,6}$ & $(2,3,0,6,0,0,0,0,0)_{3,5}$  & &\\
\hhline{~-~--}
 & $[222112212]$  & & $(1)_{2,6}$ & $(0,2,0,0,6,0,0,0,0)_{3,5}$  & &\\
\hhline{~-~--}
 & $[221222112]$  & & $(7)_{2,6}$ & $(0,0,9,0,0,2,0,0,0)_{3,5}$  & &\\
\hhline{~-~--}
 & $[221221212]$  & & $(\frac{14}{5})_{2,6}$ & $(0,0,2,4,0,0,2,0,0)_{3,5}$  & &\\
\hhline{~-~--}
 & $[221212212]$  & & $(-\frac{7}{10})_{2,6}$ & $(0,0,0,2,1,0,0,2,0)_{3,5}$  & &\\
\hhline{~-~--}
 &  \color{red}$[212221212]=(\frac{3}{2},\frac{1}{4},-2,-1,0,1,\frac{3}{2},0)$  & & $(4)_{2,6}$ & $(0,0,0,0,0,2,3,0,0)_{3,5}$  & & \\
\hhline{~-~--}
 &   \color{red}$[122222112]=(15,0,-20,0,0,10,0,0)$  & &  $(13)_{2,6}$  & $(20,0,-30,0,0,20,0,0,0)_{3,5}$  & & \\
 \hline

\cline{1-7}
\multicolumn{7}{|r|}{\textbf{Continued from previous page}}  \\
  \hline
\end{tabular}
\end{table}

%----------------------
\setcounter{table}{2}

\makeatletter 
\renewcommand{\thetable}{C\arabic{table}}

\begin{table}[h!]
%\thisfloatpagestyle{empty} 
\centering
\caption{Determination of bases of multibrackets in the adjoint representation $\Fib$. Note: $\beta=n_1\beta_1+n_2\beta_2$. \\ See page \pageref{explain} for explanation of notation.}
\label{tab:tab3}
\begin{tabular}{| c | c | c | c | c | c | c |}
\hline

$n_1, n_2$ & $X_\beta$ & & $F_1X_\beta$ & $F_2X_\beta$ & $\beta(H_1)$ & $\beta(H_2)$\\
\hline

\hline
$n_1, n_2$ & $X_\beta$ & & $F_1X_\beta$ & $F_2X_\beta$ & $\beta(H_1)$ & $\beta(H_2)$\\
\hline

\hline

\hline
\multirow{2}{*}{1,1} 
 & $[21]$ & & $(-3)_{0,1}$ & $(3)_{1,0}$  &  \multirow{2}{*}{$-2$} &  \multirow{2}{*}{$-2$} \\
\hhline{~----}
 & \color{red}$[12]=-[21]$ & & $(3)_{0,1}$ & $(-3)_{1,0}$  &  &  \\
\hline

\hline
2,1  & $[121]$ & & $(4)_{1,1}$ & $0$  &   $1$ &  $-4$ \\
\hline

\hline
1,2  & $[221]$ & & $0$ & $(4)_{1,1}$ &  $-4$ &  $1$ \\
\hline

\hline
3,1  & $[1121]$ & & $(3)_{2,1}$ & $0$ &  $3$ &  $-7$ \\
\hline

\hline
\multirow{2}{*}{2,2} 
 & $[2121]$ & & $(4)_{1,2}$ & $(4)_{2,1}$  &  \multirow{2}{*}{$-2$} &  \multirow{2}{*}{$-2$} \\
\hline

\end{tabular}
\end{table}

%---TABLE 3 - FIB LEVEL 0 (ADJOINT)
\setcounter{table}{0}

\makeatletter 
\renewcommand{\thetable}{C\arabic{table}}

\begin{table}[h!]
%\thisfloatpagestyle{empty} 
\centering
\begin{tabular}{| c | c | c | c | c | c | c |}
\hline

$n_1, n_2$ & $X_\beta$ & & $F_1X_\beta$ & $F_2X_\beta$ & $\beta(H_1)$ & $\beta(H_2)$\\
\hline

\hline
\multirow{5}{*}{5,3} 
 &\color{red}$[21112121]=(4,-\frac{4}{3},-2,\frac{1}{3})$ & & $(4,0,0,0)_{4,3}$ & $(11)_{5,2}$ &  \multirow{5}{*}{$1$} &  \multirow{5}{*}{$-9$} \\
\hhline{~----}
 & $[12112121]$ & & $(1,0,6,0)_{4,3}$ & $(8)_{5,2}$ &  & \\
 \hhline{~----}
 & $[11221121]$ & & $(0,3,3,1)_{4,3}$ & $(6)_{5,2}$  & &   \\
\hhline{~----}
 & $[11212121]$ & & $(0,-2,10,2)_{4,3}$ & $(7)_{5,2}$  &  &  \\
\hhline{~----}
 & $[11112221]$ & & $(0,0,0,16)_{4,3}$ & $(3)_{5,2}$  &  &  \\
\hline

\hline
\multirow{8}{*}{4,4} 
 & $[22112121]$ & & $(0,6,0,0)_{3,4}$ & $(14,0,0,0)_{4,3}$ &  \multirow{8}{*}{$-4$} &  \multirow{8}{*}{$-4$} \\
\hhline{~----}
 & $[21221121]$ & & $(2,3,1,0)_{3,4}$ & $(6,6,0,0)_{4,3}$  & &   \\
\hline
\end{tabular}
\end{table}

%---FIB LEVEL 0 - adjoint continued
\setcounter{table}{0}
\makeatletter 
\renewcommand{\thetable}{C\arabic{table}}

\begin{table}[h!]
%\thisfloatpagestyle{empty} 
\centering
\begin{tabular}{| c | c | c | c | c | c | c |}
\hline

$n_1, n_2$ & $X_\beta$ & & $F_1X_\beta$ & $F_2X_\beta$ & $\beta(H_1)$ & $\beta(H_2)$\\
\hline
\multirow{10}{*}{4,5} 
 & \color{red}$\scriptstyle[222112121]=(3,0,0,-2,0,\frac{1}{2},0,0,0)$ & & $(0,6,0,0)_{3,5}$ & $(18,0,0,0,0,0)_{4,4}$ &  \multirow{10}{*}{$-7$} &  \multirow{10}{*}{$-2$} \\
\hhline{~----}
 & $[221221121]$ & & $(2,3,1,0)_{3,5}$ & $(6,10,0,0,0,0)_{4,4}$  & &   \\
\hhline{~----}
 & $[221212121]$ & & $(-2,9,2,0)_{3,5}$ & $(7,0,10,0,0,0)_{4,4}$  &  &  \\
\hhline{~----}
 & $[221112221]$ & & $(0,0,15,0)_{3,5}$ & $(3,0,0,10,0,0)_{4,4}$  &  &  \\
\hhline{~----}
 & $[212221121]$ & & $(6,0,0,3)_{3,5}$ & $(0,15,0,0,4,0)_{4,4}$ &  & \\
\hhline{~----}
 & $[212212121]$ & & $(0,6,0,6)_{3,5}$ & $(0,4,8,0,0,4)_{4,4}$  & &   \\
\hhline{~----}
 & $[122221121]$ & & $(12,-6,-6,12)_{3,5}$ & $(0,0,0,0,16,0)_{4,4}$ &  &  \\
\hline
\end{tabular}
\end{table}
%%%

%---FIB LEVEL 0 - -rho module
\setcounter{table}{3}

\makeatletter 
\renewcommand{\thetable}{C\arabic{table}}

\begin{table}[h!]
%\thisfloatpagestyle{empty} 
\centering
\begin{center}
\caption{Determination of bases of multibrackets in the module $V^{-\rho}$ on level 0.  \\ Note: $\beta=n_1\beta_1+n_2\beta_2-\rho$. See page \pageref{explain} for explanation of notation.}
\label{tab:tab4}
\begin{tabular}{| c | c | c | c | c | c | c |}
\hline

$n_1, n_2$ & $X_\beta$ & & $F_1X_\beta$ & $F_2X_\beta$ & $\beta(H_1)$ & $\beta(H_2)$\\
\hline

\hline
1,0  & $[1]$ & & $(1)_{0,0}$ & $0$  &   $1$ &  $-4$ \\
\hline

\hline
0,1 & $[2]$ & & 0 & $(1)_{0,0}$ &   $-4$ &  $1$ \\
\hline

\hline
\multirow{2}{*}{1,1} 
 & $[21]$ & & $(1)_{0,1}$ & $(4)_{1,0}$  &  \multirow{2}{*}{$-2$} &  \multirow{2}{*}{$-2$} \\
\hhline{~----}
 & $[12]$ & & $(4)_{0,1}$ & $(1)_{1,0}$  &  &  \\
\hline

\hline
\multirow{2}{*}{2,1} 
 & $[121]$ & & $(2,1)_{1,1}$ &  \multirow{2}{*}{$0$} &  \multirow{2}{*}{$0$} &  \multirow{2}{*}{$-5$} \\
\hhline{~---}
 & $[112]$ & & $(0,6)_{1,1}$ &  &  &  \\
\hline

\hline
\multirow{2}{*}{1,2} 
 & $[221]$ & &  \multirow{2}{*}{$0$} & $(6,0)_{1,1}$  &  \multirow{2}{*}{$-5$} &  \multirow{2}{*}{$0$} \\
\hline

\end{tabular}
\end{center}
\end{table}

%---FIB LEVEL 0 - -rho module
\setcounter{table}{0}

\makeatletter 
\renewcommand{\thetable}{C\arabic{table}}

\begin{table}[h!]
%\thisfloatpagestyle{empty}
 \centering
\begin{tabular}{| c | c | c | c | c | c | c |}
\hline

$n_1, n_2$ & $X_\beta$ & & $F_1X_\beta$ & $F_2X_\beta$ & $\beta(H_1)$ & $\beta(H_2)$\\
\hline

\hline
\multirow{6}{*}{2,3} 
 & $[22121]$ & & $(2,1)_{1,3}$ & $(8,0,0,0)_{2,2}$ &  \multirow{6}{*}{$-6$} &  \multirow{6}{*}{$-1$} \\
\hhline{~----}
 & $[22112]$ & & $(0,6)_{1,3}$ & $(0,8,0,0)_{2,2}$ &  &  \\
\hhline{~----}
 & $[21221]$ & &  $(5,0)_{1,3}$ & $(6,0,3,0)_{2,2}$  &  &  \\
\hhline{~----}
 & $[21212]$ & & $(0,5)_{1,3}$ & $(1,2,0,3)_{2,2}$  &  &  \\
\hhline{~----}
 & $[12221]$ & & $(8,0)_{1,3}$ & $(0,0,6,0)_{2,2}$  & &  \\
\hhline{~--~-}
 & $[12212]$ & & $(0,8)_{1,3}$ & $(0,0,1,2)_{2,2}$  &  &  \\
\hline

\hline
\multirow{2}{*}{1,4} 
 &  $[22221]$ & &  \multirow{2}{*}{$0$} & $(4,0)_{1,2}$  &  \multirow{2}{*}{$-11$} &  \multirow{2}{*}{$4$} \\
\hline
\end{tabular}
\end{table}

%---FIB LEVEL 0 - -rho module

\setcounter{table}{0}

\makeatletter 
\renewcommand{\thetable}{C\arabic{table}}

\begin{table}[h!]
%\thisfloatpagestyle{empty} 
\centering
\begin{tabular}{| c | c | c | c | c | c | c |}
\hline

$n_1, n_2$ & $X_\beta$ & & $F_1X_\beta$ & $F_2X_\beta$ & $\beta(H_1)$ & $\beta(H_2)$\\
\hline

\hline
\multirow{6}{*}{5,2} 
 & \color{red}$[1121121]=(\frac{2}{3},\frac{1}{3},\frac{-1}{6},\frac{-1}{3})$ & & $(0,0,2,1,0,0)_{3,2}$ &  \multirow{6}{*}{$0$} &  \multirow{6}{*}{$3$} &  \multirow{6}{*}{$-12$} \\
\hhline{~---}
 & \color{red}$[1121112]=(0,\frac{2}{3},0,\frac{-4}{3})$ & & $(0,0,0,6,0,0)_{3,2}$ &  &  &  \\
\hhline{~---}
 & $[1112121]$ & & $(0,0,3,0,2,1)_{3,2}$ &  &  &  \\
\hhline{~---}
 & $[1112112]$ & & $(0,0,0,3,0,6)_{3,2}$ &  &  &  \\
\hhline{~---}
 & $[1111221]$ & &  $(0,0,0,0,8,0)_{3,2}$ &  &  &  \\
\hhline{~---}
 &  $[1111212]$ & & $(0,0,0,0,0,8)_{3,2}$ &  &  &  \\
\hline

\hline
\multirow{18}{*}{4,3} 
 & $[2121121]$ & & $(1,0,2,1,0,0,0,0,0,0,0,0)_{3,3}$ & $(9,8,0,-12,0,8)_{4,2}$  & \multirow{18}{*}{$-2$} &  \multirow{18}{*}{$-7$} \\
\hline
\end{tabular}
\end{table}

%---FIB LEVEL 0 - -rho module

\setcounter{table}{0}

\makeatletter 
\renewcommand{\thetable}{C\arabic{table}}

\begin{table}[h!]
%\thisfloatpagestyle{empty} 
\centering
\begin{tabular}{| c | c | c | c | c | c | c |}
\hline

$n_1, n_2$ & $X_\beta$ & & $F_1X_\beta$ & $F_2X_\beta$ & $\beta(H_1)$ & $\beta(H_2)$\\
\hline

 & $[12111212]$ & \multirow{13}{*}{5,3} & $\scriptstyle(0,0,0,0,0,2,0,0,0,0,0,9,0,0,0,0,0,0)_{4,3}$ & $(0,\frac{9}{2},0,\frac{9}{2})_{5,2}$   &  \multirow{13}{*}{$0$} &  \multirow{13}{*}{$-10$} \\
\hhline{~----}
 & $[11221121]$ & & $\scriptstyle(0,0,0,0,0,0,6,0,0,0,0,0,2,1,0,0,0,0)_{4,3}$ & $(\frac{28}{3},\frac{14}{3},\frac{-7}{3},\frac{-14}{3})_{5,2}$  & & \\
 \hhline{~----}
 & $[11221112]$ & & $\scriptstyle(0,0,0,0,0,0,0,6,0,0,0,0,0,6,0,0,0,0)_{4,3}$ & $(0,28,0,-21)_{5,2}$ &  &  \\
\hhline{~----}
 & $[11212121]$ & & $\scriptstyle(0,0,0,0,0,0,0,0,6,0,0,0,3,0,2,1,0,0)_{4,3}$ &  $(\frac{28}{3},\frac{5}{3},\frac{-5}{6},\frac{-5}{3})_{5,2}$  & & \\
\hhline{~----}
 & $[11212112]$ & & $\scriptstyle(0,0,0,0,0,0,0,0,0,6,0,0,0,3,0,6,0,0)_{4,3}$ & $(0,16,0,\frac{-15}{2})_{5,2}$ &  &  \\
\hhline{~----}
 & $[11211221]$ & & $\scriptstyle(0,0,0,0,0,0,0,0,0,0,6,0,0,0,8,0,0,0)_{4,3}$ & $(4,2,5,-2)_{5,2}$  &  &  \\
\hhline{~----}
 & $[11211212]$ & & $\scriptstyle(0,0,0,0,0,0,0,0,0,0,0,6,0,0,0,8,0,0)_{4,3}$ & $(\frac{2}{3},\frac{13}{3},\frac{-1}{6},\frac{8}{3})_{5,2}$  &  &  \\
\hline
\end{tabular}
\end{table}

\end{landscape}
\end{appendix}
% add a new chapter without a chapter # for the references
\addcontentsline{toc}{chapter}{Bibliography}

% bibs need to be single-spaced

%\fontsize{12}{11pt} \selectfont
%%\setlinespacing{1.0}
% specifies the style for the bibliography and inputs the
% bibtex file references.bib for the bibliography
%\nocite{*}
%%\setlinespacing{1.0}
%\bibliographystyle{plain}
%\bibliography{myref.bib}

\end{document}